\documentclass[sn-mathphys]{sn-jnl}
\usepackage[shortlabels]{enumitem}
\usepackage{graphicx}
\jyear{2021}
\theoremstyle{thmstyleone}
\newtheorem{theorem}{Theorem}
\newtheorem{proposition}[theorem]{Proposition}
\newtheorem{corollary}[theorem]{Corollary}
\newtheorem{lemma}[theorem]{Lemma}
\theoremstyle{thmstyletwo}

\newtheorem{remark}{Remark}
\theoremstyle{thmstylethree}
\newtheorem{definition}{Definition}
\newtheorem{assump}{Assumption}
\newtheorem{hypothesis}{Hypothesis}
\newcommand{\norm}[1]{\lVert#1 \rVert}  
\newcommand{\normb}[1]{\lVert#1 \rVert}

\def\diff{\mathrm{d}}
\def\div{\text{div}}
\def\curl{\text{curl}}
\def\tr{\text{tr}}
\def\exp{\text{exp}}

\def\sgn{\text{sgn}}
\def\llangle{\langle\langle}
\def\rrangle{\rangle\rangle}

\def\D{\mathcal D}
\def\R{\mathbb R}
\def\E{\mathbb E}

\def\X{\mathfrak X}

\begin{document}

\title[Transport noise regularises geometric transport equations]{Transport noise restores uniqueness and prevents blow-up in geometric transport equations}

\author*[1,2]{\fnm{Aythami} \sur{Bethencourt de Le\'{o}n}}\email{ab1113@ic.ac.uk}

\author[1,3]{\fnm{So} \sur{Takao}}

\affil[1]{\orgdiv{Department of Mathematics}, \orgname{Imperial College London}, \orgaddress{\city{London}, \postcode{SW7 2AZ}, \country{UK}}}

\affil[2]{\orgdiv{Department of Mathematics, Statistics and Operations Research}, \orgname{University of La Laguna}, \orgaddress{\city{San Cristóbal de La Laguna}, \postcode{38206}, \country{Spain}}}

\affil[3]{\orgdiv{UCL Centre for Artificial Intelligence}, \orgname{University College London}, \orgaddress{\city{London}, \postcode{WC1V 6BH}, \country{UK}}}

\abstract{In this work, we demonstrate well-posedness and regularisation by noise results for a class of geometric transport equations that contains, among others, the linear transport and continuity equations. This class is known as linear advection of $k$-forms. In particular, we prove global existence and uniqueness of $L^p$-solutions to the stochastic equation, driven by a spatially $\alpha$-H\"older drift $b$, uniformly bounded in time, with an integrability condition on the distributional derivative of $b$, and sufficiently regular diffusion vector fields. Furthermore, we prove that all our solutions are continuous if the initial datum is continuous. Finally, we show that our class of equations without noise admits infinitely many $L^p$-solutions and is hence ill-posed. Moreover, the deterministic solutions can be discontinuous in both time and space independently of the regularity of the initial datum. We also demonstrate that for certain initial data of class $C^\infty_{0},$ the deterministic $L^p$-solutions blow up instantaneously in the space $L^{\infty}_{loc}$. In order to establish our results, we employ characteristics-based techniques that exploit the geometric structure of our equations.}

\keywords{well-posedness by noise, transport noise, stochastic Lie transport equations, blow-up prevention, stochastic flows, geometric mechanics.}

\pacs[MSC Classification]{60H50, 60H15, 35R25, 53Z05.}
\maketitle
\pagestyle{plain}
\tableofcontents
\newpage
\pagestyle{headings}
\section{Introduction}
The aim of this work is to demonstrate that noise has a regularising effect for the geometric class of SPDEs
\begin{align}
\begin{split}
    & \diff \mathbf{K}(t,x) + \mathcal L_{b(t,x)} \mathbf{K}(t,x) \, \diff t + \sum_{i=1}^N \mathcal L_{\xi_i(t,x)} \mathbf{K}(t,x) \circ \diff W_t^i = 0,  \\
    & \mathbf{K}(0,x) = \mathbf{K}_0(x), \quad t \in [0,T], \quad x \in \mathbb R^n, \label{eq0}
\end{split}
\end{align}
where 
\begin{itemize}
    \item $\mathbf{K}(t,x)$ is a stochastic process taking values in the space of differential $k$-forms on $\mathbb R^n,$ for any $k \in \{0,1,\dots,n\}$. 
    \item $b(t,x)$ and $\xi_i(t,x),$ $i=1,\ldots,N,$ are deterministic vector fields on $\mathbb{R}^n$ satisfying some regularity conditions. 
    \item $\mathcal L_{v}$ denotes the Lie derivative \footnote{The Lie derivative generalises the notion of directional derivative to tensor fields. More concretely, $\mathcal L_{v}T$ describes the infinitesimal rate of change of a tensor field $T$ under the flow of $v$ (see Appendix \ref{tensor} for supplementary material on tensor algebra and calculus).} with respect to a vector field $v$. 
    \item $\{W_t^i\}_{i = 1, \ldots, N}$ is a family of independent Brownian motions and the notation $\circ$ indicates stochastic integration in the Stratonovich sense, which in this introduction is understood formally.
    \item $\mathbf{K}_0(x)$ is a deterministic differential $k$-form serving as initial condition.
\end{itemize}
Without the noise, i.e. $\xi_i = 0,$ $i=1,\ldots,N,$ equation \eqref{eq0} is known as the (Lie) transport equation of $k$-forms. Since the Lie derivative $\mathcal L_{b} \mathbf{K}$ can be understood as the infinitesimal rate of change of $\mathbf{K}$ when transported along the flow of the vector field $b$, the PDE $\partial_t \mathbf{K} + \mathcal{L}_b \mathbf{K} = 0,$ $\mathbf{K}(0,\cdot)=\mathbf{K}_0$ describes the advection of the $k$-form $\mathbf{K}_0$ along the flow of $b$ (see Section 6 in \cite{holm1998euler}).
The multiplicative noise term added in \eqref{eq0} is known as ``Lie transport noise" or ``transport noise" (see \cite{crisan2017solution,DieAytJNLS}), which simply perturbs the transport velocity in the deterministic equation by noise. Indeed, equation \eqref{eq0} can formally be expressed as $\partial_t \mathbf{K} + \mathcal L_{\tilde{b}} \mathbf{K} = 0$, where $\tilde{b}(t,x) := b(t,x) + \sum_{i=1}^N \xi_i(t,x) \circ \dot{W}_t^i$.

We prove that under some regularity conditions on the vector fields $b$ and $\xi_i$, $i=1,\ldots,N,$ which we state in more detail later, equation \eqref{eq0} is globally well-posed in $L^p$. Furthermore, we demonstrate that the noise term in \eqref{eq0} has the effect of {\em restoring uniqueness of solutions and preventing the development of instantaneous singularities} in the same equation without the noise. Results of this type are normally referred to as ``well-posedness/regularisation by noise" in the literature, and have attracted significant attention in the last decade. This is due to their relevance in fluid dynamics (indeed, the equations studied are often of fluid dynamics origin) and the prospect that some of the research could help tackle the famous 3D Navier-Stokes existence and smoothness problem. In this work, we establish novel well-posedness and regularisation results, generalising existing ones in the literature and unifying them under a single framework.

\subsection{Background and motivations} Two main motivations led us to consider equation \eqref{eq0}. 
Our first source of inspiration stems from the groundbreaking results in \cite{flandoli2010well}, where the authors prove that noise restores uniqueness of solutions of the stochastic linear transport equation, corresponding to the case $k=0,$ $\xi_i:=e_i,$ $i=1,\ldots,n$ in our equation \eqref{eq0}. The latter work is fundamental as it provides the first example of a PDE that becomes well-posed under the presence of noise.

Our second source of inspiration comes from \cite{holm2015variational}, where the author introduces formal variational principles that yield stochastic fluid equations with Stratonovich Lie transport noise. The resulting equations comprise a coupling between a nonlinear momentum equation and a linear equation of the form \eqref{eq0}, the latter being a natural geometric generalisation of the linear transport equation considered in \cite{flandoli2010well}. It was therefore natural to ask whether the well-posedness by noise results in \cite{flandoli2010well} could be generalised to the stochastic transport equation \eqref{eq0}. 

The well-posedness and regularisation properties that we show for \eqref{eq0} can be motivated by the observation that by (formally) rewriting \eqref{eq0} in It\^o  form and (formally) computing the expectation, we obtain
\begin{align} \label{expectationreg}
\partial_t \mathbb{E}[\mathbf{K}(t,x)] + \mathcal L_b \mathbb{E}[\mathbf{K}(t,x)] = \frac{1}{2} \sum_{i=1}^N \mathcal L^2_{\xi_i} \mathbb{E}[\mathbf{K}(t,x)],
\end{align}
where $\sum_{i=1}^N \mathcal{L}_{\xi_i}^2$ is a dissipative second order operator if some parabolicity conditions are met. We note that if $\xi_i := e_i,$ $i=1,\ldots,n,$ then $\sum_{i=1}^n \mathcal{L}_{\xi_i}^2$ is exactly the Laplacian operator $\Delta.$
Although \eqref{expectationreg} suggests that transport noise could regularise solutions in \eqref{eq0}, at least at the level of expectation, it is not clear whether this regularity would be transferred to pathwise solutions. However, we will prove that this is indeed the case.

\subsection{Physical interpretation} From a physical viewpoint, a motivation for considering equation \eqref{eq0} resides in its importance in stochastic fluid modelling, where
it represents the advection of a locally conserved quantity $\mathbf{K}$ (mathematically modelled as a differential $k$-form) along stochastic paths traversed by fluid particles. Indeed, as demonstrated in \cite{takao2019}, equation \eqref{eq0} is equivalent to the stochastic local conservation law
\begin{align} \label{conservation-law}
    \int_{\Omega_0} \mathbf{K}(0,\cdot) = \int_{\Omega_t} \mathbf{K}(t,\cdot), \quad t \in [0,T],
\end{align}
where $\Omega_0$ is any $k$-dimensional compact, orientable submanifold of $\mathbb R^n,$ and $\Omega_t := \phi_t(\Omega_0)$ is the image of $\Omega_0$ under a sufficiently smooth stochastic flow map $\phi_t$, represented by the equation
\begin{align} \label{flow-map-stratonovich}
    \phi_{s,t}(x) &= x + \int^t_s b(r,\phi_{s,r}(x)) \, \diff r + \sum_{i=1}^N \int^t_s \xi_i(r,\phi_{s,r}(x)) \circ \diff W_r^i,
\end{align}
for $0\leq s \leq t \leq T$ and $ x \in \mathbb{R}^n$, where we employed the notation $\phi_t := \phi_{0,t}$ (i.e. when the initial time is set to zero). Under the modelling philosophy of \cite{holm2015variational}, $\phi_t(x)$ represents the stochastically perturbed Lagrangian trajectory of a fluid particle located initially at the point $x \in \R^n$.

We illustrate the equivalence between \eqref{eq0} and \eqref{conservation-law} with some examples. If $\mathbf{K}(t,x) := \rho(t,x) \diff^n x$ is the mass density of the fluid, where $\diff^n x$ is the standard volume element in $\R^n$ (which is a volume form, i.e. $k=n$), conservation of mass \eqref{conservation-law} implies that $\rho$ satisfies the stochastic continuity equation with transport noise
\begin{align} \label{stoch-continuity}
    \diff \rho(t,x) + \div{(b(t,x)\rho(t,x))} \, \diff t + \sum_{i=1}^N \div{(\xi_i(t,x)\rho(t,x))} \circ \diff W_t^i = 0.
\end{align}
Note that the above equation is exactly \eqref{eq0}, since the Lie derivative of volume forms is given explicitly by $\mathcal L_{v} (\rho(t,x) \diff^n x) = \div{(\rho v)} \diff^n x$. We may also consider the case when $\mathbf{K}(t,x):=u(t,x)$ is a scalar, or a tuple of scalars $u = (u_1, \ldots, u_n)$ (this corresponds to $k=0$), which models physical quantities that make natural sense pointwise, such as the temperature or entropy per unit mass. In this case, the zero-dimensional manifold reduces to a single point $\Omega_0 = \{x\}$ and the conservation law \eqref{conservation-law} applied to each component of $u$ becomes pointwise evaluation $u_i(0, x) = u_i(t, \phi_t(x)),$ $i=1,\ldots,n$. This implies that $u$ satisfies the linear transport equation with transport noise
\begin{align} \label{lin-transport}
    \diff u(t,x) + b(t,x) \cdot Du(t,x) \, \diff t + \sum_{i=1}^N \xi_i(t,x) \cdot Du(t,x) \circ \diff W_t^i = 0,
\end{align}
which coincides with \eqref{eq0} since the Lie derivative of scalar fields is given by $\mathcal L_{v} u = v \cdot D u$. 

Yet another example arises in 3D ideal magnetohydrodynamics, where the advection of a magnetic field, interpreted geometrically as a two-form $B \cdot \diff S$ (i.e. $k=2,$ $n=3$) with $\div{(B)}=0$ (here, $\diff S$ represents the surface element), again obeys the conservation law \eqref{conservation-law} (conservation of magnetic flux) and therefore satisfies equation \eqref{eq0}. In this case, the Lie derivative reads 
\begin{align*}
\mathcal L_v (B \cdot \diff S) = (-\text{curl}(v \times B) + v \div{(B)}) \cdot \diff S = -\text{curl}(v \times B) \cdot \diff S,
\end{align*}
and \eqref{eq0} yields the stochastic vector advection equation with transport noise, studied for instance in \cite{flandoli2014noise}. We refer the readers to \cite{holm1998euler} for more instances of deterministic fluid models with advected quantities, and to \cite{holm2015variational} for their stochastic counterparts.

\subsection{Overview of related results}
The example SPDEs that we listed above have been the subject of several articles on the topic of well-posedness/regularisation by noise.
In the aforementioned work \cite{flandoli2010well}, the authors proved well-posedness by noise for the stochastic linear transport equation \eqref{lin-transport} with $\xi_i:=e_i,$ for $i=1,\ldots,n$, providing the first example of a PDE that becomes well-posed under the presence of noise. In particular, they show using a characteristics-based argument that \eqref{lin-transport} (we have retained the same notation the authors employed) admits a unique global $L^\infty$-solution, strong in the probabilistic sense, under the assumptions that the drift vector field $b$ is bounded, globally $\alpha$-H\"older continuous (uniformly in time, for some $\alpha \in (0,1)$), and $\div{(b)} \in L^q([0,T] \times \mathbb R^n),$ $q>2$, in the sense of distributions. Moreover, there exist {\em classical} examples of vector fields $b$ in the above conditions such that the deterministic transport equation is ill-posed.

In subsequent years following the publication of \cite{flandoli2010well}, effort has been made to weaken the conditions on the drift $b$ in the stochastic linear transport equation \eqref{lin-transport} such as in \cite{mohammed2015sobolev}, where only boundedness and measurability are assumed to construct solutions of Sobolev regularity. Standard PDE arguments (as opposed to characteristics-based ones) to prove the well-posedness by noise have also been developed in \cite{attanasio2011renormalized}, where the concept of renormalised solutions is exploited. Furthermore, in \cite{Flandoli2013}, the transport noise has been shown to prevent the formation of singularities for a Sobolev-type class of initial conditions and drift $b$ satisfying the so-called Ladyzhenskaya-Prodi-Serrin (LPS) condition.

Well-posedness by noise for the stochastic continuity equation \eqref{stoch-continuity} has also been shown in several works in the literature. For instance, \cite{Beck2018} applies novel duality techniques to establish regularity and uniqueness results for the transport and continuity equations when $b$ satisfies the LPS condition, sometimes considering integrability conditions on $\div{(b)}$.  
In addition, \cite{olivera2017regularization} employs an argument based on characteristics to show well-posedness in the case when $b$ is random and meets certain integrability conditions. The work \cite{gess2019stochastic} furthermore shows well-posedness of {\em kinetic solutions} of the continuity equation with a general nonlinear noise. 

We stress that {\em all} the above articles except \cite{gess2019stochastic} only treat the case $\xi_i := e_i,$ $i=1,\ldots,n$, which substantially simplifies the equations. One of the goals of this article is to remove such a strong restriction on the diffusion vector fields when establishing well-posedness and regularisation by noise results.

Interestingly, the case corresponding to $k=2$, $n=3$ in our general equation \eqref{eq0}, known as the vector advection equation in 3D ideal magnetohydrodynamics, has also been considered in the literature, but only in the case $\div(b) = 0,$ $\xi_i := e_i,$ $i=1,\ldots,n$. In particular, in \cite{flandoli2014noise}, a mechanism for preventing instantaneous singularities of the spatial supremum norm is demonstrated for a H\"older type divergence-free drift and in \cite{flandoli2018well}, well-posedness by noise is shown by assuming some integrability conditions on a divergence-free drift vector field. We observe that, in particular, the results we establish in the present paper cover well-posedness and prevention of singularities in the above case for a drift with {\em nonzero divergence}.

The ultimate goal of research in well-posedness/regularisation by noise is to prove a meaningful result for nonlinear fluid PDEs such as the 2D Boussinesq equation, 3D Euler, or 3D Navier Stokes, for which the existence of global smooth solutions in the deterministic case remains an intriguing open problem. Local well-posedness of the 2D Boussinesq and 3D Euler equations with transport noise has already been established in \cite{DieAytJNLS} and \cite{crisan2017solution}, respectively. Moreover, \cite{diegochino} generalises the latter well-posedness results by developing an abstract local theory in Hilbert spaces that includes fluid PDEs with transport noise. However, the results of these works do not shed any light on whether transport noise has a regularising effect on the deterministic solutions. An example of when noise of transport-type does indeed have a regularising effect is the class of stochastic Lagrangian-averaged equations introduced in \cite{drivas2020lagrangian}, where in the case of the 2D Boussinesq equation, global well-posedness can be demonstrated by exploiting the regularity of solutions to the parabolic system \eqref{expectationreg} (see \cite{modelling}). Surprisingly, partial results of well-posedness by noise in the case of the 3D Navier-Stokes equation have also recently been established in \cite{fl2019scaling,fl2019high}, where the added noise bears some similarity with the transport noise we consider here. While these works show promise that transport noise could help regularise solutions of deterministic nonlinear equations, it is also well-established that this is {\em not always} the case. For instance, in \cite{alonso2018burgers}, it is shown that transport noise is not sufficient to prevent formation of shocks in the inviscid Burgers' equation.

\subsection{Contributions} 
In this article, we extend the well-posedness/regularisation by noise results of \cite{flandoli2010well, Flandoli2013, flandoli2014noise, olivera2017regularization,neves2015wellposedness,flandoli2018well} and several others to the general class of equations \eqref{eq0}. In particular, we:
\begin{enumerate}
    \item Prove global existence of $L^p$-solutions, for any $p \geq 2,$ strong in the probabilistic sense, with drift and diffusion vector fields satisfying Hypotheses \ref{hypoflow} \& \ref{hypoexistence} in Section \ref{sec:main results}, and uniqueness when Hypothesis \ref{hypoexistence} is substituted by Hypothesis \ref{stronguniqueness}. We also argue that Hypothesis \ref{stronguniqueness} is sharp (in the sense that it is crucial to give a meaning to equation \eqref{eq0}) whenever $k \notin \{0,n\}$, and show that it can be substantially weakened in the case $k \in \{0,n\}$, recovering in particular the result in \cite{flandoli2010well}. Moreover, we prove that the unique weak solution of \eqref{eq0} is continuous both in time and space whenever the initial datum is taken to be continuous. 
    
    \item Demonstrate {\em two ill-posedness results of different nature} for equation \eqref{eq0} without stochastic noise (i.e. $\xi_i:=0,$ $i=1,\ldots,N$). More concretely, given $n\in \mathbb{N},$ $k \in \{0,1,\ldots,n\},$ $p \geq 2,$ and $q \in (1,2]$ defined by the relation $1/q+1/p=1,$ we show the following:
    \begin{enumerate}[(i)]
    \item For all $\alpha \in (0,1)$ if $n \geq 2$ and $\alpha \in ((q-1)/q,1)$ if $n=1,$ there exists a drift vector field $b_\alpha$ satisfying Hypotheses \ref{hypoflow}, \ref{hypoexistence}, \ref{stronguniqueness} for which we construct an infinite family of $L^p$-solutions of equation \eqref{eq0} when $kp <n$ (i.e., we establish nonuniqueness of solutions). Moreover, these solutions fail to be continuous in both time and space for initial datum of arbitrary regularity.
    \item Whenever $k\notin \{0,n\}$ and $\alpha \in ((p-1)/p,1)$, there exists a different vector field $b_{\alpha}'$ satisfying the above hypotheses such that for certain smooth compactly supported initial data, nontrivial $L^p$-solutions exit the space $L^\infty_{loc}$ at {\em any} positive time (i.e., {\em instantaneous supremum norm inflation} occurs).
    \end{enumerate}
\end{enumerate}
Thus, under fairly general conditions, the Lie transport noise term in \eqref{eq0} has the effect of (i) restoring the uniqueness of $L^p$-solutions and (ii) preventing instantaneous blow-up of the spatial supremum norm of the solutions. Since we apply characteristics-based techniques that exploit the geometric structure of equation \eqref{eq0}, along the way we also prove some regularity properties of the flow \eqref{flow-map-stratonovich} and an explicit representation for our $L^p$-solutions that may be of independent interest. To establish our results, we also prove an extension to $k$-form-valued stochastic processes of the It\^o-Wentzell formula for the evolution of a time-dependent stochastic field along a semimartingale, which is equally of independent interest and has natural applications in fluid dynamics.

\section{Preliminaries} \label{preliminaries}
We provide some important definitions and results that will be needed for our analysis. In particular, in Subsection \ref{holder}, we introduce bounded and uniform H\"older spaces of Euclidean functions defined on $[0,T] \times \mathbb{R}^n$. In Subsection \ref{notation}, we discuss general notational aspects, in particular some geometric ones that will be employed in Subsection \ref{stochasticresults}. In Subsection \ref{stochasticresults}, we state some important definitions and results from the theory of stochastic processes. We also introduce weak concepts of the Lie derivative that will be used in the paper. 

Throughout this paper, all the spaces of functions depending on time and space are understood in the Lebesgue sense (i.e. as spaces of equivalence classes, where two functions are identified if they coincide for almost every (``a.e.") $(t,x)$). We stress this only throughout this introduction. 

\subsection{Euclidean H\"older spaces} \label{holder}
Let $\alpha \in (0,1).$ We say that a bounded Borel-measurable function $f:\mathbb{R}^n \rightarrow \mathbb{R}$ is (globally) $\alpha$-H\"older continuous if
\begin{align*}
[f]_{\alpha}:=\sup_{x\neq y\in {\mathbb{R}}
^{n}}\frac{\lvert f(x)-f(y)\rvert}{\lvert x-y\rvert^{\alpha }}<\infty.
\end{align*}
We denote the space of $\alpha$-H\"older continuous functions by $C^\alpha(\mathbb{R}^n).$
In the limiting case $\alpha = 1$, we say that $f:\mathbb{R}^n \rightarrow \mathbb{R}$ is (globally) Lipschitz continuous if the Lipschitz seminorm 
\begin{align*}
    [f]_{\text{Lip}} := \sup_{x\neq y\in {\mathbb{R}}
^{n}}\frac{\lvert f(x)-f(y)\rvert}{\lvert x-y\rvert} < \infty.
\end{align*}
We denote the space of bounded $\alpha$-H\"older continuous functions by $C_b^\alpha(\mathbb{R}^n),$ which is a Banach space endowed with the norm
\begin{align*}
    \norm{f}_{C_b^\alpha} := \norm{f}_{L^\infty} + [f]_{\alpha}.
\end{align*}
This definition can be extended straightforwardly to vector-valued functions, in which case we employ the notation $C_b^\alpha(\mathbb{R}^n,\mathbb{R}^m).$ For $k\in \mathbb{N},$ we define $C^{k+\alpha}(\mathbb{R}^n)$ to be the space of $C^k$-functions $f:\mathbb{R}^n \rightarrow \mathbb{R}$, such that $D^l f$ for all $l$ with $\lvert l \rvert=k$ are $\alpha$-H\"older continuous.
Similarly, we define $C^{k+\alpha}_b(\mathbb{R}^n)$ to be the Banach space of functions $f:\mathbb{R}^n \rightarrow \mathbb{R}$ having bounded continuous derivatives of order $i=0,\ldots,k,$ and $k$-th derivatives of class $C_b^\alpha(\mathbb{R}
^n).$ The corresponding norm is defined by
\begin{align*}
\norm{f}_{C_{b}^{k+\alpha}}:=\sum_{i=0}^{k}\norm{D^{i}f}_{L^\infty}+[D^{k}f]_{\alpha}.
\end{align*}
Again, this definition can easily be extended to vector-valued functions. Let $T>0$ and $\alpha \in (0,1)$. We define the space of bounded $\alpha$-H\"older continuous functions uniformly in time $L^{\infty
}\left([0,T];C_{b}^{\alpha }({\mathbb{R}}^{n})\right) $ to be the set of
bounded Borel-measurable functions $f:[0,T]\times {\mathbb{R}}^{n}\rightarrow {\mathbb{R%
}}$ with
\begin{align*}
[f]_{\alpha ,T}:=\sup_{t\in  [0,T]}\sup_{x\neq y\in {\mathbb{R}}
^{n}}\frac{\lvert f(t,x)-f(t,y) \rvert}{\lvert x-y \rvert^{\alpha }}<\infty.
\end{align*}
This is a Banach space with respect to the norm $\norm{f}_{C_b^\alpha,T}:=\norm{f}_{L^\infty_{t,x}}+[f]_{\alpha ,T}.$ Naturally, we denote by $L^{\infty }\left( [0,T];C_{b}^{\alpha }({\mathbb{R}}^{n},{\mathbb{R}}%
^{m})\right) $ the space of vector fields $f:[0,T]\times {\mathbb{R}}%
^{n}\rightarrow {\mathbb{R}}^{m}$ having components of class $L^{\infty
}\left( [0,T];C_{b}^{\alpha }({\mathbb{R}}^{n})\right) $.

Moreover, for $k \in \mathbb{N}\cup \{0\},$ we say that $f\in L^{\infty } ( [0,T];C_{b}^{k+\alpha }({%
\mathbb{R}}^{n})) $ if all spatial derivatives of order $i$ are in $L^{\infty }\left([0,T];C_{b}^{\alpha }({\mathbb{R}%
}^{n})\right)$, $i=0,1,\dots ,k$. The corresponding norm is defined as
\begin{align*}
\norm{f}_{C_{b}^{k+\alpha},T}:=\sum_{i=0}^{k}\norm{D^{i}f}_{L_{t,x}^\infty}+[D^{k}f]_{\alpha ,T}.
\end{align*}
This definition can be extended straightforwardly to functions of multiple components, in which case we employ the notation $L^{\infty} ([0,T];C_{b}^{k+\alpha}({\mathbb{R}}^{n}, \mathbb{R}^m)).$ Finally, we denote by $C_b([0,T];C_b^{k+\alpha}(\mathbb{R}^n,\mathbb{R}^n))$ the functions of class $C([0,T]\times \mathbb{R}^n;\mathbb{R}^n) \cap L^\infty([0,T];C^{k+\alpha}_{b}(\mathbb{R}^n,\mathbb{R}^n))$.

All the spaces of functions introduced above can be defined similarly when $T=\infty,$ in which case we substitute $[0,T]$ by $[0,\infty)$ in the notations. 

\subsection{Notations} \label{notation}
\subsubsection{Differential geometry notations}

Since we will be dealing with SPDEs with solutions in the space of tensor-valued stochastic processes, it is important that we clarify some of the notations we will use throughout the paper. For completeness, we have also included supplementary materials in Appendices \ref{tensor} and \ref{back-banach}, addressed towards readers unfamiliar with tensor algebra and calculus, where all the geometric objects introduced below are defined. We stress that the geometric function spaces below are understood as spaces of equivalence classes, where two tensor fields are identified if they coincide for a.e. $x \in \R^n.$

We shall denote by $\mathcal{T}^{(r,s)}\R^n$ the $(r,s)$-tensor bundle over $\R^n$, and in particular, we denote by $\mathcal{T}\R^n$ and $\mathcal{T}^*\R^n$ the tangent and cotangent bundles of $\R^n$, respectively (note that these correspond to the cases $(r,s) = (1,0)$ and $(r,s) = (0,1),$ respectively).
In this article, we will mostly work with $k$-forms, which are the alternating tensors of type $(r,s)=(0,k).$ We employ the notations $\bigwedge^k\mathcal{T}\mathbb{R}^n, \bigwedge^k\mathcal{T}^*\mathbb{R}^n$ for the $k$-th exterior power of the tangent and cotangent bundle, respectively. $k$-forms on $\mathbb{R}^n$ are defined as sections of $\bigwedge^k\mathcal{T}^*\mathbb{R}^n$ (i.e. maps $\mathbf{K} : \mathbb{R}^n \rightarrow \bigwedge^k\mathcal{T}^*\mathbb{R}^n$ that satisfy the section property $\pi \circ \mathbf{K} = \text{id} : \mathbb{R}^n \rightarrow \mathbb{R}^n$, where $\pi : \bigwedge^k\mathcal{T}^*\mathbb{R}^n \rightarrow \mathbb{R}^n$ is the canonical projection). 

We denote the space of smooth sections of the tensor bundle $\mathcal{T}^{(r,s)}\R^n$ by $C^\infty(\R^n, \mathcal{T}^{(r,s)}\R^n),$ and $C_0^\infty(\R^n, \mathcal{T}^{(r,s)}\R^n)$ if they are assumed to have compact support. However, in the special cases of the tangent and cotangent bundles, we typically follow the notations conventionally employed in differential geometry, that is, we denote the space of smooth sections of $\mathcal{T}\R^n$ by $\mathfrak{X}(\R^n),$ and the space of smooth sections of $\mathcal{T}^*\R^n$ by $\Omega^1(\R^n)$. We also denote the space of smooth $k$-vector fields (that is, the space of smooth sections of $\bigwedge^k \mathcal{T}\R^n$) by $\mathfrak{X}^k(\R^n),$ and the space of smooth $k$-forms by $\Omega^k(\R^n).$ Moreover, for each $x \in \R^n$, we denote using angle brackets $\langle \cdot, \cdot \rangle : \mathcal{T}^{(r,s)}_x\R^n \times \mathcal{T}^{(s,r)}_x\R^n \rightarrow \R$ the natural pairing between $(r,s)$ and $(s,r)$-tensors, and by round brackets $( \cdot, \cdot ) : \mathcal{T}^{(r,s)}_x\R^n \times \mathcal{T}^{(r,s)}_x\R^n \rightarrow \R$ the canonical bundle metric on $\mathcal{T}^{(r,s)}\R^n$ induced by the Euclidean metric on $\R^n$ (see Appendix \ref{tensor} for the definition of the natural pairing and bundle metric). 

In general, we adopt the notation $\mathfrak{B}(\R^n,\mathcal{T}^{(r,s)}\R^n)$ to denote Banach sections of tensor bundles \footnote{This notation shall not to be confused with the Borel sigma-algebra associated with a metric space $E$, which we will denote by $\mathcal{B}(E)$.}, where $\mathfrak{B}$ is to be replaced with any Banach function space, for example $\mathfrak B := L^p, C^{m}_b, W^{m,p}$ (see Appendix \ref{back-banach} for precise definitions). For convenience, we also include the exception $\mathfrak B := C^m,$ which is a Fréchet space. In particular, for $m \in \mathbb{N} \cup \{0\}$ and $p \in [1,\infty],$ we denote the Sobolev norm (induced by the canonical bundle metric) of a tensor $\mathbf{K} \in W^{m,p}(\R^n,\mathcal{T}^{(r,s)}\mathbb{R}^n)$ by $\norm{\mathbf{K}}_{W^{m,p}(\R^n,\mathbb{R}^n)}$ or $\norm{\mathbf{K}}_{W^{m,p}}$ for simplicity. For example, the space of Sobolev $k$-forms will be denoted $W^{m,p}(\R^n,\bigwedge^k\mathcal{T}^*\mathbb{R}^n)$ and its corresponding Sobolev norm denoted as above.
The case $\mathfrak B := L^2$ gives a Hilbert space $L^2(\R^n,\mathcal{T}^{(r,s)}\mathbb{R}^n)$, where we denote its associated inner product (induced by the canonical bundle metric) employing double angle brackets $\llangle \mathbf{K}_1, \mathbf{K}_2\rrangle_{L^2},$ for $\mathbf{K}_1, \mathbf{K}_2 \in L^2(\R^n,\mathcal{T}^{(r,s)}\R^n)$.  

For $x \in \mathbb{R}^n$ and $R>0,$ we denote by $B(x,R)$ the open ball with center $x$ and radius $R.$ For a tensor $\mathbf{K} \in W^{m,p}_{loc}(\R^n,\mathcal{T}^{(r,s)} \mathbb{R}^n),$ we denote its norm on $B(x,R)$ by $\norm{\mathbf{K}}_{W^{m,p}_{x,R}}$, or $\norm{\mathbf{K}}_{W^{m,p}_R}$ if $x=0.$ We also denote the space of distribution-valued tensor fields by $\D'(\R^n,\mathcal{T}^{(r,s)} \mathbb R^n)$, which is equipped with the weak-$*$ topology. Given a mollifier $\rho_\epsilon$ and $\mathbf{K} \in \mathcal D'(\R^n,\mathcal{T}^{(r,s)}\mathbb{R}^n),$ we denote its mollification by $\mathbf{K}^\epsilon.$ Finally, we denote the pull-back and push-forward of a tensor field $\mathbf{K}$ by a flow of diffeomorphisms $\varphi_{s,t}$ by $\varphi_{s,t}^*\mathbf{K}$ and $(\varphi_{s,t})_* \mathbf{K},$ respectively, and its Lie derivative along a vector field $b$ by $\mathcal{L}_b \mathbf{K}$. The divergence of a vector field $b$ with respect to the Euclidean volume form is denoted by $\div{(b)}$ and its curl (when $n=3$) by $\curl{(b)}.$

Since we are primarily interested in tensor fields that are random and time-dependent, $\mathbf{K}$ may further be parameterised by time $t \in [0,T]$ and a probability element $\omega \in \Omega$. Therefore, we also consider mixed Banach spaces of $k$-forms. For example, given Banach function spaces $\mathfrak{B}_1, \mathfrak{B}_2, \mathfrak{B}_3$, we say that a tensor field $\mathbf{K}$ belongs to the mixed Banach space $\mathfrak{B}_1\left(\Omega ; \mathfrak{B}_2\left([0, T]; \mathfrak{B}_3(\R^n,\mathcal{T}^{(r,s)}\R^n)\right)\right)$ if 
\begin{align*}
\|\mathbf{K}\|_{\mathfrak{B}_{1,\omega}, \mathfrak{B}_{2,t}, \mathfrak{B}_{3,x}} := \left\|\left\|\|\mathbf{K}\|_{\mathfrak{B}_{3,x}}\right\|_{\mathfrak{B}_{2,t}}\right\|_{\mathfrak{B}_{1,\omega}} < \infty,
\end{align*}
where $\|\cdot\|_{\mathfrak{B}_{i, \bullet}}$ denotes the norm in the Banach space $\mathfrak{B}_i$ and the subscripts $\omega, t, x$ indicate whether this norm is computed in the probability space, time, or physical space, respectively. In the special case $\mathfrak{B}_1 = \mathfrak{B}_2 = \mathfrak{B}_3 = \mathfrak{B}$, we adopt the notation $\mathfrak{B}(\Omega \times [0,T] \times \R^n ; \mathcal{T}^{(r,s)}\R^n)$ for the corresponding mixed Banach space, whose norm we denote simply by $\|\cdot\|_{{\mathfrak{B}}_{\omega, t, x}}$. For example, when $\mathfrak{B} := L^p$, we obtain $L^p(\Omega \times [0, T] \times \R^n; \mathcal{T}^{(r,s)}\R^n)$, whose associated norm is $\|\mathbf{K}\|_{L^p_{\omega,t,x}} = \E\left[\int^T_0 \int_{\R^n} \lvert \mathbf{K}\rvert^p \, \diff^n x \, \diff t\right]^{1/p}$. Moreover, if the latter integral is only defined on $B(0, R) \subset \R^n$ in the physical space, we adopt the notation $\|\cdot\|_{L^p_{\omega,t,R}}$. In the special case $p=2$, we also have the corresponding inner product $\llangle\cdot, \cdot\rrangle_{L^2_{\omega,t,x}}$ in the full space and $\llangle\cdot, \cdot\rrangle_{L^2_{\omega,t,R}}$ its restriction to $B(0,R)$. 

\subsubsection{Miscellaneous notations}
Below, we introduce some of the other notations we employ throughout the paper shall any confusion arises.

\begin{itemize}
    \item We understand $\mathbb{N}$ to be the natural numbers excluding zero.
    \item We denote by $e_i,$ $i=1,\ldots,n$ the canonical basis vectors in $\mathbb{R}^n,$ that is, $e_1 = (1, 0, 0, \ldots, 0)$,  $e_2 = (0, 1, 0, \ldots, 0)$, etc.
    \item  $\text{id}_X$ represents the identity function on a set $X.$ 
    \item For $x \in \mathbb{R}^n$ and $R>0,$ we denote by $\overline{B}(x,R)$ the closure of the open ball $B(x,R)$, and by $S(x,R)$ its boundary.
    \item In general, we will employ the notation $\overline{A}$ for the closure of a set $A$ with respect to a topology (which will be clear from the context) and $A^c$ for its complement.
    \item For a matrix $P \in \mathbb{R}^{n \times n}$ or a point $x \in \mathbb{R}^n,$ we denote by $\lvert P\rvert$ (respectively $\lvert x\rvert$) the $l_2$-norm of $P$ (respectively $x$). In the matrix case, this corresponds to the Frobenius norm and in the vector case, the Euclidean norm. We also denote by $\tr{(P)}$ the trace of a matrix $P$ and $\text{Adj}{(P)}$ its adjugate matrix.
    \item Given a measurable set $A \subset \R^n,$ we denote the characteristic function of $A$ by $\mathcal{X}_A,$ and we assume by convention $\mathcal{X}_{\emptyset} := 0$ (understood as a function), where $\emptyset$ is the empty set.
    \item Given two functions  $F_1$ and $F_2$ taking values in the real line, we write $ F_1 \lesssim F_2 $ if there exists a positive constant $C$ such that $F_1 \leq C F_2.$ Such constant may differ from line to line. In this article, we make it clear which variables the constant in $\lesssim$ depends on if it is not evident from the context.
    \item Let $\rho: \mathbb{R}^n \rightarrow \mathbb{R}$ be a $C^\infty$-function such that $0 \leq \rho(x) \leq 1,$ $\rho(x) = \rho(-x),$ $x \in \mathbb{R}^n,$ $\int_{\mathbb R^n} \rho(x) \, \diff^n x = 1,$ such that it is supported on $B(0,1)$ and $\rho(x)=1,$ $x \in B(0,1/2).$ For any $\epsilon >0,$ we define $\rho^\epsilon$ by $\rho^\epsilon (x) := \epsilon^{-n} \rho(x/\epsilon).$ For a function $f \in L^1_{loc}(\R^n,\R^m),$ its mollification $\rho^\epsilon * f$ will be denoted in short by $f^\epsilon.$
    \item We denote by $D$ the gradient of a real-valued function or, in general, the Fréchet derivative of a vector-valued function, and by $D^m$ its $m$-th Fréchet derivative, where $m \in \mathbb{N}$. These may be understood in the distributional sense, which will be clear from the context. Throughout this article, whenever we apply $D^m$ to a function (or distribution) depending on both time and space, it will normally be understood as derivatives in the spatial component and will otherwise be clarified.
    \item We denote $(r,s)$-tensors in {\textbf{bold font}}, with the exception of the cases $(r,s)=(1,0)$ and $(r,s)=(0,1)$, corresponding to vectors and covectors, respectively. This is to avoid saturating the paper with bold font symbols. Tensor components are denoted by $T^{i_1,\ldots,i_r}_{j_1,\ldots,j_s}$ as discussed in Appendix \ref{tensor}. In particular, for differential $k$-forms we use the notation $K_{j_1,\ldots,j_k}$.
\end{itemize}

\subsection{Results from stochastic analysis} \label{stochasticresults}

Finally, we introduce some key results from stochastic analysis that our proofs depend on. Throughout the paper, we assume that we are working with a fixed filtered probability space $(\Omega, (\mathcal{F}_t)_{t \in [0,T]}, \mathcal{F},\mathbb{P})$ satisfying the usual conditions, together with an $N$-dimensional Brownian motion $W_t = (W_t^1,\ldots,W_t^N)$ adapted to $(\mathcal{F}_t)_t.$ We denote by $(\mathcal{F}_{s,t})_{s,t}$ the completion of the sigma-algebra generated by $W_r-W_u,$ $s \leq u \leq r \leq t,$ for $0 \leq s < t \leq T$. We might consider the case $T=\infty.$ In this section, we also consider a different Brownian motion $B_t = (B_t^1,\ldots,B_t^N)$ adapted to $(\mathcal{F}_t)_t$.

\subsubsection{Stochastic flow of diffeomorphisms} Let $b$ and $\xi_i,$ $i=1,\ldots,N$ be vector fields on $\mathbb{R}^n$ that satisfy Hypothesis \ref{hypoflow} \ref{bcondition} below (see all our assumptions in Section \ref{sec:main results}). We are interested in working with solutions $\phi_{s,t}$ of equation \eqref{flow-map-stratonovich}, which in It\^o form reads
\begin{align} \label{flownotation}
    &\phi_{s,t}(x) = x + \int_s^t \left(b + \frac{1}{2} \sum_{i=1}^N \xi_i \cdot D \xi_i \right)(r,\phi_{s,r}(x)) \, \diff r  + \sum_{i=1}^N \int_s^t \xi_i(r,\phi_{s,r}(x)) \, \diff W_r^i, 
\end{align} 
for a.e. $0 \leq s \leq t \leq T,$ $x \in \mathbb{R}^n$. The case $s=0$ will be of special interest to us, in which case \eqref{flownotation} becomes 
\begin{align} \label{flowmap} 
    \phi_t(x) = x + \int_0^t \left(b + \frac{1}{2} \sum_{i=1}^N \xi_i \cdot D \xi_i \right)(s,\phi_s(x)) \, \diff s + \sum_{i=1}^N \int_0^t \xi_i(s,\phi_s(x)) \, \diff W_s^i,
\end{align} 
for a.e. $0\leq t \leq T,$ $x \in \mathbb{R}^n$, where we denoted $\phi_t := \phi_{0,t}$. Equation \eqref{flownotation} turns out to define a ``stochastic flow of diffeomorphisms" under certain regularity assumptions on the coefficients \cite{kunita1997stochastic}. To explain precisely what we mean by this, we introduce the following definition found in \cite{flandoli2010well}, only slightly generalised to the class $C^m$ (see also \cite{itowentzell1,kunita1997stochastic}). 

\begin{definition}[Stochastic flow of $C^m$-diffeomorphisms] \label{stochasticflow}

Let $m \in \mathbb{N}$ (i.e. at least $m\geq 1).$ An $\mathcal{F}_{s,t}$-measurable random field $\phi_{s,t}: \Omega \times \mathbb{R}^n \rightarrow \mathbb{R}^n,$ $0\leq s \leq t \leq T$ solving equation \eqref{flownotation} is called a stochastic flow of $C^m$-diffeomorphisms if $\mathbb{P}$-a.s. (it can be modified in a set of measure zero in $[0,T] \times \R^n$ such that) the following properties are satisfied:
\begin{itemize}
    \item $\phi_{s,r}(x) = \phi_{t,r}(x) \circ \phi_{s,t}(x),$ for all $0\leq s \leq t \leq r \leq T,$ $x \in \R^n$,
    \item $\phi_{t,t}(x) = x,$ for all $0 \leq t \leq T,$ $x \in \R^n$,
    \item $\phi_{s,t}$ is a $C^m$-diffeomorphism, for all $0\leq s \leq t \leq T.$ Moreover, the spatial derivatives $D^i\phi_{s,t}(x),$ $D^i\phi_{s,t}^{-1}(x)$ are continuous for all $0\leq s \leq t \leq T,$ $x \in \R^n$, and $i=0,\ldots,m$.
\end{itemize}
For $0 \leq t \leq s \leq T,$ we set $\phi_{s,t}:= \phi_{t,s}^{-1}.$
\end{definition}

For $\alpha \in (0,1),$ we define a stochastic flow of $C^{m+\alpha}$-diffeomorphisms to be a stochastic flow of $C^m$-diffeomorphisms such that $\mathbb{P}$-a.s., $D^m \phi_{s,t}(x)$ is $\alpha$-H\"older continuous, {\em uniformly} in $(s,t)$.
Below, we also introduce the notion of a $C^{m+\alpha}$-semimartingale, found in \cite{kunita1997stochastic}. Note that when we talk about semimartingales $F_t(x),$ $t \in [0,T],$ $x \in \R^n$ that depend on both time and space, it is implied that the process $(F_t(x))_{t \in [0,T]}$ is a semimartingale for every $x \in \R^n$.

\begin{definition}[$C^{m+\alpha}$-semimartingale \cite{kunita1997stochastic}]
A continuous local martingale $M_t(x),$ $t \in [0,T],$ $x \in \R^n,$ with covariance $Q(t,x,y):= [M_{\cdot}(x), M_{\cdot}(y)]_t,$ $t \in [0,T],$ $x,y \in \R^n,$ is said to be a $C^{m + \alpha}$-local martingale if (i) $M_t$ and $Q(t,\cdot,\cdot)$ have modifications that are $C^{m + \alpha}$-processes, 
and (ii) the processes $D^l_x M_t(x)$ are continuous local martingales with covariance $D^l_x D^l_y Q(t,x,y),$ for all multi-indices with $\lvert l \rvert \leq m$.
A continuous semimartingale $F_t(x) = A_t(x) + M_t(x)$ expressed as the sum of a continuous local martingale and a continuous process of bounded variation is said to be a $C^{m + \alpha}$-semimartingale if (i) $A_t$ is a continuous $C^{m + \alpha}$-process such that the processes $D^l_x A_t(x)$ have bounded variation, for all $\lvert l \rvert \leq m$, and (ii) $M_t$ is a $C^{m + \alpha}$-local martingale.
\end{definition}

A classical assumption to guarantee that equation \eqref{flownotation} defines a stochastic flow of diffeomorphisms is the following.
\begin{assump}[Classical regularity of the flow equation] \label{assumpstrong}{\ } 
Let $\beta \in (0,1)$ and $m\in \mathbb{N}$. Let $b \in L^\infty([0,T];C^{m+\beta}_b(\mathbb{R}^n,\mathbb{R}^n)),$ $\xi_i  \in  L^\infty([0,T];C^{m+\beta+1}_b(\mathbb{R}^n,\mathbb{R}^n))$, $i=1,\ldots,N,$ be deterministic vector fields. Then (as a particular case of Theorem 4.7.3 in \cite{kunita1997stochastic} and the results in \cite{kunita1984stochastic}), \eqref{flownotation} admits a unique solution $\phi_{s,t}$ which is an adapted $C^{m+\beta'}$-semimartingale and a stochastic flow of $C^{m}$-diffeomorphisms for any $\beta' \in (0,\beta).$ 
\end{assump}

\subsubsection{It\^o-type formulas for $k$-form-valued processes} 
Next, we present a useful result, which comprises a generalisation of It\^o-type formulas for tensors derived by Kunita in \cite{kunita1981some}, and the classical It\^o-Wentzell formula. We call this the Kunita-It\^o-Wentzell formula for $k$-form-valued processes (advected by a stochastic flow of diffeomorphisms). We stress that this result is {\em not in the literature} and thus needs to be proven, which we do in Appendix \ref{eq:app-KIW-proof}.
\begin{theorem}[Kunita-It\^o-Wentzell (KIW) formula for $k$-form-valued processes] \label{KIWmain} Let $k \in \{0,1,\ldots,n\}$ and $\mathbf{K}: \Omega \times [0,T] \times \mathbb{R}^n \rightarrow  \bigwedge^k\mathcal{T}^*\mathbb{R}^n$ be a $k$-form-valued semimartingale satisfying $\mathbb{P}$-a.s. 
\begin{enumerate}[(a)] 
\item $\mathbf{K}(t,x)$ is continuous for all $t \in [0,T]$ and $x \in \R^n,$
\item $\mathbf{K}(t,\cdot)$ is of class $C^2$ for all $t \in [0,T].$
\end{enumerate}
Moreover, let $\mathbf{K}$ have the explicit form
\begin{align*}
    \mathbf{K}(t,x) = \mathbf{K}(0,x) + \int^t_0 \mathbf{G}(s,x) \, \diff s + \sum_{i=1}^M \int^t_0 \mathbf{H}_i(s,x) \, \diff W_s^i, 
\end{align*}
where $\mathbb{P}$-a.s. \begin{enumerate}[(i)]
\item $\int^T_0 \left(\lvert \mathbf{G}(t,x)\rvert + \sum_{i=1}^M \lvert \mathbf{H}_i(t,x)\rvert^2 \right) \diff t < \infty,$ for all $x \in \R^n,$
\item $\mathbf{G}(t,\cdot)$ and $\mathbf{H}_i(t,\cdot)$ are of class $C^1$ for all $t \in [0,T],$
\end{enumerate}
and $W_t^i$ are independent Brownian motions, $i=1,\ldots,M$. Let $\phi_t$ be an adapted solution of equation \eqref{flowmap} with $b,$ $\xi_j,$ $j=1,\ldots,N$ satisfying Hypothesis \ref{hypoflow} \ref{bcondition} and Hypothesis \ref{stronguniqueness} with $q=1$, assumed to be a stochastic flow of $C^1$-diffeomorphisms. Then $\mathbb{P}$-a.s., the $k$-form-valued process $\mathbf{K}$ pulled back by the flow $\phi_t$ satisfies the following
\begin{align} \label{KIW}
\begin{split}
    & \phi_t^* \mathbf{K}(t,x) = \mathbf{K}(0,x) + \int^t_0 \phi_s^* \mathbf{G}(s,x) \, \diff s + \sum_{i=1}^M \int^t_0 \phi_s^* \mathbf{H}_i(s,x) \, \diff W_s^i \\
    & + \int^t_0 \phi_s^* \mathcal L_b \mathbf{K} (s,x) \, \diff s 
    + \sum_{j=1}^N \int^t_0 \phi_s^* \mathcal L_{\xi_j} \mathbf{K}(s,x) \, \diff B_s^j  \\
    & + \sum_{i=1}^M \sum_{j=1}^N \int^t_0 \phi^*_s \mathcal L_{\xi_j} \mathbf{H}_i(s,x) \, \diff [W^i_{\cdot}, B^j_{\cdot}]_s + \frac12 \sum_{j=1}^N  \int^t_0 \phi^*_s \mathcal L_{\xi_j} \mathcal L_{\xi_j} \mathbf{K}(s,x) \, \diff s, 
\end{split}
\end{align}
for $t \in [0,T]$ and $x \in \mathbb{R}^n.$ Moreover, let Assumption \ref{assumpstrong} be verified with $m=3$ and $\beta \in (0,1)$.
Then we have the following formula for the push-forward by the flow 
\begin{align} \label{KIWpush}
\begin{split}
    &(\phi_t)_* \mathbf{K}(t,x) = \mathbf{K}(0,x) + \int^t_0 (\phi_s)_* \mathbf{G}(s,x) \, \diff s + \sum_{i=1}^M \int^t_0 (\phi_s)_* \mathbf{H}_i(s,x) \,  \diff W^i_s  \\
    & - \int^t_0 \mathcal L_b [(\phi_s)_* \mathbf{K}](s,x) \,  \diff s 
    - \sum_{j=1}^N \int^t_0 \mathcal L_{\xi_j} [(\phi_s)_* \mathbf{K}](s,x) \, \diff B_s^j  \\
    & + \sum_{i=1}^M \sum_{j=1}^N \int^t_0 \mathcal L_{\xi_j} [(\phi_s)_* \mathbf{G}_i](s,x) \,  \diff [W^i_{\cdot}, B^j_{\cdot}]_s + \frac12 \sum_{j=1}^N  \int^t_0 \mathcal L_{\xi_j} \mathcal L_{\xi_j} [(\phi_s)_* \mathbf{K}](s,x) \,  \diff s.
\end{split}
\end{align}
\end{theorem}
\begin{remark}
Note that for formula \eqref{KIW}, in the case $k=0$ Hypothesis 3 is not needed since $\mathcal L_{b} f = b \cdot D f$ (and hence no spatial derivatives of $b$ are involved). In the case $k=n,$ Hypothesis 3 can be substituted by $\div{(b)} \in L^1_{loc}(\R^n)$ since  $\mathcal L_{b} (\rho \diff^n x) = \div{(\rho b)} \diff^n x$.
\end{remark}
Finally, we introduce weak notions of Lie derivative on the Euclidean space (see Appendix \ref{calculo} for the definition of a Lie derivative). In the following definitions, we let $q,p \in [1,\infty]$ be two constants that satisfy $1/q+1/p =1.$

\begin{definition}[Sobolev Lie derivative] \label{lieweak}
Let $\mathbf{K} \in W^{1,p}_{loc} (\R^n,\mathcal{T}^{(r,s)} \mathbb R^n)$ and $b \in W^{1,q}_{loc}(\mathbb{R}^n,\mathbb R^n)$. We define the weak Lie derivative of $\mathbf{K}$ with respect to $b$ and denote it by $\mathcal L_b^{w} \mathbf{K} \in L^{1}_{loc} (\R^n,\mathcal{T}^{(r,s)} \mathbb R^n)$ through the global coordinate expression \eqref{Lieformula} (Appendix \ref{calculo}), where the derivatives of $\mathbf{K}$ and $b$ therein are now understood in the weak (Sobolev) sense. Its $L^2$-adjoint operator $(\mathcal L_b^w)^{*} \mathbf{K}$ can be easily computed and verified to have the coordinate expression \eqref{Lieformulaadjoint}. Throughout this paper, we will mostly omit the subscript ``$w$", since it will be clear from the context.
\end{definition}

\begin{definition}[Distributional adjoint Lie derivative and Lie derivative] \label{lieadjointdistribution}
Let $\mathbf{\Theta} \in W^{1,p}_{loc} (\R^n,\mathcal{T}^{(r,s)} \mathbb R^n)$ and $b \in L^q_{loc}(\R^n,\R^n)$. We define the distributional adjoint Lie derivative of $\mathbf{\Theta}$ with respect to $b$ (and denote it by $(\mathcal{L}_b^*)^{dist} \mathbf{\Theta} \in \D' (\R^n,\mathcal{T}^{(r,s)} \mathbb R^n)$) as the functional that assigns to any test function $\mathbf{\Gamma} \in C^{\infty}_0(\R^n,\mathcal{T}^{(r,s)} \mathbb R^n)$ the expression that would result from \eqref{Lieformulaadjoint} integrated by parts against $\mathbf{\Gamma}$ if $b$ was smooth enough, i.e.
\begin{align} \label{newformula}
\begin{split}
& \left[(\mathcal L^*_b)^{dist} \mathbf{\Theta} \right](\mathbf{\Gamma})  \\
& := \int_{\R^n} \left[ b^l \Theta^{i_1,\ldots,i_r}_{j_1, \ldots, j_s} \frac{\partial \Gamma^{i_1, \ldots, i_r}_{j_1, \ldots, j_s}}{\partial x_l}  + b^{l} \frac{\partial \left( \Theta^{l, i_2, \ldots, i_r}_{j_1, \ldots, j_s} \Gamma^{i_1, \ldots, i_r}_{j_1, \ldots, j_s}\right)}{\partial x_{i_1}}  \right. \\
&\left. \hspace{148pt} + \cdots + b^{l} \frac{\partial \left( \Theta^{i_1, \ldots, i_{r-1}, l}_{j_1, \ldots, j_s} \Gamma^{i_1, \ldots, i_r}_{j_1, \ldots, j_s} \right)}{\partial x_{i_r}} \right] (x) \, \diff^n x \\
&\quad - \int_{\R^n} \left[ b^{j_1} \frac{\partial \left(\Theta^{i_1, \ldots, i_r}_{l, j_2, \ldots, j_s} \Gamma^{i_1, \ldots, i_r}_{j_1, \ldots, j_s}\right)}{\partial x_l} + \cdots + b^{j_s} \frac{\partial \left(\Theta^{i_1, \ldots, i_r}_{j_1, \ldots, j_{s-1},l} \Gamma^{i_1, \ldots, i_r}_{j_1, \ldots, j_s}\right)}{\partial x_l} \right] (x) \, \diff^n x,
\end{split}
\end{align}
for any $\mathbf{\Gamma} \in C^{\infty}_0(\R^n,\mathcal{T}^{(r,s)} \mathbb R^n),$ where as usually we assumed summation over repeated indices. If we impose the extra assumption that $(\mathcal L^*_b)^{dist} \mathbf{\Theta}$ is represented by a function of class $L^q_{loc}(\mathbb{R}^n,\bigwedge^k \mathcal{T}^*\mathbb R^n)$ (as in Hypothesis \ref{hypoexistence}), we can define the distributional Lie derivative via the formula
\begin{align} \label{distri}
    \left[ \mathcal L_b^{dist} \mathbf{K}\right] (\mathbf{\Theta}) := \llangle \mathbf{K}, (\mathcal L_b^*)^{dist} \mathbf{\Theta} \rrangle_{L^2}.
\end{align}
\end{definition}
\begin{remark}
The RHS of \eqref{newformula} and \eqref{distri} are well-defined by H\"older's inequality. We also note that \eqref{newformula} simplifies significantly in the case when $\Theta$ is a $k$-form since the last line vanishes, which is the only case we are interested in in this article. For the necessary geometric background, we point the reader to Appendices \ref{tensor} $\&$ \ref{back-banach}.
\end{remark}
\begin{remark}
Throughout the paper, we will extensively use the distributional adjoint Lie derivative (it is necessary for defining the notion of $L^p$-weak solutions of equation \eqref{eq0} as we do in Definition \ref{weak-sol}), so to avoid cumbersome notation (and since it will be clear from the context), we will normally write $\mathcal{L}^*_b$ instead of $(\mathcal{L}_b^*)^{dist}$.
\end{remark}

\section{Assumptions and main results}\label{sec:main results}

We note that since we will be working in spaces of equivalence classes of functions ($L^p,W^{k,p},C^k,$ etc), our functions in time and space are defined up to a set of measure zero. Also, our identities in time and space make natural sense almost everywhere, so when it is clear from the context, ``a.e." will be omitted. The product sigma-algebras will always be considered to be completed. 

Let $r,l \geq 1.$ We say that a vector field $b$ is of class $L^l ([0,T] ; W^{1,r}_{loc} (\mathbb R^n, \mathbb R^n))$ if for every $R>0,$ $\mathcal{X}_{B(0,R)} b \in L^l ([0,T] ; W^{1,r} (\mathbb{R}^n, \mathbb R^n)).$
\subsection{Assumptions} 
To establish our main results, we will consider different hypotheses on the drift and diffusion vector fields of equation \eqref{eq0}. In general, we will understand that $k \in \{0,1,\ldots,n\}$ and $n \in \mathbb{N}$ are fixed unless indicated otherwise. The first hypothesis will be needed to treat the flow equation \eqref{flownotation}.
\begin{hypothesis} \label{hypoflow}{\ } 
Fix $\alpha \in (0,1).$ 
\begin{enumerate}[label=(\roman*)]
\item $b \in L^\infty\left([0,T]; C^{\alpha}_b(\mathbb R^n, \mathbb R^n)\right)$, $\xi_i \in C_b([0,T]; C^{2}_b(\mathbb R^n, \mathbb R^n))$, $i=1,\ldots,N.$ \label{bcondition}
\item There exist $0<c<C < \infty$ such that the uniform parabolicity condition 
\begin{align} \label{strongparabolicity}
    & c \lvert v\rvert^2 \leq \sum_{i,j = 1}^n \sum_{r=1}^N \xi^i_r(t,x) \xi^j_r(t,x) v_i v_j \leq C \lvert v\rvert^2,\quad t \geq 0,\quad x \in \mathbb{R}^n,\quad v \in \mathbb{R}^n,
\end{align} 
is met. 
\label{xicondition}
\end{enumerate}
\end{hypothesis}
The following hypothesis is needed to establish the existence of solutions.
\begin{hypothesis}
 \label{hypoexistence}{\ } 
Fix $q > 1.$
\begin{enumerate}[label=(\roman*)]
    \item For every test $k$-form $\mathbf{\Theta} \in C^\infty_0(\R^n,\bigwedge^k \mathcal{T}^*\mathbb R^n),$ $(\mathcal{L}^*_b)^{dist} \mathbf{\Theta} \in L^q_{loc}([0,T] \times \mathbb{R}^n;\bigwedge^k \mathcal{T}^*\mathbb R^n).$
    \label{maincondition}
    \item If $k \notin \{0,n\},$ there exists $C \in \R$ such that $C \leq \div (b) (t,x),$ uniformly in $(t,x),$ where $\div(b)$ is understood {\em in the sense of distributions}.
    \label{newcondition}
\end{enumerate}
\end{hypothesis}

\begin{remark}
It can be checked that Hypothesis \ref{hypoexistence} \ref{maincondition} always holds if $b \in L^q ([0,T] ; W^{1,q}_{loc} (\mathbb R^n, \mathbb R^n))$. However, in general Hypothesis \ref{hypoexistence} \ref{maincondition} does not require the existence of Sobolev derivatives of $b$. For example, if $k=0,$ Hypothesis \ref{hypoexistence} \ref{maincondition} is equivalent to $b \in L^q_{loc}([0,T] \times \R^n;\R^n)$ and $\div{(b)} \in L^q_{loc}([0,T] \times \R^n)$ {\em in the sense of distributions}.
\end{remark}
To prove a general uniqueness result, we will employ the following hypothesis.
\begin{hypothesis}
\label{stronguniqueness}{\ } 
Fix $q \geq 1.$ $Db \in L^q_{loc}([0,T]\times \mathbb{R}^n;\mathbb{R}^{n \times n})$.
\end{hypothesis}
At first sight, Hypothesis \ref{stronguniqueness} might appear to be a fairly strong assumption. However, it is not difficult to show that for $k \notin \{0,n\},$ {\em Hypothesis \ref{hypoexistence} \ref{maincondition} actually implies Hypothesis \ref{stronguniqueness}}. In other words, the condition we impose for uniqueness {\em is weaker} than the ones for existence. We note that Hypothesis \ref{hypoexistence} \ref{maincondition} is {\em absolutely necessary} to give rigorous meaning to the definition of weak solutions in \eqref{weak-eq}, as explained in Remark \ref{hypoexistence}. In the case $k \in \{0,n\},$ such implication does not hold and we will provide sharper conditions for uniqueness, at the expense of increasing $q$. This is done in Theorem \ref{theoremweakuniqueness} further below.
\begin{remark}
Throughout this paper, whenever we enforce any of Hypotheses \ref{hypoflow}, \ref{hypoexistence}, or \ref{stronguniqueness}, we will assume that $\alpha$ and $q$ are fixed.
\end{remark}

\subsection{Notions of solution} \label{back-notion}
We first note that the solutions we consider are global in time. For questions regarding the geometric notation that we employ, we point the reader to Subsection \ref{notation} and Appendix \ref{back-banach}. Sometimes we will choose the notation $\mathbf{K}_t(x) := \mathbf{K}(t,x)$ for the solution of \eqref{eq0} or simply $\mathbf{K}$ at our convenience. 
\begin{definition}[Classical solution] \label{classicalsol}
Let $b \in L^\infty([0,T];C^1(\mathbb{R}^n,\mathbb{R}^n))$ and $\xi_i  \in  L^\infty([0,T];C^2(\mathbb{R}^n,\mathbb{R}^n))$, $i=1,\ldots,N$. A classical solution to equation \eqref{eq0} is a $k$-form-valued process
$\mathbf{K}$ adapted to $(\mathcal{F}_t)_{t}$ such that $\mathbb{P}$-a.s. $\mathbf{K}$ has trajectories of class $C([0,T]; C^2 (\mathbb{R}^n, \bigwedge^k \mathcal{T}^*\mathbb{R}^{n}))$ and it holds 
\begin{align}
\begin{split} \label{classicalito}
& \mathbf{K}_{t}(x) - \mathbf{K}_0(x)  +  \int_0^{t} \mathcal{L}_{b} \mathbf{K}_s(x) \, \diff s +\sum_{i=1}^{N} \int_0^{t} \mathcal{L}_{\xi_{i}} \mathbf{K}_s(x) \, \diff W^{i}_{s} \\
& \hspace{110pt} - \frac{1}{2} \sum_{i=1}^{N} \int_0^{t} \mathcal{L}^2_{\xi_{i}} \mathbf{K}_s(x) \, \diff s = 0, \quad t \in [0,T], \quad x \in \mathbb{R}^n.
\end{split}
\end{align}
\begin{remark}
Equation \eqref{classicalito} can be understood as an equality in $\mathbb{R}^{\binom{n}{k}}$ (holding, of course, for almost every $(t,x) \in [0,T] \times \mathbb{R}^n$. As explained at the beginning of the section, we normally omit this).
\end{remark}
\end{definition}
\begin{definition}[Weak $L^p$-solution] \label{weak-sol}
Let $p\geq 2$. Moreover, enforce Hypothesis \ref{hypoexistence} \ref{maincondition} on the drift vector field with $q$ defined by the relation $1/q+1/p=1$, and $\xi_i \in L^\infty([0,T]; C^{2}(\mathbb R^n, \mathbb R^n))$, $i=1,\ldots,N.$  We say that a $k$-form-valued process $\mathbf{K}$ satisfies equation (\ref{eq0}) weakly in the $L^p$-sense if:
\begin{itemize}
\item $\mathbf{K} \in L^p (\Omega \times [0,T] \times \mathbb{R}^n ; \bigwedge^k \mathcal{T}^*\mathbb R^n ),$ which is the Banach space of $k$-form-valued stochastic processes with norm \footnote{This space can be identified with $L^p \left( \Omega \times [0,T] \times \mathbb{R}^{n};\mathbb{R}^{\binom{n}{k}} \right)$.}
\begin{align*}
    \norm{\mathbf{K}}_{L^p(\Omega \times [0,T] \times \mathbb R^n)} := \mathbb{E} \left[ \int_0^T \int_{\mathbb{R}^n} \lvert \mathbf{K}_s(x)\rvert^p \, \diff^n x \, \diff s \right]^{1/p} < \infty.
\end{align*}
\item For any test $k$-form $\mathbf{\Theta} \in C^\infty_0 (\R^n,\bigwedge^k \mathcal{T}^*\mathbb R^n)$, the process $\llangle \mathbf{K}_t, \mathbf{\Theta} \rrangle_{L^2}$ has an $\mathcal{F}_t$-adapted continuous modification \footnote{By a continuous modification, we mean that there exists a stochastic process in the same equivalence class of functions that $\mathbb{P}$-a.s. has continuous paths.} and
\begin{align} \label{weak-eq}
\begin{split}
& \llangle \mathbf{K}_{t}, \mathbf{\Theta} \rrangle_{L^2} - \llangle \mathbf{K}_0, \mathbf{\Theta} \rrangle_{L^2} + \int_0^{t} \llangle  \mathbf{K}_s,(\mathcal{L}^*_{b})^{dist} \mathbf{\Theta} \rrangle_{L^2} \diff s  \\
&\hspace{50pt} + \sum_{i=1}^{N} \int_0^{t} \llangle \mathbf{K}_s, \mathcal{L}^*_{\xi_{i}} \mathbf{\Theta} \rrangle_{L^2} \diff  W^{i}_{s}  - \frac12 \sum_{i=1}^{N}  \int_0^{t} \llangle \mathbf{K}_s, \mathcal{L}^*_{\xi_{i}} \mathcal{L}^*_{\xi_{i}} \mathbf{\Theta} \rrangle_{L^2} \diff s = 0,
\end{split}
\end{align}
understood as an equality in $L^2(\Omega \times [0,T]).$
\end{itemize}
\end{definition}
\begin{remark}
We note that Hypothesis \ref{hypoexistence} \ref{maincondition} is the weakest possible condition on the drift $b$ to make sense of solutions of \eqref{eq0} without assuming spatial regularity of $\mathbf{K},$ since it is the weakest condition that ensures that $\llangle  \mathbf{K}_s,(\mathcal{L}^*_{b})^{dist} \mathbf{\Theta} \rrangle_{L^2}$ is well-defined. Also, we observe that it is natural to ask for $\eqref{weak-eq}$ to be satisfied as an equality in $L^2(\Omega \times [0,T]),$ since $\mathbf{K} \in L^p_{loc} (\Omega \times [0,T] \times \mathbb{R}^n ; \bigwedge^k \mathcal{T}^*\mathbb R^n ) \subset L^2_{loc} (\Omega \times [0,T] \times \mathbb{R}^n ; \bigwedge^k \mathcal{T}^*\mathbb R^n ).$ It is equivalent to imposing that \eqref{weak-eq} holds for almost every $(\omega,t) \in L^2(\Omega \times [0,T]).$ We remind that $\llangle \cdot, \cdot \rrangle_{L^2}$ stands for the $L^2$-inner product of tensor fields, which we briefly introduced in Subsection \ref{notation} and defined in detail in Appendix \ref{back-banach}. 
\end{remark}
\begin{remark}
The term $- \frac{1}{2} \sum_{i=1}^{N}  \int_0^{t} \llangle \mathbf{K}_s, \mathcal{L}^*_{\xi_{i}} \mathcal{L}^*_{\xi_{i}} \mathbf{\Theta} \rrangle_{L^2} \diff s$ in equation \eqref{weak-eq} does not indicate that the equation has a parabolic nature (see Remark 14 in \cite{flandoli2010well}).
\end{remark}
\begin{remark}
Equation \eqref{weak-eq} can be naturally reformulated in terms of distributions as
\begin{align*} 
& \mathbf{K}_{t} - \mathbf{K}_0 + \int_0^{t} \mathcal{L}^{dist}_{b} \mathbf{K}_s  \, \diff s  + \sum_{i=1}^{N} \int_0^{t} \mathcal{L}^{dist}_{\xi_{i}} \mathbf{K}_s \, \diff  W^{i}_{s} - \frac12 \sum_{i=1}^{N}  \int_0^{t} \mathcal{L}^{dist}_{\xi_{i}} \mathcal{L}^{dist}_{\xi_{i}} \mathbf{K}_s \, \diff s = 0,
\end{align*}
where $\mathbf{K}_t$ is to be understood as the distribution $\mathbf{K}_t(\mathbf{\Theta}) := \llangle \mathbf{K}_t, \mathbf{\Theta} \rrangle_{L^2},$ $\mathbf{\Theta} \in C^\infty_0 (\R^n,\bigwedge^k \mathcal{T}^*\mathbb R^n)$.
\end{remark}

\subsection{Main results}
The main goals of this paper are establishing global well-posedness of \eqref{eq0}, showing that this fails consistently in the deterministic case, and demonstrating that the noise in \eqref{eq0} also prevents the instantaneous norm blow-up taking place in the deterministic case. For this, we first prove the following result of existence, uniqueness, and regularity of the flow of equation \eqref{flownotation}.
\begin{theorem}[Regularity of the stochastic flow] \label{flowfla}
Consider equation \eqref{flownotation} with drift and diffusion vector fields satisfying Hypothesis \ref{hypoflow}. Then we have the following:
\begin{enumerate}
    \item For every $s \in [0,T]$ and $x \in \mathbb{R}^n,$ equation \eqref{flownotation} has a unique continuous adapted solution $(\phi_{s,t}(x))_{t \in [s,T]}$. Moreover, $\phi_{s,t},$ $0\leq s\leq t \leq T,$ is a stochastic flow of $C^{1+ \alpha'}$-diffeomorphisms (see Definition \ref{stochasticflow}), for any $\alpha' \in (0,\alpha).$ \label{primeritem}
    \item Let $\{b^j\}_{j = 1}^\infty, \{\xi_i^j\}_{j=1}^\infty  \subset L^\infty\left([0,T]; C_b^\alpha(\mathbb R^n, \mathbb R^n)\right),$ $i=1,\ldots,N$ be sequences of vector fields such that $b^j \rightarrow b$ and $\xi_i^j \rightarrow \xi_i$  in $L^\infty\left([0,T]; C_b^\alpha(\mathbb R^n, \mathbb R^n) \right)$, $i=1,\ldots,N,$ $j \rightarrow \infty.$ Let $\{\phi^j_{s,t}\}_{j=1}^\infty$ be the corresponding sequence of associated stochastic flows. Then for any $r \geq 1$, we have
    \begin{align}
        &\lim_{j \rightarrow \infty} \sup_{x \in \mathbb R^n} \sup_{0\leq s \leq T} \mathbb E\left[\sup_{s \leq t \leq T} \vert \phi_{s,t}^j(x) - \phi_{s,t}(x)\rvert^r\right] = 0, \label{phi-convergence}\\
        &\sup_{j \in \mathbb N} \sup_{x \in \mathbb R^n} \sup_{0 \leq s \leq T} \mathbb E\left[\sup_{s \leq t \leq T} \lvert D\phi_{s,t}^j(x)\rvert^r\right] < \infty, \label{bound-phi-deriv}\\
        &\lim_{j \rightarrow \infty} \sup_{x \in \mathbb R^n} \sup_{0\leq s \leq T} \mathbb E\left[\sup_{s \leq t \leq T} \lvert D\phi_{s,t}^j(x) - D\phi_{s,t}(x) \rvert^r\right] = 0.\label{phi-deriv-convergence}
    \end{align}
    \label{segundoitem} \item Let $\widetilde{r}(x) := \sqrt{1+\lvert x\rvert^2}$, $x \in \mathbb{R}^n$ and $r\geq 2$ be any constant such that $\alpha r > n$. Then for any $\alpha' \in (0, \alpha - n/r)$ and $\epsilon > n/r$, there exists a $C^{\alpha'}(\R^n; \R^{n \times n})$-modification of $[\tilde{r}]^{-\epsilon}D\phi_{s,t}$ such that
    \begin{align}\label{eq:JM-estimate}
    \sup_{0 \leq s \leq T} \mathbb{E}\left[\sup_{s \leq t \leq T} \left(\sup_{\substack{x,y \in\R^n \\ x \neq y}}\frac{\big\lvert [\widetilde{r}(x)]^{-\epsilon}D\phi_{s,t}(x) - [\widetilde{r}(y)]^{-\epsilon}D\phi_{s,t}(y) \big\rvert}{\lvert x-y \rvert^{\alpha'}}\right)^r\, \right] < \infty.
    \end{align}
    \label{terceritem}
\end{enumerate}
\end{theorem}
Since Theorem \ref{flowfla} establishes the existence and uniqueness of a stochastic flow of $C^1$-diffeomorphisms, we are in place to state our main theorems in full rigor.
\begin{theorem}[Existence of stochastic solutions] \label{maintheoremforms}
Let $p\geq 2$ and $\mathbf{K}_0 \in L^p (\R^n,\bigwedge^k \mathcal{T}^*\mathbb R^n)$. Moreover, let Hypotheses \ref{hypoflow} $\&$ \ref{hypoexistence} hold with $q$ defined by the relation $1/q+1/p=1$. Then equation \eqref{eq0} has a global, strong in the probabilistic sense, $L^p$-solution (i.e. a solution in the sense of Definition \ref{weak-sol}) with initial condition $\mathbf{K}_0$. Moreover, we can explicitly construct this solution via the formula $\mathbf{K}_t(x) := (\phi_t)_*\mathbf{K}_0(x),$ where we remind the readers that $(\phi_t)_*$ denotes the push-forward of the flow $\phi_t,$ assumed to be a solution of the flow equation \eqref{flowmap}. 
\end{theorem}

The following theorem guarantees uniqueness of solutions.

\begin{theorem}[Uniqueness of stochastic solutions] \label{theoremuniqueness}
In the conditions of Theorem \ref{maintheoremforms}, substitute Hypothesis \ref{hypoexistence} by Hypothesis \ref{stronguniqueness}. Let $\mathbf{K}^1$ and $\mathbf{K}^2$ be two $L^p$-solutions of \eqref{eq0} with initial condition $\mathbf{K}_0$. Then for every $t \in [0,T],$ we have $\mathbf{K}^1_t = \mathbf{K}^2_t$ (the equality is understood naturally as an equality in $L^p$).
\end{theorem}
As mentioned earlier, for $k \notin \{0,n\},$ Hypothesis \ref{hypoexistence} \ref{maincondition} implies Hypothesis \ref{stronguniqueness}. Hence existence implies uniqueness (although we do not need Hypothesis \ref{hypoexistence} \ref{newcondition} for uniqueness) and Hypothesis \ref{stronguniqueness} is sharp since it is the minimal condition necessary to provide a well-defined notion of weak solutions. In the case $k \in\{0,n\},$ the hypotheses in Theorem \ref{theoremuniqueness} can be reduced substantially. Indeed: 
\begin{theorem}[Improved uniqueness for $k \in \{0,n\}$] \label{theoremweakuniqueness}
Theorem \ref{theoremuniqueness} holds in the case $k=0$ when Hypothesis \ref{stronguniqueness} is replaced by $\div{(b)} \in L^r([0,T] \times \mathbb{R}^n)$ for some $r>2$, and in the case $k=n$ when Hypothesis \ref{stronguniqueness} is replaced by $\div{(b)} \in L^{r}([0,T] \times \mathbb{R}^n)$ for some $r>2q,$ where $q$ is defined by the relation $1/q+1/p=1$. We stress here that $\div{(b)}$ is understood in the sense of distributions. 
\end{theorem}
We will also conclude the following regularity result for our solutions.
\begin{corollary}[Continuity of the stochastic solutions] \label{strongexistence}
Let $p\geq 2$ and $\mathbf{K}_0 \in \left(C \cap L^p \right)(\R^n,\bigwedge^k \mathcal{T}^*\mathbb R^n)$. Moreover, let Hypotheses \ref{hypoflow}, \ref{hypoexistence}, \ref{stronguniqueness} hold with $q$ defined by the relation $1/q+1/p=1$. Then the unique global $L^p$-solution of \eqref{eq0} has a continuous modification (i.e. there is a $k$-form-valued process in the same equivalence class that $\mathbb{P}$-a.s. is of class $C([0,T] \times \R^n; \bigwedge^k\mathcal{T}^*\mathbb R^n)$).
\end{corollary}
Finally, we show that in the deterministic case, the linear Lie transport equation is ill-posed in $L^p$ due to nonuniqueness of solutions, and that instantaneous blow-up of the supremum norm takes place.

\begin{theorem}[Ill-posedness in the deterministic case] \label{ill}
In the deterministic case, i.e. $\xi_i:=0,$ $i=1,\ldots,N,$ neither Theorem \ref{theoremuniqueness} nor Corollary \ref{strongexistence} are true. In particular, for $n \in \mathbb{N},$ $k \in \{0,1,\ldots,n\},$ and $p \geq 2$ with $q$ defined by $1/q+1/p=1,$ there exist:
\begin{enumerate}
    \item An autonomous vector field $b_{\alpha}$ of class $(C_b^{\alpha}\cap W^{1,q}_{loc})(\mathbb{R}^n,\mathbb{R}^n)$ for any $\alpha \in (0,1)$ if $n \geq 2$ and for any $\alpha \in ((q-1)/q,1)$ if $n=1,$ such that if $kp <n,$ there exists $\mathbf{K}_0 \in C_0^\infty(\R^n,\bigwedge^k\mathcal{T}^*\mathbb R^n)$ (and hence also of class $L^p(\R^n,\bigwedge^k\mathcal{T}^*\mathbb R^n)$) and a parametric family of different solutions of class $L^p([0,T]\times \mathbb{R}^n;\bigwedge^k\mathcal{T}^*\mathbb R^n)$ to equation \eqref{deq0} with drift $b_\alpha$ and initial condition $\mathbf{K}_0.$ Moreover, such solutions are discontinuous in both time and space. 
    \item An autonomous divergence-free vector field $b'_{\alpha}$ of class $(C_b^{\alpha}\cap W^{1,q}_{loc})(\mathbb{R}^n,\mathbb{R}^n)$ for any $\alpha \in ((p-1)/p,1),$ such that in the case $k\notin \{0,n\},$  there exists $\mathbf{K}_0 \in C_0^\infty(\R^n,\bigwedge^k\mathcal{T}^*\mathbb R^n)$ for which the $L^p$-solution $\mathbf{K}_t$ to the deterministic equation with drift $b_\alpha'$ and initial condition $\mathbf{K}_0$ satisfies
    \begin{align*}
    & \sup_{x \in B(0,\epsilon)} \lvert \mathbf{K}_t(x) \rvert = \infty,
    \end{align*}
    for any $0<t \leq T$ and $\epsilon >0.$
\end{enumerate}
\end{theorem}

\section{Regularity of the stochastic flow}  \label{flowsection}

This section is devoted to establishing Theorem \ref{flowfla}. In order to show Items \ref{primeritem} $\&$ \ref{segundoitem}, we extend the proof of the flow regularity result in \cite{flandoli2010well} to allow for our more general type of noise. The proof in \cite{flandoli2010well} is based on the It\^o-Tanaka trick and sharp parabolic estimates. We encounter some extra difficulties and show representations of some specific solutions of parabolic equations with coefficients having low regularity in time. To establish Item \ref{terceritem}, we exploit the extra regularity properties of the flow that arise as a result of the It\^o-Tanaka trick, to perform certain estimates on the flow equation, which guarantee the validity of certain uniform estimates in \cite{leahy2015some}. We note that the existence and uniqueness of strong path-wise solutions to a general class of SDEs containing \eqref{flownotation} was established in \cite{ruso}, which also provide the first example in the literature of well-posedness by noise in the ODE case. 

We proceed to demonstrate Theorem \ref{flowfla}. Since our proofs of Items \ref{primeritem} $\&$ \ref{segundoitem} comprise a natural extension of the strategy in \cite{flandoli2010well}, the latter reference can be checked for completeness. For simplicity, for fixed $s\geq0$ and $x \in \mathbb{R}^n,$ we occasionally employ the notation $X_t := \phi_{s,t}(x),$ $t \in [s,T]$. Without loss of generality, we only treat the case $N=1$ here (i.e. we only consider one diffusion vector field $\xi = (\xi^1,\ldots,\xi^n)$). Incorporating more noise terms is a straightforward extension. 

\begin{proof}[Proof of Theorem \ref{flowfla}.]
{\bf{Proof of Item \ref{primeritem}}.} The essence of our argument is as follows: we show that the solution to equation \eqref{flownotation} is unique and can moreover be expressed as a composition of functions which are regular enough (in fact, they are remarkably regular taking into account our assumptions on the drift and diffusions!) to guarantee $C^{1+\alpha'}$-smoothness. More concretely, for fixed $s \in [0,T),$ $x\in \mathbb{R}^n,$ we prove that the process $\Psi_\lambda(t,\phi_{s,t}(x)) := \phi_{s,t}(x) + \psi_\lambda(t,\phi_{s,t}(x)),$ $t \geq s,$ where $\psi_\lambda$ is a solution of a parabolic equation on the extended domain $[0,\infty)\times \mathbb{R}^n$ depending on a parameter $\lambda \geq 1$, solves a flow equation of the type \eqref{flownotation} with new coefficients $\widetilde{b},$ $\widetilde{\xi}$ enjoying more regularity than the original ones. Moreover, this equation admits a smooth enough flow of diffeomorphisms $\varphi_{s,t}$ (this technique is called the It\^o-Tanaka trick or Zvonkin transform. See \cite{itotanaka} for one of the original motivations). By applying results treating the existence, uniqueness, and regularity of solutions of parabolic equations, we can show that $\psi_\lambda$ enjoys enough regularity.

Moreover, we provide an implicit representation for $\psi_\lambda$, which we employ to conclude that $\Psi_\lambda$ enjoys desired regularity and is invertible if $\lambda$ is large enough. Finally, we prove that the regularity of $\Psi_\lambda$ and $\varphi_{s,t}$ is inherited by the flow $\phi_{s,t},$ since $\phi_{s,t} = \Psi_\lambda^{-1} \circ \varphi_{s,t} \circ \Psi_\lambda$. In order to facilitate the understanding of our argument, we divide this proof into five steps. In Step 1 and Step 2, we treat the regularity of the functions $\psi_\lambda$ and $\Psi_\lambda$, whereas Step 3 is devoted to the It\^o-Tanaka trick and the improved regularity of the flow.

\noindent \textbf{Step 1.}
First, we extend $b$ and $\xi$ to the whole real time line $[0,\infty)$ by simply defining
\begin{align*}
& b(t,x) := b(T,x), \\
& \xi(t,x) := \xi(T,x),  \quad t > T,\quad x \in \mathbb{R}^n.
\end{align*}
By Hypothesis \ref{hypoflow}, we have
\begin{align*}
    b \in L^\infty([0,\infty); C_b^\alpha(\mathbb{R}^n,\mathbb{R}^n)),\quad \xi \in C_b ([0,\infty);C_b^{2+\alpha}(\mathbb{R}^n,\mathbb{R}^n)).
\end{align*}
For $\lambda \geq 1,$ consider the parabolic equation
\begin{align} \label{parabolic}
    \partial_t \psi_\lambda(t,x) + L \psi_\lambda(t,x) - \lambda \psi_\lambda(t,x) = -b(t,x), \quad t \geq 0, \quad x \in \mathbb{R}^n,
\end{align}
with 
\begin{align} \label{eq:L-operator}
    L \psi_\lambda(t,x) := \frac12 \xi^j(t,x) \xi^k(t,x) \partial_j \partial_k \psi_\lambda(t,x) + (b^k(t,x) +  \xi^j(t,x) \partial_j \xi^k(t,x) )\partial_k \psi_\lambda(t,x),
\end{align}
where summation over repeated indices is assumed. The notion of a solution to equation \eqref{parabolic} is not standard, since $b$ is only of class $L^\infty$ in time instead of continuous. Following \cite{krylov2009elliptic}, we prescribe that a function $\psi_\lambda:[0,\infty)\times \mathbb{R}^n \rightarrow \mathbb{R}^n$ in the space $L^\infty([0,\infty);C_b^{2+\alpha}(\mathbb{R}^n,\mathbb{R}^n))$ is a solution to \eqref{parabolic}, if it satisfies the integral equation
\begin{align} \label{krylovnotion}
\psi_\lambda(t,x) - \psi_\lambda(s,x) = \int_s^t \big[ -L \psi_\lambda(r,x) + \lambda \psi_\lambda(r,x) - b(r,x) \big] \diff r, 
\end{align}
for $0\leq s \leq t$ and $x \in \mathbb{R}^n$. From this, it follows that $\psi_\lambda$ is Lipschitz continuous in time and hence differentiable in time a.e. for fixed $x.$ By Theorems 2.4 $\&$ 2.5 in \cite{krylov2009elliptic} (set $f:=-b.$ It is easy to check that Hypotheses 2.1 $\&$ 2.2 in \cite{krylov2009elliptic} are met in the present context), there exists a unique solution $\psi_\lambda$ to \eqref{parabolic} in the Banach space $L^\infty([0,\infty); C_b^{2+\alpha}(\mathbb{R}^n, \mathbb{R}^n))$, which moreover, satisfies the bound
\begin{align*}
    \sup_{t \geq 0}\norm{\psi_\lambda(t,\cdot)}_{C_b^{2+\alpha}} \lesssim \sup_{t\geq 0} \norm{b(t,\cdot)}_{C_b^{\alpha}},
\end{align*}
where the constant in $\lesssim$ only depends on the parabolicity constants \eqref{strongparabolicity}, $\alpha,$ and $n.$

\noindent \textbf{Step 2.} Next, we introduce the function $\Psi_\lambda: [0,\infty) \times \mathbb{R}^n \rightarrow \mathbb{R}^n$, defined as $\Psi_\lambda(t,x) := x + \psi_\lambda(t,x)$. By means of the following lemma (originally appearing in \cite{flandoli2010well}. See Lemma 6), we can establish that $\Psi_\lambda$ enjoys higher regularity than $\psi_\lambda$.
\begin{lemma} \label{lemazo}
For $\lambda$ large enough, the condition 
\begin{align} \label{condition}
& \norm{D \psi_\lambda}_{L^\infty_{t,x}} <1
\end{align} 
is met, and the following properties are satisfied:
\begin{enumerate}
    \item Uniformly in $t \in [0,\infty),$ $\Psi_\lambda(t,\cdot)$ has bounded first and second spatial derivatives, and $D^2 \Psi_\lambda$ is globally $\alpha$-H\"older continuous. \label{point-1}
    \item For any $t \in [0,\infty),$ $\Psi_\lambda(t,\cdot): \mathbb R^n \rightarrow \mathbb R^n$ is a diffeomorphism of class $C^2.$ \label{point-2}
    \item Uniformly in $t \in [0, \infty),$ $\Psi_\lambda^{-1}(t, \cdot)$ has bounded first and second derivatives. \label{point-3}
\end{enumerate}
\end{lemma}

In order to apply Lemma \ref{lemazo} to our function $\psi_\lambda,$ we need to show that for large enough $\lambda$, $\psi_\lambda$ satisfies condition \eqref{condition}. For this, we prove the following intermediate result.
\begin{lemma} \label{generalheat}
Let $\lambda >0$ and $\widetilde{f} \in L^\infty([0,\infty);C^\alpha_b (\mathbb{R}^n, \mathbb{R}^n)).$ Then the parabolic system
\begin{align} \label{krylovheat}
   \partial_t v_\lambda(t,x) + L_0 v_\lambda(t,x) - \lambda v_\lambda (t,x) = \widetilde{f}(t,x),\quad t \geq 0, \quad x \in \mathbb{R}^n,
\end{align}
where $(L_0 v)(t,x) := \frac{1}{2}\xi^j(t,x) \xi^k(t,x) \partial_{j} \partial_{k} v(t,x)$, has a unique solution $v_\lambda \in L^\infty([0,\infty);C^{2+\alpha}_b (\mathbb{R}^n, \mathbb R^n)).$ Moreover, there exists a fundamental solution $p(t,x;s,y),$ $0 \leq t \leq s,$ $x, y \in \mathbb{R}^n$ for the operator $\partial_t + L_0$, such that $v_\lambda$ can be constructed via the explicit formula
\begin{align} \label{casot}
    & v_\lambda(t,x) := \int_{t}^{\infty}\int_{\mathbb{R}^n} \mathrm{exp}\left((-\lambda(s-t))\right) p(t,x;s,y)\widetilde{f}(s,y) \, \diff y \, \diff s , \quad t \geq 0, \quad x \in \mathbb{R}^n.
\end{align}
Furthermore, $p$ is infinitely differentiable in space and there exist constants $\epsilon, C>0$ such that for $i,m \in \mathbb{N} \cup \{0\}$ with $0 \leq i < 2m,$ we have
\begin{align} \label{estimadacalor}
    & \lvert D^i p(t,x;s,y)\rvert \leq C (s-t)^{-(n+i)/2m} \mathrm{exp} \left( \frac{-\epsilon \lvert x-y \rvert^{2m}}{s-t} \right)^{1/(2m-1)},
\end{align}
for all $0 \leq t < s$, $x,y \in \mathbb{R}^n$.
\end{lemma}
\begin{proof}[Proof of Lemma \ref{generalheat}.]
By Theorems 2.4 $\&$ 2.5 in \cite{krylov2009elliptic}, it is known that \eqref{krylovheat} admits a unique solution of class $L^\infty([0,\infty);C^{2+\alpha}_b (\mathbb{R}^n, \mathbb R^n))$. However, the explicit representation \eqref{casot} does not follow immediately. We start by assuming that $\widetilde{f} \in C_b([0,\infty);C_b^\alpha (\mathbb{R}^n, \mathbb{R}^n)).$ If $\lambda = 0,$ Lemma \ref{generalheat} is a classical result (See Theorem 2 in Chapter 9 of \cite{friedman} and Theorem 5 in the same chapter for the adjoint system. Conditions (A1), (A2), (A3) therein are clearly met under our assumptions). We claim that for $\lambda >0,$
\begin{align} \label{casotsmooth}
    & v_\lambda(t,x) := \int_{t}^{\infty}\int_{\mathbb{R}^n} \exp{(-\lambda(s-t))} p(t,x;s,y)\widetilde{f}(s,y) \, \diff y \, \diff s ,  \quad t \geq 0, \quad x \in \mathbb{R}^n,
\end{align}
solves \eqref{krylovheat}. Indeed, taking the time derivative on both sides of \eqref{casotsmooth}, we have
\begin{align}
    \begin{split}
    \partial_t v_\lambda(t,x) &=  \int_{t}^{\infty}\int_{\mathbb{R}^n} \exp{(-\lambda(s-t))} \partial_t p(t,x;s,y) \widetilde{f}(s,y) \, \diff y \, \diff s \\
    & \quad + \lambda \int_{t}^{\infty}\int_{\mathbb{R}^n} \exp{(-\lambda(s-t))} p(t,x;s,y)  \widetilde{f}(s,y) \, \diff y \, \diff s \label{krylov1}
    \end{split} \\
    &=  \exp{(\lambda t)} \int_{t}^{\infty}\int_{\mathbb{R}^n} \partial_t p(t,x;s,y) (\exp{(-\lambda s)} \widetilde{f}(s,y)) \, \diff y \, \diff s +  \lambda v_\lambda(t,x), \label{eq:leftoff}
\end{align}
where we used the Leibniz rule in \eqref{krylov1} and took into account that $\lim_{s \rightarrow t^+} p(t,x;s,y) = 0$, which follows from \eqref{estimadacalor} with $i=0$. Now, define
\begin{align*}
    u_\lambda(t,x) := \int_{t}^{\infty}\int_{\mathbb{R}^n} p(t,x;s,y) \exp{(-\lambda s)} \widetilde{f}(s,y) \, \diff y \, \diff s,
\end{align*}
which is a solution to the PDE
\begin{align*}
    \partial_t u_\lambda(t,x) + L_0 u_\lambda(t, x) = \exp{(-\lambda t)} \widetilde{f}(t,x),
\end{align*}
by the classical result \cite{friedman} (Theorem 2 in Chapter 9). Then, we have
\begin{align}
    \eqref{eq:leftoff} &=  \exp{(\lambda t)} \,\partial_t u_\lambda(t, x) +  \lambda v_\lambda(t,x) \label{krylov2} \\
    &=  \exp{(\lambda t)} \left(-L_0 u_\lambda(t,x) + \exp{(-\lambda t)} \widetilde{f}(t,x)\right) + \lambda v_\lambda(t,x),
    \nonumber
\end{align}
where we have again used Leibniz rule in \eqref{krylov2} to take the time derivative outside the integrals. Hence
\begin{align*}
    \partial_t v_\lambda(t,x) &= \exp{(\lambda t)} (-L_0 u_\lambda(t,x) + \exp{(-\lambda t)} \widetilde{f}(t,x)) + \lambda v_\lambda(t,x) \\
    &= -L_0 v_\lambda(t,x) + \lambda v_\lambda(t,x) + \widetilde{f}(t,x), 
\end{align*}
for $t\geq 0$ and $x \in \mathbb{R}^n,$ so we have proven our claim that the representation \eqref{casotsmooth} solves \eqref{krylovheat}. \\\\
Now we treat the general case $\widetilde{f} \in L^\infty([0,\infty);C^\alpha (\mathbb{R}^n, \mathbb{R}^n)).$ For $\epsilon >0,$ consider the time mollification of $f$ (formally, $f$ first needs to be extended by zero to the whole real line, then mollified and restricted), which we denote $\widetilde{f}^\epsilon \in C_b([0,\infty);C^\alpha (\mathbb{R}^n, \mathbb{R}^n)).$ From the first part of the proof, we know that
\begin{align*} 
    & v^\epsilon_\lambda(t,x) := \int_{t}^{\infty}\int_{\mathbb{R}^n} \exp{(-\lambda(s-t))} p(t,x;s,y)\widetilde{f}^\epsilon(s,y) \, \diff y \, \diff s ,  \quad t \geq 0, \quad x \in \mathbb{R}^n,
\end{align*}
solves \eqref{krylovheat} with $\widetilde{f}$ substituted by $\widetilde{f}^\epsilon.$ By the dominated convergence theorem (take into account the mollifier properties, which in particular yield $\norm{\widetilde{f}^{\epsilon}}_{L^\infty_{t,x}} \leq \norm{\widetilde{f}}_{L^\infty_{t,x}}$), we also obtain
\begin{align*} 
    & \lim_{\epsilon \rightarrow 0} v^\epsilon_\lambda(t,x) \rightarrow \int_{t}^{\infty}\int_{\mathbb{R}^n} \exp{(-\lambda(s-t))} p(t,x;s,y)\widetilde{f}(s,y) \,\diff y \, \diff s = v_{\lambda}(t,x),
\end{align*}
for $t \geq 0$, $x \in \mathbb{R}^n$ pointwise. Now, fix a test function $\theta \in C_0^\infty(\mathbb{R}^n,\mathbb{R}^n).$ Then we must have
\begin{align} \label{rollo}
\begin{split}
   &  \langle v^\epsilon_\lambda(t,\cdot) - v^\epsilon_\lambda(s,\cdot), \theta(\cdot) \rangle_{L^2} \\
   & = \int_s^t \left( - \langle v^\epsilon_\lambda(r,\cdot), (L_0^* \theta)(r,\cdot) \rangle_{L^2} + \lambda \langle v^\epsilon_\lambda (r,\cdot), \theta(\cdot) \rangle_{L^2} + \langle \widetilde{f}(r,\cdot), \theta(\cdot) \rangle_{L^2} \right) \diff r, 
\end{split}
\end{align}
for $0 \leq t < s$, where we have denoted by $L_0^*$ the $L^2$-adjoint operator of $L_0.$ Hence, taking limits in \eqref{rollo} by means of the bounded convergence theorem (use the fact that $v^\epsilon_\lambda \rightarrow v_\lambda$ a.e.), we conclude that $v_\lambda$ satisfies \eqref{krylovheat} weakly. It can also be checked that 
\begin{align*}
& v_\lambda \in L^\infty([0,\infty);C^{3}_b (\mathbb{R}^n, \mathbb R^n))   \subset L^\infty([0,\infty);C^{2+\alpha}_b (\mathbb{R}^n, \mathbb R^n)) 
\end{align*}
by taking into account its integral representation \eqref{casotsmooth} and using property \eqref{estimadacalor}, so it is indeed a solution in the sense \eqref{krylovnotion}. Since \eqref{krylovheat} has one and only one solution, it must be \eqref{casot}. The proof is now complete.
\end{proof}
\noindent \textbf{Step 3.}
We can now show that $\norm{D \psi_\lambda}_{L^\infty_{t,x}} \rightarrow 0,$ as $\lambda \rightarrow \infty.$ By setting 
\begin{align}\label{eq:tilde-f-definition}
\widetilde{f}(t,x) := -b(t,x)-\left(b^k(t,x) +  \xi^j(t,x) \partial_j \xi^k(t,x) \right)\partial_k \psi_\lambda,
\end{align}
then by Lemma \ref{generalheat}, we must have the following implicit formula
\begin{align*}
\psi_\lambda(t,x) = \int_{t}^{\infty}\int_{\mathbb{R}^n} \exp{(-\lambda(s-t))} p(t,x;s,y)\widetilde{f}(s,y) \,\diff y \,\diff t , \quad t \geq 0, \quad x \in \mathbb{R}^n,
\end{align*}
for the solution to the parabolic system \eqref{parabolic}.
From estimate \eqref{estimadacalor} with $i=1$ (consider the change of variables $u=s-t)$, it follows that
\begin{align*}
\lvert D \psi_\lambda(t,x) \rvert &\lesssim \int_{0}^{\infty}\int_{\mathbb{R}^n} e^{-\lambda u} u^{-(n+1)/2} \text{exp} \left(\frac{-\lvert x-y \rvert^2}{u} \right)  \widetilde{f}(s,y) \, \diff y \, \diff u  \\
&\lesssim \norm{\widetilde{f}}_{L_{t,x}^\infty} \int_{0}^{\infty}\int_{\mathbb{R}^n} e^{-\lambda u} u^{-1/2} \text{exp} \left[ -\lvert z \rvert^2 \right] \, \diff z \, \diff u  \\
&\lesssim \norm{\widetilde{f}}_{L_{t,x}^\infty} \int_{0}^{\infty} e^{-\lambda u} u^{-1/2} \, \diff u = \norm{\widetilde{f}}_{L_{t,x}^\infty} \frac{1}{\lambda^{1/2}}\int_{0}^{\infty} e^{-r} r^{-1/2} \, \diff r \\
&\lesssim \frac{\norm{\widetilde{f}}_{L_{t,x}^\infty}}{\lambda^{1/2}}, 
\end{align*}
where we have made the change of variables $y-x = z \sqrt{u}$ and $r = \lambda u$. We stress that the constants $\lesssim$ in the inequalities above are independent of $\lambda,$ since inequality \eqref{estimadacalor} only involves the fundamental solution of \eqref{parabolic}. Therefore, there exists a constant $C>0$ that is independent of $\lambda$, such that
\begin{align*}
    & \norm{D \psi_\lambda}_{L_{t,x}^\infty} \leq  C \lambda^{-1/2} \|\widetilde{f}\|_{L^\infty_{t,x}}.
\end{align*}
Hence from \eqref{eq:tilde-f-definition}, we have the estimate
\begin{align*}
\norm{D \psi_\lambda}_{L^\infty_{t,x}} \leq  C \lambda^{-1/2} \left(\norm{b}_{L^\infty_{t,x}} + \left( \norm{b}_{L^\infty_{t,x}} + \norm{\xi}_{L^\infty_{t,x}} \norm{D \xi}_{L^\infty_{t,x}} \right) \norm{D \psi_\lambda}_{L^\infty_{t,x}} \right),  
\end{align*}
which, after minor rearranging, gives us
\begin{align} \label{eq:grad-phi-lambda-ineq}
\left(1-  C \lambda^{-1/2} \left(\norm{b}_{L^\infty_{t,x}} + \norm{\xi}_{L^\infty_{t,x}} \norm{D \xi}_{L^\infty_{t,x}} \right) \right)\norm{D \psi_\lambda}_{L^\infty_{t,x}} \leq  C \lambda^{-1/2} \norm{b}_{L^\infty_{t,x}}. 
\end{align}
From the last inequality \eqref{eq:grad-phi-lambda-ineq}, it follows that by choosing $\lambda$ large enough, one can obtain the bound \eqref{condition}, ensuring that Lemma \ref{lemazo} can be invoked. 

\noindent \textbf{Step 4.} 
We now detail the most important step in proving Item \ref{primeritem}, which is to carry out the so-called It\^o-Tanaka trick. Define the vector fields
\begin{align*}  
    & \widetilde{b}(t,y) := \lambda \psi_\lambda(t, \Psi_\lambda^{-1}(t,y)) + \frac{1}{2} \left ( \xi \cdot D \xi\right) (t,\Psi_\lambda^{-1}(t,y)) , \\
    & \widetilde{\xi}(t,y) := \xi(t,\Psi_\lambda^{-1}(t,y)) D \Psi_\lambda (t, \Psi_\lambda^{-1}(t,y)),
\end{align*}
and consider for $s \in [0,T],$ $y \in \mathbb{R}^n,$ the following SDE in It\^o form:
\begin{align} \label{conjugated}
    Y_t = y + \int_s^t \widetilde{b}(u, Y_u) \, \diff u + \int_s^t \widetilde{\xi}(u, Y_u) \, \diff W_u , \quad t \in [s,T].
\end{align}
We claim that equation \eqref{conjugated} is equivalent to the characteristic equation \eqref{flownotation} in the following sense: if $\phi_{s,t}(x)$ solves \eqref{flownotation}, then $Y_t := \Psi_\lambda(t,\phi_{s,t}(x))$ solves \eqref{conjugated} with $y = \Psi_\lambda(s,x)$. This can be shown by applying It\^o's Lemma to the function $\Psi_\lambda$. However, in doing so, one needs to be careful since this function does not enjoy enough regularity in time to satisfy the assumptions of the classical It\^o's Lemma. To this end, we use the following weak-in-time version of It\^o's Lemma, established in \cite{flandoli2010well} (see Lemma 3).
\begin{lemma}
\label{weakito}
Let $\alpha \in (0,1)$ and $\Psi \in L^\infty([0,\infty);C_b^{2+\alpha}(\mathbb{R}^n)),$ satisfying
\begin{align*}
   \Psi(t,x) -\Psi(s,x) = \int_s^t v(r,x) \, \diff r, \quad 0 \leq s \leq t, \quad x \in \mathbb{R}^n,
\end{align*}
with $v \in L^\infty([0,\infty);C_b^\alpha (\mathbb{R}^n)).$ Let $(X_t)_{t \geq 0}$ be a stochastic process of the form
\begin{align*}
    X_t = X_0 + \int_0^t b(s,X_s) \, \diff s + \int_0^t \xi(s,X_s) \, \diff W_s,
\end{align*}
where $b,\xi:[0,\infty) \times \mathbb{R}^n \rightarrow \mathbb{R}^n $ such that 
\begin{align*}
X_0 + \int_0^t b(s,x) \, \diff s + \int_0^t \xi(s,x) \, \diff W_s, \quad t \geq 0,
\end{align*}
is an It\^o process for all $x \in \R^n$. Then $\mathbb{P}$-a.s.
\begin{align*}
\Psi(t,X_t) &= \Psi(0,X_0) + \int_0^t \left( v + (D\Psi )^T b \right) (s,X_s)  \, \diff s \\
&\quad + \frac{1}{2} \int_0^t \text{tr}{\left( \xi^T \left( H \Psi \right) \xi \right)} (s,X_s) \, \diff s + \int_0^t \left((D \Psi \right)^T \xi)(s,X_s)  \, \diff W_s, \quad t \geq 0,
\end{align*}
where $H \Psi$ denotes the (spatial) Hessian matrix of $\Psi$.
\end{lemma}
\begin{remark}
Lemma \ref{weakito} can be extended in a straightforward manner to allow for multiple diffusion terms and applied to functions $\Psi \in L^\infty([0,\infty);C_b^{2+\alpha} (\mathbb{R}^n, \mathbb{R}^n))$ component-wise.
\end{remark}
Fix $s \geq 0$ and $x\in \mathbb{R}^n.$ Applying Lemma \ref{weakito} to $X_t := \phi_{s,t}(x)$, we obtain
\begin{align*}
\Psi_\lambda(t,X_t) &= \Psi_\lambda(s,x) + \int_s^t v(r,X_r) \, \diff r + \int_s^t  \left( (D \Psi_\lambda )^T  \left(b + \frac{1}{2} \xi^j \partial_j \xi \right) \right) (r,X_r) \, \diff r \\
&\quad + \frac{1}{2} \int_s^t \left( \tr{[ \xi^T ( H_X \Psi_\lambda ) \xi  ]} \right) (r,X_r) \, \diff r + \int_s^t (  (D \Psi_\lambda )^T \xi  ) (r,X_r) \, \diff W_r, 
\end{align*}
where $t \geq s$ and
\begin{align*}
    v &:= \lambda \psi_\lambda - \frac12 \xi^i \xi^j \partial_i \partial_j \psi_\lambda - \left(b^j + \frac{1}{2} \xi^i \partial_i \xi^j \right) \partial_j\psi_\lambda -b, \\
    \left (D \Psi_\lambda \right)^T
    \left(b + \frac{1}{2}\xi^j \partial_j \xi \right) 
    &= b + b^j \partial_j \psi_\lambda +  \frac{1}{2}\xi^j \partial_j \xi  +  \frac{1}{2}\xi^i \partial_i \xi^j \partial_j \psi_\lambda, \text{ and }\\
    \frac{1}{2} \text{tr}\left( \xi^T \left( H_X \Psi_\lambda \right) \xi \right) &= \frac{1}{2} \xi^i \xi^j \partial_i \partial_j \psi_\lambda,
\end{align*}
which yields \eqref{conjugated}. We note that $\widetilde{b},\widetilde{\xi} \in L^\infty([0,T];C_b^{1+\alpha}(\mathbb{R}^n,\mathbb{R}^n))$ and by the classical results in \cite{kunita1997stochastic}, equation \eqref{conjugated} has a unique strong solution which defines a flow of $C^{1+\alpha'}$-diffeomorphisms $\varphi_{s,t}$, $0\leq s \leq t \leq T,$ for any $\alpha' \in (0,\alpha).$ As in \cite{flandoli2010well}, owing to the regularity shown for $\psi_\lambda$ in Step 1 and for $\Psi_\lambda$ in Step 2, and using the fact that $\phi_{s,t} = \Psi_\lambda^{-1} \circ \varphi_{s,t} \circ \Psi_\lambda,$ we conclude that $\phi_{s,t}$ enjoys the extra regularity specified in Theorem \ref{flowfla} and the solution to \eqref{flownotation} is unique, which proves Item \ref{primeritem}.

\noindent{\bf{Proof of Item \ref{segundoitem}}.} The argument follows almost verbatim the techniques in \cite{flandoli2010well}, so we omit it.

\noindent{\bf{Proof of Item \ref{terceritem}.}}
The following lemma can be derived as a special case of Corollary 5.3 in \cite{leahy2015some}, where H\"older properties of flows of SDEs with stochastic coefficients are analysed. Our lemma is obtained by setting $l=0,$ $\delta=\alpha,$ $X_s := \sup_{s\leq t \leq T} \lvert D \phi_{s,t}(x)\rvert,$ for $s \in [0,T]$ and $x \in \R^n,$ applying Corollary 5.3 in \cite{leahy2015some} with these constants, and then taking suprema in $0 \leq s \leq T.$ We note that the inequality shown in Corollary 5.3 in \cite{leahy2015some} is also valid for the suprema since their constant $N$ is independent of $X$ and our constant $C$ in Lemma \ref{JM-lemma} is independent of $s.$
\end{proof}

\begin{lemma}\label{JM-lemma}
Let $\alpha \in (0,1]$ and $r\geq 2$ be a constant satisfying $\alpha r > n$. Moreover, assume that there exists a constant $C>0$ such that
\begin{align}
    &\sup_{x \in \mathbb R^n} \sup_{0\leq s \leq T} \mathbb{E}\left[\sup_{s \leq t \leq T} \lvert D \phi_{s,t}(x)\rvert^r\right] \leq C, \label{eq:D-phi-Lp}\\
    &\sup_{0\leq s \leq T} \mathbb{E}\left[\sup_{s \leq t \leq T} \lvert D \phi_{s,t}(x) - D \phi_{s,t}(y)\rvert^r\right] \leq C\lvert x-y\rvert ^{\alpha r}, \quad x,y \in \R^n. \label{eq:D-phi-Lp-difference}
\end{align}
Then for any $\alpha' \in (0, \alpha - n/r)$ and $\epsilon > n/r$, property \eqref{eq:JM-estimate} holds.
\end{lemma}

Hence, in order to establish \eqref{eq:JM-estimate}, it suffices to verify estimates \eqref{eq:D-phi-Lp} and \eqref{eq:D-phi-Lp-difference}, which we claim to be true in our setting. Since the proof of this claim is technical and lengthy, we include it in Appendix \ref{twoproperties}.

\begin{remark} \label{inversa}
We note that Items \ref{segundoitem} $\&$ \ref{terceritem} in Theorem \ref{flowfla} also hold for the inverse flow $\phi_{s,t}^{-1}$. Indeed, by using the flow properties (and {\em formal} Stratonovich notation for simplicity), we observe that
\begin{align*} 
    \phi_{s,t}(\phi_{s,t}^{-1}(y)) = \phi_{s,t}^{-1}(y) + \int_s^t b(r,\phi_{s,r}(\phi_{s,t}^{-1}(y))) \, \diff r + \sum_{i=1}^N \int_s^t \xi_i(r,\phi_{s,r}(\phi_{s,t}^{-1}(y))) \circ \diff W_r^i,
\end{align*}
which can be rewritten as
\begin{align*} 
    \phi_{s,t}^{-1}(y) = y - \int_s^t b(r,\phi_{r,t}^{-1}(y)) \, \diff r - \sum_{i=1}^N \int_s^t \xi_i(r,\phi_{r,t}^{-1}(y)) \circ \diff W_r^i,
\end{align*}
where $0 \leq s \leq t,$ $y \in \mathbb{R}^n,$ so for fixed $t \in [0,T],$ the inverse flow satisfies a backward SDE with a sign change in the drift and diffusion coefficients and hence enjoys similar properties.
\end{remark}
We present the following important consequence of Theorem \ref{flowfla} \eqref{terceritem}. 
\begin{corollary} \label{consequences}
Let $R>0$ and $r \geq 1.$ We have

\begin{align} 
    & \lim_{n \rightarrow \infty}  \mathbb{E} \left [ \int_0^T \int_{B(0,R)} \lvert \phi_t^{n}(x)-\phi_t(x)\rvert^r \diff^n x \, \diff t \right]   = 0, \label{weakflow} \\ 
    & \lim_{n \rightarrow \infty}   \mathbb{E} \left [ \int_0^T \int_{B(0,R)} \lvert D \phi_t^{n}(x)- D \phi_t(x)\rvert^r \diff^n x \, \diff t \right]  = 0. \label{weakjacobian}
\end{align}
\end{corollary}
Finally, we state a remarkable result shown in \cite{flandoli2010well} (see Theorem 11) that translates easily into our context.
\begin{theorem} \label{regularidad1}
Let the drift and diffusion vector fields satisfy Hypothesis \ref{hypoflow} and let $r>q\geq 2$. If $\div{(b)} \in L^{r}([0,T] \times \mathbb{R}^n)$, then $\mathbb{P}$-a.s. $J \phi \in L^{q}([0,T];W^{1,q}_{loc}(\mathbb{R}^n,\mathbb{R}^n)).$
\end{theorem}

\begin{proof}
The result follows by a simple modification of the proof of Theorem 11 in \cite{flandoli2010well} (note that the employed parabolic estimates in Sobolev spaces also hold for our general system by Theorem 9, Chapter 7 in \cite{sobolevkrylov}).
\end{proof}

In the following result, we show that if the drift and diffusion vector fields of \eqref{flownotation} are divergence-free {\em in the sense of distributions}, the stochastic flow constructed in Theorem \ref{flowfla} is a.s. volume-preserving. This is proven in \cite{flandoli2010well} in the case $\xi_i := e_i,$ $i=1, \ldots,n,$ and we follow exactly the same argument.
\begin{theorem}
\label{volume} 
Let $b$ and $\xi_i,$ $i=1,\ldots,N$ satisfy Hypothesis \ref{hypoflow}. Assume that $\div{(b)}=\div{(\xi_i)} = 0,$ $i=1,\ldots,N,$ in the sense of distributions. Then $\mathbb{P}$-a.s. the flow of \eqref{flownotation} is volume-preserving, i.e. $J\phi _{s,t}(x)=1$, for $0\leq s\leq
t\leq T$ and $x\in {\mathbb{R}}^{n}$.
\end{theorem}

\begin{proof}
We assume without loss of generality that $N=1.$ Let $\{b^j\}_{j = 1}^\infty$ be a sequence of vector fields such that $\div{(b^j)}= 0,$ $b^j \rightarrow b$ in $L^\infty\left([0,T]; C_b^\alpha(\mathbb R^n, \mathbb R^n) \right)$, $j \rightarrow \infty$ (these can be constructed as in the proof of Lemma 9 in \cite{flandoli2010well}). Let $\{\phi^j_{s,t}\}_{j=1}^\infty$ be the corresponding sequence of associated stochastic flows (with fixed diffusion vector field $\xi$). Applying Theorem II.3.1 in \cite{kunita1984stochastic}, we obtain $\mathbb{P}$-a.s. $J\phi^j _{s,t}(x)=1$, for all $0\leq s\leq
t\leq T,$ $x\in {\mathbb{R}}^{n},$ and $j \in \mathbb{N}$. Fix $s\in (0,T]$ and $x\in {\mathbb{R}}^{n}.$ By taking into account \eqref{phi-deriv-convergence}, there existence a subsequence $\{D\phi_{s,t}^{j_k}\}_{k=1}^\infty$ (depending on $(s,x)$) such that $\mathbb{P}$-a.s.
\begin{align*}
\lvert D\phi _{s,t}^{j_k}(x)-D\phi _{s,t}(x) \rvert
^{2} \rightarrow 0, \quad j_k \rightarrow \infty.
\end{align*}
Since $J\phi$ is well-defined, we must have $J \phi_{s,t}(x) = 1$. The proof is complete.
\end{proof}

\section{Existence of stochastic solutions} \label{existence}
In this section, we prove our main theorem of existence of solutions to equation \eqref{eq0}, namely Theorem \ref{maintheoremforms}. Only sometimes throughout the rest of this paper (when the notation becomes especially convoluted in a proof), the $L^r$ norms of measurable objects defined on the measure spaces $(\Omega,\mathcal{F},\mathbb{P}),$ $([0,T],\lambda([0,T]),\lambda),$ and $(\mathbb{R}^n,\lambda(\mathbb{R}^n),\lambda),$ 
where $\lambda(\cdot)$ represents the Lebesgue sigma-algebra and $\lambda$ the Lebesgue measure, may be denoted by $\norm{\cdot}_{L^r_\omega},$ $\norm{\cdot}_{L^r_t},$ and $\norm{\cdot}_{L^r_x},$ respectively. We may also use notations of the type $\norm{\cdot}_{L^r_x,L^q_t}$ for the mixed norms in the product spaces. 

\subsection{Existence for regular drift and initial condition}

First, we establish the existence of weak solutions in the case when the drift vector field $b$ is smooth in space. Then we extend these results to the case when $b$ is only $\alpha$-H\"older continuous and satisfies Hypothesis \ref{hypoexistence} by means of a limiting argument. 
\begin{proposition}[Existence of classical solutions for regular drift and initial condition]  \label{strong-sol-smooth}
Enforce Assumption \ref{assumpstrong} with $m=3$ and $\beta \in (0,1),$ and let $\mathbf{K}_0 \in C^2(\R^n,\bigwedge^k \mathcal{T}^*\mathbb R^n)$. Then there exists a strong solution to equation \eqref{eq0} of the form $\mathbf{K}_t(x) := (\phi_t)_* \mathbf{K}_0(x)$, where $\phi_t$ is the flow of \eqref{flowmap}.
\end{proposition}
\begin{proof}
For the sake of simplicity, we assume without loss of generality that $N=1$. By Assumption \ref{assumpstrong} with $m=3$, $\phi_t$ is a stochastic flow of $C^3$-diffeomorphisms. 

Applying Theorem \ref{KIW}, we have
\begin{align*} 
    &(\phi_t)_* \mathbf{K}_0(x) =  \mathbf{K}_0(x) - \int^t_0 \mathcal L_b [(\phi_s)_* \mathbf{K}_0] (x) \,  \diff s 
    - \int^t_0 \mathcal L_{\xi} [(\phi_s)_* \mathbf{K}_0](x) \, \diff W_s \nonumber \\
    &\hspace{130pt} + \frac12 \int^t_0 \mathcal L_{\xi} \mathcal L_{\xi} [(\phi_s)_* \mathbf{K}_0](x) \,  \diff s,
\end{align*}
for $t \in [0,T]$ and $x \in \mathbb{R}^n.$ Hence $\mathbf{K}_t$ solves \eqref{eq0} in the sense of Definition \ref{classicalsol}. For this, we note that $\mathbb{P}$-a.s., $\mathbf{K}$ has paths of class $C([0,T];C^2(\mathbb{R}^n,\bigwedge^k \mathcal{T}^*\mathbb R^n))$, since by definition,
\begin{align}  \label{repr}
     ((\phi_t)_* \mathbf{K}_0)_{j_1,\ldots,j_k}(x) = (K_0)_{i_1,\ldots,i_k}(\psi_t(x)) \frac{\partial \psi_t^{i_1}}{\partial x_{j_1}}(x) \cdots \frac{\partial \psi_t^{i_k}}{\partial x_{j_k}}(x)
\end{align}
holds in global coordinates, where $\psi_t := \phi_t^{-1}$ is the inverse flow and summation over repeated indices is assumed.
\end{proof}
Next, we prove the existence of weak solutions of the linear transport of $k$-forms \eqref{eq0} when the drift vector field $b$ is smooth in space and the initial condition is of low regularity. First we establish some useful preliminary lemmas. In the following, we let $J\phi$ represent the {\em determinant} of the Jacobian matrix of a diffeomorphism $\phi$.

\subsection{Auxiliary results}
\begin{lemma} \label{coord-change}
Given a $k$-form $\mathbf{K} \in L^1_{loc}(\R^n,\bigwedge^k \mathcal{T}^*\mathbb R^n)$, a test $k$-form $\mathbf{\Theta} \in C_0^\infty (\R^n,\bigwedge^k \mathcal{T}^*\mathbb R^n),$ and a $C^1$-diffeomorphism $\phi$, we have the identity
\begin{align*}
    \int_{\mathbb R^n} (  \phi_* \mathbf{K}(x), \mathbf{\Theta}(x) )  \, \diff^n x = \int_{\mathbb R^n} \left( \mathbf{K}(y), \lvert J\phi(y)\rvert \left(\phi^*\mathbf{\Theta}^\sharp(y)\right)^\flat \right) \diff^n y,
\end{align*}
where $\sharp$ and $\flat$ are the sharp and flat operators, respectively (see Appendix \ref{tensor} for the definitions).
\end{lemma}
\begin{remark}
Since we are only using the Euclidean metric, for a $k$-form with expression $\mathbf{\Theta}(x) = \Theta_{i_1,\ldots,i_k}(x) \diff x_{i_1}(x) \wedge \ldots  \wedge \diff x_{i_k}(x)$, the corresponding $k$-vector field $\mathbf{\Theta}^\sharp$ reads
\begin{align*}
    \mathbf{\Theta}^\sharp = \Theta_{i_1,\ldots,i_k}(x) \frac{\partial}{\partial x_{i_1}}(x) \wedge \ldots \wedge \frac{\partial}{\partial x_{i_k}}(x),
\end{align*}
and similarly for a $k$-vector field $\mathbf{V}(x) = V^{i_1,\ldots,i_k}(x) (\partial/\partial x_{i_1})(x) \wedge \ldots \wedge (\partial / \partial x_{i_k})(x)$, the corresponding $k$-form $\mathbf{V}^\flat$ reads
\begin{align*}
    \mathbf{V}^\flat =  V^{i_1,\ldots,i_k}(x) \diff x_{i_1}(x) \wedge \ldots  \wedge \diff x_{i_k}(x),
\end{align*}
where $x \in \mathbb{R}^n$ and we assume summation over repeated indices as usually.
\end{remark}
\begin{proof}
By standard tensor identities, we have
\begin{align*}
    (\phi_* \mathbf{K}(x), \mathbf{\Theta}(x) ) = \left<\phi_* \mathbf{K}(x), \mathbf{\Theta}^\sharp(x) \right> = \left<\mathbf{K}(\phi^{-1}(x)), \phi^*\mathbf{\Theta}^\sharp(\phi^{-1}(x))\right>.
\end{align*}
Hence, by the change of coordinates $y = \phi^{-1}(x)$, we obtain
\begin{align*}
    & \int_{\mathbb R^n} (\phi_* \mathbf{K}(x), \mathbf{\Theta}(x)) \, \diff^n x = \int_{\mathbb R^n} \left<\mathbf{K}(\phi^{-1}(x)) , \phi^*\mathbf{\Theta}^\sharp(\phi^{-1}(x))\right>  \, \diff^n x \\
    &= \int_{\mathbb R^n} \left<\mathbf{K}(y),  \phi^*\mathbf{\Theta}^\sharp(y) \right> \lvert J\phi(y)\rvert \, \diff^n y = \int_{\mathbb R^n} \left( \mathbf{K}(y), \lvert J\phi(y) \rvert (\phi^*\mathbf{\Theta}^\sharp(y))^\flat \right ) \diff^n y.
\end{align*}
\end{proof}
\begin{remark} \label{remark:tilde-notation}
When changing coordinates in Lemma \ref{coord-change}, we observe that the term $\lvert J\phi(y)\rvert (\phi^*\mathbf{\Theta}^\sharp(y))^\flat$ appears. For simplicity in our future calculations, we will often denote it by $\widetilde{\mathbf{\Theta}}.$
\end{remark}
We introduce a lemma that we will employ several times.
\begin{lemma}\label{newbound}
Let $\mathbf{\Theta} \in C_0^\infty (\R^n,\bigwedge^k \mathcal{T}^*\mathbb R^n)$, $\phi_t$ be a flow of $C^1$-diffeomorphisms. Then $\mathbb{P}$-a.s. we have the following properties
\begin{enumerate}
\item $\widetilde{\mathbf{\Theta}}_t$ (see previous remark for the notation used) has compact support, uniformly in $t \in [0,T]$. \label{theta-tilde-support}
\item $\|\widetilde{\mathbf{\Theta}}\|_{L^\infty_{t,x}}< \infty$.
\label{theta-tilde-bound}
\end{enumerate}
\end{lemma}
\begin{proof}
Let $R>0$ such that the support of $\mathbf{\Theta}$ is contained in the ball $B(0,R)$. Since $\mathbb{P}$-a.s., the map $\phi_{\cdot}^{-1}(\cdot)$ is continuous in both time and space (and noting that the image of a compact set under a continuous map is compact), we can choose $R^*:\Omega \rightarrow (0,\infty]$ \footnote{More concretely, $R^* := \sup_{t \in [0,T], x \in B(0,R) }\lvert \phi_t^{-1}(x)\rvert,$ which is measurable and hence a random variable.} such that $R^* < \infty$ $a.s.$ and
\begin{align}\label{eq:r-star}
    & \{\phi_t^{-1}(y):\hspace{0.1cm} t \in [0,T], \hspace{0.1cm} y \in B(0,R) \} \subset B(0,R^*).
\end{align}
This implies that the support of $\widetilde{\mathbf{\Theta}}$ is contained in $B(0,R^*)$, which proves \ref{theta-tilde-support}. We observe that by definition, we have the explicit expression (assume summation over repeated indices)
\begin{align} \label{pullback} 
     (\phi_t^* \mathbf{\Theta}^\sharp)^{j_1,\ldots,j_k}(x) = \Theta_{i_1,\ldots,i_k}(\phi_t(x)) \frac{\partial \psi_t^{i_1}}{\partial x_{j_1}}(\phi_t(x)) \cdots \frac{\partial \psi_t^{i_k}}{\partial x_{j_k}}(\phi_t(x))
\end{align}
for the pull-back of the $k$-vector field $\mathbf{\Theta}^\sharp$ in global coordinates. Since $\phi_t$ is a flow of $C^1$-diffeomorphisms, both its (first order) derivatives and the derivatives of the inverse flow are $\mathbb{P}$-a.s. continuous in time and space. Hence $\mathbb{P}$-a.s. $\widetilde{\mathbf{\Theta}}_t$ is uniformly bounded on $B(0,R^*)$, since every continuous function on a compact set reaches a maximum. \footnote{We note that in our case, since we consider the Euclidean metric, the $k$-vector field $\mathbf{\Theta}^\sharp$ has the same coordinate expression as $\mathbf{\Theta}$, and the $k$-form $(\phi_t^* \mathbf{\Theta}^\sharp)^\flat$ has the same coordinate expression as $\phi_t^*\mathbf{\Theta}^\sharp.$} 
We can therefore conclude \ref{theta-tilde-bound}.
\end{proof}
\begin{lemma}  \label{nuevo}
Let $\mathbf{\Theta} \in C_0^\infty (\R^n,\bigwedge^k \mathcal{T}^*\mathbb R^n)$ and $\phi_t$ be a stochastic flow of $C^2$-diffeomorphisms. Moreover, assume that the drift vector field $b$ satisfies Hypothesis \ref{hypoexistence} \ref{maincondition} and $\xi \in C([0,T];C^2(\mathbb{R}^n,\mathbb{R}^n))$. Then $\mathbb{P}$-a.s. 
\begin{enumerate}
\item $\widetilde{\mathcal{L}_b^* \mathbf{\Theta}}(t)$ has compact support, uniformly in $t \in [0,T],$ and $\|\widetilde{\mathcal{L}_b^* \mathbf{\Theta}}(t)\|_{L_x^q}$ is of class $L^q([0,T])$. 
\label{lie-tilde-bound}
\item $\widetilde{\mathcal{L}_\xi^* \mathbf{\Theta}}(t)$ has compact support, uniformly in $t \in [0,T],$ and $\|\widetilde{\mathcal{L}_\xi^* \mathbf{\Theta}}(t)\|_{L_x^\infty}$ is of class $L^\infty([0,T])$.
\label{lie-tilde-bound-xi}
\item $\widetilde{\mathcal{L}_\xi^* \widetilde{\mathcal{L}_\xi^* \mathbf{\Theta}}}(t)$ has compact support, uniformly in $t \in [0,T],$ and $\|\widetilde{\mathcal{L}_\xi^* \widetilde{\mathcal{L}_\xi^* \mathbf{\Theta}}}(t)\|_{L_x^\infty}$ is of class $L^\infty([0,T]).$ 
\label{lie-tilde-bound-xi2}
\end{enumerate}
\end{lemma}
\begin{proof}
Let $\mathbf{\Theta}$ be supported on $B(0,R).$ By Hypothesis \ref{hypoexistence} \ref{maincondition}, $\mathcal{L}^*_b \mathbf{\Theta}$ is well-defined, and therefore
\begin{align*}
\widetilde{\mathcal{L}^*_b \mathbf{\Theta}}(t,x) = \lvert J \phi_t(x)\rvert (\phi_t^* (\mathcal{L}^*_b \mathbf{\Theta})^\sharp )^\flat(x), \quad t \in [0,T], \quad x \in \mathbb{R}^n, 
\end{align*}
is also well-defined. Hence, by the arguments in the proof of Lemma \ref{newbound} \ref{theta-tilde-support}, it must be supported on $B(0,R^*)$, where $R^*$ is defined as in \eqref{eq:r-star}. By a similar argument, one can check that $\widetilde{\mathcal{L}_\xi^* \mathbf{\Theta}}(t)$ and $\widetilde{\mathcal{L}_\xi^* \widetilde{\mathcal{L}_\xi^* \mathbf{\Theta}}}(t)$ are well-defined (take into account the regularity assumed on $\phi_t$ and $\xi$). Moreover, $\widetilde{\mathcal{L}_\xi^* \mathbf{\Theta}}(t)$ is supported on $B(0,R^*)$, and $\widetilde{\mathcal{L}_\xi^* \widetilde{\mathcal{L}_\xi^* \mathbf{\Theta}}}(t)$ is supported on $B(0,R^{**}),$ where $R^{**}:\Omega \rightarrow (0,\infty]$ satisfies $R^{**} < \infty$ $a.s.$ and is chosen such that
\begin{align*}
    & \{\phi_t^{-1}(y):\hspace{0.1cm} t \in [0,T], \hspace{0.1cm} y \in B(0,R^*) \} \subset B(0,R^{**}).
\end{align*}
Moreover, by applying H\"older's inequality, we can obtain 
\begin{align} \label{primeracota}
    \norm{\widetilde{\mathcal{L}^*_b \mathbf{\Theta}}(t)}_{L^q_x} \lesssim \norm{J \phi_t}_{L^\infty_{R^*}} \norm{D \psi_t}^k_{L^\infty_{R^*}} \norm{\mathcal{L}^*_b \mathbf{\Theta}}_{L^q_{R^*}},
\end{align}
where we have employed the notation $\psi_t := \phi_t^{-1}$ and used the explicit formula for the pull-back of a $k$-vector field \eqref{pullback}. Since $R^* < \infty$ a.s., we have a.s. $\norm{J \phi_{\cdot}}_{L^\infty_{R^*}} \in L^\infty([0,T]),$ $\norm{D \psi_{\cdot}}_{L^\infty_{R^*}} \in L^\infty([0,T]),$ and by Hypothesis \ref{hypoexistence} \ref{maincondition} $\norm{\mathcal{L}^*_b \mathbf{\Theta}}_{L^q_{R^*}} \in L^q([0,T]),$ so \eqref{primeracota} is of class $L^q([0,T])$ and we conclude \ref{lie-tilde-bound}. For \ref{lie-tilde-bound-xi}, the argument is similar by noting that a.s., $\norm{\mathcal{L}^*_\xi \mathbf{\Theta}}_{L^\infty_{R^*}} \in L^\infty([0,T])$ holds. Finally, for \ref{lie-tilde-bound-xi2}, observe that
\begin{align}  \label{cotapesada}
    \norm{\widetilde{\mathcal{L}_\xi^* \widetilde{\mathcal{L}_\xi^* \mathbf{\Theta}}}(t)}_{L^\infty_x} \lesssim \norm{J \phi_t}_{L^\infty_{R^{**}}} \norm{D \psi_t}^k_{L^\infty_{R^{**}}} \norm{\mathcal{L}_\xi^*  \widetilde{\mathcal{L}^*_\xi \mathbf{\Theta}}}_{L^\infty_{R^{**}}},
\end{align}
where $\mathcal{L}_\xi^* \widetilde{\mathcal{L}^*_\xi \mathbf{\Theta}}(t,x) = \mathcal{L}_\xi^* (\lvert J \phi_t\rvert (\phi_t^* (\mathcal{L}^*_\xi \mathbf{\Theta})^\sharp )^\flat)(x).$ By taking into account the definition of the pull-back of a vector field \eqref{pullback} and of the adjoint Lie derivative \eqref{Lieformulaadjoint}, we observe that $\mathcal{L}_\xi^* \widetilde{\mathcal{L}^*_\xi \mathbf{\Theta}}$ is expressed in coordinates as a sum that only contains products where at most two partial derivatives of the flow $\phi_t,$ its inverse $\psi_t,$ and $\xi$ are involved. By our assumptions, these are locally uniformly bounded in space and time, so for $\mathbb{P}$-a.s. $\omega$ such that $R^{**}(\omega) < \infty,$ $\mathcal{L}_\xi^* \widetilde{\mathcal{L}^*_\xi \mathbf{\Theta}}(t,x)$ must be uniformly bounded in space and time. Therefore \eqref{cotapesada} is a.s. of class $L^\infty([0,T])$ and the proof is complete.
\end{proof}
A final short lemma will be needed.
\begin{lemma}
Let Hypothesis \ref{hypoflow} and Hypothesis \ref{hypoexistence} \ref{newcondition} hold and $\phi_t$ be the unique stochastic flow of $C^1$-diffeomorphisms defined by \eqref{flowmap}. Given $\epsilon >0,$ let $b^\epsilon$ and $\xi_i^\epsilon,$ $i=1,\ldots,N$ be the mollifications of $b$ and $\xi_i,$ $i=1,\ldots,N,$ respectively, and $\phi_t^\epsilon$ be the flow of \eqref{flowmap} with drift vector field $b^\epsilon$ and diffusion vector fields $\xi_i^\epsilon,$ $i=1,\ldots,N$. Then, for any $R>0$ and $r \geq 2$, we have the following
\begin{align} \label{otracosa1}
& \sup_{\epsilon >0} \sup_{x \in \mathbb R^n} \mathbb E\left[\sup_{0\leq t\leq T} [J \phi^\epsilon_t(x)]^{-r} \right] < \infty, \\
\label{otracosa2}
& \sup_{x \in \mathbb R^n} \mathbb E\left[\sup_{0\leq t\leq T} [J \phi_t(x)]^{-r} \right] < \infty.
\end{align}
\end{lemma}
\begin{proof}
We assume without loss of generality that $N=1.$ Let $r \geq 2.$ For $\epsilon >0,$ define $f^\epsilon_t(x):=[J \phi^\epsilon_t(x)]^{-r},$ $f_t(x):=[J \phi_t(x)]^{-r},$ and
\begin{align*}
    & g_t^\epsilon(x):=\exp{\left( -r \int_0^t \div \, b^\epsilon(s,\phi^\epsilon_s(x)) \, \diff s - r \int_0^t \div \, \xi^\epsilon(s,\phi^\epsilon_s(x)) \, \diff W_s \right)},
\end{align*}
for $t \in [0,T], x \in \R^n$. We note that $f^\epsilon$ and $f$ are well-defined since $\mathbb{P}$-a.s. $J \phi^\epsilon,$ $J\phi > 0 $ because $\phi_t^\epsilon,$ $\phi_t$ (respectively) are orientation-preserving diffeomorphisms for all $t \in [0,T]$. By classical results, we know that for a.e. $(\omega,t,x) \in \Omega \times [0,T] \times \R^n$ 
\begin{align*}
    & \log J \phi_t(x) = \int_0^t \div \, b^\epsilon(s,\phi^\epsilon_s(x)) \, \diff s + \int_0^t \div \, \xi^\epsilon(s,\phi^\epsilon_s(x)) \, \diff W_s,
\end{align*}
so $\text{log} \, f^\epsilon_t(x) = \text{log} \, g^\epsilon_t(x)$ and hence $f^\epsilon_t(x) = g^\epsilon_t(x)$. Moreover, by the mollifier properties and Hypothesis \ref{hypoexistence} \ref{newcondition}, we have $ -\div \, b^\epsilon \leq -C.$ Indeed
\begin{align*}
    -\div \, b^\epsilon (t,x) &= \int_{\R^n} -\div \, b(t,x-y) \rho^\epsilon(y) \, \diff^n y \leq  \int_{\R^n} - C \rho^\epsilon(y) \, \diff^n y \\
    &= -C \int_{\R^n} \rho^\epsilon(y) \, \diff^n y = -C ,
\end{align*}
for all $t \in [0,T]$ and $x \in \R^n.$ Also, since $D\xi \in L^\infty([0,T] \times \R^n ; \R^n)$ by Hypothesis \ref{hypoflow} \ref{bcondition}, and taking into account the mollifier properties, there exists $C'>0$ such that the uniform bound $ \lvert \div \, \xi^\epsilon\rvert \leq C'$ holds. Hence, by setting $X_s^\epsilon(x) := - \int_0^s \div \, \xi^\epsilon(u,\phi_u^\epsilon(x)) \, \diff W_u$ for $s \in [0,T]$ and $x \in \R^n$, we obtain 
\begin{align}
\sup_{0 \leq s \leq t} g_s^\epsilon(x) &\leq \exp{\left(r \lvert C\rvert T +  r \sup_{0 \leq s \leq t}  X_s^\epsilon(x) \right)}  \nonumber \\
& \lesssim \sup_{0 \leq s \leq t} \exp{\left(X_s^\epsilon(x) \right)^{r}}, \quad t \in [0,T], \quad x \in \R^n, \label{divergence}
\end{align}
where the constant in the last inequality above is independent of $\epsilon.$ Applying Lemma \ref{weakito}, the diffusion coefficient of the process $\exp{(X_s^\epsilon(x))}$ is $\exp{(X_u^\epsilon(x))} \, \div \, \xi^\epsilon(u,\phi_u^\epsilon(x)),$ and therefore its quadratic variation can be calculated and bounded as
\begin{align*}
    \left[ \exp{\left( X_s^\epsilon(x) \right)} \right]_t &= \frac{1}{2} \int_0^t \left(\exp{(X_s^\epsilon(x))} \, \div \, \xi^\epsilon(s,\phi_s^\epsilon(x))\right)^2 \diff s \\
    &\leq  \frac{(C')^2}{2} \int_0^t \sup_{0 \leq u \leq s} \exp{(X_u^\epsilon(x))}^2 \, \diff s.
\end{align*}
Hence by Burkholder-Davis-Gundy and H\"older's inequalities, we have
\begin{align*}
\E \left[ \sup_{0 \leq s \leq t} \exp{\left( X_s^\epsilon(x) \right)^{r}} \right] &\lesssim \E \left[ \left[ \exp{\left(X_s^\epsilon(x) \right)} \right]^{r/2}_t \right] \\
&\lesssim \int_0^t \E \left[  \sup_{0 \leq u \leq s} \exp{(X_u^\epsilon(x))}^r \right] \diff s  , \quad t \in [0,T], \quad x \in \R^n,
\end{align*}
where the constants in the inequalities above are independent of $\epsilon.$ By applying Gr\"onwall's inequality, we find that \eqref{divergence} is finite, independently of $x,$ so by taking suprema in time and space we obtain \eqref{otracosa1}.

To prove \eqref{otracosa2}, we first note that by taking into account \eqref{weakjacobian}, we obtain that for any $R>0,$ $J \phi^\epsilon \rightarrow J \phi$ in $L^2(\Omega \times [0,T] \times B(0,R); \bigwedge^k \mathcal{T}^*\mathbb R^n)$ when $\epsilon \rightarrow 0,$ and hence there exists a subsequence (which we call again $\epsilon$) such that for a.e. $(\omega,t,x) \in \Omega \times [0,T] \times \R^n,$ $f_t^\epsilon(x) \rightarrow f_t(x)$ when $\epsilon \rightarrow 0.$ Moreover, since $\mathbb{P}$-a.s. $f$ is continuous in time and space, then for every $x \in \R^n,$ the map $f_{\cdot}(x)$ is continuous and hence reaches a maximum $f_{t^*_x}(x) = \sup_{0\leq t \leq T} f_t(x)$. 
Hence by applying the bound \eqref{divergence} with $t=T$ and invoking Fatou's Lemma, we have
\begin{align*}
   \E \left[ \sup_{0\leq t \leq T} f_t(x) \right] &= \E \left[ f_{t^*_x}(x) \right] = \E \left[ \lim_{\epsilon \rightarrow 0} f_{t^*_x}^\epsilon(x) \right] = \E \left[ \lim_{\epsilon \rightarrow 0} g_{t^*_x}^\epsilon(x) \right] = \E \left[ \liminf_{\epsilon \rightarrow 0} g_{t^*_x}^\epsilon(x) \right]\\
   &\stackrel{\eqref{divergence}}{\lesssim} \E \left[ \liminf_{\epsilon \rightarrow 0} \sup_{0 \leq t \leq T} \exp{\left(X_t^\epsilon(x) \right)^{r}} \right] \\
   &\stackrel{(\text{Fatou})}{\lesssim} \liminf_{\epsilon \rightarrow 0} \E \left[ \sup_{0 \leq t \leq T} \exp{\left(X_t^\epsilon(x) \right)^{r}} \right].
\end{align*}
Finally, by \eqref{otracosa1} we have $\liminf_{\epsilon \rightarrow 0} \E \left[ \sup_{0 \leq t \leq T} \exp{\left(X_t^\epsilon(x) \right)^{r}} \right] < \infty$ independently of $x.$ The proof is complete.
\end{proof}

\subsection{Existence for regular drift and weak initial condition}

\begin{proposition}[Existence of $L^p$-solutions for regular drift and weak initial condition] \label{weak-sol-smooth}
Let $p \geq 2$ and $\mathbf{K}_0 \in L^p (\R^n,\bigwedge^k \mathcal{T}^*\mathbb R^n)$. Let $b,$ $\xi_i,$ $i=1,\ldots,N$ satisfy Assumption \ref{assumpstrong} with $m=3$ and $\beta \in (0,1)$ and Hypothesis \ref{hypoexistence} \ref{newcondition}. Then there exists an $L^p$-solution to equation \eqref{eq0} of the form $\mathbf{K}_t(x) := (\phi_t)_* \mathbf{K}_0(x),$ where $\phi_t$ is the flow of \eqref{flowmap}.
\end{proposition}
\begin{proof}
{\bf Step 1. $\mathbf{K}$ satisfies \eqref{eq0} weakly.} We assume without loss of generality that $N=1$. Consider the mollified initial datum $\mathbf{K}_0^\epsilon := \rho^\epsilon * \mathbf{K}_0$. By the embedding $L^p_{loc} \subset L^q_{loc}$ for $p \geq q$ and the mollifier properties, $\mathbf{K}_0^\epsilon \rightarrow \mathbf{K}_0,$ $\epsilon \rightarrow 0,$ strongly in $L^r_{loc},$ for any $r \in [1,p],$ and therefore also weakly. Moreover, by Proposition \ref{strong-sol-smooth}, $\mathbf{K}_t^\epsilon(x) := (\phi_t)_*\mathbf{K}_0^\epsilon(x)$ is a strong solution to \eqref{eq0} with initial datum $\mathbf{K}_0^\epsilon$ and must satisfy the equation weakly. Hence, for any test $k$-form $\mathbf{\Theta} \in C^\infty_0 (\R^n,\bigwedge^k \mathcal{T}^*\mathbb R^n),$ we have $\mathbb{P}$-a.s. 
\begin{align*}
    & \llangle (\phi_t)_*\mathbf{K}_0^\epsilon, \mathbf{\Theta} \rrangle_{L^2} - \llangle \mathbf{K}_0^\epsilon, \mathbf{\Theta} \rrangle_{L^2} = -\int^t_0 \llangle (\phi_s)_*\mathbf{K}_0^\epsilon, \mathcal L_b^* \mathbf{\Theta} \rrangle_{L^2} \diff s \\
    & \hspace{70pt} - \int^t_0 \llangle (\phi_s)_*\mathbf{K}_0^\epsilon, \mathcal L_\xi^* \mathbf{\Theta} \rrangle_{L^2} \diff W_s
    + \frac12 \int^t_0 \llangle (\phi_s)_*\mathbf{K}_0^\epsilon, \mathcal L_\xi^* \mathcal L_\xi^* \mathbf{\Theta} \rrangle_{L^2} \diff s,
\end{align*}
for all $t \in [0,T].$ We also note that $\llangle (\phi_t)_*\mathbf{K}_0^\epsilon, \mathbf{\Theta} \rrangle_{L^2}$ is a.s. continuous in time, since from Theorem \ref{flowfla}, $\mathbb{P}$-a.s. $\psi_t$ and its spatial derivatives are continuous in $(t,x).$  We analyse term-by-term convergence in the last equality for fixed $t \in [0,T]$. 
\begin{itemize}
\item By the weak convergence $\mathbf{K}_0^\epsilon \rightarrow \mathbf{K}_0$ in $L^2_{loc}$, the second term on the left-hand side converges as $\llangle \mathbf{K}_0^\epsilon, \mathbf{\Theta} \rrangle_{L^2} \rightarrow \llangle \mathbf{K}_0, \mathbf{\Theta} \rrangle_{L^2}.$
\item We claim that the first term on the left-hand side converges a.s. as
\begin{align*}
    & \llangle (\phi_t)_*\mathbf{K}_0^\epsilon, \mathbf{\Theta} \rrangle_{L^2} = \llangle \mathbf{K}_0^\epsilon, \widetilde{\mathbf{\Theta}}_t \rrangle_{L^2} \rightarrow \llangle \mathbf{K}_0, \widetilde{\mathbf{\Theta}}_t \rrangle_{L^2} = \llangle (\phi_t)_*\mathbf{K}_0, \mathbf{\Theta} \rrangle_{L^2},
\end{align*}
where $\widetilde{\mathbf{\Theta}}_t(x) = \lvert J \phi_t(x)\rvert (\phi_t^*\mathbf{\Theta}^\sharp )^\flat(x)$ as explained in Lemma \ref{coord-change}. For this, note that a.s.
\begin{align*}
     \lvert \llangle \mathbf{K}_0^\epsilon- \mathbf{K}_0, \widetilde{\mathbf{\Theta}}_t \rrangle_{L^2} \rvert \leq \norm{\mathbf{K}_0^\epsilon-\mathbf{K}_0}_{L^1_{R^*}} \norm{\widetilde{\mathbf{\Theta}}}_{L^\infty_{t,x}} \rightarrow 0,
\end{align*}
where we have taken into account the following facts (i) $\widetilde{\mathbf{\Theta}}_t$ has compact support contained in $B(0,R^*),$ uniformly in $t \in [0,T],$ where $R^* < \infty$ a.s., (ii) $\mathbb{P}$-a.s. $\norm{\widetilde{\mathbf{\Theta}}}_{L^\infty_{t,x}} < \infty$ (use Lemma \ref{newbound} \ref{theta-tilde-support} \& \ref{theta-tilde-bound} for (i) \& (ii)), and (iii) $\mathbf{K}_0^\epsilon \rightarrow \mathbf{K}_0$ in $L^1(B(0,R^*),\bigwedge^k \mathcal{T}^*\mathbb R^n) \subset L^p(\R^n,\bigwedge^k \mathcal{T}^*\mathbb R^n)$. 
\item By the bounded convergence theorem, the first term on the right-hand side converges a.s. as
\begin{align*}
    & \int^t_0 \llangle (\phi_s)_*\mathbf{K}_0^\epsilon, \mathcal L_b^* \mathbf{\Theta} \rrangle_{L^2} \diff s = \int^t_0 \llangle \mathbf{K}_0^\epsilon, \widetilde{\mathcal L_b^* \mathbf{\Theta}}(s) \rrangle_{L^2} \diff s \\
    & \hspace{70pt} \rightarrow \int^t_0 \llangle \mathbf{K}_0, \widetilde{\mathcal L_b^* \mathbf{\Theta}}(s) \rrangle_{L^2} \diff s = \int^t_0 \llangle (\phi_s)_*\mathbf{K}_0, \mathcal L_b^* \mathbf{\Theta} \rrangle_{L^2} \diff s.
\end{align*}
First note that for such a regular drift $b,$ it is easy to check that Hypothesis \ref{hypoexistence} \ref{maincondition} holds (our assumptions on $b$ are substantially stronger) so we can apply Lemma \ref{nuevo} \ref{lie-tilde-bound}. We claim that a.s. we have
\begin{align*}
    &  \lvert \llangle \mathbf{K}_0^\epsilon-\mathbf{K}_0, \widetilde{\mathcal L_b^* \mathbf{\Theta}}(s) \rrangle_{L^2}\rvert \leq \norm{\mathbf{K}_0^\epsilon-\mathbf{K}_0}_{L_{R^*}^p} \norm{\widetilde{\mathcal L_b^* \mathbf{\Theta}}}_{L^q_{R^*}} \rightarrow 0,
\end{align*}
pointwise in $s \in [0,T],$ since (i) from the mollifier properties $\mathbf{K}_0^\epsilon \rightarrow \mathbf{K}_0$ in $L^p,$ (ii) $\widetilde{\mathcal L_b^* \mathbf{\Theta}}(s)$ is supported on $B(0,R^*)$, uniformly in $s \in [0,T],$ (iii) $\widetilde{\mathcal L_b^* \mathbf{\Theta}}(s)$ is of spatial class $L^q$ (use Lemma \ref{nuevo} \ref{lie-tilde-bound} for (ii) \& (iii)). Moreover, we have the following bound in $L^1([0,T]).$
\begin{align*}
\lvert \llangle \mathbf{K}_0^\epsilon, \widetilde{\mathcal L_b^* \mathbf{\Theta}}(s) \rrangle_{L^2}\rvert \leq \norm{\mathbf{K}_0^\epsilon}_{L^{p}} \norm{\widetilde{\mathcal L_b^* \mathbf{\Theta}}(s)}_{L^{q}} \leq \norm{\mathbf{K}_0}_{L^{p}} \norm{\widetilde{\mathcal L_b^* \mathbf{\Theta}}(s)}_{L^q}, \end{align*}
since $\mathbf{K}_0 \in L^p (\R^n,\bigwedge^k \mathcal{T}^*\mathbb R^n)$ and $\norm{\widetilde{\mathcal L_b^* \mathbf{\Theta}}}_{L^q} \in L^q([0,T])$ from Lemma \ref{nuevo} \ref{lie-tilde-bound}. 
\item The It\^o-correction term converges a.s. as
\begin{align*}
    \frac12 \int^t_0 \llangle (\phi_s)_*\mathbf{K}_0^\epsilon, \mathcal L_\xi^* \mathcal L_\xi^* \mathbf{\Theta} \rrangle_{L^2} \diff s \rightarrow \frac12 \int^t_0 \llangle (\phi_s)_*\mathbf{K}_0, \mathcal L_\xi^* \mathcal L_\xi^* \mathbf{\Theta} \rrangle_{L^2} \diff s,
\end{align*}
by similar arguments as the ones for the last term by taking into account Lemma \ref{nuevo} \ref{lie-tilde-bound-xi2}.
\item Finally, by the stochastic dominated convergence theorem, the martingale term converges as
\begin{align*}
    \int^t_0 \llangle (\phi_s)_*\mathbf{K}_0^\epsilon, \mathcal L_\xi^* \mathbf{\Theta} \rrangle_{L^2} \diff W_s \rightarrow \int^t_0 \llangle (\phi_s)_*\mathbf{K}_0, \mathcal L_\xi^* \mathbf{\Theta} \rrangle_{L^2} \diff W_s,
\end{align*}
in probability (apply the previous strategies using Lemma \ref{nuevo} \ref{lie-tilde-bound-xi}).
\end{itemize}
{\bf Step 2. $\mathbf{K}$ is of class $L^p$.} We observe that by definition
\begin{align} \label{explicitfor}
     ((\phi_t)_* \mathbf{K}_0)_{j_1,\ldots,j_k}(x) = (K_0)_{i_1,\ldots,i_k}(\psi_t(x)) \frac{\partial \psi_t^{i_1}}{\partial x_{j_1}}(x) \cdots \frac{\partial \psi_t^{i_k}}{\partial x_{j_k}}(x),
\end{align}
in global coordinates, where $\psi_t := \phi_t^{-1}$. Hence, by considering the change of variables $x = \phi_t(y),$ we have
\begin{align}
    &\mathbb E\left[\int^T_0 \int_{\mathbb R^n} \lvert \mathbf{K} (t,x)\rvert^p \, \diff^nx \, \diff t \right] = \mathbb E\left[\int^T_0 \int_{\mathbb R^n} \lvert (\phi_t)_* \mathbf{K}_0(x)\rvert^p \diff^nx \, \diff t \right] \nonumber \\
    &\leq \mathbb E\left[\int^T_0 \int_{\mathbb R^n} \lvert \mathbf{K}_0(y)\rvert^p \lvert D\psi_t(\phi_t(y))\rvert^{pk} \lvert J\phi_t(y)\rvert \, \diff^ny \, \diff t \right] \nonumber \\
    &\leq  \int_{\mathbb R^n} \lvert \mathbf{K}_0(y)\rvert^p  \left(\mathbb E\left[\int^T_0   \lvert D\psi_t(\phi_t(y))\rvert^{2pk}   \diff t \right] \right)^\frac12 \left(\mathbb E\left[\int^T_0   \lvert J\phi_t(y)\rvert^2  \diff t \right] \right)^{\frac{1}{2}}  \diff^n y \label{cotak3} \\
    &\lesssim \left( \sup_{y \in \mathbb R^n} \mathbb E\left[\sup_{0\leq t\leq T} \lvert D\psi_t(\phi_t(y))\rvert^{2pk} \right] \right)^{\frac12} \left(\sup_{y \in \mathbb R^n} \mathbb E\left[\sup_{0\leq t\leq T} \lvert J\phi_t(y)\rvert^2\right]\right)^{\frac12} \|\mathbf{K}_0\|_{L^p}^p \nonumber \\
    &\lesssim \left( \sup_{y \in \mathbb R^n} \mathbb E\left[\sup_{0\leq t\leq T} \lvert D\psi_t(\phi_t(y))\rvert^{2pk} \right] \right)^{\frac12}, \label{cotak5}
\end{align}
where we have used Tonelli's theorem to change the order of integration and H\"older's inequality in the space $\Omega \times [0,T]$ in \eqref{cotak3}, and taken into account that
\begin{align*}
&\sup_{y \in \mathbb R^n} \mathbb E\left[\sup_{0\leq t\leq T} \lvert J \phi_t(y)\rvert^2\right] \lesssim \sup_{y \in \mathbb R^n} \mathbb E\left[\sup_{0\leq t\leq T} \lvert D\phi_t(y)\rvert^2\right] \\
&\hspace{50pt} \lesssim \sup_{y \in \mathbb R^n} \mathbb E\left[\sup_{0\leq t\leq T} \lvert D\phi^\epsilon_t(y)\rvert^2\right] + \sup_{y \in \mathbb R^n} \mathbb E\left[\sup_{0\leq t\leq T} \lvert D\phi^\epsilon_t(y)-D\phi_t(y)\rvert^2\right],
\end{align*}
by \eqref{bound-phi-deriv} \& \eqref{phi-deriv-convergence} , and the fact that $\mathbf{K}_0 \in L^p (\R^n,\bigwedge^k \mathcal{T}^*\mathbb R^n)$ in \eqref{cotak5}. To conclude, we show that \eqref{cotak5} is finite. For this, we first note that $D\psi_t(\phi_t(y)) = [D \phi_t(y)]^{-1},$ for $t \in [0,T]$ and $y \in \R^n,$ where $[D \phi_t(y)]^{-1}$ denotes the inverse of the Fréchet matrix $D \phi_t(y).$ Therefore, we have $D\psi_t(\phi_t(y)) = [J \phi_t(y)]^{-1} \text{Adj}(D \phi_t(y))$ and
\begin{align*}
     &\sup_{y \in \mathbb R^n}  \mathbb E\left[\sup_{0\leq t\leq T} \lvert D\psi_t(\phi_t(y))\rvert^{2pk} \right] \\
     &\quad \leq \sup_{y \in \mathbb R^n}  \mathbb E\left[\sup_{0\leq t\leq T} [J \phi_t(y)]^{-4pk}  \right]^{\frac12} \sup_{y \in \mathbb R^n} \mathbb E\left[\sup_{0\leq t\leq T} \lvert \text{Adj}(D \phi_t(y))\rvert^{4pk} \right]^{\frac12} \\
     &\quad \lesssim \sup_{y \in \mathbb R^n} \mathbb E\left[\sup_{0\leq t\leq T} [J \phi_t(y)]^{-4pk}  \right]^{\frac12} \sup_{y \in \mathbb R^n}   \mathbb E\left[\sup_{0\leq t\leq T} \lvert D \phi_t(y)\rvert^{4pk(n-1)} \right]^{\frac{1}{2}} \\
     &\quad \lesssim  \sup_{y \in \mathbb R^n} \mathbb E\left[\sup_{0\leq t\leq T} [J \phi_t(y)]^{-4pk}  \right]^{\frac12} < \infty,
\end{align*}
where we have applied H\"older's inequality, taken into account that all the terms in $\text{Adj}(D \phi_t(y))$ are sums of products of $(n-1)$ components of the matrix $D \phi_t(y),$ and applied \eqref{phi-deriv-convergence}. Finally, the term in the last line is finite due to \eqref{otracosa2}.
\end{proof}

\subsection{Existence for weak drift and initial condition}
We proceed to reduce the regularity on the drift and establish our main existence theorem.
\begin{theorem}[Existence of $L^p$-solutions for weak drift and initial condition] \label{nonsmooth-existence}
Let $p\geq 2$ and $\mathbf{K}_0 \in L^p(\R^n,\bigwedge^k \mathcal{T}^*\mathbb R^n)$. Let $b,$ $\xi_i,$ $i=1,\ldots,N$ satisfy Hypotheses \ref{hypoflow} $\&$ \ref{hypoexistence} with $q$ such that $1/q+1/p=1$. Then there exists an $L^p$-solution to \eqref{eq0}.
\end{theorem} 
\begin{proof}
For simplicity, we treat the case $N=1.$ First, we take the mollified vector field $b^\epsilon := \rho^\epsilon * b$ and consider equation \eqref{eq0} with drift vector field $b^\epsilon$ instead of $b$. Let $\phi^\epsilon_t$ be the smooth flow of \eqref{flowmap} with drift vector field $b^\epsilon$ and define $\mathbf{K}^\epsilon_t(x) := (\phi_t^\epsilon)_* \mathbf{K}_0(x)$, which is a weak solution of \eqref{eq0} by Proposition \ref{weak-sol-smooth}. By imitating the argument in Step 2 of the last proposition, considering the change of coordinates $x = \phi_t^\epsilon(y)$ and taking into account properties \eqref{bound-phi-deriv} and \eqref{otracosa1}, we obtain
\begin{align} \label{greatbound}
    &\mathbb E\left[\int^T_0 \int_{\mathbb R^n} \lvert \mathbf{K}^{\epsilon}_t(x)\rvert^p \, \diff^nx \, \diff t \right] \leq C < \infty,
\end{align}
where the constant $C$ depends on $p$ but is independent of $0 <\epsilon < 1.$ Hence there exists a subsequence $\epsilon_n \rightarrow 0$ and a $k$-form-valued stochastic process $\mathbf{K} \in L^p ( \Omega \times [0,T] \times \R^n; \bigwedge^k \mathcal{T}^*\mathbb R^n)$ such that $\mathbf{K}^{\epsilon_n}$ converges weakly to $\mathbf{K}$ in $L^p( \Omega \times [0,T] \times \R^n; \bigwedge^k \mathcal{T}^*\mathbb R^n)$ and, without loss of generality, also in $L^2( \Omega \times [0,T] \times \R^n; \bigwedge^k \mathcal{T}^*\mathbb R^n)$. Indeed, bound \eqref{greatbound} holds in $L^2$ for $\mathbf{K}^{\epsilon_n}$ by following Step 2 in Proposition \ref{weak-sol-smooth} with $p=2$ and applying H\"older's inequality in space with conjugate exponents $p'=p/2$ and $q'=p/(p-2)$ in \eqref{cotak3}. Hence, a new subsequence converging weakly in $L^2( \Omega \times [0,T] \times \R^n; \bigwedge^k \mathcal{T}^*\mathbb R^n)$ can be subtracted by iteration. Its weak limit must again be $\mathbf{K}$. This means that
\begin{align*}
    &\llangle \mathbf{K}^{\epsilon_n}_t(x) -\mathbf{K}_t(x), \mathbf{F}(\omega,t,x) \rrangle_{L^2_{\omega,t,x}} \rightarrow 0, \quad n \rightarrow \infty,
\end{align*}
for any $\mathbf{F} \in L^2(\Omega \times [0,T] \times \R^n; \bigwedge^k \mathcal{T}^*\mathbb R^n)$. From now on, we shall write $\epsilon$ instead of $\epsilon_n$ for the sake of simplicity. We fix a test $k$-form $\mathbf{\Theta} \in C^\infty_0(\R^n,\bigwedge^k \mathcal{T}^*\mathbb R^n)$. Note that the process $\llangle \mathbf{K}^{\epsilon},\mathbf{\Theta} \rrangle_{L^2}$ is adapted by definition and converges to $\llangle \mathbf{K},\mathbf{\Theta} \rrangle_{L^2},$ weakly in $L^{2}(\Omega \times
[0,T]),$ from the weak convergence of $\mathbf{K}^\epsilon$ in $L^2(\Omega \times [0,T] \times \R^n; \bigwedge^k \mathcal{T}^*\mathbb R^n)$. Since $\llangle \mathbf{K}^{\epsilon}, \mathbf{\Theta} \rrangle_{L^2}$ is adapted and the set of adapted processes in $L^{2}(\Omega \times [0,T])$ is closed (and therefore weakly closed), $\llangle \mathbf{K}, \mathbf{\Theta} \rrangle_{L^2}$ is adapted.
Taking into account this last fact, it is not difficult to show that $\llangle \mathbf{K}, \mathcal{L}_{\xi}^* \mathbf{\Theta} \rrangle_{L^2}$ is also adapted.

Moreover, by H\"older's inequality and using the fact that $\mathbf{K} \in L^p( \Omega \times [0,T] \times \R^n; \bigwedge^k \mathcal{T}^*\mathbb R^n)$ together with Hypothesis \ref{hypoexistence} \ref{maincondition}, we have $\llangle \mathbf{K}, \mathcal{L}_{b}^* \mathbf{\Theta}  \rrangle_{L^2} \in L^1(\Omega \times [0,T]),$ so its time integral is well defined. It can also be checked that $\llangle \mathbf{K}, \mathcal{L}_{\xi}^* \mathcal{L}_{\xi}^* \mathbf{\Theta} \rrangle_{L^2}$ is of class $L^1(\Omega \times [0,T]),$ so their time integrals are well-defined. Moreover, $\llangle \mathbf{K}, \mathcal{L}_{\xi}^* \mathbf{\Theta} \rrangle_{L^2} \in L^2(\Omega \times [0,T])$. 

We know that $\mathbf{K}^\epsilon$ satisfies equation \eqref{weak-eq}, and hence it also satisfies it weakly in $L^2(\Omega \times [0,T]).$ We claim that this property is transferred in the weak limit, so $\llangle \mathbf{K}, \mathbf{\Theta} \rrangle_{L^2}$ also satisfies \eqref{weak-eq}, weakly in $L^2(\Omega \times [0,T]).$ We include the proof of this claim in Appendix \ref{weaklimit} to prevent this proof from becoming excessively convoluted. Therefore we may pass to the weak $L^{2}(\Omega \times \lbrack 0,T])$-limit in the equation for $\mathbf{K}$ and \eqref{weak-eq} holds for almost every $  \left ( \omega ,t \right) \in \Omega \times [0,T] $. The proof is complete.
\end{proof}

Until now, we have constructed a solution to our equation but we have not proven that there exists a solution with the particular form $(\phi_t)_* \mathbf{K}_0(x).$
To show this, we require the following technical lemma.

\begin{lemma}\label{lemma:theta-push-convergence}
Let $b$ and $\xi_i,$ $i=1,\ldots,N$ satisfy Hypothesis \ref{hypoflow} and $\{b^m\}_{m = 1}^\infty$ be a sequence of drift vector fields such that $b^m \rightarrow b$ in $L^\infty([0,T]; C^\alpha_b(\R^n, \R^n)),$ $m \rightarrow \infty.$ Moreover, let $\{\phi_t^m\}_{m=1}^\infty$ be the corresponding sequence of stochastic flows solving \eqref{flownotation} with drift vector fields $b^m$ and fixed diffusion vector fields $\xi_i,$ $i=1,\ldots,N$. Let $r\geq 2$ such that $\alpha r > (n+1)/2^{k+1}$ and $R>0.$ Then for any multi-vector field $\mathbf{U} \in C_0^\infty (\R^n,\bigwedge^k \mathcal{T}\mathbb R^n)$ with compact support on $B(0,R),$ we have
\begin{align}\label{eq:theta-push-convergence}
\mathbb{E}\left[ \int_0^T \int_{B(0,R^*(m))} \left((\phi_t^{m})^* \mathbf{U}(x)  - (\phi_t)^* \mathbf{U}(x) \right)^r \diff^n x \, \diff t \right] \rightarrow 0, \quad m \rightarrow \infty,
\end{align}
where $R^*(m) <\infty$ a.s. is chosen such that 
\begin{align*}
\resizebox{1. \textwidth}{!}{
\begin{math}
    \Big\{(\phi_t^{m})^{-1}(y):\hspace{0.1cm} t \in [0,T], \hspace{0.1cm} y \in B(0,R) \Big\} \cup \Big\{\phi_t^{-1}(y):\hspace{0.1cm} t \in [0,T],\hspace{0.1cm} y \in B(0,R)\Big\} \subset B(0,{R}^*(m)),
\end{math}
}
\end{align*}
and we assume $B(0,R^*(m)) \subset B(0,R')$ \footnote{Note that this is guaranteed for fixed $m,$ but $R'$ would depend on $m,$ so this assumption is important.} with deterministic $R' < \infty,$ for all $m\in \mathbb{N}.$
\end{lemma}
Since this lemma is quite technical, we prove it in Appendix \ref{uglylemma}. The following proposition concludes this existence section.
\begin{theorem}
In the conditions of Theorem \ref{nonsmooth-existence}, there exists an $L^p$-solution to \eqref{eq0} of the form $\mathbf{K}_t(x) := (\phi_t)_* \mathbf{K}_0(x),$ where $\phi_t$ is the flow of \eqref{flowmap}.
\end{theorem}
\begin{proof}
{\bf Step 1.} First we assume that $\mathbf{K}_0$ is supported on $B(0,R).$ We will show that we can choose a sequence converging to a weak solution of \eqref{eq0}, which we denote by $\mathbf{K}^\epsilon$ for convenience, such that for every test $k$-form $\mathbf{\Theta} \in C^\infty_0(\R^n,\bigwedge^k \mathcal{T}^*\mathbb R^n)$, we have $\llangle \mathbf{K}^\epsilon - (\phi_t)_* \mathbf{K}_0,\mathbf{\Theta} \rrangle_{L^2} \rightarrow 0,$ weakly in $L^2(\Omega \times [0,T]).$ From this, we can conclude that $\langle \langle (\phi_t)_* \mathbf{K}_0, \mathbf{\Theta} \rangle \rangle_{L^2}$ satisfies equation \eqref{weak-eq} for almost every $(\omega, t) \in \Omega \times [0,T].$ Since by the flow properties $(\phi_t)_* \mathbf{K}_0$ is adapted and a.s. continuous, it is a solution in the sense prescribed in Definition \ref{weak-sol}.

Indeed, we select a convergent (sub)sequence $\mathbf{K}^\epsilon$ by applying the same argument as we did previously in the proof of Theorem \ref{nonsmooth-existence}. Fix a test $k$-form $\mathbf{\Theta} \in C^\infty_0(\R^n,\bigwedge^k \mathcal{T}^*\mathbb R^n)$, which we assume without loss of generality to be supported on $B(0,R)$.  We claim that $\llangle \mathbf{K}^\epsilon, \mathbf{\Theta} \rrangle_{L^2}$ converges to $\llangle (\phi_t)_* \mathbf{K}_0,\mathbf{\Theta} \rrangle_{L^2},$ weakly in $L^2 ( \Omega \times [0,T]).$
Fix a test function $f \in L^2(\Omega \times [0,T]).$ Let us adopt the notation $\|\cdot\|_{L^r_{\omega,t,R}}$ to denote the $L^r$-norm of a tensor on the set $\Omega \times [0,T] \times B(0,R)$. Recalling that $\mathbf{K}^\epsilon_t = (\phi_t^\epsilon)_* \mathbf{K}_0$, where $\phi_t^\epsilon$ is the stochastic flow of \eqref{eq0} with $b^\epsilon$, the mollified drift, instead of $b$, we have 
\begin{align*}
    & \left\lvert \mathbb{E} \left[ \int_0^T  \llangle \mathbf{K}_t^\epsilon - (\phi_t)_* \mathbf{K}_0, \mathbf{\Theta} \rrangle_{L^2} f \, \diff t \right] \right\rvert \\
    & \leq   \mathbb{E} \left[ \int_0^T \left\lvert \langle \langle \mathcal{X}_{B(0,R)} \mathbf{K}_0,  \lvert J\phi_t^\epsilon\rvert ((\phi_t^\epsilon)^* \mathbf{\Theta}^\sharp )^\flat - \lvert J\phi_t\rvert (\phi_t^* \mathbf{\Theta}^\sharp )^\flat \rangle \rangle_{L^2}  f \right\rvert \diff t \right] \\
    & =  \mathbb{E} \left[ \int_0^T \left\lvert \langle \langle \mathcal{X}_{B(0,R)} \mathbf{K}_0,  \lvert J\phi_t^\epsilon\rvert ((\phi_t^\epsilon)^* \mathbf{\Theta}^\sharp )^\flat - \lvert J\phi^\epsilon_t\rvert (\phi_t^* \mathbf{\Theta}^\sharp )^\flat \rangle \rangle_{L^2}  f \right\rvert \diff t \right] \\
    & \quad +   \mathbb{E} \left[ \int_0^T \left\lvert \langle \langle \mathcal{X}_{B(0,R)} \mathbf{K}_0,  \lvert J\phi_t^\epsilon\rvert ((\phi_t)^* \mathbf{\Theta}^\sharp )^\flat - \lvert J\phi_t\rvert (\phi_t^* \mathbf{\Theta}^\sharp )^\flat \rangle \rangle_{L^2}  f \right\rvert \diff t \right] \\
    & \leq \norm{\mathbf{K}_{0}f}_{L^2_{\omega,t,R}}  \normb{ J\phi_t^\epsilon \,( ((\phi_t^\epsilon)^* \mathbf{\Theta}^\sharp )^\flat - (\phi_t^* \mathbf{\Theta}^\sharp )^\flat ) }_{L^2_{\omega,t,R}} \\
    &\hspace{100pt} + \norm{\mathbf{K}_{0}f}_{L^2_{\omega,t,R}}  \normb{ (J\phi_t^\epsilon- J\phi_t ) (\phi_t^* \mathbf{\Theta}^\sharp )^\flat}_{L^2_{\omega,t,R}},
\end{align*}
where we took into account Lemma \ref{coord-change} in the second line and H\"older's inequality in the last step. We check that the last lines above converges to zero. First we note
\begin{align*}
    & \norm{\mathbf{K}_{0}f}_{L^2_{\omega,t,R}} \lesssim \norm{\mathbf{K}_{0}}_{L^2_R} \norm{f}_{L^2_{\omega,t}} < \infty.
\end{align*}
Moreover, choosing constants $r \geq 2$ such that $2r\alpha > (n+1)/2^{k+1}$ and $s > 1$ satisfying $1/r + 1/s = 1$, we have 
\begin{align*}
    & \normb{J\phi_t^\epsilon ( ((\phi_t^\epsilon)^* \mathbf{\Theta}^\sharp )^\flat - (\phi_t^* \mathbf{\Theta}^\sharp)^\flat ) }_{L^2_{\omega,t,R}} \leq \norm{J\phi_t^\epsilon}_{L^{2s}_{\omega,t,R}} \normb{ (\phi_t^\epsilon)^* \mathbf{\Theta}^\sharp  - \phi_t^* \mathbf{\Theta}^\sharp}_{L^{2r}_{\omega,t,R}} \rightarrow 0,
\end{align*}
where we have used \eqref{weakjacobian} to control the Jacobian determinant term together with Lemma \ref{lemma:theta-push-convergence} to obtain convergence of the second term. 
By applying similar techniques we can also derive
\begin{align*}
    & \normb{\left(J\phi_t^\epsilon- J\phi_t\right) (\phi_t^* \mathbf{\Theta}^\sharp )^\flat}_{L^2_{\omega,t,R}} \rightarrow 0.
\end{align*}
Note that the hypothesis in Lemma \ref{lemma:theta-push-convergence} is verified with $R'=R$ (the hypothesis would not hold had we not assumed that $\mathbf{K}$ is compactly supported, so this was crucial). Hence, our argument is complete in the case when $\mathbf{K}_0$ has compact support.

{\bf Step 2.} Now we treat the general case $\mathbf{K}_0 \in L^p(\R^n,\bigwedge^k \mathcal{T}^*\mathbb R^n),$ without restriction on the support. We choose $\zeta \in C_{0}^{\infty }(\mathbb R^n)$ supported on $B(0,1)$ such that $\zeta
(x)=1,$ $x \in B(0,1/2)$. For every $m \in \mathbb{N},$ let $\mathbf{K}_{0}^{m}(x):=\mathbf{K}_{0}\left( x\right) \zeta (x/m),$ which is compactly supported. By taking Step 1, we can construct a corresponding solution $\mathbf{K}^m$ with initial datum $\mathbf{K}^m_0.$ Therefore we have 
\begin{align}
\label{nequation}
\begin{split}
    & \llangle \mathbf{K}^m_t, \mathbf{\Theta} \rrangle_{L^2} - \llangle \mathbf{\mathbf{K}}_0^m, \mathbf{\Theta} \rrangle_{L^2} = -\int^t_0 \llangle \mathbf{K}^m_s, \mathcal L_b^* \mathbf{\Theta} \rrangle_{L^2} \diff s \\
    & \hspace{50pt} - \int^t_0 \llangle \mathbf{K}^m_s, \mathcal L_\xi^* \mathbf{\Theta} \rrangle_{L^2} \diff W_s
    + \frac12 \int^t_0 \llangle \mathbf{K}^m_s, \mathcal L_\xi^* \mathcal L_\xi^* \mathbf{\Theta} \rrangle_{L^2} \diff s,
\end{split}
\end{align}
$t \in [0,T],$ for any test $k$-form $\mathbf{\Theta} \in C^\infty_0(\R^n,\bigwedge^k \mathcal{T}^*\mathbb R^n)$. We know from the first part of our proof that $\mathbb{P}$-a.s.,
\begin{align*} 
& \mathbf{K}^{m}_t(x) = (\phi_t)_* (\mathbf{K}_{0}(\cdot) \zeta (\cdot/m))(x) = (\phi_t)_* \mathbf{K}_{0}(x) \zeta (\phi_{t}^{-1}(x)/m), \quad t \in [0,T], \quad x \in \mathbb{R}^n,
\end{align*}
where we have employed the property $(\phi_t)_*(\mathbf{K} f)(x) =(\phi_t)_*\mathbf{K}(x)(\phi_t)_*f(x)$ and taken into account that $(\phi_t)_*(x) = f(\phi_t^{-1}(x))$ for smooth scalar functions $f: \mathbb{R}^n \rightarrow \mathbb{R}$. It is clear that a.s., the sequence $\mathbf{K}^{m}_t$ converges to $\mathbf{K}_t :=(\phi_t)_* \mathbf{K}_{0}$ pointwise in $(t,x)$. We want to show that $\mathbf{K}_t$ also satisfies \eqref{nequation}. To this end, we first fix a test $k$-form $\mathbf{\Theta} \in C^\infty_0(\R^n,\bigwedge^k \mathcal{T}^*\mathbb R^n)$ supported on $B(0,R)$. As in the proof of Lemma \ref{newbound}, we choose a random variable $R^* : \Omega \rightarrow (0, \infty]$, which is finite a.s., such that 
\begin{align*}
    & \{\phi_t^{-1}(x): \hspace{0.1cm} t \in [0,T], \hspace{0.1cm} x \in B(0,R) \} \subset B(0,R^*).
\end{align*}
Let $\omega$ belong to the set of probability one such that $R^*(\omega) < \infty$ and set $R'(\omega) := \max(R,R^*(\omega))$. Then we have for $\mathbb{N} \ni m> \lfloor 2R'(\omega) \rfloor + 1$, 
\begin{align*}
\mathcal{X}_{B(0,R)} \mathbf{K}^{m}_t(x) & = \mathcal{X}_{B(0,R)} (\phi_t)_* \mathbf{K}_{0}(x) \zeta (\phi_{t}^{-1}(x)/m) \\
& = \mathcal{X}_{B(0,R)} (\phi_t)_* \mathbf{K}_{0}(x) = \mathcal{X}_{B(0,R)} \mathbf{K}_t(x), \quad t \in [0,T], \quad x \in \mathbb{R}^n.
\end{align*}
Therefore $\mathcal{X}_{B(0,R)} \mathbf{K}^{m}_t(x) = \mathcal{X}_{B(0,R)} \mathbf{K}_t(x).$
We conclude that for $\mathbb{P}$-a.s. $\omega,$ if $m \equiv m(\omega)> \lfloor 2R'(\omega) \rfloor + 1,$ we always have $\mathbf{K}^m = \mathbf{K}$ and hence the terms $\llangle \mathbf{K}^m_t, \mathbf{\Theta} \rrangle_{L^2},$ $\int^t_0 \llangle \mathbf{K}^m_s, \mathcal L_b^* \mathbf{\Theta} \rrangle_{L^2} \diff s,$ and $\int^t_0 \llangle \mathbf{K}^m_s, \mathcal L_\xi^* \mathcal L_\xi^* \mathbf{\Theta} \rrangle_{L^2} \diff s$ in \eqref{nequation} become $\llangle \mathbf{K}_t, \mathbf{\Theta} \rrangle_{L^2},$ $\int^t_0 \llangle \mathbf{K}_s, \mathcal L_b^* \mathbf{\Theta} \rrangle_{L^2} \diff s,$ $\int^t_0 \llangle \mathbf{K}_s, \mathcal L_\xi^* \mathcal L_\xi^* \mathbf{\Theta} \rrangle_{L^2} \diff s$ respectively, which implies a.s. convergence. To finish the argument, we claim that
\begin{align*}
& \int^t_0 \llangle \mathbf{K}^m_s, \mathcal L_\xi^* \mathbf{\Theta} \rrangle_{L^2} \diff W_s \rightarrow \int^t_0 \llangle \mathbf{K}_s, \mathcal L_\xi^* \mathbf{\Theta} \rrangle_{L^2} \diff W_s
\end{align*}
in probability, $t \in [0,T]$. For this, it suffices to establish 
\begin{align} \label{finalfact}
& \mathbb{E} \left[\left(\int^t_0 \llangle \mathbf{K}^m_s - \mathbf{K}_s, \mathcal L_\xi^* \mathbf{\Theta} \rrangle_{L^2} \diff W_s \right)^2 \right] = \mathbb{E} \left[\int^t_0 \llangle \mathbf{K}^m_s - \mathbf{K}_s, \mathcal L_\xi^* \mathbf{\Theta} \rrangle^2_{L^2} \diff s \right] \rightarrow 0, 
\end{align}
where we have used It\^o's isometry and employed the fact that square mean convergence implies convergence in probability. Claim \eqref{finalfact} can be verified by applying the dominated convergence theorem in the space $\Omega \times [0,T]$ to the sequence of functions $\llangle \mathbf{K}^m_s, \mathcal L_\xi^* \mathbf{\Theta} \rrangle_{L^2}$, where the hypotheses are satisfied since (i) clearly, a.s. $\llangle \mathbf{K}^m_s - \mathbf{K}_s, \mathcal L_\xi^* \mathbf{\Theta} \rrangle^2_{L^2} \rightarrow 0,$ (ii) by using H\"older's inequality, we have the uniform bound
\begin{align*}
& \llangle \mathbf{K}^m_s, \mathcal L_\xi^* \mathbf{\Theta} \rrangle^2_{L^2_x} \leq \norm{\mathbf{K}^m_s}_{L^2_x}^2 \norm{\mathcal{L}_\xi^* \mathbf{\Theta}}_{L^2_x}^2 \leq \norm{\mathbf{K}_s}_{L^2_x}^2 \norm{\mathcal{L}_\xi^* \mathbf{\Theta}}_{L^2_x}^2 \lesssim \norm{\mathbf{K}_s}_{L_x^2}^2,
\end{align*}
where we have taken into account that $\lvert \mathbf{K}^{m}_t(x)\rvert \leq \lvert \mathbf{K}_t(x)\rvert$ and have used that $\norm{\mathcal{L}_\xi^* \mathbf{\Theta}}_{L^2_x}^2 \in L^\infty([0,T])$. Moreover, (iii) $\norm{\mathbf{K}}_{L_x^2}^2 \in L^1(\Omega \times [0,T])$ by usual strategies as in the bound for $\mathbb E\left[\int^T_0 \int_{\mathbb R^n} \lvert \mathbf{K} (t,x)\rvert^p \, \diff^n x \, \diff t \right]$ in Step 2 of the proof of Proposition \ref{weak-sol-smooth}.
\end{proof}

\section{Uniqueness of stochastic solutions} \label{nonsmooth-unique}
\subsection{Uniqueness with weak differentiability of the drift} \label{unicidadfuerte}
This subsection is devoted to establishing our general uniqueness result (Theorem \ref{theoremuniqueness}). Let us assume we are under the hypotheses in Theorem \ref{theoremuniqueness}. For convenience, we employ the notation $\mathbf{K}(t,x)$ for the solution of \eqref{eq0}. For simplicity, we treat the case $N=1$. Due to the linearity of the equation, it is sufficient to show that if $\mathbf{K}(0,\cdot) = 0$, then $\mathbf{K}(t,\cdot) = 0$ a.s., for every $t \in [0,T]$. 

First, fix $\epsilon>0$ and consider the {\em smooth-in-space} $k$-form-valued process $\mathbf{K}^\epsilon(t,x) := (\rho^\epsilon * \mathbf{K}(t,\cdot))(x)$, which must satisfy \footnote{Stochastic Fubini's theorem can be applied to equation \eqref{smoothened} to switch spatial and temporal integration by a similar (but simpler) argument to the one we employ to justify its use in the proof of Theorem \ref{theoremuniqueness}.}
\begin{align}
\label{smoothened}
\begin{split}
    & \mathbf{K}^\epsilon(t,x) = -\int^t_0 (\rho^\epsilon * \mathcal L_b \mathbf{K})(s,x) \, \diff s - \int^t_0 (\rho^\epsilon * \mathcal L_{\xi} \mathbf{K})(s,x) \, \diff W_s \\
    & \hspace{90pt} + \frac{1}{2} \int_0^{t} (\rho^\epsilon * (\mathcal{L}_{\xi})^2 \mathbf{K})(s,x) \, \diff s = 0, \quad t \in [0,T], \quad x \in \mathbb{R}^n,
\end{split}
\end{align}
where $\mathcal L_b \mathbf{K},$ $\mathcal L_\xi \mathbf{K},$ $(\mathcal L_\xi)^2 \mathbf{K}$ are understood as Lie derivatives {\em in the sense of distributions}, owing to the low spatial regularity of $\mathbf{K}$ (see Definition \ref{lieadjointdistribution}). We note that \eqref{distri} is well-defined since by H\"older's inequality and Hypothesis \ref{hypoexistence} \ref{maincondition}, we have
\begin{align*}
    & \llangle \mathbf{K}, \mathcal L_b^{*} \mathbf{\Theta} \rrangle_{L^2} \leq \norm{\mathbf{K}}_{L^p_{\omega,t,x}} \norm{\mathcal L_b^{*} \mathbf{\Theta}}_{L^q_{t,x}}.
\end{align*}
Since all the terms in (\ref{smoothened}) are now smooth in space, we can apply the KIW formula for $k$-forms (Theorem \ref{KIWmain} \eqref{KIW}) to obtain
\begin{align} 
\begin{split}
    \label{KIWapplication}
    &\phi_t^* \mathbf{K}^\epsilon(t,x) = - \int^t_0 \phi_s^* (\mathcal L_b \mathbf{K})^\epsilon(s,x) \, \diff s - \int^t_0 \phi_s^* (\mathcal L_\xi \mathbf{K})^\epsilon \, \diff W_s \\
    &+ \frac{1}{2} \int^t_0 \phi_s^* (\mathcal L^2_\xi \mathbf{K})^\epsilon (s,x) \, \diff s + \int^t_0 \phi_s^* \mathcal L_b \mathbf{K}^\epsilon (s,x) \, \diff s \\
    &+ \int^t_0 \phi_s^* \mathcal L_{\xi} \mathbf{K}^\epsilon(s,x) \, \diff W_s  + \int^t_0 \phi^*_s \mathcal L_{\xi} (\mathcal L_{\xi} \mathbf{K})^\epsilon \, \diff [W_{\cdot}, W_{\cdot}]_s + \frac12 \int^t_0 \phi^*_s \mathcal L^2_{\xi} \mathbf{K}^\epsilon(s,x) \, \diff s,
\end{split}
\end{align}
for all $t \in [0,T]$ and $x \in \mathbb{R}^n.$ Note that the hypotheses of Theorem \ref{KIWmain} are satisfied since:
\begin{itemize}
    \item By the properties of mollifiers, $\mathbb{P}$-a.s., $(\mathcal L_b \mathbf{K})^{\epsilon}$ is continuous in time (indeed, see equation \eqref{smoothened}) and smooth in space, whereas $(\mathcal L_{\xi} \mathbf{K})^{\epsilon}$ and $((\mathcal{L}_{\xi})^2 \mathbf{K})^\epsilon$ are just smooth in space.
    \item By taking into account the property $D_y\rho^\epsilon(x-y) = -D_x\rho^\epsilon(x-y)$ and applying the definition of mollification of distribution-valued $k$-forms (see Definition \ref{distrimolliformula}), one can check that $(\mathcal L_b \mathbf{K})^\epsilon(t,x)$ is the smooth time-dependent $k$-form
    \begin{align*}
        & (\mathcal L_b \mathbf{K})^\epsilon(t,x) \\
        &= \int_{\mathbb R^n} \left(\frac{\partial \rho^\epsilon}{\partial x_l}(x-y) \mathbf{K}(t,y) b^l(t,y) - \rho^\epsilon(x-y) \mathbf{K}(t,y) \frac{\partial b^l}{\partial y_l}(t,y) \right) \diff^n y
        \\ 
        & + \int_{\mathbb R^n} \rho^\epsilon(x-y) \left(\frac{\partial b^{l}}{\partial y_{i_1}}(t,y) K_{l,i_2,\ldots,i_k}(t,y) \right. \\
        &\hspace{130pt} \left. + \cdots + \frac{\partial b^{l}}{\partial y_{i_k}}(t,y) K_{i_1,\ldots,i_{k-1}, l}(t,y)\right) \diff^n y.
    \end{align*}
    Let $\rho^{\epsilon}$ be supported on $B(0,R)$. Then $\mathbb{P}$-a.s. we have
    \begin{align*}
        & \int_0^T \left\lvert(\mathcal L_b \mathbf{K})^\epsilon(t,x)\right\rvert \diff t \\
        & \lesssim \int_{\mathbb R^n} \int_0^T \left\lvert \frac{\partial \rho^\epsilon}{\partial x_l}(x-y) \mathbf{K}(t,y) b^l(t,y) - \rho^\epsilon(x-y) \mathbf{K}(t,y) \frac{\partial b^l}{\partial y_l}(t,y) \right\rvert \diff t \, \diff^n y
        \\ 
        & \quad + \int_{\mathbb R^n} \int_0^T \left\lvert \rho^\epsilon(x-y) \left(\frac{\partial b^{l}}{\partial y_{i_1}}(t,y) K_{l,i_2,\ldots,i_k}(t,y) \right. \right.\\
        &\hspace{120pt} \left. \left. + \cdots + \frac{\partial b^{l}}{\partial y_{i_k}}(t,y) K_{i_1,\ldots,i_{k-1}, i_k}(t,y)\right) \right\rvert \diff t \, \diff^n y 
        \\
        & \lesssim \int_{\R^n} \left( \frac{\partial \rho^{\epsilon}}{\partial x_l}(x-y) + \rho^{\epsilon}(x-y) \right)\norm{\mathbf{K}(\cdot,y)}_{L^p_t} \\
        &\hspace{150pt} \times \left( \norm{b(\cdot,y)}_{L^q_t} + \norm{Db(\cdot,y)}_{L^q_t} \right) \, \diff^n y \\
        & \lesssim \left( \left\lvert \left\lvert \frac{\partial \rho^{\epsilon}}{\partial x_l}\right\rvert \right\rvert_{L^\infty} + \norm{\rho^{\epsilon}}_{L^\infty} \right)\norm{\mathbf{K}}_{L^p_{t,x}} \left( \norm{b}_{L^q_{B(x,R)}} + \norm{Db}_{L^q_{B(x,R)}} \right) < \infty,
    \end{align*}
    for every $x \in \R^n,$ where we have taken into account Hypotheses \ref{hypoflow} \ref{bcondition} $\&$ \ref{stronguniqueness}.
    \item By a similar argument, one can also check that $\mathbb{P}$-a.s.
    \begin{align}
        \int_0^T \left( \left\lvert\big((\mathcal{L}_\xi)^2 \mathbf{K}\big)^\epsilon(t,x)\right\rvert + \big\vert(\mathcal L_\xi \mathbf{K})^\epsilon(t,x)\big\rvert^2 \right) \diff t < \infty, \quad x \in \R^n.
    \end{align}
\end{itemize}
Therefore, \eqref{KIWapplication} can be rewritten as
\begin{align} \label{commutatorito}
\begin{split}
    \phi_t^* \mathbf{K}^\epsilon(t,x) &= \int^t_0 \phi_s^* [\mathcal L_b, \rho^\epsilon *] \mathbf{K}(s,x) \, \diff s + \int^t_0 \phi_s^*[\mathcal L_{\xi}, \rho^\epsilon *] \mathbf{K}(s,x) \, \diff W_s  \\
    & \quad + \frac12 \int^t_0 \phi_s^*[\mathcal L_{\xi},[\mathcal L_{\xi}, \rho^\epsilon *]] \mathbf{K}(s,x) \, \diff s,\quad t \in [0,T], \quad x \in \mathbb{R}^n,
\end{split}
\end{align}
where
\begin{align} \label{commutator}
[\mathcal L_b, \rho^\epsilon *] \mathbf{K} := \mathcal L_b \mathbf{K}^\epsilon - (\mathcal L_b \mathbf{K})^\epsilon,
\end{align}
denotes the commutator between the linear operators $\mathcal L_b$ and $\rho^\epsilon *(\cdot),$ generalising the DiPerna-Lions commutator appearing in \cite{diperna1989ordinary, lions1996mathematical}, and 
\begin{align} \label{commutator2}
[\mathcal L_{\xi},[\mathcal L_{\xi}, \rho^\epsilon *]] \mathbf{K} := \mathcal L^2_\xi \mathbf{K}^\epsilon -2 \mathcal L_\xi (\mathcal L_\xi \mathbf{K})^\epsilon + (\mathcal L^2_\xi \mathbf{K})^\epsilon
\end{align}
is the Stratonovich-to-It\^o correction of the commutator (a double commutator). A novelty of this paper consists in showing that the commutators \eqref{commutator} \& \eqref{commutator2} satisfy a DiPerna-Lions type estimate, and converge weakly to zero as $\epsilon \rightarrow 0$. This is a-priori far from being obvious since they contain terms which at first glance seem too singular. We remark that in the special case $\mathcal L_\xi = \xi \cdot \nabla$, a similar estimate for the double commutator was obtained in \cite{punshon2018boltzmann}, Proposition 3.4.

We will show that \eqref{commutatorito} vanishes weakly as $\epsilon \rightarrow 0$, thus establishing uniqueness of solutions. To treat commutators \eqref{commutator} \& \eqref{commutator2}, we establish generalisations to $k$-forms of the classical DiPerna-Lions commutator estimates (see  \cite{diperna1989ordinary, lions1996mathematical, flandoli2010well}). Our first lemma will be applied to treat the first commutator \eqref{commutator}.

\begin{lemma} \label{commutator-estimate-1}
For $p\geq2$, given a $k$-form $\mathbf{K} \in L^p_{loc} (\R^n,\bigwedge^k \mathcal{T}^*\mathbb R^n)$, a vector field $b \in W^{1,q}_{loc}\left(\mathbb R^n, \mathbb R^n \right),$ and a test $k$-form $\mathbf{\Theta} \in C^\infty_0 (\R^n,\bigwedge^k \mathcal{T}^*\mathbb R^n)$ supported on $B(0,R)$ for some $R>0$, for $0 <\epsilon < 1$, there exists a constant $C>0$  independent of $\mathbf{\Theta},\mathbf{K},b,\epsilon$, such that
\begin{align}\label{eq:commutator-estimate}
    \left\lvert\int_{\mathbb R^n} \left([\mathcal L_b, \rho^\epsilon *] \mathbf{K}(x), \mathbf{\Theta}(x)\right) \diff^n x \right\rvert \leq C \norm{\mathbf{\Theta}}_{L^\infty_R}\norm{\mathbf{K}}_{L^p_{R+1}} \norm{b}_{W^{1,q}_{R+1}},
\end{align}
Moreover, $ \left([\mathcal L_b, \rho^\epsilon *] \mathbf{K}, \mathbf{\Theta} \right)$ converges to zero  in $L^1$ as $\epsilon \rightarrow 0$.
\end{lemma}

\begin{proof}
We split the proof in two steps, where we first prove the commutator estimate \eqref{eq:commutator-estimate} and then prove the $L^1$-convergence of $ \left([\mathcal L_b, \rho^\epsilon *] \mathbf{K}, \mathbf{\Theta} \right)$ to zero as $\epsilon \rightarrow 0$.

\noindent\textbf{Step 1.} By the definition of the commutator, we have
\begin{align} \label{twointegrals}
\begin{split}
    & \int_{\mathbb R^n} \left([\mathcal L_b, \rho^\epsilon *] \mathbf{K}(x), \mathbf{\Theta}(x)\right) \diff^n x \\
    &\quad = \int_{\mathbb R^n} \left(\mathcal L_b \mathbf{K}^\epsilon(x), \mathbf{\Theta}(x)\right) \diff^n x - \int_{\mathbb R^n} \left( (\mathcal L_b \mathbf{K})^\epsilon(x), \mathbf{\Theta}(x)\right) \diff^n x.
\end{split}
\end{align}
The first integral on the left-hand side becomes
\begin{align}\label{eq:commutator-first-term}
\begin{split}
    \int_{\mathbb R^n} \left(\mathcal L_b \mathbf{K}^\epsilon(x), \mathbf{\Theta}(x)\right) \diff^n x &= \int_{\mathbb R^n} \left( \mathbf{K}^\epsilon(x), \mathcal L^*_b \mathbf{\Theta}(x)\right) \diff^n x \\
    &= \int_{\mathbb R^n} K^\epsilon_{i_1,\ldots,i_k}(x) (\mathcal L^*_b \mathbf{\Theta})_{i_1,\ldots,i_k}(x) \, \diff^n x,
\end{split}
\end{align}
where the coordinate expression for the adjoint Lie derivative reads (see \eqref{Lieformulaadjoint})
\begin{align} \label{lie-T-theta-sharp}
\begin{split}
    (\mathcal L_b^*\mathbf{\Theta})_{i_1,\ldots,i_k}(x) &=-\frac{\partial}{\partial x_l} (b^l(x) \Theta_{i_1,\ldots,i_k}(x)) + \Theta_{l, i_2,\ldots,i_k}(x) \frac{\partial b^{i_1}}{\partial x_l}(x) \\
    &\quad + \cdots + \Theta_{ i_1,\ldots,i_{k-1}, l}(x) \frac{\partial b^{i_k}}{\partial x_l}(x).
\end{split}
\end{align}
Hence by definition of tensor mollification (Definition \ref{molliformula}), we obtain
\begin{align*}
    & \eqref{eq:commutator-first-term} = - \int_{\mathbb R^n} \int_{\mathbb R^n} \rho^\epsilon(x-y) K_{i_1,\ldots,i_k}(y) \frac{\partial}{\partial x_l} (b^l(x) \Theta_{i_1,\ldots,i_k}(x)) \, \diff^n y \, \diff^n x \\
    & \hspace{40pt} + \int_{\mathbb R^n} \int_{\mathbb R^n} \rho^\epsilon(x-y) K_{i_1,\ldots,i_k}(y)\left(\frac{\partial b^{i_1}}{\partial x_l}(x) \Theta_{l,i_2,\ldots,i_k}(x) \right. \\
    &\hspace{150pt} \left. + \cdots + \frac{\partial b^{i_k}}{\partial x_l}(x) \Theta_{i_1,\ldots,i_{k-1}, l}(x)\right) \diff^n y \, \diff^n x \\
    & \hspace{20pt} = \int_{\mathbb{R}^n} \int_{\mathbb{R}^n} \Bigg(\frac{\partial \rho^\epsilon}{\partial x_l}(x-y) K_{i_1,\ldots,i_k}(y) b^l(x) \Theta_{i_1,\ldots,i_k}(x) \\
    & \hspace{70pt} + \rho^\epsilon(x-y) K_{i_1,\ldots,i_k}(y) \left(\frac{\partial b^{i_1}}{\partial x_l}(x) \Theta_{l,i_2,\ldots,i_k}(x) \right.\\
    & \hspace{150pt} \left.+ \cdots + \frac{\partial b^{i_k}}{\partial x_l}(x) \Theta_{i_1,\ldots,i_{k-1}, l}(x)\right)\Bigg) \diff^n y \, \diff^n x.
\end{align*}
For the second integral in \eqref{twointegrals}, we take into account the property $D_y\rho^\epsilon(x-y) = -D_x\rho^\epsilon(x-y)$ and apply the definition of mollification of distribution-valued $k$-forms (see Definition \ref{distrimolliformula}), obtaining
\begin{align*}
    &\int_{\mathbb R^n} \left( (\mathcal L_b \mathbf{K})^\epsilon(x), \mathbf{\Theta}(x)\right) \diff^n x = \int_{\mathbb R^n} \left( \mathbf{K}(y),\mathcal L_b^*(\rho^\epsilon * \mathbf{\Theta})(y) \right) \diff^n y \nonumber \\
    & = \int_{\mathbb R^n} \int_{\mathbb R^n} \Bigg(\frac{\partial \rho^\epsilon}{\partial x_l}(x-y) K_{i_1,\ldots,i_k}(y) b^l(y) \Theta_{i_1,\ldots,i_k}(x) \\
    &\hspace{100pt} - \rho^\epsilon(x-y) K_{i_1,\ldots,i_k}(y) \frac{\partial b^l}{\partial y_l}(y) \Theta_{i_1,\ldots,i_k}(x) \Bigg) \diff^n y \, \diff^nx 
    \\ 
    &\quad + \int_{\mathbb R^n} \int_{\mathbb R^n} \rho^\epsilon(x-y) K_{i_1,\ldots,i_k}(y) \Bigg(\frac{\partial b^{i_1}}{\partial y_l}(y) \Theta_{l,i_2,\ldots,i_k}(x) \\
    &\hspace{150pt} + \cdots + \frac{\partial b^{i_k}}{\partial y_l}(y) \Theta_{i_1,\ldots,i_{k-1}, l}(x)\Bigg) \diff^n y \, \diff^n x.
\end{align*}
Combining the above expressions, we derive
\begin{align}
    &\int_{\mathbb R^n} \left([\mathcal L_b, \rho^\epsilon *] \mathbf{K}(x), \mathbf{\Theta}(x)\right) \diff^n x \nonumber \\
    &= \int_{\mathbb R^n} \int_{\mathbb R^n} \epsilon^{-n}  \frac{\partial}{\partial y_l} \rho\left(\frac{y-x}{\epsilon}\right) K_{i_1,\ldots,i_k}(y) \left(\frac{b^l(y)-b^l(x)}{\epsilon}\right) \Theta_{i_1,\ldots,i_k}(x) \, \diff^n y \, \diff^nx \nonumber \\
    &\hspace{30pt}+\int_{\mathbb R^n} \int_{\mathbb R^n} \rho^\epsilon(x-y) \Bigg(\frac{\partial b^l}{\partial y_l}(y) \Theta_{i_1,\ldots,i_k}(x) K_{i_1,\ldots,i_k}(y) \nonumber\\
    &\hspace{50pt} -  \left(\frac{\partial b^{i_1}}{\partial y_l}(y) - \frac{\partial b^{i_1}}{\partial x_l}(x)\right) \Theta_{l, i_2,\ldots,i_k}(x)K_{i_1,\ldots,i_k}(y) \\
    &\hspace{50pt} - \cdots - \left(\frac{\partial b^{i_k}}{\partial y_l}(y) - \frac{\partial b^{i_k}}{\partial x_l}(x)\right) \Theta_{i_1,\ldots,i_{k-1}, l}(x)K_{i_1,\ldots,i_k}(y) \Bigg) \diff^n y  \, \diff^nx \nonumber \\
    & \hspace{150pt} =: \int_{\mathbb{R}^n} I_1^\epsilon(x) \, \diff^nx + \int_{\mathbb{R}^n} I_2^\epsilon(x) \, \diff^n x. \label{sumofintegrals} 
\end{align}
Applying Young's convolution inequality (take into account $0<\epsilon<1$ and that $\mathbf{\Theta}$ has support in $B(0,R)$), we have
\begin{align} \label{dominatedi2}
    \left\lvert\int_{\mathbb{R}^n} I_2^\epsilon(x) \, \diff^n x \right\rvert \lesssim \norm{\mathbf{\Theta}}_{L^\infty_R} \norm{\mathbf{K}}_{L^p_{R+1}} \norm{b}_{W^{1,q}_{R+1}},
\end{align}
where we have also employed the fact that $\norm{\rho^\epsilon}_{L^1} = 1$, independently of $\epsilon$.
To obtain an analogous estimate for the first integral in \eqref{sumofintegrals}, we follow the proof of Lemma 2.3 in \cite{lions1996mathematical}. 
First, we note that by Taylor's theorem, we have
\begin{align*}
    b^l(y)-b^l(x) = \int_0^1 (y-x) \cdot Db^l\left(x+\lambda(y-x)\right) \diff \lambda.
\end{align*}
Hence
\begin{align*}
    &I^\epsilon_1(x) = \epsilon^{-n} \int_{ B(x,\epsilon)} \frac{\partial}{\partial y_l} \rho\left(\frac{y-x}{\epsilon}\right) K_{i_1,\ldots,i_k}(y) \left(\frac{b^l(y)-b^l(x)}{\epsilon}\right) \Theta_{i_1,\ldots,i_k}(x) \, \diff^n y \\
    &= \int_{B(0,1)} \frac{\partial}{\partial w_l} \rho\left(w \right) K_{i_1,\ldots,i_k}(x + \epsilon w) \left(\int_0^1 w \cdot Db^l\left(x+\lambda \epsilon w\right) \diff \lambda \right) \Theta_{i_1,\ldots,i_k}(x) \, \diff^n w,
\end{align*}
and therefore
\begin{align*}
    \lvert I^\epsilon_1(x)\rvert &\lesssim \norm{D \rho}_{L^\infty_1} \norm{\mathbf{\Theta}}_{L^\infty_R} \norm{\mathbf{K}}_{L^p_{R+1}} \left(\int_{B(0,1)} \int_0^1 \lvert Db\left(x+\lambda \epsilon w \right)\rvert^q \, \diff \lambda \, \diff^n w \right)^{1/q}.
\end{align*}
Finally, by observing that
\begin{align*}
    \int_{B(0,1)}\int_0^1 \lvert Db\left(x+\lambda \epsilon w\right)\rvert^q \, \diff \lambda \, \diff^n w = (\lvert Db\rvert^q*\chi^\epsilon)(x),
\end{align*}
where $ \chi^\epsilon(z) := \int^1_0 \frac{1}{(\epsilon \lambda)^n} \mathcal{X}_{B(0,\epsilon \lambda)}(z) \, \diff \lambda$ satisfies $\norm{\chi^\epsilon}_{L^1_\epsilon} = \text{meas}(B(0,1))$, and applying H\"older's inequality in space, we obtain
\begin{align}
    & \int_{B(0,R)} \lvert I^\epsilon_1(x)\rvert \, \diff^n x \lesssim \norm{D \rho}_{L^\infty_1} \norm{\mathbf{\Theta}}_{L^\infty_R} \norm{\mathbf{K}}_{L^p_{R+1}} \left(\int_{B(0,R)} ( \lvert Db\rvert^q*\chi^\epsilon)(x) \, \diff^n x \right)^{1/q} \nonumber \\
    & \lesssim \norm{D \rho}_{L^\infty_1} \norm{\mathbf{\Theta}}_{L^\infty_R} \norm{\mathbf{K}}_{L^p_{R+1}} \norm{Db}_{L^q_{R+1}} \norm{\chi^\epsilon}_{L^1_\epsilon}^{1/q} \lesssim \norm{\mathbf{\Theta}}_{L^\infty_R} \norm{\mathbf{K}}_{L^p_{R+1}} \norm{b}_{W^{1,q}_{R+1}}, \label{I1-estimate-b}
\end{align} 
where we have made use of Young's convolution inequality in the second line. Note that none of the constants that appear in the inequalities thus far depend on $\mathbf{\Theta},\mathbf{K},b,\epsilon.$ This proves our bound \eqref{eq:commutator-estimate}.

\noindent\textbf{Step 2.} We now show the $L^1$-convergence of $\left([\mathcal L_b, \rho^\epsilon *] \mathbf{K}, \mathbf{\Theta}\right)$. Taking into account the properties of mollifiers and applying the dominated convergence theorem (use bound \eqref{dominatedi2}), we obtain 
\begin{align*}
I^\epsilon_2 \rightarrow \left( \mathbf{K},\mathbf{\Theta} \right) \div{(b)} 
\end{align*}
in $L^1$ as $\epsilon \rightarrow 0$. Next, for any $R>0$, taking into account the density of $C^\infty_0$ functions in the space $W^{1,q}_{R+1}$, we consider a sequence of $C^\infty_0$-vector fields $\{b_\delta\}_{\delta >0}$ such that $b_\delta \rightarrow b$ in $W^{1,q}_{R+1}$. Define
\begin{align*}
    I_1^\epsilon[b_\delta] &:= \epsilon^{-n} \int_{B(x,\epsilon)} \frac{\partial}{\partial y_l} \rho\left(\frac{y-x}{\epsilon}\right) K_{i_1,\ldots,i_k}(y) \left(\frac{b^l_\delta(y)-b^l_\delta(x)}{\epsilon}\right) \Theta_{i_1,\ldots,i_k}(x) \, \diff^n y \\
    & = \int_{B(0,1)} \frac{\partial \rho}{\partial w_l}\left(w\right) K_{i_1,\ldots,i_k}(x+\epsilon w) \left(\frac{b^l_\delta(x+\epsilon w)-b^l_\delta(x)}{\epsilon}\right) \Theta_{i_1,\ldots,i_k}(x) \, \diff^n w.
\end{align*}
We claim that the convergence
\begin{align}\label{eq:convergence-I[b]-epsilon}
I_1^\epsilon[b_\delta] \rightarrow \left( \mathbf{K}(x),\mathbf{\Theta}(x)\right) \int_{B(0,1)} w_i \frac{\partial b^j_\delta}{\partial x_i}(x) \frac{\partial \rho}{\partial w_j}(w) \, \diff^n w = -\left( \mathbf{K}(x),\mathbf{\Theta}(x)\right) \div{(b_\delta(x))},
\end{align}
holds in $L^1$ as $\epsilon \rightarrow 0$, where the last equality was obtained by integrating by parts, and using $\frac{\partial w_i}{\partial w_j} = \delta_{ij}$ together with $\int_{B(0,1)} \rho(w) \,\diff^n w = 1$.
To show this convergence, we observe that
\begin{align*}
    &\int_{B(0,R)} \left\lvert I_1^\epsilon[b_\delta] - \left( \mathbf{K}(x),\mathbf{\Theta}(x)\right) \int_{B(0,1)} w_i \frac{\partial b^j_\delta}{\partial x_i}(x) \frac{\partial \rho}{\partial w_l}(w) \, \diff^n w \right\rvert \diff^n x \\
    &\lesssim \|\mathbf{\Theta}\|_{L^\infty_R} \|D\rho\|_{L^\infty_1} \int_{B(0,1)}\int_{B(0,R)} \left\lvert K_{i_1,\ldots,i_k}(x+\epsilon w)\left(\frac{b^l_\delta(x+\epsilon w)-b^l_\delta(x)}{\epsilon}\right) \right.\\
    &\left.\hspace{190pt} - w_i K_{i_1,\ldots,i_k}(x) \frac{\partial b^l_\delta}{\partial x_i}(x)\right\rvert \diff^n x \,\diff^n w,
\end{align*}
where we used Fubini's theorem to switch the order of integration. Clearly, if we can show that the integral above vanishes in the limit $\epsilon \rightarrow 0$, then our claim follows. To this end, we apply the triangle inequality and H\"older's inequality to get the following bound on the inner integral
\begin{align*}
    &\int_{B(0,R)} \left\lvert K_{i_1,\ldots,i_k}(x+\epsilon w)\left(\frac{b^l_\delta(x+\epsilon w)-b^l_\delta(x)}{\epsilon}\right) - w_i K_{i_1,\ldots,i_k}(x) \frac{\partial b^l_\delta}{\partial x_i}(x) \right\rvert \diff^n x \\
    &\leq \int_{B(0,R)} \big\lvert K_{i_1,\ldots,i_k}(x+\epsilon w) - K_{i_1,\ldots,i_k}(x) \big\rvert \left\lvert w_i \frac{\partial b^l_\delta}{\partial x_i}(x) \right\rvert \diff^n x \\
    &\quad + \int_{B(0,R)} \big\lvert K_{i_1,\ldots,i_k}(x+\epsilon w)\big\rvert \left\lvert\left(\frac{b^l_\delta(x+\epsilon w)-b^l_\delta(x)}{\epsilon}\right) - w_i \frac{\partial b^l_\delta}{\partial x_i}(x) \right\rvert\diff^n x \\
    &\leq \left(\int_{B(0,R)} \big\lvert  K_{i_1,\ldots,i_k}(x+\epsilon w) - K_{i_1,\ldots,i_k}(x) \big\rvert^p \,\diff^n x\right)^{\frac1p} \left(\int_{B(0,R)} \left\lvert w_i \frac{\partial b^l_\delta}{\partial x_i}(x) \right\rvert^q \diff^n x\right)^{\frac1q} \\
    &\quad + \left(\int_{B(0,R)} \big\lvert K_{i_1,\ldots,i_k}(x+\epsilon w)\big\rvert^p \diff^n x\right)^{\frac1p} \\
    &\hspace{100pt} \times \left(\int_{B(0,R)} \left\lvert \left(\frac{b^l_\delta(x+\epsilon w)-b^l_\delta(x)}{\epsilon}\right) - w_i \frac{\partial b^l_\delta}{\partial x_i}(x) \right\rvert^q\diff^n x\right)^{\frac1q} \\
    &=: J_1^\epsilon(w) + J_2^\epsilon(w).
\end{align*}
We first treat the term $J_1^\epsilon(w)$. Noting that $C_0^\infty$ functions are dense in $L^p_{R}$, consider a sequence of $k$-forms $\{\mathbf{K}^\eta\}_{\eta > 0}$ of class $C_0^\infty$ such that $\mathbf{K}^\eta \rightarrow \mathbf{K}$ in $L^p_{R}$ as $\eta \rightarrow 0$. Applying Minkowski's inequality, we obtain
\begin{align}
    &\left(\int_{B(0,R)} \left\lvert K_{i_1,\ldots,i_k}(x+\epsilon w) - K_{i_1,\ldots,i_k}(x) \right\rvert^p \diff^n x\right)^{\frac1p} \\
    &\leq \left(\int_{B(0,R)} \left\lvert K_{i_1,\ldots,i_k}^\eta(x+\epsilon w) - K_{i_1,\ldots,i_k}(x+\epsilon w) \right\rvert^p \diff^n x\right)^{\frac1p} \nonumber \\
    &\quad + \left(\int_{B(0,R)} \left\lvert K_{i_1,\ldots,i_k}^\eta(x+\epsilon w) - K_{i_1,\ldots,i_k}^\eta(x) \right\rvert^p \diff^n x\right)^{\frac1p} \\
    &\quad + \left(\int_{B(0,R)} \left\lvert K_{i_1,\ldots,i_k}^\eta(x) - K_{i_1,\ldots,i_k}(x) \right\rvert^p \diff^n x\right)^{\frac1p} \nonumber\\
    &\lesssim \|\mathbf{K}^\eta(\, \cdot + \epsilon w) - \mathbf{K}^\eta(\cdot)\|_{L^p_{R}} + \|\mathbf{K}^\eta(\cdot) - \mathbf{K}(\cdot)\|_{L^p_{R}}. \label{eq:K-eta-inequality}
\end{align}
Now fix $\eta>0$ and $w \in B(0,1)$. By the continuity of $\mathbf{K}^\eta$, we have the pointwise convergence $\mathbf{K}^\eta(x + \epsilon w) \rightarrow \mathbf{K}^\eta(x),$ as $\epsilon \rightarrow 0$. Moreover, since $\mathbf{K}^\eta(x + \epsilon w) \leq \|\mathbf{K}^\eta\|_{L_{R+1}^\infty} < \infty$ (take into account $0<\epsilon<1$), we have $\|\mathbf{K}^\eta(\,\cdot + \epsilon w) - \mathbf{K}^\eta(\cdot)\|_{L^p_{R}} \rightarrow 0$ as $\epsilon \rightarrow 0$ by the dominated convergence theorem.
Thus
\begin{align}
    &\lim_{\epsilon \rightarrow 0} \left(\int_{B(0,R)} \left\lvert K_{i_1,\ldots,i_k}(x+\epsilon w) - K_{i_1,\ldots,i_k}(x) \right\rvert^p \diff^n x\right)^{\frac1p} \nonumber\\
    &\quad = \lim_{\eta \rightarrow 0} \lim_{\epsilon \rightarrow 0} \left(\int_{B(0,R)} \left\lvert K_{i_1,\ldots,i_k}(x+\epsilon w) - K_{i_1,\ldots,i_k}(x) \right\rvert^p \diff^n x\right)^{\frac1p} \nonumber\\
    &\quad \stackrel{\eqref{eq:K-eta-inequality}}{\lesssim} \lim_{\eta \rightarrow 0} \lim_{\epsilon \rightarrow 0} \|\mathbf{K}^\eta(\,\cdot + \epsilon w) - \mathbf{K}^\eta(\cdot)\|_{L^p_{R}} + \lim_{\eta \rightarrow 0} \|\mathbf{K}^\eta(\cdot) - \mathbf{K}(\cdot)\|_{L^p_{R}} = 0, \label{eq:epsilon-eta-limit}
\end{align}
which yields the convergence $J_1^\epsilon(w) \rightarrow 0$ as $\epsilon \rightarrow 0$.

The convergence $J_2^\epsilon(w) \rightarrow 0$ as $\epsilon \rightarrow 0$ also follows from the dominated convergence theorem. Indeed, by differentiability of $b_\delta$, we have that $(b^l_\delta(x+\epsilon w)-b^l_\delta(x))/\epsilon$ converges pointwise to $w_i \frac{\partial b^l_\delta}{\partial x_i}(x)$ (sum over repeated indices assumed) and its absolute value is furthermore bounded by $\|Db_\delta\|_{L^\infty_{R+1}}$ due to the mean value theorem, where we took into account $w \in B(0,1)$.

Furthermore, we have the bound $J_1^\epsilon(w) + J_2^\epsilon(w) \lesssim \norm{\mathbf{K}}_{L^p_{{R}+1}} \norm{Db_\delta}_{L^{q}_{{R}+1}}$, which can be shown by following similar arguments as when showing the bound \eqref{I1-estimate-b}. This allow us to conclude that
\begin{align*}
\int_{B(0,1)} \left(J_1^\epsilon(w) + J_2^\epsilon(w)\right) \diff^n w \rightarrow 0
\end{align*}
by the bounded convergence theorem, thus proving our claim.

Finally, by the linearity of $I_1^\epsilon$ (in the $b$-argument) and using estimate \eqref{I1-estimate-b}, we have
\begin{align*}
    &\norm{I_1^\epsilon[b] +\left( \mathbf{K},\mathbf{\Theta}\right) \div{(b)} }_{L^1_R} \\
    &\leq \norm{I_1^\epsilon[b] - I_1^\epsilon[b_\delta]}_{L^1_R} + \norm{I_1^\epsilon[b_\delta] + \left( \mathbf{K},\mathbf{\Theta}\right) \div{(b_\delta)}}_{L^1_R} \\
    &\hspace{150pt} + \norm{\left( \mathbf{K},\mathbf{\Theta}\right) \div{(b)} - \left( \mathbf{K},\mathbf{\Theta}\right) \div{(b_\delta)}}_{L^1_R} \\
    & \lesssim \norm{\mathbf{K}}_{L^p_{R+1}}\norm{\mathbf{\Theta}}_{L^\infty_R} \left(\norm{b-b_\delta}_{W^{1,q}_{R+1}} + \norm{\div{(b)} - \div{(b_\delta)}}_{L^q_{R}} \right) \\
    &\hspace{150pt} + \norm{I_1^\epsilon[b_\delta] + \left( \mathbf{K},\mathbf{\Theta}\right) \div{(b_\delta)}}_{L^1_{R}} \rightarrow 0,
\end{align*}
as $\epsilon \rightarrow 0$, followed by $\delta \rightarrow 0$ (limits taken in this precise order, following arguments similar to \eqref{eq:epsilon-eta-limit}. Note that the constant in the last inequality is independent of $\delta$). Therefore we have shown that $I_1^\epsilon + I_2^\epsilon  \rightarrow -\left( \mathbf{K}, \mathbf{\Theta}\right) \div{(b)} + \left( \mathbf{K},\mathbf{\Theta}\right) \div{(b)} = 0$ in $L^1$ as $\epsilon \rightarrow 0$.
\end{proof}

The following lemma will be applied to treat the double commutator \eqref{commutator}.

\begin{lemma}\label{commutator-estimate-2}
For $p \geq 2$, given a $k$-form $\mathbf{K} \in L^p_{loc} (\R^n,\bigwedge^k \mathcal{T}^*\mathbb R^n)$, a vector field $\xi \in C^2\left(\mathbb R^n, \mathbb R^n \right),$ and a test $k$-form $\mathbf{\Theta} \in C^\infty_0 (\R^n,\bigwedge^k \mathcal{T}^*\mathbb R^n)$ supported on $B(0,R)$ for some $R>0$. Then for $0 <\epsilon < 1$, there exists a constant $C>0$ independent of $\mathbf{\Theta},\mathbf{K},\xi,\epsilon$ such that
\begin{align} \label{comm-ineq-2}
\begin{split}
    &\left\lvert\int_{\mathbb R^n} \left([\mathcal L_{\xi},[\mathcal L_{\xi}, \rho^\epsilon *]] \mathbf{K}(x), \mathbf{\Theta}(x)\right) \diff^n x\right\rvert \\
    &\hspace{50pt}\leq C \norm{\mathbf{\Theta}}_{L^\infty_{R}}\norm{\mathbf{K}}_{L^p_{R+1}} \left(\norm{\xi}_{L^\infty_{R+1}}\norm{D^2\xi}_{L^q_{R+1}} + \norm{D \xi}_{L^\infty_{R+1}}\norm{D \xi}_{L^q_{R+1}} \right).
\end{split}
\end{align}
Moreover, $ \left( [\mathcal L_{\xi},[\mathcal L_{\xi}, \rho^\epsilon *]] \mathbf{K}, \mathbf{\Theta} \right)$ converges to zero in $L^1$ as $\epsilon \rightarrow 0$.
\end{lemma}
\begin{proof}
By a straightforward but lengthy calculation using the explicit formula \eqref{lie-T-theta-sharp} and the property that $D_x\rho(x-y) = -D_y\rho(x-y)$, we find
\begin{align*}
    &\left( [\mathcal L_\xi,[\mathcal L_\xi,\rho^\epsilon*]]\mathbf{K}(x), \mathbf{\Theta}(x) \right) \\
    &= \int_{\mathbb R^n} \frac{\partial^2 \rho}{\partial y_m \partial y_l}(y-x) K_{i_1,\ldots,i_k}(y)\Theta^{i_1,\ldots,i_k}(x)(\xi^l(y) - \xi^l(x))(\xi^m(y)-\xi^m(x)) \,\diff^n y \\
    & \quad + \int_{\mathbb R^n} \frac{\partial \rho^\epsilon}{\partial y_l}(y-x)\Bigg[K_{i_1,\ldots,i_k}(y)\Theta^{i_1,\ldots,i_k}(x)\Bigg(2 \frac{\partial \xi^m}{\partial y_m}(y)(\xi^l(y)-\xi^l(x)) \\
    & \quad+ \xi^m(y)\left(\frac{\partial \xi^l}{\partial y_m}(y) - \frac{\partial \xi^l}{\partial x_m}(x)\right) + \frac{\partial \xi^l}{\partial x_m}(x) (\xi^m(y)-\xi^m(x))\Bigg) \\
    & \quad- 2\Theta^{i_1,\ldots,i_k}(x) \big(\xi^l(y) - \xi^l(x)\big) \Bigg(K_{m,i_2,\ldots,i_k}(y) \left(\frac{\partial \xi^m}{\partial y_{i_1}}(y) - \frac{\partial \xi^m}{\partial x_{i_1}}(x)\right) \\
    &\hspace{50pt} + \cdots + K_{i_1,\ldots,i_{k-1}, m}(y) \left(\frac{\partial \xi^m}{\partial y_{i_k}}(y) - \frac{\partial \xi^m}{\partial x_{i_k}}(x)\right) \Bigg)\Bigg] \diff^n y\\
    & \quad  + \int_{\mathbb R^n} \rho^\epsilon(y-x) \Theta^{i_1,\ldots,i_k}(x)\Bigg[K_{i_1,\ldots,i_k}(y) \frac{\partial}{\partial y_m}\left(\xi^m(y)\frac{\partial \xi^l}{\partial y_l}(y)\right)\\
    & \qquad - K_{l,i_2,\ldots,i_k}(y) \Bigg[\left(\xi^m(y)\frac{\partial^2 \xi^k}{\partial y_m \partial y_{i_1}}(y) - \xi^m(x)\frac{\partial^2 \xi^k}{\partial x_m \partial x_{i_1}}(x)\right) \\
    &\qquad + 2 \frac{\partial \xi^m}{\partial y_m}(y)\left(\frac{\partial \xi^l}{\partial y_{i_1}}(y) - \frac{\partial \xi^l}{\partial x_{i_1}}(x) \right) + \frac{\partial \xi^m}{\partial x_{i_1}}(x) \left(\frac{\partial \xi^l}{\partial y_m}(y) - \frac{\partial \xi^l}{\partial x_m}(x)\right) \\
    & \qquad - \frac{\partial \xi^l}{\partial y_m}(y)\left(\frac{\partial \xi^m}{\partial y_{i_1}}(y) - \frac{\partial \xi^m}{\partial x_{i_1}}(x)\right) \Bigg]\\
    & \qquad - \cdots - K_{i_1,\ldots,i_{k-1}, l}(y) \Bigg[\left(\xi^m(y)\frac{\partial^2 \xi^k}{\partial y_m \partial y_{i_k}}(y) - \xi^m(x)\frac{\partial^2 \xi^k}{\partial x_m \partial x_{i_k}}(x)\right) \\
    & \qquad + 2 \frac{\partial \xi^m}{\partial y_m}(y)\left(\frac{\partial \xi^l}{\partial y_{i_k}}(y) - \frac{\partial \xi^l}{\partial x_{i_k}}(x) \right) + \frac{\partial \xi^m}{\partial x_{i_k}}(x) \left(\frac{\partial \xi^l}{\partial y_m}(y) - \frac{\partial \xi^l}{\partial x_m}(x)\right) \\
    & \qquad - \frac{\partial \xi^l}{\partial y_m}(y)\left(\frac{\partial \xi^m}{\partial y_{i_k}}(y) - \frac{\partial \xi^m}{\partial x_{i_k}}(x)\right) \Bigg]\\
    & \qquad + 2 K_{l,m,i_3, \ldots, i_k}(y)\left(\frac{\partial \xi^l}{\partial y_{i_1}}(y) - \frac{\partial \xi^l}{\partial x_{i_1}}(x)\right)\left(\frac{\partial \xi^m}{\partial y_{i_2}}(y) - \frac{\partial \xi^m}{\partial x_{i_2}}(x)\right) \\
    &+ \cdots + 2 K_{i_1, \ldots, i_{k-2}, l, m}(y)\left(\frac{\partial \xi^l}{\partial y_{i_{k-1}}}(y) - \frac{\partial \xi^l}{\partial x_{i_{k-1}}}(x)\right)\left(\frac{\partial \xi^m}{\partial y_{i_k}}(y) - \frac{\partial \xi^m}{\partial x_{i_k}}(x)\right)\Bigg] \diff^n y\\
    &=: I^\epsilon_1(x) + I^\epsilon_2(x) + I^\epsilon_3(x).
\end{align*}
By Young's convolution inequality, one can verify that for the third integral, we have the bound
\begin{align*}
    \left\lvert I^\epsilon_3(x)\right\rvert \lesssim \norm{\mathbf{\Theta}}_{L^\infty_{R}} \norm{\mathbf{K}}_{L^p_{R+1}} \left(\norm{\xi}_{L^\infty_{R+1}}\norm{D^2\xi}_{L^q_{R+1}} + \norm{D \xi}_{L^\infty_{R+1}}\norm{D \xi}_{L^q_{R+1}}\right).
\end{align*}
For the second integral, we apply a similar strategy as carried out in the proof of Lemma \ref{commutator-estimate-1} to show that
\begin{align*}
    \left\lvert I^\epsilon_2(x)\right\rvert \lesssim \norm{\mathbf{\Theta}}_{L^\infty_{R}}\norm{\mathbf{K}}_{L^p_{R+1}} \left(\norm{\xi}_{L^\infty_{R+1}}\norm{D^2\xi}_{L^q_{R+1}} + \norm{D \xi}_{L^\infty_{R+1}}\norm{D \xi}_{L^q_{R+1}}\right).
\end{align*}
For the first integral, using
\begin{align*}
    \xi^l(y)-\xi^l(x) = \int_0^1 (y-x) \cdot D \xi^l\left(x+\lambda(y-x)\right) \diff \lambda,
\end{align*}
we observe that
\begin{align*}
    &\int_{\mathbb R^n} I^\epsilon_1(x) \,\diff^nx \\
    &= \epsilon^{-n} \int_{\mathbb R^n} \int_{\mathbb R^n} K_{i_1,\ldots, i_k}(y)\Theta^{i_1,\ldots,i_k}(x) \frac{\partial^2 \rho}{\partial y^l \partial y^m}\left(\frac{y-x}{\epsilon}\right) \\
    &\hspace{80pt} \times \left(\frac{\xi^l(y)-\xi^l(x)}{\epsilon}\right)\left(\frac{\xi^m(y)-\xi^m(x)}{\epsilon}\right) \diff^n y \, \diff^n x \\
    &= \int_{B(0,R)} \int_{B(0,1)} K_{i_1,\ldots, i_k}(x+\epsilon \lambda w) D^2_{lm}\rho(w)\left(\int_0^1 w \cdot D \xi^l\left(x+\epsilon \lambda w\right) \diff \lambda \right)\\
    &\hspace{80pt} \times\left(\int_0^1 w \cdot D \xi^m\left(x+\epsilon \lambda w\right) \diff \lambda \right)\Theta^{i_1,\ldots,i_k}(x) \,\diff^n w\,\diff^n x.
\end{align*}
Hence, we have
\begin{align*}
\left\lvert\int_{\mathbb R^n} I_1^{\epsilon}(x)\,\diff^n x \right\rvert
& \lesssim \norm{\mathbf{\Theta}}_{L^\infty_R}\norm{\mathbf{K}}_{L^p_{R+1}} \norm{\lvert D \xi\rvert^2}_{L^q_{R+1}} \\
& \lesssim \norm{\mathbf{\Theta}}_{L^\infty_R}\norm{\mathbf{K}}_{L^p_{R+1}} \norm{D \xi}_{L^\infty_{R+1}}\norm{D \xi}_{L^q_{R+1}},
\end{align*}
where $\chi^\epsilon(z) := \int^1_0 \frac{1}{(\epsilon \lambda)^n} \mathcal{X}_{B(0,\epsilon \lambda)}(z) \, \diff \lambda$ as in Lemma \ref{commutator-estimate-1} and the constant in $\lesssim$ is independent of $\epsilon$. Combining all of the above, we obtain \eqref{comm-ineq-2}. None of the constants appearing in this proof depend on $\mathbf{\Theta},\mathbf{K},\xi,\epsilon.$ Finally, $L^1$-convergence is not difficult to prove.
\end{proof}
As a corollary of Lemmas \ref{commutator-estimate-1} $\&$ \ref{commutator-estimate-2}, we obtain the following result. 

\begin{corollary} \label{necesario}
Let $\phi$ be a $C^1$-diffeomorphism, $\mathbf{K} \in L^p_{loc} (\R^n,\bigwedge^k \mathcal{T}^*\mathbb R^n)$ for $p \geq 2$, $b \in W^{1,q}_{loc}(\mathbb R^n, \mathbb R^n),$ $\xi \in C^2(\mathbb R^n, \mathbb R^n),$ and $\mathbf{\Theta} \in C^\infty_0(\R^n,\bigwedge^k \mathcal{T}^*\mathbb R^n)$ a test $k$-form with support on $B(0,R).$ Then for any $0 < \epsilon < 1,$ choosing $R^* < \infty$ as in Lemma \ref{newbound} (in this case, $R^*$ is deterministic), we have
\begin{align}
    &\bullet \quad \left\lvert \int_{\mathbb R^n} \left( \phi^*[\mathcal L_b, \rho^\epsilon *] \mathbf{K}(x), \,\mathbf{\Theta}(x) \right) \diff^n x \right\rvert \lesssim \|\widetilde{\mathbf{\Theta}}\|_{L^\infty_{R^*}}\|\mathbf{K}\|_{L^p_{R^*+1}}\|b\|_{W^{1,q}_{R^*+1}}, \label{comm-ineq-b} \\ 
    \begin{split}
    &\bullet \quad \left\lvert\int_{\mathbb R^n} \left ( \phi^*[\mathcal L_{\xi},[\mathcal L_{\xi}, \rho^\epsilon *]] \mathbf{K}(x), \,\mathbf{\Theta}(x)\right ) \diff^n x\right\rvert \\ &\quad\lesssim \norm{\widetilde{\mathbf{\Theta}}}_{L^\infty_{R^*}}\norm{\mathbf{K}}_{L^p_{R^*+1}} \left(\norm{\xi}_{L^\infty_{R^*+1}}\|D^2\xi\|_{L^q_{R^*+1}} + \norm{D \xi}_{L^\infty_{R^*+1}}\norm{D \xi}_{L^q_{R^*+1}}\right), \label{comm-ineq-xi}
    \end{split}
\end{align}
where the constants above are independent of $\mathbf{\Theta},\mathbf{K},b,\xi,\epsilon$. Moreover, both $ \left( \phi^*[\mathcal L_b, \rho^\epsilon *] \mathbf{K}, \mathbf{\Theta} \right)$ and $\left( \phi^*[\mathcal L_{\xi},[\mathcal L_{\xi}, \rho^\epsilon *]] \mathbf{K}, \mathbf{\Theta} \right)$ converge to zero in $L^1$ as $\epsilon \rightarrow 0.$
\end{corollary}
\begin{proof}
Using Lemma \ref{coord-change} (note $\phi^* = (\phi^{-1})_*$), we get
\begin{align*}
    &\left\lvert\int_{\mathbb R^n} \left(\phi^*[\mathcal L_b, \rho^\epsilon *] \mathbf{K}(x), \mathbf{\Theta}(x)\right) \diff^n x\right\rvert = \left\lvert\int_{\mathbb R^n} \left([\mathcal L_b, \rho^\epsilon *] \mathbf{K}(x), \widetilde{\mathbf{\Theta}}(x)\right) \diff^n x\right\rvert,
\end{align*}
where $\widetilde{\mathbf{\Theta}}(x) = \lvert J\phi^{-1}(x)\rvert (\phi_* \mathbf{\Theta}^\sharp(x))^\flat$. Applying Lemma \ref{commutator-estimate-1}, we obtain the estimate \eqref{comm-ineq-b}. Likewise, the estimate \eqref{comm-ineq-xi} can be obtained by applying Lemma \ref{commutator-estimate-2}. The $L^1$-convergence claimed at the end of this lemma also follows easily from Lemmas \ref{commutator-estimate-1} $\&$ \ref{commutator-estimate-2} and estimates \eqref{comm-ineq-b} $\&$ \eqref{comm-ineq-xi}.
\end{proof}

We are now ready to prove our uniqueness result.
\begin{proof}[Proof of Theorem \ref{theoremuniqueness}]
We showed earlier by applying the KIW formula that the pull-back of the mollified $k$-form-valued process $\mathbf{K}^\epsilon$ satisfies \eqref{commutatorito}. We observe that in order to prove the uniqueness of solutions to equation \eqref{eq0}, it suffices to show that $\phi_t^*\mathbf{K}(t,\cdot)$ is zero a.s. in the weak sense for all $t \in [0,T]$ if $\mathbf{K}(0,\cdot) = 0,$ since $\mathbb{P}$-a.s. $\phi_t$ is a $C^1$-diffeomorphism for all $t\in [0,T]$.

Fixing a test function $\mathbf{\Theta} \in C^\infty_0 (\R^n,\bigwedge^k \mathcal{T}^*\mathbb R^n)$ with support in $B(0,R),$ taking into account \eqref{eq0}, and using stochastic Fubini's theorem to switch the spatial and time integrals, we have
\begin{align} \label{equnicidad}
\begin{split}
    &\int_{\mathbb R^n}\left(\phi_t^* \mathbf{K}^\epsilon(t,x),\mathbf{\Theta}(x) \right) \diff^nx = \int^t_0 \int_{\mathbb R^n}\left(\phi_s^*[\mathcal L_b, \rho^\epsilon *] \mathbf{K}(s,x),\mathbf{\Theta}(x) \right) \diff^nx \, \diff s \\
    & \quad + \int^t_0 \int_{\mathbb R^n}\left(\phi_s^* [\mathcal L_{\xi}, \rho^\epsilon *] \mathbf{K}(s,x),\mathbf{\Theta}(x) \right) \diff^nx  \, \diff W_s \\
    & \quad + \int_0^t \int_{\mathbb{R}^n} \left( \phi_s^* [ \mathcal L_\xi,[\mathcal L_\xi,\rho^\epsilon*]]\mathbf{K}(s,x), \mathbf{\Theta}(x) \right) \diff^nx \, \diff s.
\end{split}
\end{align}
We note that stochastic Fubini's theorem can be applied in \eqref{equnicidad}, since the uniform bounds \eqref{comm-ineq-b} $\&$ \eqref{comm-ineq-xi} together with Lemma \ref{newbound} yield
\begin{align} \label{fubinibound}
    \left\lvert\int_{\mathbb R^n}\left(\phi_s^* [\mathcal L_b, \rho^\epsilon *] \mathbf{K}(s,x),\mathbf{\Theta}(x) \right) \diff^nx\right\rvert \lesssim  \norm{\widetilde{\mathbf{\Theta}}_s}_{L^\infty_{R^*}}\norm{\mathbf{K}_s}_{L^p_{R^*+1}}\norm{b_s}_{W^{1,q}_{R^*+1}},
\end{align}
and 
\begin{align} \label{fubinibound2}
\begin{split}
    & \left\lvert \int_{\mathbb{R}^n} \left( \phi_s^* [\mathcal L_\xi,[\mathcal L_\xi,\rho^\epsilon*]]\mathbf{K}(s,x), \mathbf{\Theta}(x) \right) \diff^n x \right\rvert  \\
    & \quad \lesssim \norm{\widetilde{\mathbf{\Theta}}_s}_{L^\infty_{R^*}}\norm{\mathbf{K}_s}_{L^p_{R^*+1}} \left(\norm{\xi_s}_{L^\infty_{R^*+1}}\norm{D^2\xi_s}_{L^\infty_{R^*+1}} + \norm{D \xi_s}_{L^\infty_{R^*+1}}\norm{D \xi_s}_{L^\infty_{R^*+1}}\right),
\end{split}
\end{align}
uniformly in $\epsilon$, where $R^* < \infty$ a.s. is chosen such that
\begin{align*}
    \left \{ \phi_t^{-1}(x), \hspace{1mm} t \in [0,T] ,\hspace{1mm} x \in B(0,R) \right \} \subset B(0,R^*),
\end{align*}
where a.s., we know that $\norm{\widetilde{\mathbf{\Theta}}_{\cdot}}_{L^\infty_{R^*}} \in L^\infty([0,T])$. Moreover, by the definition of solution we have $\norm{\mathbf{K}_{\cdot}}_{L^p_{R^*+1}} \in L^p([0,T])$ a.s., and by Hypothesis \ref{stronguniqueness}, we have $\norm{b_{\cdot}}_{W^{1,q}_{R^*+1}} \in L^q([0,T])$, so the RHS of $\eqref{fubinibound}$ is a.s. in $L^1([0,T])$ by a standard application of H\"older's inequality. We now check term-by-term convergence as $\epsilon \rightarrow 0$ in \eqref{equnicidad} for fixed $t \in [0,T]$. 
\begin{itemize}
    \item {The term before the equality.} We claim that a.s.
    \begin{align*}
    &\int_{\mathbb R^n}(\phi_t^* \mathbf{K}^\epsilon(t,x),\mathbf{\Theta}(x)) \, \diff^nx = \int_{\mathbb R^n} \left( \mathbf{K}^\epsilon(t,y), \widetilde{\mathbf{\Theta}}_t(y) \right) \, \diff^ny \\
    &\rightarrow \int_{\mathbb R^n} \left( \mathbf{K}(t,y), \widetilde{\mathbf{\Theta}}_t(y) \right) \, \diff^ny = \int_{\mathbb R^n}\left(\phi_t^* \mathbf{K}(t,x),\mathbf{\Theta}(x) \right) \diff^nx.
    \end{align*}
    For this, we note that a.s.
    \begin{align*}
    \int_{\mathbb R^n} ( \mathbf{K}^\epsilon(t,y)-\mathbf{K}(t,y), \widetilde{\mathbf{\Theta}}_t(y)) \, \diff^ny  \leq \norm{\mathbf{K}^\epsilon(t,\cdot)-\mathbf{K}(t,\cdot)}_{L^1_{R^*}} \norm{\widetilde{\mathbf{\Theta}}_t}_{L^\infty_{R^*}} \rightarrow 0,
    \end{align*}
    pointwise in $t,$ since $\mathbf{K}^\epsilon$ converges to $\mathbf{K}$ in $L^1_{loc}.$
    \item {The first term after the equality.} 
    By Corollary \ref{necesario} \eqref{comm-ineq-b}, we have that
    \begin{align*}
        \left\lvert\int_{\mathbb R^n}\left(\phi_s^* [\mathcal L_b, \rho^\epsilon *] \mathbf{K}(s,x),\mathbf{\Theta}(x) \right) \diff^nx\right\rvert \lesssim  \norm{\widetilde{\mathbf{\Theta}}_s}_{L^\infty_{R^*}}\norm{\mathbf{K}_s}_{L^p_{R^*+1}}\norm{b_s}_{W^{1,q}_{R^*+1}},
    \end{align*}
    where a.s. $\norm{\widetilde{\mathbf{\Theta}}_{\cdot}}_{L^\infty_{R^*}} \in L^\infty([0,T])$, $\norm{\mathbf{K}_{\cdot}}_{L^p_{R^*+1}} \in L^p([0,T]),$ and $\norm{b_{\cdot}}_{W^{1,q}_{R^*+1}} \in L^q([0,T])$. Hence, by the dominated convergence theorem we obtain a.s.
    \begin{align*}
        \lim_{\epsilon \rightarrow 0}\int^t_0 \int_{\mathbb R^n}\left(\phi_s^*[\mathcal L_b, \rho^\epsilon *] \mathbf{K}(s,x),\mathbf{\Theta}(x) \right) \diff^nx \, \diff s = 0.
    \end{align*}
    \item {The double commutator term in the last line.} By \eqref{comm-ineq-xi}, we have 
    \begin{align*}
        & \left\lvert\int_{\mathbb R^n} \left( \phi_s^* [\mathcal L_\xi, [\mathcal L_\xi, \rho^\epsilon *]] \mathbf{K}(s,x), \mathbf{\Theta}(x) \right) \diff^n x \right\rvert \\
        &\quad \lesssim \norm{\widetilde{\mathbf{\Theta}}_s}_{L^\infty_{R^*}} \norm{\mathbf{K}_s}_{L^p_{R^*+1}} \left(\norm{\xi_s}_{L^\infty_{R^*+1}}\norm{D ^2\xi_s}_{L^\infty_{R^*+1}} + \norm{D \xi_s}_{L^\infty_{R^*+1}}\norm{D \xi_s}_{L^\infty_{R^*+1}}\right).
    \end{align*}
    It is easy to show that the bound is in $L^1([0,T])$, and hence by the dominated convergence theorem we have a.s.
    \begin{align*}
        \lim_{\epsilon \rightarrow 0} \frac12 \int^t_0 \int_{\mathbb R^n} \left( \phi_s^* [\mathcal L_\xi, [\mathcal L_\xi, \rho^\epsilon *]] \mathbf{K}(s,x), \mathbf{\Theta}(x) \right) \diff^nx \,\diff s = 0.
    \end{align*}
    \item {The martingale term in the last line.} Similarly, by the stochastic dominated convergence theorem, we obtain
    \begin{align*}
        \lim_{\epsilon \rightarrow 0} \left\lvert\int^t_0 \int_{\mathbb R^n}\left(\phi_s^* [\mathcal L_{\xi}, \rho^\epsilon *] \mathbf{K}(s,x),\mathbf{\Theta}(x) \right) \diff^nx  \, \diff W_s \right\rvert = 0
    \end{align*}
    in probability.
\end{itemize}
\end{proof}

\subsection{Uniqueness without weak differentiability of the drift} \label{unicidaddebil}
This subsection is devoted to proving Theorem \ref{theoremweakuniqueness}.
\begin{lemma} \label{commutator-sharp}
Let $p\geq 2,$ $\mathbf{K} \in L^p_{loc} (\R^n,\bigwedge^k \mathcal{T}^*\mathbb R^n)$, $b \in L^\infty_{loc}(\mathbb{R}^n,\mathbb{R}^n)$, and a test $k$-form $\mathbf{\Theta} \in C^\infty_0 (\R^n,\bigwedge^k \mathcal{T}^*\mathbb R^n)$ supported on $B(0,R)$ for some $R>0$.
\begin{enumerate}
    \item{If $k=0$ and $\div{(b)} \in L^q(\mathbb{R}^n),$ then
    \begin{align} \label{primeruni}
    \begin{split}
    & \left\lvert\int_{\mathbb R^n} \left([\mathcal L_b, \rho^\epsilon *] \mathbf{K}(x), \mathbf{\Theta}(x)\right) \diff^n x \right\rvert \\
    & \qquad \leq 2 \norm{\mathbf{K}}_{L^p_{R+1}} \left( \norm{\div{(b)}}_{L^q_{R+1}} \norm{\mathbf{\Theta}}_{L_R^\infty} + \norm{b}_{L^\infty_{R+1}} \norm{D \mathbf{\Theta}}_{L^q_R} \right).
    \end{split}
    \end{align}}
    \item{If $k=n$ then
    \begin{align} \label{terceruni}
    \left\lvert\int_{\mathbb R^n} \left([\mathcal L_b, \rho^\epsilon *] \mathbf{K}(x), \mathbf{\Theta}(x)\right) \diff^n x \right\rvert \leq 2 \norm{\mathbf{K}}_{L^p_{R+1}}  \norm{b}_{L^\infty_{R+1}} \norm{D \mathbf{\Theta}}_{L^q_R}.
    \end{align}} 
\end{enumerate}
Moreover, in all cases $ \left([\mathcal L_b, \rho^\epsilon *] \mathbf{K}, \mathbf{\Theta} \right)$ converges to zero in $L^1$ as $\epsilon \rightarrow 0$.
\end{lemma}
\begin{proof}
By following our proof of Lemma \ref{commutator-estimate-1}, we have
\begin{align} \label{casosespeciales}
\begin{split}
    &  \int_{\mathbb R^n} \left([\mathcal L_b, \rho^\epsilon *] \mathbf{K}(x), \mathbf{\Theta}(x)\right) \diff^n x  \\
    &= -\int_{\mathbb R^n} \int_{\mathbb R^n} \rho^{\epsilon}\left(x-y\right) K_{i_1,\ldots,i_k}(y) \left(b^l(x)-b^l(y)\right) \frac{\partial \Theta_{i_1,\ldots,i_k}}{\partial x_l}(x) \, \diff^n y \, \diff^nx \\
    & \quad - \int_{\mathbb R^n} \int_{\mathbb R^n} \rho^\epsilon \left(x-y\right) K_{i_1,\ldots,i_k}(y) \left(\frac{\partial b^l}{\partial x_l}(x)\right) \Theta_{i_1,\ldots,i_k}(x) \, \diff^n y \, \diff^nx \\
    &\quad+\int_{\mathbb R^n} \int_{\mathbb R^n} \rho^\epsilon(x-y) \Bigg(\frac{\partial b^l}{\partial y_l}(y) \Theta_{i_1,\ldots,i_k}(x) K_{i_1,\ldots,i_k}(y) \\
    &\qquad - \left(\frac{\partial b^{i_1}}{\partial y_l}(y) - \frac{\partial b^{i_1}}{\partial x_l}(x)\right) \Theta_{l, i_2,\ldots,i_k}(x)K_{i_1,\ldots,i_k}(y) \\
    &\qquad - \cdots - \left(\frac{\partial b^{i_k}}{\partial y_l}(y) - \frac{\partial b^{i_k}}{\partial x_l}(x)\right) \Theta_{i_1,\ldots,i_{k-1}, l}(x)K_{i_1,\ldots,i_k}(y) \Bigg)\diff^n y  \, \diff^n x.
\end{split}
\end{align}
Inequality \eqref{primeruni} follows by observing that all the terms in \eqref{casosespeciales} except for the one in the first line cancel, and applying Young's convolution inequality. \eqref{terceruni} can be established by similar inspection mechanisms. To establish the $L^1$-convergence, we follow similar arguments as in the proof of Lemma \ref{commutator-estimate-1}.
\end{proof}

\begin{corollary} 
Let $\phi$ be a $C^1$-diffeomorphism with inverse $\psi,$ $\mathbf{K} \in L^p_{loc} (\R^n,\bigwedge^k \mathcal{T}^*\mathbb R^n)$, $b \in L^\infty_{loc}(\mathbb{R}^n,\mathbb{R}^n),$ and a test $k$-form $\mathbf{\Theta} \in C^\infty_0(\R^n,\bigwedge^k \mathcal{T}^*\mathbb R^n)$ supported on $B(0,R).$ 
\begin{enumerate}
    \item{If $k=0$ and $\div{(b)} \in L^{r}(\mathbb{R}^n),$ for any $r>2,$ we have
    \begin{align} \label{cotapotente1}
    \begin{split}
    &\left\lvert\int_{\mathbb R^n} \left( \phi^* [\mathcal L_b, \rho^\epsilon *] \mathbf{K}(x), \mathbf{\Theta}(x)\right) \diff^n x \right\rvert \\
    &\lesssim \norm{\mathbf{K}}_{L^p_{R^* +1}} \left( \norm{\div{(b)}}_{L^q_{R^*+1}} \norm{J \psi}_{L_{R^*}^\infty} + \norm{b}_{L^\infty_{R^*+1}} \left(\norm{D \psi}^2_{L^\infty_{R^*}} + \norm{D J \psi}_{L^q_{R^*}} \right) \right),
    \end{split}
    \end{align}} 
    \item{if $k=n$ and $\div{(b)} \in L^{r}(\mathbb{R}^n),$ for any $r>2q,$ we have
    \begin{align} \label{cotapotente3}
    \begin{split}
    &\left\lvert\int_{\mathbb R^n} \left( \phi^* [\mathcal L_b, \rho^\epsilon *] \mathbf{K}(x), \mathbf{\Theta}(x)\right) \diff^n x \right\rvert  \\
    &\lesssim \norm{\mathbf{K}}_{L^p_{R^*+1}}  \norm{b}_{L^\infty_{R^*+1}} \left(\norm{D \psi}^2_{L^\infty_{R^*}} + \norm{D J \psi}_{L^{2q}_{R^*}} \right) \norm{DJ\phi}_{L_{R^*}^{2q}} \norm{D\psi}_{L_{R^*}^{\infty}},
    \end{split}
    \end{align}} 
\end{enumerate}
where, as usually, $R^*$ is chosen such that $\{ \phi^{-1}(x), \hspace{1mm} x \in B(0,R) \} \subset B(0,R^*)$ and the constants in $\lesssim$ do not depend on $\phi,\mathbf{K},b,\epsilon.$ Moreover, in all cases $\left( \phi^* [\mathcal L_b, \rho^\epsilon *] \mathbf{K}, \mathbf{\Theta} \right)$ converges to zero in $L^1$ as $\epsilon \rightarrow 0$.
\end{corollary}
\begin{proof}
We first establish \eqref{cotapotente1}. Using Lemma \ref{coord-change} and \eqref{primeruni} in Lemma \ref{commutator-sharp}, we get
\begin{align*}
    &\left\lvert \int_{\mathbb R^n} \left( \phi^* [\mathcal L_b, \rho^\epsilon *] \mathbf{K}(x), \mathbf{\Theta}(x)\right) \diff^n x \right\rvert = \left\lvert\int_{\mathbb R^n} \left([\mathcal L_b, \rho^\epsilon *] \mathbf{K}(x), \widetilde{\mathbf{\Theta}}(x)\right) \diff^n x \right\rvert \\
    &\quad \lesssim \norm{\mathbf{K}}_{L^p_{R^*+1}} \left( \norm{\div{(b)}}_{L^q_{R^*+1}} \norm{\widetilde{\mathbf{\Theta}}}_{L_{R^*}^\infty} + \norm{b}_{L^\infty_{R^*+1}} \norm{D \widetilde{\mathbf{\Theta}}}_{L^q_{R^*}} \right).
\end{align*}
Since $k=0,$ we have $\widetilde{\mathbf{\Theta}}(y) = \lvert J\psi(y) \rvert \mathbf{\Theta}(\psi(y))$, which is supported on $B(0,R^*)$ by Lemma \ref{newbound}. Finally, observe that 
\begin{align} \label{k=0}
\begin{split}
\norm{ D \widetilde{\mathbf{\Theta}}}_{L_{R^*}^{q}} &\leq \norm{[D\psi]^{T} ( D \mathbf{\Theta} \circ \psi )
J\psi}
_{L_{{R^*}}^{q}}+\norm{ ( \mathbf{\Theta} \circ
\psi) DJ\psi}_{L_{{R^*}}^{q}} \\
&\leq \norm{D\psi}^2_{{L_{{R^*}}^{\infty }}} \norm{D\mathbf{\Theta}}_{L_{{R^*}}^{q}}+\norm{\mathbf{\Theta}}_{L_{{R^*}}^{\infty
}}\norm{DJ\psi} _{L_{{R^*}}^{q}},
\end{split}
\end{align}
where we have taken into account Theorem \ref{regularidad1}.
To show \eqref{cotapotente3}, applying Lemma \ref{coord-change} and \eqref{terceruni} in Lemma \ref{commutator-sharp}, we obtain
\begin{align*} 
    \left\lvert \int_{\mathbb R^n} \left( \phi^* [\mathcal L_b, \rho^\epsilon *] \mathbf{K}(x), \mathbf{\Theta}(x)\right) \diff^n x \right\rvert &= \left\lvert \int_{\mathbb R^n} \left( [\mathcal L_b, \rho^\epsilon *] \mathbf{K}(x), \widetilde{\mathbf{\Theta}}(x)\right) \diff^n x \right\rvert \\ & \lesssim  \norm{\mathbf{K}}_{L^p_{R+1}}  \norm{b}_{L^\infty_{R+1}} \norm{D \widetilde{\mathbf{\Theta}}}_{L^q_R}.
\end{align*}
Since $k=n,$ we have $\widetilde{\mathbf{\Theta}}(y) = \lvert J\psi(y)\rvert (\phi_* \mathbf{\Theta}^\sharp(y))^\flat,$ with $(\phi_* \mathbf{\Theta}^\sharp)(y) = \mathbf{\Theta}(\psi(y)) \lvert (J\phi) (\psi(y))\rvert.$ By the same argument in \eqref{k=0} and using H\"older's inequality, we have the control
\begin{align*}
\norm{ D \widetilde{\mathbf{\Theta}}}_{L_{R^*}^{q}} \leq \left( ||D\psi||^2_{{L_{{R^*}}^{\infty }}} \norm{D\mathbf{\Theta}}_{L_{{R^*}}^{2q}}+\norm{\mathbf{\Theta}}_{L_{{R^*}}^{\infty
}}\norm{DJ\psi} _{L_{{R^*}}^{2q}} \right) \norm{DJ\phi}_{L_{R^*}^{2q}} \norm{D\psi}_{L_{R^*}^{\infty}},
\end{align*}
where we have taken into account Theorem \ref{regularidad1}. Finally, it is not difficult to show $L^1$-convergence of $\left( \phi^* [\mathcal L_b, \rho^\epsilon *] \mathbf{K}, \mathbf{\Theta} \right).$
\end{proof}

Finally, by following the same arguments as in the proof of Theorem \ref{theoremuniqueness} in Section \ref{nonsmooth-unique}, we can apply the dominated convergence theorem to conclude Theorem \ref{theoremweakuniqueness}.

\section{Ill-posedness of the deterministic equation} \label{counterexample}
This section is devoted to establishing Theorem \ref{ill}. Without noise, \eqref{eq0} becomes
\begin{align} \label{deq0}
\begin{split}
    & \partial_t \mathbf{K}_t(x) + \mathcal L_b \mathbf{K}_t(x) = 0,  \\
    & \mathbf{K}_0(\cdot) = \mathbf{K}_0, \quad t \in [0,T], \quad x \in \mathbb{R}^n.
\end{split}
\end{align} 

We start by noting that the natural equivalent of Definition \ref{weak-sol} in the deterministic case is the following.
\begin{definition}[Deterministic weak $L^p$-solution] \label{weak-sol-det}
Let $p \geq 2$ and $b  \in W^{1,q}_{loc}(\mathbb{R}^n,\mathbb{R}^n)$ with $1/q+1/p=1.$ We say that a $k$-form-valued time-dependent function $\mathbf{K}$ satisfies equation \eqref{deq0} weakly in the $L^p$-sense if
\begin{itemize}
\item $\mathbf{K} \in L^p ([0,T] \times \mathbb{R}^n ; \bigwedge^k \mathcal{T}^*\mathbb R^n).$
\item For any test $k$-form $\mathbf{\Theta} \in C^\infty_0 (\bigwedge^k \mathcal{T}^*\mathbb R^n)$, the function $\llangle \mathbf{K}_t, \mathbf{\Theta} \rrangle_{L^2}$ is continuous (a.e.) and \eqref{weak-eq} with $\xi_i=0,$ $i=1,\ldots,N$ is satisfied.
\end{itemize}
\end{definition}
\subsection{Nonuniqueness of deterministic solutions}
During the following arguments, we choose $T>1.$ If $T \leq 1,$ the solutions to the characteristic ODE \eqref{dflow} below are exactly the same but restricted. Let $\alpha \in (0,1)$ and $R>0$. Consider the autonomous vector field 
\begin{align} \label{particularb}
    b_{\alpha,R}(x) = \frac{1}{1-\alpha}\frac{x}{\lvert x \rvert} (\min \{\lvert x\rvert,R\})^\alpha, \quad x \in \R^n,
\end{align}
where we assumed the convention $\frac{0}{0}=0.$ A straightforward computation shows that $b_{\alpha,R} \in C^\alpha_b(\R^n,\R^n)$. Moreover, $b_{\alpha,R} \in W^{1,q}_{loc}(\R^n,\R^n)$ for all $q \in [1,2]$ when $n \geq 2$ and when $n=1,\alpha \in ((q-1)/q,1)$, since its partial derivatives behave like $1/\lvert x\rvert^{1-\alpha}$ near zero. The associated characteristic ODE reads
\begin{align}\label{dflowgeneral}
\begin{split}
    & \partial_t \phi_t(x) = b_{\alpha,R}(\phi_t(x)), \\
    & \phi_0(x) = x, \quad t \in [0,T], \quad x \in \mathbb{R}^n.
\end{split}
\end{align}
First of all, notice that by means of expressing the equation in terms of the re-scaled variables
$y=x/R$ and $\tau=t/R^{1-\alpha}$, it suffices to consider the case $R=1$ without loss of generality. Thus, to simplify notation, from now on we set $R=1$ and denote $b_\alpha := b_{\alpha,1}$. Hence \eqref{dflowgeneral} becomes
\begin{align}\label{dflow}
\begin{split}
    & \partial_t \phi_t(x) = b_{\alpha}(\phi_t(x)), \\
    & \phi_0(x) = x.
\end{split}
\end{align}
For $x\neq 0$, it can be verified that the unique solution to \eqref{dflow} has the explicit form
\begin{align}
\resizebox{1.0 \textwidth}{!}{
\begin{math}
    \phi_t(x)=
\begin{cases} \label{phi-eq}
    \phi^0_t(x)= \dfrac{x}{\lvert x\rvert} \left( t+\lvert x\rvert^{1-\alpha}\right)^{\frac{1}{1-\alpha}}, \quad  0 \leq t + \lvert x\rvert^{1-\alpha} \leq 1, \quad 0< \lvert x\rvert \leq 1, \\ 
    \phi^1_t(x)= \dfrac{x}{\lvert x\rvert}\left(\frac{1}{1-\alpha}(t + \lvert x\rvert^{1-\alpha}-1)+1 \right),\quad t + \lvert x\rvert^{1-\alpha} > 1,\quad 0<\lvert x\rvert \leq  1, \\ 
    \phi^2_t(x)=  \dfrac{x}{\lvert x\rvert} \left(\frac{t}{1-\alpha}+ \lvert x\rvert\right), \quad \lvert x\rvert > 1,
\end{cases}
\end{math}
}
\end{align}
for $t \in [0,T].$ Moreover, one can check that for every $t \in [0,T],$ $\phi_t$ maps $\mathbb{R}^n \backslash \{0\}$ continuously and bijectively onto the set $\mathcal{A}_{t},$ where
\begin{align*}
    \mathcal{A}_{t} :=
\begin{cases}
     \mathbb{R}^n \backslash \overline{B}(0,t^{\frac{1}{1-\alpha}}), \quad  0\leq t \leq 1, \\ 
     \mathbb{R}^n \backslash  \overline{B}(0,\frac{t}{1-\alpha}-\frac{\alpha}{1-\alpha}), \quad  t > 1.
\end{cases}
\end{align*}
We have assumed the convention $\overline{B}(0,0) := \{0\}.$ The solution $\phi$ in \eqref{phi-eq} is of class $C^\infty$ in both time and space everywhere except on the sets
\begin{align*}
 \mathcal{S}^{\phi}_{1} := \left\{(t,x) \in [0,1] \times \R^n \backslash \{0\}: \lvert x\rvert = (1 - t)^{\frac{1}{1-\alpha}} \right\}
\end{align*}
and
\begin{align*}
 \mathcal{S}^{\phi}_{2} := \left\{(t,x) \in [0,T] \times \R^n: \lvert x\rvert = 1 \right\}.
\end{align*}
Nevertheless, we know that $\phi$ is at least continuous on $\mathcal{S}^{\phi} := \mathcal{S}^{\phi}_{1}\cup \mathcal{S}^{\phi}_{2}$. Let us first show that $\phi_t \in C^1(\R^n\backslash \{0\},\mathcal{A}_{t}),$ for all $t \in [0,T]$. To that purpose, by computing the first order (spatial) partial derivatives of $\phi_t$, we find that 
\begin{align*}
\resizebox{1.05 \textwidth}{!}{
\begin{math}
\frac{\partial \phi^{i}_t}{\partial x_j}(x) =
\begin{cases} 
   \left( t + \lvert x\rvert^{1-\alpha} \right)^{\frac{\alpha}{1-\alpha}} \dfrac{x_i x_j}{\lvert x\rvert^{2-\alpha}} + (t+\lvert x\rvert^{1-\alpha})^{\frac{1}{1-\alpha}} \dfrac{\delta_{ij} \lvert x\rvert - \frac{x_i}{\lvert x\rvert}x_j}{\lvert x\rvert^2}, \quad 0\leq t + \lvert x\rvert^{1-\alpha} \leq 1, \quad 0<\lvert x\rvert\leq 1,\\
   \dfrac{\delta_{ij} \lvert x\rvert - \frac{x_i}{\lvert x\rvert}x_j}{\lvert x\rvert^2} \left(\dfrac{t-1 + \lvert x\rvert^{1-\alpha}}{1-\alpha}+ 1\right) + \dfrac{x_i x_j}{\lvert x\rvert^{2-\alpha}}, \quad t + \lvert x\rvert^{1-\alpha} > 1,\quad 0< \lvert x\rvert \leq 1, \\ 
   \dfrac{\delta_{ij} \lvert x\rvert - \frac{x_i}{\lvert x\rvert}x_j}{\lvert x\rvert^2}\left( \dfrac{t}{1-\alpha}+ \lvert x\rvert\right) + \dfrac{x_i x_j}{\lvert x\rvert^2}, \quad \lvert x\rvert > 1,
\end{cases}
\end{math}
}
\end{align*}
for $i,j \in \{1,\ldots,n\}$. It is easy to check that all partial derivatives coincide on $\mathcal{S}^{\phi}$. Indeed,
\begin{align*}
\frac{\partial (\phi^0)^i_{t}}{\partial x_j}(x) & =\frac{\partial (\phi^1)^i_{t}}{\partial x_j}(x), \quad (t,x) \in  \mathcal{S}^{\phi}_{1},\\
\frac{\partial (\phi^1)^i_{t}}{\partial x_j}(x) & =\frac{\partial (\phi^2)^i_{t}}{\partial x_j}(x), \quad  (t,x) \in  \mathcal{S}^{\phi}_{2}, \\ 
\frac{\partial (\phi^0)^i_{t}}{\partial x_j}(x) & =\frac{\partial (\phi^2)^i_{t}}{\partial x_j}(x), \quad  (t,x)\in  \mathcal{S}^{\phi}_{1}\cap \mathcal{S}^{\phi}_{2}.
\end{align*}
Similarly, it can be verified that for fixed $x\in\R^n\backslash \{0\},$ $\phi_{\cdot}(x) \in C^1([0,T])$. Therefore, we have shown that $\phi\in C^1([0,T]\times \R^n\backslash \{0\}; \R^n\backslash \{0\}).$
Furthermore, by means of the Inverse Function Theorem or by direct computation, one can also check that the inverse of $\phi_{t}$, denoted by $\psi_{t} := \phi^{-1}_t$, is also of class $\psi\in C^1 \left(\bigcup_{t \in [0,T]}\left( \{t\}\times\mathcal{A}_{t} \right);\R^n\backslash \{0\} \right).$ Moreover, the explicit form of $\psi$ is 
\begin{align}
\resizebox{1.0 \textwidth}{!}{
\begin{math}
    \psi_t(x)=
\begin{cases} \label{phi-eq:inverse}
    \psi^0_t(x)= \dfrac{x}{\lvert x\rvert} \left( \lvert x\rvert^{1-\alpha}-t\right)^{\frac{1}{1-\alpha}}, \quad  t^{\frac{1}{1-\alpha}}<\lvert x\rvert\leq 1, \quad 0\leq t < 1, \\ 
    \psi^1_t(x) = \dfrac{x}{\lvert x\rvert}\left((1-\alpha)\lvert x\rvert+\alpha-t \right)^{\frac{1}{1-\alpha}}, \quad 1 < \lvert x\rvert\leq 1+ \frac{t}{1-\alpha}, \quad 0 < t <1, \\
    \psi^2_t(x)= \dfrac{x}{\lvert x\rvert}\left((1-\alpha)\lvert x\rvert+\alpha-t \right)^{\frac{1}{1-\alpha}}, \quad \frac{t-\alpha}{1-\alpha}  < \lvert x\rvert\leq 1+ \frac{t}{1-\alpha}, \quad t \geq 1, \\
    \psi^3_t(x)=  \dfrac{x}{\lvert x\rvert} \left(\lvert x\rvert-\frac{t}{1-\alpha} \right), \quad \lvert x\rvert > 1+ \frac{t}{1-\alpha}, \quad t >0.
\end{cases}
\end{math}
}
\end{align}
Consequently, $\psi$ is of class $C^\infty$ in time and space everywhere except on the singular sets 
\begin{align*}
 \mathcal{S}^{\psi}_{1} := \left\{(t,x) \in [0,1] \times \R^n : \lvert x\rvert =1 \right\}
\end{align*}
and
\begin{align*}
 \mathcal{S}_{2}^{\psi} := \left\{(t,x) \in [0,T] \times \R^n : \lvert x\rvert = 1+\frac{t}{1-\alpha} \right\}.
\end{align*}

For $x = 0$, the ODE \eqref{dflow} has infinitely many solutions of the form
\begin{align*}
    \varphi^{v}_{t}=
\begin{cases} 
    v  t^{\frac{1}{1-\alpha}}, \quad  0\leq t \leq 1, \\ 
     v \left( \frac{1}{1-\alpha}(t-1)+1\right), \quad  t > 1,
\end{cases}
\end{align*}
where $v$ is any unitary vector, i.e. $v\in \mathbb{R}^{n}$ with $\lvert v \rvert=1$.  Moreover, fixing $t_0$ such that $0 < t_0 \leq t \leq T$, we can construct via translation a shifted family of solutions
\begin{align*}
    \varphi^{v}_{t-{t_0}}=
\begin{cases} 
     v  \left(t-t_{0}\right)^{\frac{1}{1-\alpha}}, \quad  0\leq t-t_{0} \leq 1, \\ 
     v \left( \frac{1}{1-\alpha}(t-t_{0}-1)+1\right), \quad  t-t_{0} > 1.
\end{cases}
\end{align*}
By expressing uniquely any element $x\in B(0,t)\backslash \{0\}$ in the form $x = t_0 v,$ where $t_0 \in (0,t]$ and $\lvert v \rvert=1$, we can define a diffeomorphism $\Phi_t: B(0,t)\backslash \{0\} \rightarrow B(0,t^{\frac{1}{1-\alpha}})\backslash \{0\}$ via the relation $t_0 v \mapsto \varphi^{v}_{t-{t_0}}.$ More explicitly, we have
\begin{align*}
\resizebox{1.0 \textwidth}{!}{
\begin{math}
    \Phi_t(x)=
\begin{cases} 
     \Phi_t^{0}(x)=\dfrac{x}{\lvert x\rvert}\left(t-\lvert x\rvert\right)^{\frac{1}{1-\alpha}} ,\quad 0\leq t-\lvert x\rvert\leq 1, \quad 0< \lvert x\rvert< t, \\ 
     \Phi_t^{1}(x)= \dfrac{x}{\lvert x\rvert}\left(\frac{1}{1-\alpha}(t-\lvert x\rvert-1)+1\right), \quad t-\lvert x\rvert> 1,\quad \lvert x\rvert>0, \quad t> 1,
\end{cases}
\end{math}
}
\end{align*}
for $t \in (0,T]$ and $x \in B(0,t) \backslash \{0\}$. One can check that $ \Phi :\bigcup_{t\in (0,T]} \left(\{t\}\times B(0,t)\backslash \{0\} \right) \rightarrow \R^n\backslash \{0\}$ is of class $C^{\infty}$ in time and space everywhere except on the singular set
\begin{align*}
 \mathcal{S}^{\Phi} := \left\{(t,x) \in [1,T] \times \R^n : \lvert x\rvert = t-1 \right\}.
\end{align*}
The inverse of $\Phi_t$, denoted by $\Psi_t := \Phi_t^{-1}$, can be computed explicitly as
\begin{align} \label{psi-eq:inverse}
\resizebox{1.0 \textwidth}{!}{ 
\begin{math}
    \Psi_t(x)=
\begin{cases} 
     \Psi_t^{0}(x)=\dfrac{x}{\lvert x\rvert}\left(t-\lvert x\rvert^{1-\alpha}\right), \quad 0<\lvert x\rvert\leq t^{\frac{1}{1-\alpha}}, \quad 0 < t \leq 1, \\ 
     \Psi_t^{1}(x)=\dfrac{x}{\lvert x\rvert}\left(t-\lvert x\rvert^{1-\alpha}\right), \quad 0 < \lvert x\rvert\leq 1, \quad t >1, \\ 
     \Psi_t^{2}(x)= \dfrac{x}{\lvert x\rvert}\left((\alpha-1)\lvert x\rvert-\alpha+t\right), \quad 1< \lvert x\rvert\leq 1+\frac{t-1}{1-\alpha}, \quad t >1,
\end{cases}
\end{math}
}
\end{align}
with singular set
\begin{align*}
 \mathcal{S}^{\Psi} := \left\{(t,x) \in [1,T]\times \R^{n}: \lvert x\rvert = 1\right\}.
\end{align*}
One can check that $\Phi$ is of class $C^1\left(\bigcup_{t\in (0,T]} \left(\{t\}\times B(0,t)\backslash \{0\} \right);\R^n\backslash \{0\}\right)$ and its inverse $\Psi$ is of class $C^1\left(\bigcup_{t\in (0,T]} \left(\{t\}\times B(0,t^{\frac{1}{1-\alpha}})\backslash \{0\} \right);\R^n\backslash \{0\}\right).$

Now, we will construct explicit locally in space strong solutions to the transport equation \eqref{eq0} with vector field $b_{\alpha}$. In the next step, we will employ this to construct $L^p$-solutions to \eqref{eq0} (i.e. in the sense of Definition \ref{weak-sol-det}). Since the aim of the first part of this section is showing that equation \eqref{eq0} is ill-posed in $L^p$ spaces and not developing a deterministic theory, we will assume for simplicity that the support of the initial datum $\mathbf{K}_0$ is contained in $B(0,1/2),$ and we will choose $T \leq 1-1/2^{1-\alpha}$. The general case works similarly, although the computations become more involved. Before stating an important result, let us first introduce new notation that will simplify the exposition of the arguments hereafter. 

We define the time dependent open sets
\begin{align*}
\mathcal{A}_{t}^0 := B(0,1)\backslash  \overline{B}(0,t^{\frac{1}{1-\alpha}}), 
\end{align*}
where $t \in [0,T]$. Observe that $\mathcal{A}_t^0$ is the open set where $\psi_t^0$ is defined (see \eqref{phi-eq:inverse}). Similarly, the set
\begin{align*}
& \mathcal{B}_{t}^0 :=
\begin{cases}
B(0,t^{\frac{1}{1-\alpha}})\backslash \{0\}, \quad 0 < t \leq T, \\
\emptyset, \quad t = 0,
\end{cases}  
\end{align*}
is exactly the open set where
$\Psi^{0}_{t}$ is defined. 

With this notation at hand, we present the following result:

\begin{proposition} \label{strongsolution}
Let $\mathbf{K}_0,\mathbf{\Gamma} \in C^\infty(\R^n,\bigwedge^k  \mathcal{T}^*\mathbb R^n)$ be supported on $B(0,1/2)$ and $0<T \leq 1-1/2^{1-\alpha}$. Then the $k$-form-valued function defined by
\begin{align} \label{K:form:sol:gamma}
\mathbf{K}_t^\mathbf{\Gamma}(x) = \begin{cases} 
      (\phi^0_t)_* \mathbf{K}_0(x), & x \in \mathcal{A}^0_t, \quad t\in [0,T],\\
      (\Phi^{0}_t)_* \mathbf{\Gamma}(x), & x \in \mathcal{B}^0_t, \quad t \in [0,T], \\
      0, & \textit{elsewhere}, 
   \end{cases}
\end{align}
solves the Lie transport equation \eqref{eq0} strongly on the open sets $\mathcal{A}^0_t,$ $\mathcal{B}^0_t,$ $t \in [0,T].$
\end{proposition}

\begin{proof}\label{s:521}
Before starting the computation, we first note that by taking into account the regularity shown for $\psi$ and $\Psi$ in the previous steps, the $k$-form-valued function $\mathbf{K}^\mathbf{\Gamma}$ defined in \eqref{K:form:sol:gamma} is of class $\mathbf{K}^\mathbf{\Gamma} \in  C^{\infty}\left( \bigcup_{t\in(0,T)} \left( \{t\} \times ( \mathcal{A}_t^0 \cup \mathcal{B}_t^0) \right); \bigwedge^k \mathcal{T}^*\mathbb R^n \right).$ We have to verify that the following identity holds
\begin{align} \label{ultima}
\partial_t \mathbf{K}_t^\mathbf{\Gamma}(x) + \mathcal L_{b_\alpha} \mathbf{K}_t^\mathbf{\Gamma}(x) = 0,
\end{align}
for $x \in\mathcal{A}_{t}^0 \cup \mathcal{B}_{t}^0$ and $t \in (0,T).$ To that purpose, we divide the arguments into two parts.

\noindent\textbf{Step 1: $\mathbf{K}^\mathbf{\Gamma}$ verifies the equation strongly on $\mathcal{A}_{t}^0,$ $t \in (0,T)$.}
Using the fact that $\phi^0_t$ \eqref{phi-eq} is a smooth local flow, we can define the Lie derivative in terms of the local flow as $\mathcal{L}_{b_\alpha}\mathbf{K}^{\mathbf{\Gamma}}_t =\left.\frac{\partial}{\partial s}\right\rvert_{s=0} (\phi_s^0)^* \mathbf{K}^{\mathbf{\Gamma}}_t$ on $\mathcal{A}_{t}^0$. A standard differential geometry manipulation shows that indeed $\mathbf{K}^\mathbf{\Gamma}_t$ solves \eqref{ultima} strongly on $\mathcal{A}_{t}^0,$ $t \in (0,T)$ . 

However, showing that $\mathbf{K}^\mathbf{\Gamma}$ solves \eqref{ultima} strongly on $\mathcal{B}_{t}^0,$ $t \in (0,T),$ requires a different argument since $\Phi^0_t$ is not a local flow (in particular, the semigroup property fails). To bypass this issue, we provide a direct proof in the next step.

\noindent\textbf{Step 2: $\mathbf{K}^\mathbf{\Gamma}$ verifies the equation strongly on $\mathcal{B}_t^0,$ $t \in (0,T)$.}
Before starting with the proof, we derive a useful evolution equation for $\Psi^{0}_{t}$. To simplify the notation, we will just write $\Phi_{t}$ instead of $\Phi^{0}_{t}$ and $\Psi_t$ instead of $\Psi^0_t$ (this is actually correct by \eqref{psi-eq:inverse}). Recalling that  $\Psi_t(\Phi_t(y)) = y,$ for any $y \in B (0,t) \backslash \{0\}$ we infer that
\begin{align} \label{flowinv}
\begin{split}
    0 = \frac{\partial}{\partial t}\Psi_t(\Phi_t(y)) &= \frac{\partial \Psi_t}{\partial t}(\Phi_t(y)) + \frac{\partial \Psi_t}{\partial x_i}(\Phi_t(y)) \frac{\partial \Phi_t^i}{\partial t}(y) \\
    &= \frac{\partial \Psi_t}{\partial t}(\Phi_t(y)) + \frac{\partial \Psi_t}{\partial x_i}(\Phi_t(y)) b_\alpha^i(\Phi_t(y)),
\end{split} 
\end{align}
$t \in (0,T),$ where we have used the chain rule and the fact that $\Phi_t$ solves \eqref{dflow}. Setting $x = \Phi_t(y)$ and rearranging the terms slightly we obtain 
\begin{align} \label{eq:back-flow}
    \frac{\partial}{\partial t}\Psi_t(x) = -b_\alpha^i(x) \frac{\partial \Psi_t}{\partial x_i}(x), \quad x \in \mathcal{B}_{t}^0, \quad t \in (0,T).
\end{align}
We are ready to verify that \eqref{ultima} holds. We write the left-hand side of \eqref{ultima} in coordinates as
\begin{align}  
     ((\Phi_t)_* \mathbf{\Gamma})_{j_1,\ldots,j_k}(x) = & \, \Gamma_{i_1,\ldots,i_k}(\Psi_t(x)) \frac{\partial \Psi_t^{i_1}}{\partial x_{j_1}}(x) \cdots \frac{\partial \Psi_t^{i_k}}{\partial x_{j_k}}(x), \label{eq:comput:derivative}
\end{align}
$x \in \mathcal{B}_{t}^0,$ $t \in (0,T).$ Applying identity \eqref{eq:back-flow} on the right hand side of \eqref{eq:comput:derivative}, we obtain
\begin{align}
\begin{split}
    &\frac{\partial}{\partial t}((\Phi_t)_* \mathbf{\Gamma})_{j_1,\ldots,j_k}(x) = -\frac{\partial \Gamma_{i_1,\ldots,i_k}}{\partial y_l}(\Psi_t(x)) b_\alpha^r(x)\frac{\partial \Psi^l_t}{\partial x_r}(x) \frac{\partial \Psi_t^{i_l}}{\partial x_{j_1}}(x) \cdots \frac{\partial \Psi_t^{i_k}}{\partial x_{j_k}}(x) \\ 
    & \quad - \Gamma_{i_1,\ldots,i_k}(\Psi_t(x)) \left(\frac{\partial b_\alpha^r}{\partial x_{j_1}}(x)\frac{\partial \Psi^{i_1}_t}{\partial x_r}(x) + b_\alpha^r(x)\frac{\partial^2 \Psi^{i_1}_t}{\partial x_r x_{j_1}}(x)\right) \frac{\partial \Psi_t^{i_2}}{\partial x_{j_2}}(x)  \cdots \frac{\partial \Psi_t^{i_k}}{\partial x_{j_k}}(x)  \\
    &- \cdots - \Gamma_{i_1,\ldots,i_k}(\Psi_t(x)) \frac{\partial \Psi_t^{i_1}}{\partial x_{j_1}}(x)\cdots \frac{\partial \Psi_t^{i_{k-1}}}{\partial x_{j_{k-1}}}(x) \\
    &\hspace{150pt} \times \left(\frac{\partial b_\alpha^r}{\partial x_{j_k}}(x)\frac{\partial \Psi^{i_k}_t}{\partial x_r}(x) + b_\alpha^r(x)\frac{\partial^2 \Psi^{i_k}_t}{\partial x_r x_{j_k}}(x)\right), \label{derivativepush}
\end{split}
\end{align}
for $x \in \mathcal{B}_{t}^0,$ $t \in (0,T)$. Using the explicit formula for Lie derivatives \eqref{Lieformula}, we also compute the right-hand side of \eqref{ultima} as
\begin{align} 
\begin{split}
        &(\mathcal L_{b_\alpha} [(\Phi_t)_* \mathbf{\Gamma}])_{j_1,\ldots,j_k}(x) =  b_\alpha^l(x)\frac{\partial \Gamma_{i_1,\ldots,i_k}}{\partial y_r}(\Psi_t(x))\frac{\partial \Psi^r_t}{\partial x_l}(x) \frac{\partial \Psi_t^{i_1}}{\partial x_{j_1}}(x) \cdots \frac{\partial \Psi_t^{i_k}}{\partial x_{j_k}}(x) \\
        &\quad + \Gamma_{i_1,\ldots,i_k}(\Psi_t(x))\left(b_\alpha^l(x) \frac{\partial^2 \Psi_t^{i_1}}{\partial x_{j_1} \partial x_l}(x) + \frac{\partial \Psi_t^{i_1}}{\partial x_{l}}(x) \frac{\partial b_\alpha^{l}}{\partial x_{j_1}}(x) \right) \frac{\partial \Psi_t^{i_2}}{\partial x_{j_2}}(x) \cdots \frac{\partial \Psi_t^{i_k}}{\partial x_{j_k}}(x) \\
        &+ \cdots + \Gamma_{i_1,\ldots,i_k}(\Psi_t(x)) \frac{\partial \Psi_t^{i_1}}{\partial x_{j_1}}(x) \cdots \frac{\partial \Psi_t^{i_{k-1}}}{\partial x_{j_{k-1}}}(x) \\
        &\hspace{150pt} \times \left(b_\alpha^l(x) \frac{\partial^2 \Psi_t^{i_k}}{\partial x_{j_k} \partial x_l}(x) + \frac{\partial \Psi_t^{i_k}}{\partial x_{l}}(x) \frac{\partial b_\alpha^{l}}{\partial x_{j_k}}(x) \right),\label{eq:Lie:derivative:1}
\end{split}
\end{align}
$x\in\mathcal{B}_{t}^0,$ $t \in (0,T)$. 
Comparing the expressions \eqref{derivativepush} and \eqref{eq:Lie:derivative:1}, we clearly see that
\begin{align*}
          & -\frac{\partial}{\partial t}((\Phi_t)_* \mathbf{\Gamma})_{j_1,\ldots,j_k}(x) = (\mathcal L_{b_\alpha} [(\Phi_t)_* \mathbf{\Gamma}])_{j_1,\ldots,j_k}(x) 
\end{align*}
holds for $x\in \mathcal{B}_{t}^0$ and $t\in (0,T)$. The proof is now complete.
\end{proof}
We now show the existence of $L^p$-solutions (in the sense of Definition \ref{weak-sol-det}) to \eqref{eq0} with a restriction on the dimension. More precisely, we will show the following result.
\begin{proposition} \label{prop:exist:Lp}
Let $kp< n,$ $\mathbf{K}_0,\mathbf{\Gamma} \in C^\infty(\R^n,\bigwedge^k  \mathcal{T}^*\mathbb R^n)$ be supported on $B(0,1/2),$ and $0<T \leq 1-1/2^{1-\alpha}$. Then the $k$-form-valued function $\mathbf{K}^\mathbf{\Gamma}$ defined in \eqref{K:form:sol:gamma} complemented with the initial datum $\mathbf{K}^\mathbf{\Gamma}_0 := \mathbf{K}_0$ is an $L^p$-solution of the Lie transport equation \eqref{eq0} in the sense of Definition \ref{weak-sol-det}.
\end{proposition}
\begin{proof}
We split the argument into three steps. 

\noindent\textbf{Step 1. Demonstrating that $\mathbf{K}^\mathbf{\Gamma} \in L^p ([0,T] \times \mathbb{R}^n ; \bigwedge^k \mathcal{T}^*\mathbb R^n)$.}
It is easy to check that $(\phi^0_{\cdot})_* \mathbf{K}_0 \in L^p([0,T] \times \mathbb{R}^n ; \bigwedge^k \mathcal{T}^*\mathbb R^n)$ since $\mathbf{K}_0 \in C^\infty_0(\R^n,\bigwedge^k  \mathcal{T}^*\mathbb R^n)$ and $\psi^0 \in C^1(\bigcup_{t \in [0,T]} \left(\{t\} \times \mathcal{A}_t^0 \right); \R^n \backslash \{0\}),$ where we recall that $\psi^0_t := (\phi^0_t)^{-1}$. We claim that $(\Phi^0_{\cdot})_* \mathbf{\Gamma} \in L^p([0,T] \times \mathbb{R}^n ; \bigwedge^k \mathcal{T}^*\mathbb R^n)$. Indeed, by noting that $\Psi^0_t(x) = \frac{x}{\lvert x\rvert} \left(t - \lvert x\rvert^{1-\alpha} \right)$, we have
\begin{align*}
\frac{\partial (\Psi^0_t)^i}{\partial x_j}(x) = \frac{\delta_{ij} \lvert x\rvert - \frac{x_i x_j}{\lvert x\rvert} }{\lvert x\rvert^2}(t - \lvert x\rvert^{1-\alpha}) - (1-\alpha) \frac{x_ix_j}{\lvert x\rvert^{2-\alpha}}, \quad t \in (0,T], \quad x \in \mathcal{B}^0_t,
\end{align*}
for $i,j \in \{1,\ldots,n\}$.

We observe that the partial derivatives $\frac{\partial (\Psi^0_t)^i}{\partial x_j}$ behave at worst like $1/\lvert x\rvert$ near zero. Therefore, using the explicit formula \eqref{repr} for the push-forward, we obtain that
\begin{align*}
    & \left\lvert \int_{B(0,1)} \left((\Phi^0_t)_*\mathbf{\Gamma}(x)\right)^p \diff^n x \right\rvert \lesssim \norm{\mathbf{\Gamma}}_{L^\infty_{1}} \left\lvert \int_{B(0,1)} \lvert x\rvert^{-k p} \, \diff^n x \right\rvert < \infty,
\end{align*}
where we used the hypothesis $kp < n$.

\noindent\textbf{Step 2. $\mathbf{K}^\mathbf{\Gamma}_{t}$ satisfies equation \eqref{eq0} in weak form.}
We have to check that
\begin{align} \label{ecuaciondebil}
 \llangle  \mathbf{K}_t^\mathbf{\Gamma}, \mathbf{\Theta} \rrangle_{L^2} - \llangle  \mathbf{K}_0, \mathbf{\Theta} \rrangle_{L^2}  + \int_0^t \llangle  \mathbf{K}_s^\mathbf{\Gamma},\mathcal{L}^*_{b_\alpha} \mathbf{\Theta} \rrangle_{L^2} \diff s  = 0, \quad t \in [0,T].
\end{align}
By observing that $\mathbf{K}^\mathbf{\Gamma}$ is supported on $B(0,1),$ \eqref{ecuaciondebil} can be rewritten as
\begin{align*} 
 & \llangle \mathcal{X}_{\mathcal{A}_{t}^0}  \mathbf{K}_t^\mathbf{\Gamma}, \mathbf{\Theta} \rrangle_{L^2} + \llangle \mathcal{X}_{\mathcal{B}_{t}^0}  \mathbf{K}_t^\mathbf{\Gamma}, \mathbf{\Theta} \rrangle_{L^2} - \llangle  \mathcal{X}_{\mathcal{A}_{0}^0} \mathbf{K}_0, \mathbf{\Theta} \rrangle_{L^2} - \llangle  \mathcal{X}_{\mathcal{B}_{0}^0} \mathbf{K}_0, \mathbf{\Theta} \rrangle_{L^2} \\
 & \quad + \int_0^t \llangle  \mathcal{X}_{\mathcal{A}_s^0} \mathbf{K}_s^\mathbf{\Gamma},\mathcal{L}^*_{b_\alpha} \mathbf{\Theta} \rrangle_{L^2} \diff s + \int_0^t \llangle  \mathcal{X}_{\mathcal{B}_s^0} \mathbf{K}_s^\mathbf{\Gamma},\mathcal{L}^*_{b_\alpha} \mathbf{\Theta} \rrangle_{L^2} \diff s  = 0,
\end{align*}
for $t \in [0,T],$ where we recall that $\mathcal{X}_{\emptyset} = 0$ (and hence $\mathcal{X}_{\mathcal{B}_{0}^0} = 0$) as assumed by our convention in Notations (Section \ref{notation}). We will first prove
\begin{align} \label{firstpart}
 & \llangle \mathcal{X}_{\mathcal{A}_{t}^0}  \mathbf{K}_t^\mathbf{\Gamma}, \mathbf{\Theta} \rrangle_{L^2}  - \llangle  \mathcal{X}_{\mathcal{A}_{0}^0} \mathbf{K}_0, \mathbf{\Theta} \rrangle_{L^2} + \int_0^t \llangle  \mathcal{X}_{\mathcal{A}_s^0} \mathbf{K}_s^\mathbf{\Gamma},\mathcal{L}^*_{b_\alpha} \mathbf{\Theta} \rrangle_{L^2} \diff s  = 0.
\end{align}
By noting that $\partial \mathcal{A}_{t}^0 = S(0,1) \cup S(0,t^{\frac{1}{1-\alpha}})$ and applying the Reynolds transport theorem, we observe that
\begin{align} \label{primeraparte}
\begin{split}
    \partial_t \llangle \mathcal{X}_{\mathcal{A}_{t}^0} \mathbf{K}_t, \mathbf{\Theta} \rrangle_{L^2} &= \int_{\mathcal{A}_{t}^0} \partial_t (K^\mathbf{\Gamma}_t)_{i_1,\ldots,i_k}(x) \Theta_{i_1,\ldots,i_k}(x) \, \diff^n x \\
    & + \int_{S(0,t^{\frac{1}{1-\alpha}})} (K_t^\mathbf{\Gamma})_{i_1,\ldots,i_k}(x) \Theta_{i_1,\ldots,i_k}(x) v_t(x) \cdot n(x) \, \diff S_t,
\end{split}
\end{align}
where $\diff S_t$ is the surface element, $v_t(x)$ is the velocity of the boundary $S(0,r^{\frac{1}{1-\alpha}})$ at $r=t$ and $n(x)$ is the outward pointing unit normal to $S(0,t^{\frac{1}{1-\alpha}})$ at the point $x \in S(0,t^{\frac{1}{1-\alpha}})$. We also assumed summation over repeated indices as usual. For the first term after the equality in \eqref{primeraparte}, we have 
\begin{align*}
    & \int_{\mathcal{A}_{t}^0} \partial_t (K^\mathbf{\Gamma}_t)_{i_1,\ldots,i_k}(x) \Theta_{i_1,\ldots,i_k}(x) \, \diff^n x = - \llangle \mathcal{X}_{\mathcal{A}_{t}^0} \mathcal{L}_{b_\alpha} \mathbf{K}_t,\mathbf{\Theta} \rrangle_{L^2},
\end{align*}
since by Proposition \ref{strongsolution}, $\mathbf{K}_t$ satisfies \eqref{eq0} strongly on $\mathcal{A}_{t}^0,$ $t \in (0,T)$. By observing that 
\begin{align*}
v_t(x) = \frac{\diff}{\diff t} \left(t^{\frac{1}{1-\alpha}} n(x)\right) = \frac{1}{1-\alpha} t^\frac{\alpha}{1-\alpha} n(x),
\end{align*}
we rewrite the second term after the equality in \eqref{primeraparte} as
\begin{align*} 
    &  \int_{S(0,t^{\frac{1}{1-\alpha}})} (K_t^\mathbf{\Gamma})_{i_1,\ldots,i_k}(x) \Theta_{i_1,\ldots,i_k}(x) v_t(x) \cdot n(x) \, \diff S_t \\
    &\quad = t^{\frac{\alpha}{1-\alpha}}(1-\alpha)^{-1}\int_{S(0,t^{\frac{1}{1-\alpha}})} (K_t^\mathbf{\Gamma})_{i_1,\ldots,i_k}(x) \Theta_{i_1,\ldots,i_k}(x) \, \diff S_t.
\end{align*}
Now we compute $\llangle  \mathcal{X}_{\mathcal{A}_{t}^0} \mathbf{K}_t^\mathbf{\Gamma},\mathcal{L}^*_{b_\alpha} \mathbf{\Theta} \rrangle_{L^2}$ by integration by parts
\begin{align} 
    &\llangle  \mathcal{X}_{\mathcal{A}_{t}^0} \mathbf{K}^\mathbf{\Gamma}_t,\mathcal{L}^*_{b_\alpha} \mathbf{\Theta} \rrangle_{L^2} = \int_{\mathcal{A}_{t}^0} (K^\mathbf{\Gamma}_t)_{i_1,\ldots,i_k}(x) {(\mathcal{L}^*_{b_\alpha} \Theta)}_{i_1,\ldots,i_k}(x) \, \diff^n x \nonumber \\
    &\hspace{10pt} =  \int_{\mathcal{A}_{t}^0} (K^\mathbf{\Gamma}_t)_{i_1,\ldots,i_k}(x) \left(-\frac{\partial}{\partial x_l} \big(b_\alpha^l(x) \Theta_{i_1,\ldots,i_k}(x)\big) + \Theta_{l,i_2,\ldots,i_k}(x) \frac{\partial b_\alpha^{i_1}}{\partial x_l}(x) \right. \nonumber \\
    &\hspace{180pt} \left. + \cdots + \Theta_{i_1,\ldots,i_{k-1},l}(x) \frac{\partial b_\alpha^{i_k}}{\partial x_l}(x) \right) \diff^n x \nonumber \\
    &\hspace{10pt} = \llangle \mathcal{X}_{\mathcal{A}_{t}^0} \mathcal{L}_{b_\alpha} \mathbf{K}^\mathbf{\Gamma}_t,\mathbf{\Theta} \rrangle_{L^2} - \int_{\mathcal{A}_{t}^0} \div \left(b_\alpha(x) (K^\mathbf{\Gamma}_t)_{i_1,\ldots,i_k}(x)\Theta_{i_1,\ldots,i_k}(x)\right) \diff^n x, \label{nose}
\end{align}
where we have used the coordinate expression of $\mathcal{L}^*_{b_\alpha}\mathbf{\Theta}$ (see \eqref{Lieformulaadjoint}). By the divergence theorem and explicit expression \eqref{particularb} of $b_\alpha$, we obtain
\begin{align*} 
    &-\int_{\mathcal{A}_{t}^0} \div \left(b_\alpha(x) (K_t^\mathbf{\Gamma})_{i_1,\ldots,i_k}(x)\Theta_{i_1,\ldots,i_k}(x)\right) \diff^n x \\
    &\quad = -\int_{S(0,t^{\frac{1}{1-\alpha}})} b_\alpha(x) \cdot n(x)(K_t^\mathbf{\Gamma})_{i_1,\ldots,i_k}(x)\Theta_{i_1,\ldots,i_k}(x) \, \diff S_t \\
    &\quad = -t^{\frac{\alpha}{1-\alpha}}(1-\alpha)^{-1}\int_{S(0,t^{\frac{1}{1-\alpha}})} (K_t^\mathbf{\Gamma})_{i_1,\ldots,i_k}(x) \Theta_{i_1,\ldots,i_k}(x) \, \diff S_t.
\end{align*}
By putting together all the previous calculations, we can conclude that \eqref{firstpart} holds. Finally, one can apply the same techniques to prove 
\begin{align*} 
 & \llangle \mathcal{X}_{\mathcal{B}_{t}^0}  \mathbf{K}_t^\mathbf{\Gamma}, \mathbf{\Theta} \rrangle_{L^2}  - \llangle  \mathcal{X}_{\mathcal{B}_{0}^0} \mathbf{K}_0, \mathbf{\Theta} \rrangle_{L^2} + \int_0^t \llangle  \mathcal{X}_{\mathcal{B}_s^0} \mathbf{K}_s^\mathbf{\Gamma},\mathcal{L}^*_{b_\alpha} \mathbf{\Theta} \rrangle_{L^2} \diff s  = 0,
\end{align*}
for $t \in [0,T].$ We conclude that \eqref{ecuaciondebil} holds.

\noindent\textbf{Step 3. $\mathbf{K}^\mathbf{\Gamma}$ is weakly continuous.}
Fix a test $k$-form $\mathbf{\Theta} \in C^\infty_0 (\bigwedge^k \mathcal{T}^*\mathbb R^n)$. Then we have
\begin{align*}
   & \llangle \mathbf{K}^\mathbf{\Gamma}_t,\mathbf{\Theta} \rrangle_{L^2} = 
   \llangle \mathcal{X}_{\mathcal{A}_{t}^0} (\phi_t)_* \mathbf{K}_0, \mathbf{\Theta}  \rrangle_{L^2} + \llangle \mathcal{X}_{\mathcal{B}_t^0}(\Phi_t)_* \mathbf{\Gamma},\mathbf{\Theta} \rrangle_{L^2}, \quad  t \in [0,T].  
\end{align*}
By the regularity of $\phi_t$ on $\mathcal{A}_{t}^0$ and due to the fact that $\mathcal{A}_{t}^0$ varies smoothly with $t,$ it is not difficult to show that $\llangle \mathcal{X}_{\mathcal{A}_{t}^0} (\phi_t)_* \mathbf{K}_0, \mathbf{\Theta}  \rrangle_{L^2}$ is continuous in $t \in [0,T]$. Since $\mathcal B_t^0$ varies smoothly with $t$ and we have a $L^p_{loc}$ control of $(\Phi_t)_* \mathbf{\Gamma}$ on $\mathcal B_t^0,$ it can also be shown that $\llangle \mathcal{X}_{\mathcal{B}_t^0}(\Phi_t)_* \mathbf{\Gamma},\mathbf{\Theta} \rrangle_{L^2}$ is continuous in $t \in [0,T]$.
\end{proof}

\subsection{Instantaneous blow-up of deterministic solutions}
Let $n \geq 2.$ Moreover, let $\alpha \in (0,1)$ and $R>0.$ We define $b_{\alpha,R}'$ by the relations
\begin{align*}
& (b_{\alpha,R}')^1(x) = 0, \\
& (b_{\alpha,R}')^2(x) = (\min \{\lvert x_1 \rvert, R\})^{\alpha}, \\
& (b_{\alpha,R}')^i(x) = 0, \quad i=3,\ldots,n, \\
\end{align*}
where $x = (x_1,\ldots,x_n) \in \R^n$ and we note that the constant $R$ has been introduced to obtain boundedness. One can check that $b_{\alpha,R}' \in C_b^\alpha(\mathbb R^n, \mathbb R^n) \cap  W^{1,q}_{loc}(\mathbb R^n, \mathbb R^n)$ provided that
\begin{align} \label{qcondition}
\alpha \in ((q-1)/q,1).
\end{align}
The associated characteristic ODE is
\begin{align} \label{dflow*}
\begin{split}
	& \partial_t \phi_t(x) = b_{\alpha,R}'(\phi_t(x)), \\
    & \phi_0(x) = x, \quad t \in [0,T], \quad x \in \mathbb{R}^n.
\end{split}
\end{align}
From now on we choose $R=1$ without loss of generality as in the previous subsection and employ the notation $b_{\alpha}'$.  The solution to \eqref{dflow*} can be explicitly computed to be
\begin{align*}
& (\phi^1_t)(x) = x_1, \\
& (\phi^2_t)(x) = x_2 + \lvert x_1\rvert^{\alpha} t, \quad \lvert x_1 \rvert \leq 1, \\
& (\phi^2_t)(x) = x_2 + t, \quad \lvert x_1 \rvert > 1, \\
& (\phi^i_t)(x) = x_i, \quad i=3,\ldots,n, 
\end{align*}
for all $t \in [0,T]$. Moreover, the inverse can be easily verified to be given by 
\begin{align*}
& (\psi^1_t)(x) = x_1, \\
& (\psi^2_t)(x) = x_2 - \lvert x_1\rvert^{\alpha} t, \quad \lvert x_1 \rvert \leq 1, \\
& (\psi^2_t)(x) = x_2 - t, \quad \lvert x_1 \rvert > 1, \\
& (\psi^i_t)(x) = x_i, \quad i=3,\ldots,n,
\end{align*}
for all $t \in [0,T]$. Therefore we have
\begin{align*}
\frac{\partial \psi^i_t}{\partial x_i}(x) = 1, \\
\end{align*}
for all $i = 1, \ldots, n$ (no sum assumed over repeated indices) and
\begin{align*}
& \frac{\partial \psi^2_t}{\partial x_1}(x) =
\begin{cases}
-\text{sgn}(x_1)\alpha\lvert x_1\rvert^{\alpha-1} t, \quad \lvert x_1 \rvert < 1, \\
0, \quad \lvert x_1 \rvert > 1,
\end{cases} 
\end{align*}
where $\text{sgn}$ is the sign function. All the other partial derivatives $\frac{\partial \phi_t^i}{\partial x_j}$ (i.e. when $i \neq j$ and $(i,j) \neq (2,1)$) are zero. Let $k \notin \{0,n\}$ and choose the initial datum
\begin{align*}
\mathbf{K}_0:= \mathcal{X}_{B(0,1)} \diff x_2 \wedge \diff x_3 \wedge \ldots \wedge \diff x_{k+1},
\end{align*}
which is trivially of class  $L^p$. By direct computation (or simply by integrating the equations), it can be checked that the solution to the deterministic Lie transport equation \eqref{deq0} with vector field $b_\alpha'$ must be given by $(\phi_t)_* \mathbf{K}_0$ as expected.

We also have (take into account the formula for the push-forward \eqref{explicitfor})
\begin{align*} 
     & (\mathbf{K}_t)_{1,3,\ldots,k+1}(x) = ((\phi_t)_* \mathbf{K}_0)_{1,3,\ldots,k+1}(x)  \\
     & = \mathcal{X}_{B(0,1)} \frac{\partial \psi_t^2}{\partial x_1}(x) \frac{\partial \psi_t^3}{\partial x_3}(x) \cdots \frac{\partial \psi_t^{k}}{\partial x_k}(x) \frac{\partial \psi_t^{k+1}}{\partial x_{k+1}}(x)\\
     & = -\mathcal{X}_{B(0,1)}\text{sgn}(x_1)\alpha \lvert x_1\rvert^{\alpha-1} t.
\end{align*}
Hence, for every $\epsilon, t >0$,
\begin{align*}
    & \sup_{x \in B(0,\epsilon)} \left \lvert ((\phi_t)_* \mathbf{K}_0)_{1,3,\ldots,k+1}(x) \right \rvert = \infty,
\end{align*}
which clearly implies that for every $\epsilon, t >0$,
\begin{align*}
& \sup_{x \in B(0,\epsilon)} \lvert \mathbf{K}_t(x) \rvert = \infty.
\end{align*}
Furthermore, we observe that the only other nonzero component of $\mathbf{K}_t$ is $(K_t)_{1,2,3,\ldots,k}(x) = 1$ (recall the above observation regarding the partial derivatives of $\psi_t$). 
Therefore $\mathbf{K}$ is of class $L^p$ provided that $\alpha \in ((p-1)/p,1)$. Moreover, $(p-1)/p \geq (q-1)/q$ since $p \geq 2 \geq q,$ so in this case, condition \eqref{qcondition} is met. It is also clear that $\mathbf{K}_t$ is weakly continuous in time, so it a solution to \eqref{deq0} in the sense of Definition \ref{weak-sol-det}.

Finally, note that $\mathbf{K}_0$ can be easily modified to have $C^\infty_c$-regularity in a way that $\mathbf{K}_t$ would still be a solution of \eqref{dflow*}. For example, choose $\mathbf{K}_0$ such that (i) $\mathbf{K}_0:= \diff x_2 \wedge \diff x_3 \wedge \ldots \wedge \diff x_{k+1}$ on $B(0,1/2),$ (ii) $\mathbf{K}_0$ vanishes outside $B(0,1),$ (iii) $\mathbf{K}_0$ is of class $C^\infty,$ and (iv) $(K_0)_{i_1,i_2,\ldots,i_k} = 0$ when $(i_1,i_2,\ldots,i_k) \neq (2,3,\dots,k+1).$ It is not difficult to show that formation of instantaneous singularities via the same mechanism persists in this case. This means that {\em regardless of the regularity of the initial datum}, instantaneous blow-up of the supremum norm can take place in the deterministic case.
\newpage

\section*{Acknowledgements}
{\small
The authors would like to thank Jos\'e Sabina de Lis, Wei Pan, James-Michael Leahy, Nicolai V. Krylov, V\'ictor M. Jim\'enez, Darryl D. Holm, Franco Flandoli, Marta de Le\'on Contreras, and Diego Alonso Or\'an for amazing mathematical discussions and feedback that tremendously helped us put together this work and formalise important aspects. ABdL has been funded by the EPSRC grant at the Mathematics of Planet Earth Centre of Doctoral Training and the Margarita Salas grant awarded by the Ministry of Universities granted by Order UNI/551/2021 of May 26, as well as the European Union Next Generation EU Funds. ST acknowledges the Schr\"odinger scholarship scheme for much of the funding during this work.}

\section*{Declarations}
On behalf of all authors, the corresponding author states that there are no competing interests.

\begin{appendices}

\section{Technical proofs} \label{technical}

\subsection{Proof of estimates \eqref{eq:D-phi-Lp} and \eqref{eq:D-phi-Lp-difference} in Lemma \ref{JM-lemma} (Step 4 of the proof of Theorem \ref{flowfla})} \label{twoproperties}

{\bf{Step 4.1.}} For simplicity, we set $\xi := 1$, but our argument can easily be generalised to allow for general diffusion vector fields of class $L^\infty([0,T]; C^{2+\alpha}_b(\mathbb R^n, \mathbb R^n))$. To avoid confusion with the time integration variables, we employ the notation $p$ instead of $r$ for the number $r\geq 2$ in Lemma \ref{JM-lemma}. We divide the proof into three steps because it involves long computations. Estimate \eqref{eq:D-phi-Lp} follows easily from \eqref{phi-convergence}, but we outline a self-contained proof here since the arguments and the result will be used in the next steps. Since for fixed $s \in [0,T],$ the process $Y_t(x) = \Psi_\lambda(t,\phi_{s,t}(x)),$ $t \geq s,$ $x\in \R^n$ satisfies the modified SDE \eqref{conjugated} and $\phi_{s,t}(x) = \Psi_\lambda^{-1}(t,Y_t(x))$ with $y = \Psi_\lambda(s,x)$, we have the following equation for its (spatial) derivative
\begin{align}
\begin{split}
\label{conjugated-sde-deriv}
    & DY_t(x) =  D\Psi_\lambda(s,x) + \lambda \int^t_s D\psi_\lambda(r,\phi_{s,r}(x)) D\Psi_\lambda^{-1}(r,Y_r(x)) DY_r(x) \, \diff r \\
    & \hspace{70pt} + \int^t_s D^2\Psi_\lambda(r, \phi_{s,r}(x)) D\Psi_\lambda^{-1}(r,Y_r(x)) DY_r(x) \, \diff W_r,
\end{split}
\end{align}
where $0\leq s \leq t \leq T,$ $x \in \mathbb{R}^n$. To prove the first bound \eqref{eq:D-phi-Lp}, we observe that $D\phi_{s,t}(x) = D\Psi_\lambda^{-1}(t,Y_t(x)) DY_t(x)$ and
$$
\mathbb{E}\left[\sup_{s \leq t \leq T}\lvert D\phi_{s,t}(x)\rvert^p\right] \leq \sup_{y \in \R^n} \sup_{s \leq t \leq T} \lvert D\Psi_\lambda^{-1}(t,y)\rvert^p  \mathbb{E}\left[\sup_{s \leq t \leq T}\lvert DY_t(x)\rvert^p\right].
$$
Since by Lemma \ref{lemazo}, $D\Psi_\lambda^{-1}(t,\cdot)$ is bounded uniformly in time, we are only left to show the bound \\
$\sup_{0 \leq s \leq T} \mathbb{E}\left[\sup_{s \leq t \leq T}\lvert DY_t(x)\rvert^p\right] < \infty$, which can be proven directly from \eqref{conjugated-sde-deriv} by a standard argument using Gr\"onwall's lemma and Burkholder-Davis-Gundy inequality, together with the fact that $D\psi_\lambda, D\Psi_\lambda, D\Psi_\lambda^{-1}, D^2\Psi_\lambda$ are bounded {\em uniformly in time}. 

{\bf{Step 4.2.}} It is easy to check that the constants appearing in all the inequalities $\lesssim$ in the rest of the proof are {\em independent of $s$}, so we will not stress this again to avoid repetition. We claim that for fixed $\beta \in (0, 1],$
\begin{align}
\mathbb{E}\left[\sup_{s \leq t \leq T}\lvert\phi_{s,t}(x) - \phi_{s,t}(y)\rvert^p\right] \lesssim \lvert x - y\rvert^{\beta p} \label{phi-diff-estimate}
\end{align}
holds for any $x,y \in \mathbb{R}^n$. To verify this, we first note that
\begin{align*}
\mathbb{E}\left[\sup_{s \leq t \leq T}\lvert \phi_{s,t}(x) - \phi_{s,t}(y)\rvert^p\right] &= \mathbb{E}\left[\sup_{s \leq t \leq T}\lvert \Psi_\lambda^{-1}(t,Y_t(x)) - \Psi_\lambda^{-1}(t,Y_t(y))\rvert^p\right] \\
&\lesssim \mathbb{E}\left[\sup_{s \leq t \leq T}\lvert Y_t(x) - Y_t(y)\rvert^p\right],
\end{align*}
where we have taken into account that the Lipschitz seminorm $[\Psi_\lambda^{-1}(t,\cdot)]_{\text{Lip}}$ is bounded uniformly in time since $\Psi_\lambda^{-1} \in L^\infty([0,T];C^{2}_b(\R^n,\R^n))$ by Lemma \ref{lemazo}. By using the modified SDE \eqref{conjugated}, we also have
\begin{align*}
    \lvert Y_t(x) - Y_t(y)\rvert & \leq \lvert \Psi_\lambda(s,x) - \Psi_\lambda(s,y)\rvert \\
    & + \lambda \int^t_s \lvert \psi_\lambda(r,\Psi_\lambda^{-1}(r,Y_r(x))) - \psi_\lambda(r,\Psi_\lambda^{-1}(r,Y_r(y)))\rvert \, \diff r \\
    & + \left\lvert \int^t_s \left(D\Psi_\lambda(r,\Psi_\lambda^{-1}(r,Y_r(x))) - D\Psi_\lambda(r,\Psi_\lambda^{-1}(r,Y_r(y)))\right) \diff W_r\right\rvert.
\end{align*}
Employing the fact that $\psi_\lambda$ and $\Psi_\lambda^{-1}$ are globally Lipschitz, we obtain
\begin{align} \label{gronwall0}
\begin{split}
    & \lvert Y_t(x) - Y_t(y)\rvert \leq \lvert \Psi_\lambda(s,x) - \Psi_\lambda(s,y)\rvert + C_1 \int^t_s \lvert Y_r(x) - Y_r(y)\rvert \, \diff r  \\
    & \quad + \sup_{s \leq u \leq t} \left\lvert \int^u_s \left(D\Psi_\lambda(r,\Psi_\lambda^{-1}(r,Y_r(x))) - D\Psi_\lambda(r,\Psi_\lambda^{-1}(r,Y_r(y)))\right) \diff W_r\right\rvert,
\end{split}
\end{align} 
where $C_1 = \lambda \sup_{0 \leq t \leq T}[\psi_\lambda(t,\cdot)]_{\text{Lip}} \sup_{0 \leq t \leq T}[\Psi_\lambda^{-1}(t,\cdot)]_{\text{Lip}}$ is a finite constant. Applying Gr\"onwall's inequality to \eqref{gronwall0}, we get
\begin{align*}
    & \lvert Y_t(x) - Y_t(y)\rvert \leq \text{exp}(C_1T) \,\bigg(\lvert \Psi_\lambda(s,x) - \Psi_\lambda(s,y)\rvert \\
    &\quad  + \sup_{s \leq u \leq t}\left\lvert \int^u_s \left(D\Psi_\lambda (r,\Psi_\lambda^{-1}(r,Y_r(x))) - D\Psi_\lambda(r,\Psi_\lambda^{-1}(r,Y_r(y))) \right) \diff W_r\right\rvert \bigg). 
\end{align*}
By raising both sides to the power of $p$ and applying the supremum in time, we have
\begin{align} \label{gronwall}
\begin{split}
&  \sup_{s \leq u \leq t} \lvert Y_u(x) - Y_u(y)\rvert^p \leq \text{exp}(p\,C_1T) \,\bigg( \lvert \Psi_\lambda(s,x) - \Psi_\lambda(s,y)\rvert^p  \\
& \quad + \sup_{s \leq u \leq t}\left\lvert\int^u_s \left(D\Psi_\lambda (r,\Psi_\lambda^{-1}(r,Y_r(x))) - D\Psi_\lambda(r,\Psi_\lambda^{-1}(r,Y_r(y))) \right) \diff W_r\right\rvert^p \bigg).
\end{split}
\end{align}
Now using Burkholder-Davis-Gundy inequality, there exists a constant $C^* > 0$ (dependent on $p$) such that
\begin{align} \label{gronwall2}
    &\mathbb{E}\left[\sup_{s \leq u \leq t}\left\lvert\int^u_s \left(D\Psi_\lambda(r,\Psi_\lambda^{-1}(r,Y_r(x))) - D\Psi_\lambda(r,\Psi_\lambda^{-1}(r,Y_r(y)))\right) \diff W_r\right\rvert^p \right] \nonumber \\
    &\leq C^* \mathbb{E}\left[\left(\int^t_s \lvert D\Psi_\lambda(r,\Psi_\lambda^{-1}(r,Y_r(x)) - D\Psi_\lambda(r,\Psi_\lambda^{-1}(r,Y_r(y))\rvert^2 \, \diff r \right)^{p/2}\right] \nonumber \\
    &\leq C_2 \int^t_s \mathbb{E}\left[\sup_{s \leq u \leq r}\lvert Y_u(x) - Y_u(y)\rvert^p \right] \diff r,
\end{align}
\footnote{Explicitly, $C_2 = C^* T^{\frac{p-2}{2}} \left(\sup_{0 \leq t \leq T} [D\Psi_\lambda(t,\cdot)]_{\text{Lip}} \sup_{0 \leq t \leq T}[ \Psi_\lambda^{-1}(t,\cdot)]_{\text{Lip}} \right)^p$.} 
where in the last line we have taken into account that $p \geq 2.$ Hence, taking expectations in \eqref{gronwall} and applying Gr\"onwall's inequality to the function 
\begin{align*}
& f(t) := \mathbb{E}\left[\sup_{s \leq u \leq t}\lvert Y_u(x) - Y_u(y)\rvert^p\right], \quad t \in [0,T],
\end{align*}
(use inequality \eqref{gronwall2}), we obtain the estimate
\begin{align}
    \mathbb{E}\left[\sup_{s \leq u \leq t} \lvert Y_u(x) - Y_u(y)\rvert^p\right] \leq C_3 \lvert \Psi_\lambda(s,x) - \Psi_\lambda(s,y)\rvert^p \lesssim \lvert x-y\rvert^{\beta p}, \label{eq:Y-diff-bound}
\end{align}
for any $x,y \in \mathbb{R}^n,$ where $C_3 = \text{exp}[C_2T \,\text{exp}(p\,C_1T)] < \infty$ and we have used the fact that the $\beta$-H\"older seminorm $[\Psi_\lambda(t,\cdot)]_{\beta}$ is uniformly bounded in time.

{\bf{Step 4.3.}} Finally, we establish estimate \eqref{eq:D-phi-Lp-difference}. We do not write down the explicit constants in the remaining inequalities for simplicity. For this, employing a similar argument as in the previous steps, we start by noting that
\begin{align}
    &\mathbb{E}\left[\sup_{s \leq t \leq T} \lvert D\phi_{s,t}(x) - D\phi_{s,t}(y)\rvert^p\right] \nonumber   \\
    & = \mathbb{E}\left[\sup_{s \leq t \leq T} \lvert D\Psi_\lambda^{-1}(t,Y_t(x)) DY_t(x) - D\Psi_\lambda^{-1}(t,Y_t(y)) DY_t(y)\rvert^p\right] \nonumber \\
    &\lesssim \mathbb{E}\left[\sup_{s \leq t \leq T} \lvert D\Psi_\lambda^{-1}(t,Y_t(x))\rvert^p \lvert DY_t(x) - DY_t(y)\rvert^p\right] \nonumber \\
    &\hspace{30pt} +\mathbb{E}\left[\sup_{s \leq t \leq T} \lvert DY_t(y)\rvert^p \lvert D\Psi_\lambda^{-1}(t,Y_t(x)) - D\Psi_\lambda^{-1}(t,Y_t(y))\rvert^p\right] \nonumber \\
    &\lesssim \mathbb{E}\left[\sup_{s \leq t \leq T} \lvert DY_t(x) - DY_t(y)\rvert^p\right] \label{horrible1} \\
    & \hspace{50pt} + \mathbb{E}\left[\sup_{s \leq t \leq T} \lvert D\Psi_\lambda^{-1}(t,Y_t(x)) - D\Psi_\lambda^{-1}(t,Y_t(y))\rvert^{2p}\right]^{1/2}, \label{horrible2}
\end{align}
where we have applied H\"older's inequality in the last line and have taken into account the fact that $\mathbb{E}\left[\sup_{s \leq t \leq T}\lvert DY_t(x)\rvert^{2p}\right]$ is uniformly bounded in $s$ as shown at the beginning of this proof. The last term can be treated easily as
\begin{align*}
    & \eqref{horrible2} \lesssim \mathbb{E}\left[\sup_{s \leq t \leq T} \lvert Y_t(x) - Y_t(y)\rvert^{2p}\right]^{1/2} \lesssim \lvert x-y\rvert^{\alpha p}, 
\end{align*}
where we have used \eqref{eq:Y-diff-bound} with $\beta = \alpha$ and the fact that $D\Psi_\lambda^{-1}$ is globally Lipschitz uniformly in time. For the term \eqref{horrible1}, using \eqref{conjugated-sde-deriv} and defining
\begin{align*}
    P_t(x) &:= D\psi_\lambda(t,\phi_{s,t}(x)) D\Psi_\lambda^{-1}(t,Y_t(x)),\\
    Q_t(x) &:= D^2\Psi_\lambda(t,\phi_{s,t}(x)) D\Psi_\lambda^{-1}(t,Y_t(x)), \\
    M_t(x) &:= \int^t_s Q_r(x)DY_r(x) \, \diff W_r,
\end{align*}
for $0 \leq s \leq t \leq T,$ $x \in \mathbb{R}^n,$ we obtain the following inequality
\begin{align} \label{gronwall4}
    & \lvert DY_t(x) - DY_t(y)\rvert \leq \lvert D\Psi_\lambda(s,x) - D\Psi_\lambda(s,y)\rvert + \lvert M_t(x) - M_t(y)\rvert \nonumber \\
    &\qquad  + \lambda \int^t_s \lvert P_r(x) DY_r(x) - P_r(y) DY_r(y)\rvert \, \diff r \nonumber \\
    \begin{split}
    &\quad \leq \lvert D\Psi_\lambda(s,x) - D\Psi_\lambda(s,y)\rvert + \sup_{s\leq r \leq t}\lvert M_r(x) - M_r(y)\rvert \\
    &\qquad + \lambda \left(\int^t_s \lvert P_r(x)\rvert \lvert DY_r(x) - DY_r(y)\rvert \, \diff r + \int^t_s \lvert DY_r(y)\rvert \lvert P_r(x) - P_r(y)\rvert \, \diff r\right).
    \end{split}
\end{align}
By the same argument used in Step 4.2 to derive \eqref{gronwall2}, we can apply Burkholder-Davis-Gundy inequality to obtain
\begin{align*}
    & \mathbb{E}\left[\sup_{s \leq u \leq t} \lvert M_u(x) - M_u(y)\rvert^p\right] \lesssim \int_s^t \mathbb{E}\left[\sup_{s \leq u \leq r}\lvert DY_u(x) - DY_u(y)\rvert^p\right] \diff r \\
    & \hspace{150pt} + \mathbb{E}\left[\int_s^t \lvert DY_r(y)\rvert^p \lvert Q_r(x) - Q_r(y)\rvert^p \, \diff r\right].
\end{align*}
By substituting the last inequality in \eqref{gronwall4}, taking the power of $p$, and applying supremum and expectation, we have 
\begin{align} \label{gronwall5}
\begin{split}
    \eqref{horrible1}  &\lesssim \lvert D\Psi_\lambda(s,x) - D\Psi_\lambda(s,y)\rvert^p  \\
    & \quad + \int_s^t \mathbb{E}\left[\sup_{s \leq u \leq r} \lvert DY_u(x) - DY_u(y)\rvert^p\right] \diff r \\
    & \quad + \mathbb{E}\left[\int_s^t \lvert DY_r(y)\rvert^p \lvert Q_r(x) - Q_r(y)\rvert^p \, \diff r\right] \\
    & \quad +  \int^t_s \E\left[ \sup_{s \leq u \leq r} \lvert P_u(x)\rvert^p \lvert DY_u(x) - DY_u(y)\rvert^p \right] \diff r  \\
    & \quad + \E \left[\int^t_s \lvert DY_u(y)\rvert^p \lvert P_u(x) - P_u(y)\rvert^p \, \diff r\right],
\end{split}
\end{align}
where we have also used H\"older's inequality and Tonelli's theorem to switch the domains of integration in the last line. By applying Gr\"onwall's lemma in \eqref{gronwall5} (these techniques are similar as the ones in Step 4.2), we arrive at
\begin{align}
& \mathbb{E}\left[\sup_{s \leq u \leq t} \lvert DY_u(x) - DY_u(y)\rvert^p\right] \lesssim \Bigg(\lvert D\Psi_\lambda(s,x) - D\Psi_\lambda(s,y)\rvert^p \label{eq:DY-diff1}  \\
&+ \mathbb{E}\left[\int^t_s \lvert DY_r(y)\rvert^p \lvert P_r(x) - P_r(y)\rvert^p \, \diff r\right] + \mathbb{E}\left[\int^t_s \lvert DY_r(y)\rvert^p \lvert Q_r(x) - Q_r(y)\rvert^p \, \diff r\right] \Bigg) \label{eq:DY-diff2} \\
& \hspace{120pt} \times \text{exp} \left(T + \int_0^T \E \left[ \sup_{s \leq u \leq T} \lvert P_r(x)\rvert^p \right] \diff s \right). \label{eq:DY-diff3}
\end{align}
The term in \eqref{eq:DY-diff3} is clearly controlled. The term after the inequality in \eqref{eq:DY-diff1} can be controlled easily as
\begin{align*}
    \lvert D\Psi_\lambda(s,x) - D\Psi_\lambda(s,y)\rvert^p \lesssim \lvert x - y\rvert^{\alpha p}.
\end{align*}
For the first term in \eqref{eq:DY-diff2}, we apply H\"older's inequality to obtain 
\begin{align*}
    &\mathbb{E}\left[\int^t_s \lvert DY_r(y)\rvert^p \lvert P_r(x) - P_r(y)\rvert^p \, \diff r\right] \lesssim \mathbb{E}\left[\int^t_s \lvert P_r(x) - P_r(y)\rvert^{2p} \, \diff r\right]^{1/2} \\
    &\lesssim \mathbb{E}\left[\sup_{s\leq r \leq T} \lvert D\psi_\lambda(r,\phi_{s,r}(x))\rvert^{2p} \lvert D\Psi_\lambda^{-1}(r,Y_t(x)) - D\Psi_\lambda^{-1}(r,Y_t(y))\rvert^{2p}\right]^{1/2} \\
    &\quad + \mathbb{E}\left[\sup_{s\leq r \leq T} \lvert D\Psi_\lambda^{-1}(r,Y_r(y))\rvert^{2p} \lvert D\psi_\lambda(r,\phi_{s,r}(x)) - D\psi_\lambda(r,\phi_{s,r}(y))\rvert^{2p}\right]^{1/2} \\
    &\lesssim \mathbb{E}\left[\sup_{s\leq r \leq T} \lvert D\Psi_\lambda^{-1}(r,Y_r(x)) - D\Psi_\lambda^{-1}(r,Y_r(y))\rvert^{2p}\right]^{1/2} \!\! \\
    & \quad + \mathbb{E}\left[\sup_{s\leq r \leq T} \lvert D\psi_\lambda(r,\phi_{s,r}(x)) - D\psi_\lambda(r,\phi_{s,r}(y))\rvert^{2p}\right]^{1/2} \lesssim \lvert x-y\rvert^{\alpha p},
\end{align*}
where in the last inequality, we have employed \eqref{phi-diff-estimate} and \eqref{eq:Y-diff-bound} with $\beta = \alpha,$ together with the fact that $D\Psi_\lambda^{-1}, D\psi_\lambda$ are globally Lipschitz uniformly in time.
For the second term in \eqref{eq:DY-diff2}, we get the bound
\begin{align*}
    &\mathbb{E}\left[\int^T_s \lvert DY_r(y)\rvert^p \lvert Q_r(x) - Q_r(y)\rvert^p \, \diff r\right] \lesssim \mathbb{E}\left[\int^t_s \lvert Q_r(x) - Q_r(y)\rvert^{2p} \, \diff r\right]^{1/2} \\
    &\lesssim \mathbb{E}\left[\sup_{s\leq r \leq t} \lvert D\Psi_\lambda^{-1}(r,Y_r(x)) - D\Psi_\lambda^{-1}(r,Y_r(y))\rvert^{2p}\right]^{1/2} \!\! \\
    &\hspace{70pt} + \mathbb{E}\left[\sup_{s\leq r \leq t}\lvert D^2 \Psi_\lambda (r,\phi_{s,r}(x)) - D^2 \Psi_\lambda(r,\phi_{s,r}(y))\rvert^{2p}\right]^{1/2} \\
    &\lesssim \mathbb{E}\left[\sup_{s \leq r \leq t}\lvert Y_r(x)-Y_r(y)\rvert^{2 \alpha p}\right]^{1/2} + \mathbb{E}\left[\sup_{s\leq r \leq t}\lvert \phi_{s,r}(x) - \phi_{s,r}(y)\rvert^{2\alpha p}\right]^{1/2}\\
    &\lesssim \lvert x-y\rvert^{\alpha p},
\end{align*}
where in the fourth line, we employed the fact that the $\alpha$-H\"older seminorms $[D\Psi_\lambda^{-1}(t,\cdot)]_{\alpha}, [D^2 \Psi_\lambda (t,\cdot)]_{\alpha}$ are bounded uniformly in time and in the last line, we applied \eqref{phi-diff-estimate} and \eqref{eq:Y-diff-bound} with $\beta = 1$. Our proof is complete.

\subsection{Proof of the claim in the last paragraph of the proof of Proposition \ref{nonsmooth-existence}} \label{weaklimit}

We fix a test $k$-form $\mathbf{\Theta} \in C_0^\infty(\R^n,\bigwedge^k \mathcal{T}^*\mathbb R^n)$ and check term by term convergence in the following equation 
\begin{align*}
    & \llangle \mathbf{K}^\epsilon_{\cdot}, \mathbf{\Theta} \rrangle_{L^2} - \llangle \mathbf{K}_0, \mathbf{\Theta} \rrangle_{L^2} = -\int^{\cdot}_0 \llangle \mathbf{K}^\epsilon, \mathcal L_{b^\epsilon}^* \mathbf{\Theta} \rrangle_{L^2} \diff s - \int^{\cdot}_0 \llangle \mathbf{K}^\epsilon, \mathcal L_\xi^* \mathbf{\Theta} \rrangle_{L^2} \diff W_s \\
    & \hspace{170pt} + \frac12 \int^{\cdot}_0 \llangle \mathbf{K}^\epsilon, \mathcal L_\xi^* \mathcal L_\xi^* \mathbf{\Theta} \rrangle_{L^2} \diff s,
\end{align*}
weakly in $L^{2}(\Omega \times [0,T])$.
\begin{itemize}
    \item{The first term on the left-hand side converges as}
    \begin{align*}
         \llangle \mathbf{K}^\epsilon_\cdot,\mathbf{\Theta} \rrangle_{L^2}  \rightarrow  \llangle \mathbf{K}_\cdot,\mathbf{\Theta} \rrangle_{L^2},
    \end{align*}
    weakly in $L^{2}(\Omega \times [0,T])$ as $\epsilon \rightarrow 0,$ by definition due to the weak convergence of $\mathbf{K}^\epsilon$.
    \item{The Stratonovich-to-It\^o correction term in the second line.} We claim that
    \begin{align*}
        \int_0^\cdot \llangle  \mathbf{K}_s^\epsilon, \mathcal L_\xi^*\mathcal L_\xi^* \mathbf{\Theta} \rrangle_{L^2} \diff s \rightarrow \int_0^\cdot \llangle \mathbf{K}_s, \mathcal L_\xi^*\mathcal L_\xi^* \mathbf{\Theta} \rrangle_{L^2} \diff s,
    \end{align*}
    weakly in $L^{2}(\Omega \times [0,T]).$ Indeed, we observe that the mapping $h \rightarrow \int_0^\cdot h \,\diff s $ is linear continuous from the space of adapted processes of class $L^2(\Omega \times [0,T])$ into itself, and therefore also weakly continuous. Hence   $\int_0^\cdot \llangle  \mathbf{K}_s^\epsilon, \mathcal L_\xi^*\mathcal L_\xi^* \mathbf{\Theta} \rrangle_{L^2} \diff s \rightarrow \int_0^\cdot \llangle \mathbf{K}_s, \mathcal L_\xi^*\mathcal L_\xi^* \mathbf{\Theta} \rrangle_{L^2} \diff s,$ weakly in $L^2(\Omega \times [0,T]).$
    \item{The diffusion term on the right-hand side.} 
    We have
    \begin{align*} 
        \int_0^\cdot \llangle  \mathbf{K}_s^\epsilon, \mathcal L_\xi^* \mathbf{\Theta} \rrangle_{L^2} \diff W_s \rightarrow \int_0^\cdot \llangle \mathbf{K}_s, \mathcal L_\xi^* \mathbf{\Theta} \rrangle_{L^2} \diff W_s,
    \end{align*}
    weakly in $L^{2}(\Omega \times [0,T])$, since the mapping $h \rightarrow \int_0^\cdot h \,\diff W_s $ is linear continuous from the space of adapted processes of class $L^2(\Omega \times [0,T])$ into itself.
    \item{The drift term on the right-hand side.} We check that
    \begin{align*}
    \int_0^\cdot   \llangle \mathbf{K}_s^\epsilon, \mathcal L_{b^\epsilon}^* \mathbf{\Theta} \rrangle_{L^2} \diff s \rightarrow  \int_0^\cdot \llangle \mathbf{K}_s, \mathcal L_{b}^* \mathbf{\Theta} \rrangle_{L^2}  \diff s,
    \end{align*}
    weakly in $L^2(\Omega \times [0,T])$. To this purpose, one can verify that for $f \in L^2(\Omega \times [0,T])$,
    \begin{align} 
        & \left\lvert \int_0^T  \mathbb{E} \left[ \int_0^t  \left( \llangle \mathbf{K}_s^\epsilon, \mathcal L_{b^\epsilon}^* \mathbf{\Theta} \rrangle_{L^2}  - \llangle \mathbf{K}_s, \mathcal L_{b}^* \mathbf{\Theta} \rrangle_{L^2} \right)  \diff s \,f(\cdot,t) \right]  \diff t \right\rvert \nonumber  \\ 
        & \quad \leq \left\lvert \int_0^T \mathbb{E} \left[ \int_0^t   \llangle \mathbf{K}_s^\epsilon - \mathbf{K}_s, \mathcal L_{b}^* \mathbf{\Theta} \rrangle_{L^2}  f(\cdot,t) \, \diff s \right ]  \diff t \right\rvert \label{dominada1} \\
        & \qquad + \left\lvert\int_0^T \mathbb{E} \left[ \int_0^t  \llangle \mathbf{K}_s^\epsilon, \mathcal L_{b^\epsilon}^* \mathbf{\Theta} - \mathcal L_{b}^* \mathbf{\Theta} \rrangle_{L^2}  f(\cdot,t) \, \diff s   \right] \diff t \right\rvert. \label{dominada2}
    \end{align}
    The term \eqref{dominada1} converges to zero by applying the dominated convergence theorem to the function $F_1^\epsilon(t) := \mathbb{E} \left[ \int_0^t   \llangle \mathbf{K}_s^\epsilon - \mathbf{K}_s, \mathcal L_{b}^* \mathbf{\Theta} \rrangle_{L^2} f(\cdot,t)  \, \diff s \right ]$. For this, we note that by defining the test function $\mathbf{\Gamma}_{t}(\omega,s,x) := f(\omega,t) \mathcal L_b^* \mathbf{\Theta}(x)$, \footnote{It should be understood that $b$ depends on $s$.} we have 
    \begin{align*}
        & F_1^\epsilon := \mathbb{E} \left[ \int_0^t   \llangle \mathbf{K}_s^\epsilon - \mathbf{K}_s, \mathcal L_{b}^* \mathbf{\Theta} \rrangle_{L^2} f(\cdot,t)  \, \diff s \right ] = \llangle \mathbf{K}_s^\epsilon - \mathbf{K}_s,\mathbf{\Gamma}_{t}  \rangle \rangle_{L^2_{\omega,s,x}} \rightarrow 0,
    \end{align*}
    by the weak convergence of $\mathbf{K}_t^\epsilon$ to $\mathbf{K}_t$ in $L^2 ( \Omega \times [0,T] \times \R^n; \bigwedge^k \mathcal{T}^*\mathbb R^n)$. By choosing $R$ such that $\mathbf{\Theta}$ has support on $B(0,R),$ we also have the uniform bound
    \begin{align*}
        & \left \lvert \llangle \mathbf{K}_s^\epsilon - \mathbf{K}_s,\mathbf{\Gamma}_t  \rangle \rangle_{L^2_{\omega,s,x}} \right \rvert \leq \mathbb{E} \left[ \norm{\mathcal{X}_{B(0,R)}(\mathbf{K}_s^\epsilon - \mathbf{K}_s)}_{L^p_{s,x}} \right] \mathbb{E} \left[ \norm{ \mathcal L_{b}^* \mathbf{\Theta} }_{L^q_{s,x}} \lvert f(\cdot,t) \rvert \right] \\ 
        &\quad  \lesssim \norm{\mathcal{X}_{B(0,R)}(\mathbf{K}_s^\epsilon-\mathbf{K}_s)}_{L^p_{\omega,s,x}} \|\mathcal L_{b}^* \mathbf{\Theta}\|_{L^q_{s,x}}  \mathbb{E} \left[ \lvert f(\cdot,t) \rvert  \right ]  \lesssim  \mathbb{E} \left[ \lvert f(\cdot,t) \rvert \right ]   < \infty,
    \end{align*}
    where we have taken into account the uniform bound \eqref{greatbound}, Hypothesis \ref{hypoexistence} \ref{maincondition}, and the embedding $L^p_{loc} \subset L^2_{loc},$ $p \geq 2$. Since $f \in L^{2}(\Omega \times [0,T]) \subset  L^1(\Omega \times [0,T]),$ we have $\mathbb{E} \left[ \lvert f(\cdot,t) \rvert \right] \in L^1([0,T]).$ Similarly, the term \eqref{dominada2} also converges to zero by applying the dominated convergence theorem to $F_2^\epsilon(t) := \mathbb{E} \left[ \int_0^t  \llangle \mathbf{K}_s^\epsilon, \mathcal L_{b^\epsilon}^* \mathbf{\Theta} - \mathcal L_{b}^* \mathbf{\Theta} \rrangle_{L^2} f(\cdot,t) \, \diff s \right].$ The proof is now complete.
\end{itemize}

\subsection{Proof of Lemma \ref{lemma:theta-push-convergence}} \label{uglylemma}
We prove our result by induction on $k$ (i.e., the degree of the multivector field $\mathbf{U}$).

\noindent\textbf{Step 1.}
First consider the base case $k = 1$ (i.e. $\mathbf{U}$ is a vector field). Employing the notation $\psi_t := \phi_t^{-1}$ as usually, the pullback of $\mathbf{U}$ has the global coordinate expression
\begin{align}\label{eq:pullback-of-vf}
(\phi_t^* \mathbf{U})^i(x) = U^j(\phi_t(x)) \frac{\partial \psi_t^i}{\partial y_j}(\phi_t(x)), \quad i \in \{1, \ldots,n\}.
\end{align}
Hence, we have
\begin{align} 
    & \norm{(\phi_\cdot^m)^* \mathbf{U} (\cdot)  - (\phi_\cdot)^* \mathbf{U} (\cdot) }^r_{L^r_{\omega,t,R^*(m)}} \nonumber \\
    &\lesssim \mathbb{E}\left[\int^T_0 \int_{B(0,R^*(m))} \lvert \mathbf{U}(\phi_t^m(x))\rvert^r \lvert D\psi_t^m(\phi_t^m(x)) - D\psi_t(\phi_t(x))\rvert^r \, \diff^n x \, \diff t \right] \nonumber \\
    &\quad + \mathbb{E}\left[\int^T_0 \int_{B(0,R^*(m))} \lvert \mathbf{U}(\phi_t^m(x))-\mathbf{U}(\phi_t(x))\rvert^r \lvert D\psi_t(\phi_t(x))\rvert^r \, \diff^n x \, \diff t \right] \nonumber \\
    &\lesssim \norm{\mathbf{U}}_{L^\infty}^r \mathbb{E}\left[\int^T_0 \int_{B(0,R^*(m))} \lvert D\psi_t^m(\phi_t^m(x)) - D\psi_t(\phi_t(x))\rvert^r \, \diff^n x \, \diff t \right] \label{vector+flow1}  \\
    \begin{split}
    &\quad + \norm{D \mathbf{U}}_{L^\infty}^r \mathbb{E}\left[\int^T_0 \int_{B(0,R^*(m))}  \lvert\phi_t^m(x)-\phi_t(x)\rvert^{2r} \, \diff^n x \, \diff t \right]^{1/2} \\
    &\hspace{50pt} \times \mathbb{E}\left[\int^T_0 \int_{B(0,R^*(m))} \lvert D\psi_t(\phi_t(x))\rvert^{2r} \, \diff^n x \, \diff t \right]^{1/2}. \label{vector+flow2} 
    \end{split}
\end{align}
We now proceed to prove that the term \eqref{vector+flow1} converges to zero. Dealing with the term \eqref{vector+flow2} is easy since by standard arguments, we have the bound
\begin{align*}
&\mathbb{E}\left[\int^T_0 \int_{B(0,R^*(m))} \lvert D\psi_t(\phi_t(x))\rvert^{2r} \, \diff^n x \, \diff t \right]^{1/2} \\
&\quad = \mathbb{E}\left[\int^T_0 \int_{B(0,R^*(m))} \lvert D\psi_t(y)\rvert^{2r} \lvert J\psi_t(y)\rvert\, \diff^n y \, \diff t \right]^{1/2} \\
&\qquad \leq \|D\psi_\cdot(\cdot)\|_{L_{\omega, t, R'}^{4r}}^r \|D\psi_\cdot(\cdot)\|_{L_{\omega, t, R'}^{2n}}^{n/2} < \infty,
\end{align*}
and taking into account \eqref{weakflow}, we obtain convergence to zero. Hence, it suffices only to establish the convergence of \eqref{vector+flow1}. We observe
\begin{align}
    &\mathbb{E}\left[\int^T_0 \int_{B(0,R^*(m))} \lvert D\psi_t^m(\phi_t^m(x)) - D\psi_t(\phi_t(x))\rvert^r \, \diff^n x \, \diff t \right] \label{eq:Dpsi(phi) difference} \\
    &\lesssim  \mathbb{E}\left[\int^T_0 \int_{B(0,R^*(m))} \lvert D\psi_t^m(\phi_t^m(x)) - D\psi_t(\phi_t^m(x))\rvert^r \, \diff^n x \, \diff t  \right] \label{splitting1} \\
    &\hspace{30pt} + \mathbb{E}\left[\int^T_0 \int_{B(0,R^*(m))} \lvert D\psi_t(\phi_t^m(x)) - D\psi_t(\phi_t(x))\rvert^r \, \diff^n x \, \diff t \right]. \label{splitting2}
\end{align}
Changing coordinates $y = \phi^m_t(x)$, we obtain convergence of the first term after the inequality as
\begin{align*}
    \eqref{splitting1} &= \|D\psi_\cdot^{m}(\phi_\cdot^{m}(\cdot))- D\psi_\cdot( \phi_\cdot^{m}(\cdot))\|^r_{L^r_{\omega,t,R^*(m)}} \\
    &= \mathbb{E}\left[\int^T_0 \int_{B(0,R)} \lvert D\psi_t^{m}(y) - D\psi_t(y)\rvert^r \lvert J\psi_t^{m}(y)\rvert \, \diff^n y \, \diff t  \right] \\
    & \leq \|D\psi_\cdot^{m}(\cdot) - D\psi_\cdot(\cdot)\|^{r}_{L^{2r}_{\omega,t,R}} \|D\psi_\cdot^{m}(\cdot)\|_{L^{2n}_{\omega,t,R}}^n \rightarrow 0,
\end{align*}
where we have taken into account property \eqref{weakjacobian} of the flow. Fix $\alpha' \in (\frac{1}{4r}, \alpha-\frac{n}{4r})$\footnote{We can guarantee that $\alpha-\frac{n}{4r} > \frac{1}{4r}$ since we have imposed $r\alpha > (n+1)/4.$}.
To establish the convergence of \eqref{splitting2}, we
set $\widetilde{r}(x) := \sqrt{1+|x|^2}$ and observe that
\begin{align*}
& \lvert D\psi_t(\phi_t^m(x)) - D\psi_t(\phi_t(x))\rvert = \lvert\widetilde{r}(\phi_t(x))\widetilde{r}^{-1}(\phi_t(x))(D\psi_t(\phi_t^m(x)) - D\psi_t(\phi_t(x)))\rvert\\
& \leq \lvert \widetilde{r}(\phi_t(x))\rvert \Big(\left\lvert \widetilde{r}^{-1}(\phi_t^m(x))D\psi_t(\phi_t^m(x)) - \widetilde{r}^{-1}(\phi_t(x))D\psi_t(\phi_t(x))\right\rvert \\
& \hspace{150pt} + \lvert D\psi_t(\phi_t^m(x))\rvert \big\lvert \widetilde{r}^{-1}(\phi_t^m(x)) - \widetilde{r}^{-1}(\phi_t(x))\big\rvert \Big)\\
& \leq \lvert \widetilde{r}(\phi_t(x))\rvert \left(C_1(\omega,t) \lvert\phi_t^m(x) - \phi_t(x)\rvert^{\alpha'} + C_2 \lvert D\psi_t(\phi_t^m(x))\rvert \lvert \phi_t^m(x) - \phi_t(x)\rvert\right),
\end{align*}
where
\begin{align*}
C_1(\omega,t) := \sup_{\substack{x,y \in \R^n \\ x \neq y}}\frac{\lvert \widetilde{r}^{-1}(x)D\psi_{s,t}(x) - \widetilde{r}^{-1}(y)D\psi_{s,t}(y) \rvert}{\lvert x-y\rvert^{\alpha'}}
\end{align*}
is finite $\mathbb{P}$-a.s. for all $0 \leq t \leq T$ by property \eqref{eq:JM-estimate} of the flow with $\epsilon=1$ and
\begin{align*}
C_2 := \sup_{\substack{x,y \in \R^n \\ x \neq y}} \frac{\lvert \widetilde{r}^{-1}(x) - \widetilde{r}^{-1}(y)\rvert}{\lvert x-y\rvert}
\end{align*}
is finite since $\widetilde{r}^{-1}$ is globally Lipschitz. Hence, by applying H\"older's inequality in the whole space, we get 
\begin{align*}
    &\eqref{splitting2} = \|D\psi_\cdot(\phi_\cdot^{m}(\cdot)) - D\psi_\cdot( \phi_\cdot(\cdot))\|^r_{L^r_{\omega,t,R^*(m)}} \\
    &\lesssim \mathbb{E}\left[\int^T_0 \int_{B(0,R')} \lvert \widetilde{r}(\phi_t(x))\rvert^r \left( C_1(\cdot,t) \lvert \phi_t^m(x) - \phi_t(x)\rvert^{\alpha'} \right)^r \diff^n x \, \diff t \right] \\
    & \qquad + \mathbb{E}\left[\int^T_0 \int_{B(0,R')} \lvert \widetilde{r}(\phi_t(x))\rvert^r \left(C_2\lvert D\psi_t(\phi_t^m(x))\rvert \lvert \phi_t^m(x) - \phi_t(x)\rvert\right)^r \diff^n x  \, \diff t\right]
\end{align*}
\begin{align*}
    &\lesssim \|\widetilde{r}(\phi_\cdot(\cdot))\|_{L^{2r}_{\omega,t,R'}}^{r} \left(\mathbb{E}\left[\int^T_0 \int_{B(0,R')} C_1^{2r}(\cdot,t) \lvert \phi_t^m(x) - \phi_t(x)\rvert^{2r\alpha'} \, \diff^n x \, \diff t\right]^{1/2} \right.\\
    &\left. \qquad + C_2^{2r} \mathbb{E}\left[\int^T_0 \int_{B(0,R')} \lvert D\psi_t(\phi_t^m(x)))\rvert^{2r} \lvert \phi_t^m(x) - \phi_t(x)\rvert^{2r} \, \diff^n x \, \diff t\right]^{1/2}\right) \\
    &= \|\widetilde{r}(\phi_\cdot(\cdot))\|_{L^{r}_{\omega,t,R'}}^r [\text{(I)} + \text{(II)}].
\end{align*}
By our usual arguments, we have the control
\begin{align*}
    \|\widetilde{r}(\phi_\cdot(\cdot))\|_{L^{2r}_{\omega,t,R'}}^r &= \mathbb{E}\left[\int^T_0 \int_{B(0,R')} \lvert \widetilde{r}(y)\rvert^{2r} \lvert J\phi_t(y)\rvert \, \diff^n y \, \diff t \right]^{1/2} \\
    &\leq \|\widetilde{r}\|_{L^\infty_{R'}}^{r}\|D\phi_\cdot (\cdot)\|_{L^{n}_{\omega,t,R'}}^{n/2} < \infty.
\end{align*}
Moreover, for term (I) we have the convergence
\begin{align*}
    & \text{(I)} \lesssim \mathbb{E}\left[\sup_{0 \leq t \leq T} C_1^{4r}(\cdot, t)\right]^{1/4} \mathbb{E}\left[\int^T_0 \int_{B(0,R')} \lvert \phi_t^m(x) - \phi_t(x)\rvert^{4\alpha'r} \diff^n x \, \diff t\right]^{1/4} \rightarrow 0,
\end{align*}
where we have applied H\"older's inequality, taken into account property \eqref{eq:JM-estimate} with $4r$ instead of $r$ to control the weighted H\"older norm of $C_1,$ and employed \eqref{weakflow} to establish convergence of the second term (note that this can be done since $\alpha' \geq \frac{1}{4r}$ and therefore $4\alpha' r \geq 1$). For the second term (II), we have
\begin{align*}
    \text{(II)} &\lesssim \mathbb{E}\left[\int^T_0 \int_{B(0,R')} \lvert D\psi_t(y)\rvert^{4r} \lvert J \phi_t^m(x)\rvert \, \diff^n y \, \diff t\right]^{1/4} \|\phi_\cdot^m(\cdot) - \phi_\cdot(\cdot)\|_{L^{4r}_{\omega,t,R'}}^r \\
    &\leq \|D\psi_\cdot(\cdot)\|_{L^{8r}_{\omega,t,R'}}^r\|D\phi_\cdot^m(\cdot)\|_{L^{2n}_{\omega,t,R'}}^{n/4} \|\phi_\cdot^m(\cdot) - \phi_\cdot(\cdot)\|_{L^{4r}_{\omega,t,R'}}^r \rightarrow 0,
\end{align*}
where we have employed our usual strategies and the fact that $\|D\phi_\cdot^m(\cdot)\|_{L^s_{\omega,t,R'}}$ is uniformly bounded for any $s \geq 1$, by property \eqref{bound-phi-deriv} of the flow. This concludes the convergence \eqref{eq:theta-push-convergence} when $\mathbf{U}$ is a vector field.

\noindent\textbf{Step 2.} We now proceed to perform the induction step. Here, we only do the case $k=2$ for the sake of communicating the idea more clearly; however, proving this in the general case is straightforward, albeit notationally convoluted. If $\mathbf{U}$ is a two-vector field, we have
\begin{align}\label{eq:pullback of 2-vector field}
(\phi_t^* \mathbf{U})^{ij}(x) = U^{kl}(\phi_t(x)) \frac{\partial \psi_t^i}{\partial y_k}(\phi_t(x)) \frac{\partial \psi_t^j}{\partial y_l}(\phi_t(x)), \quad i,j \in \{1,\ldots,n\}.
\end{align}
Hence, for fixed $i,j \in \{1,\ldots,n\},$ we obtain
\begin{align}\label{eq:difference of 2-vector fields}
    &((\phi_t^m)^* \mathbf{U})^{ij}(x) - (\phi_t^* \mathbf{U})^{ij}(x) \nonumber\\
    &= \frac{\partial (\psi_t^m)^j}{\partial y_l}(\phi_t^m(x))\left(U^{kl}(\phi_t^m(x))\frac{\partial (\psi_t^m)^i}{\partial y_k}(\phi_t^m(x)) - U^{kl}(\phi_t(x))\frac{\partial \psi_t^i}{\partial y_k}(\phi_t(x))\right) \nonumber \\
    &\quad + U^{kl}(\phi_t(x))\frac{\partial \psi^i_t}{\partial y_k}(\phi_t(x)) \left(\frac{\partial (\psi_t^m)^j}{\partial y_l}(\phi_t^m(x)) - \frac{\partial \psi^j_t}{\partial y_l}(\phi_t(x))\right).
\end{align} 
Define 
\begin{align*}
& v^{l,i}(t,x) := U^{kl}(\phi_t(x)) \frac{\partial \psi_t^i}{\partial y_k}(\phi_t(x)), \qquad v_m^{l,i}(t,x) := U^{kl}(\phi_t^m(x)) \frac{\partial (\psi_t^m)^i}{\partial y_k}(\phi_t^m(x)). 
\end{align*}
Observe that for fixed index $l$, these are identical to the expression \eqref{eq:pullback-of-vf} of the pullback of a vector field that we considered in the previous step. \footnote{In the general $k$ case, one can set $v$ and $v_m$ similarly so that they adopt the expression of the pullback of $(k-1)$-vector fields (for example, in the case $k=3$, they should take the form \eqref{eq:pullback of 2-vector field}).} We also define $v$ and $v_m$ to be the tuples $v := (v^1, \ldots, v^n),$ where $v^l:=(v^{l,1},\ldots,v^{l,n}),$ and $v_m := (v^1_m, \ldots, v^n_m),$ where $v_m^l:=(v_m^{l,1},\ldots,v_m^{l,n}),$ respectively, for $l \in \{1,\ldots,n\}.$ For every $l$, the sequence $(v_m^l)_{m \in \mathbb{N}}$ satisfies
\begin{align}\label{eq:inductive hypothesis}
\mathbb{E}\left[ \int_0^T \int_{B(0,R^*(m))} \left\lvert v_m^l(t,x)  - v^l(t,x) \right\rvert^{2r} \diff^n x \, \diff t \right] \rightarrow 0, \quad m \rightarrow \infty,
\end{align}
by the inductive hypothesis. Notice that since $r > (n+1)/(\alpha 2^{k+1}),$ we have $2r > (n+1)/(\alpha 2^{k}),$ which satisfies our assumption at Step $k-1$. 

Next, applying Cauchy-Schwartz inequality in \eqref{eq:difference of 2-vector fields} to get 
\begin{align*}
    \lvert (\phi_t^m)^* \mathbf{U})^{ij}(x) - (\phi_t^*\mathbf{U})^{ij}(x)\rvert &\leq \lvert D(\psi_t^m)^j(\phi_t(x))\rvert \lvert C_m^i(t,x) - C^i(t,x)\rvert \\
    &\quad + \lvert C^i(t,x)\rvert \lvert D(\psi_t^m)^j(\phi_t(x)) - D\psi_t^j(\phi_t(x))\rvert, 
\end{align*}
where $C^i := (\sum_l (v^{l,i})^2)^\frac12$ and $C^i_m := (\sum_l (v^{l,i}_m)^2)^\frac12$, then applying the triangle inequality
\begin{align*}
    \Big(\sum_{ij} \lvert A_{ij} + B_{ij}\rvert^2 \Big)^\frac12 \leq \Big(\sum_{ij} A_{ij}^2\Big)^\frac12 + \Big(\sum_{ij} B_{ij}^2\Big)^\frac12,
\end{align*}
on the RHS (take $A_{ij} := \lvert D(\psi_t^m)^j(\phi_t(x))\rvert \lvert C_m^i(t,x) - C^i(t,x)\rvert$ and $B_{ij} := \lvert C^i(t,x)\rvert \lvert D(\psi_t^m)^j(\phi_t(x)) - D\psi_t^j(\phi_t(x))\rvert$), we have
\begin{align*}
    \lvert (\phi_t^m)^* \mathbf{U}(x) - \phi_t^* \mathbf{U}(x)\rvert &\leq \lvert D\psi_t^m(\phi_t^m(x))\rvert \lvert v_m(t,x) - v(t,x)\rvert \\
    &\quad + \lvert v(t,x)\rvert \lvert D\psi_t^m(\phi_t^m(x)) - D\psi_t(\phi_t(x))\rvert,
\end{align*}
where the norm $\lvert \cdot \rvert$ on $v$ (likewise for $v_m$) is defined as
\begin{align*}
    \lvert v(t,x)\rvert = \left(\sum_{l=1}^n \vert v^l(t,x)\rvert^2\right)^{1/2} = \left(\sum_{l=1}^n \sum_{i=1}^n (v^{l, i}(t,x))^2 \right)^{1/2}.
\end{align*}
This yields
\begin{align*} 
    &\norm{(\phi_\cdot^m)^* \mathbf{U} (\cdot)  - (\phi_\cdot)^* \mathbf{U} (\cdot) }^r_{L^r_{\omega,t,R^*(m)}} \\
    &\lesssim \mathbb{E}\left[\int^T_0 \int_{B(0,R^*(m))} \lvert D\psi_t^m(\phi_t^m(x))\rvert^r \lvert v_m(t,x) - v(t,x)\rvert^r \, \diff^n x \, \diff t \right] \\
    &\quad + \mathbb{E}\left[\int^T_0 \int_{B(0,R^*(m))} \lvert v(t,x)\rvert^r \lvert D\psi_t^m(\phi_t^m(x)) - D\psi_t(\phi_t(x))\rvert^r \, \diff^n x \, \diff t \right] \\
    &= I_1 + I_2.
\end{align*}
Now using H\"older's inequality, we obtain the following estimate for the first term
\begin{align*}
    I_1 &\leq \|D\psi_\cdot^m(\phi_\cdot^m(\cdot))\|_{L^{2r}_{\omega, t, R^*(m)}}^r \|v_m(\cdot,\cdot) - v(\cdot,\cdot)\|_{L^{2r}_{\omega, t, R^*(m)}}^r.
\end{align*}
By our usual arguments, we have the control
\begin{align*}
    \|D\psi_\cdot^m(\phi_\cdot^m(\cdot))\|_{L^{2r}_{\omega, t, R^*(m)}}^r &= \mathbb{E}\left[\int^T_0 \int_{B(0,R)} \lvert D\psi_t^m(y)\rvert^{2r} \lvert J\psi_t^m(y)\rvert \, \diff^n y \, \diff t \right]^{1/2} \\
    &\leq \|D\psi^m_\cdot(\cdot)\|_{L^{4r}_{\omega,t,R}}^r \|D\psi^m_\cdot(\cdot)\|_{L^{2n}_{\omega,t,R}}^{n/2} < \infty,
\end{align*}
where we used property \eqref{bound-phi-deriv}. Noting that
\begin{align*}
    \lvert v_m(t,x) - v(t,x)\rvert^{2r} = \left(\sum_{l=1}^n \lvert v_m^l(t,x) - v^l(t,x)\rvert^2 \right)^{r} \lesssim \sum_{l=1}^n \lvert v_m^l(t,x) - v^l(t,x)\rvert^{2r},
\end{align*}
we also get 
\begin{align*}
    \|v_m(\cdot,\cdot) - v(\cdot,\cdot)\|_{L^{2r}_{\omega, t, R^*(m)}}^r \lesssim \sum_{l=1}^n \|v_m^l(\cdot,\cdot) - v^l(\cdot,\cdot)\|_{L^{2r}_{\omega, t, R^*(m)}}^r \rightarrow 0,
\end{align*}
where we used the inductive hypothesis \eqref{eq:inductive hypothesis} to deduce the convergence in the last line. Hence $I_1 \rightarrow 0$.
Applying H\"older's inequality, we also have
\begin{align*}
    I_2 &\leq \|v(\cdot,\cdot)\|_{L^{2r}_{\omega, t, R^*(m)}}^r \|D\psi_\cdot^m(\phi_\cdot^m(\cdot)) - D\psi_\cdot(\phi_\cdot(\cdot))\|_{L^{2r}_{\omega, t, R^*(m)}}^r \\
    &\leq \|U(\cdot)\|_{L^\infty_x}^r\|D\psi_\cdot(\phi_\cdot(\cdot))\|_{L^{2r}_{\omega, t, R^*(m)}}^r \|D\psi_\cdot^m(\phi_\cdot^m(\cdot)) - D\psi_\cdot(\phi_\cdot(\cdot))\|_{L^{2r}_{\omega, t, R^*(m)}}^r \rightarrow 0,
\end{align*}
where we used that $\|D\psi_{\cdot}(\phi_{\cdot}(\cdot))\|_{L^{2r}_{\omega, t, R^*(m)}}^r < \infty$ by standard arguments and the convergence $\|D\psi_{\cdot}^m(\phi_{\cdot}^m(\cdot)) - D\psi_{\cdot}(\phi_{\cdot}(\cdot))\|_{L^{2r}_{\omega, t, R^*(m)}}^r \rightarrow 0$ holds by following similar arguments as in Step 1 of the proof, where we established convergence of \eqref{eq:Dpsi(phi) difference}. The same arguments hold in the general $k$ case. This concludes our proof.

\subsection{Proof of KIW formula for $k$-form-valued processes (Theorem \ref{KIWmain})}\label{eq:app-KIW-proof}
We first establish \eqref{KIW}. We divide our proof into several steps.

\noindent\textbf{Step 1. Setup and preparation.}
For simplicity and without loss of generality, we choose $N=M=1$ so that the only diffusion vector field in the flow equation \eqref{flowmap} is $\xi =(\xi^1,\ldots,\xi^n)$. Given a multi-vector field $\mathbf{U} = u_1 \wedge \ldots \wedge u_k \in \mathfrak X^k(\mathbb R^n),$ we consider the following real-valued continuous $\mathcal{F}_t$-semimartingale 
\begin{align} \label{F-eq-full}
& F_t(x) := \langle \phi^*_t \mathbf{K}(t,x),\mathbf{U}(x) \rangle = K_{i_1,\ldots,i_k}(t,\phi_t(x)) \prod_{\alpha=1}^k J^{i_\alpha}_{j_\alpha}(t,x) u_\alpha^{j_\alpha}(x),
\end{align}
for $t \in [0,T]$ and $x \in \mathbb{R}^n.$ In \eqref{F-eq-full}, $\langle \cdot , \cdot \rangle$ represents the dual pairing, we employed the notation
$$J^{i_\alpha}_{j_\alpha}(t,x) := \frac{\partial \phi^{i_\alpha}_t}{\partial x_{j_{\alpha}}}(x), \quad t \in [0,T], \quad x \in \mathbb{R}^n,$$
and we stress that summation over repeated indices $i_\alpha=1,\ldots ,n,$ $\alpha=1,\ldots, k$ is implied as usually.

\noindent{\bf Step 2. Equation for $F_t$ in simple form.} 
We employ differential notation to avoid excessively convoluted formulas. It\^o's product rule, we have
\begin{align*}
\diff F_t(x) &= \left(\prod_{\alpha=1}^k J^{i_\alpha}_{j_\alpha}(t,x) u_\alpha^{j_\alpha}(x)\right) \diff  K_{i_1,\ldots,i_k}(t,\phi_t(x)) \\
&\quad + K_{i_1,\ldots,i_k}(t,\phi_t(x)) \, \diff \left(\prod_{\alpha=1}^k J^{i_\alpha}_{j_\alpha}(t,x) u_\alpha^{j_\alpha}(x)\right) \\
&\quad + \diff \left[K_{i_1,\ldots,i_k}(\cdot,\phi_{\cdot}(x)), \prod_{\alpha=1}^k J^{i_\alpha}_{j_\alpha}(\cdot,x) u_\alpha^{j_\alpha}(x) \right]_t, \\
\end{align*}
where $[\cdot,\cdot]_t$ denotes cross-variation. By applying the It\^o-Wentzell formula to $K_{i_1,\ldots,i_k}(t,\phi_t(x))$ and using the flow equation \eqref{flowmap}, one can check that $\mathbb{P}$-a.s.
\begin{align*} 
    &\diff K_{i_1,\ldots,i_k}(t,\phi_t(x)) =  G_{i_1,\ldots,i_k}(t,\phi_t(x)) \, \diff t + H_{i_1,\ldots,i_k}(t,\phi_t(x)) \, \diff W_t \\
    &\quad + (b \cdot D K_{i_1,\ldots,i_k}) (t,\phi_t(x)) \, \diff t + (\xi \cdot D K_{i_1,\ldots,i_k})(t,\phi_t(x)) \, \diff B_t \\
    &\quad + (\xi \cdot D H_{i_1,\ldots,i_k})(t,\phi_t(x)) \, \diff [W_{\cdot}, B_{\cdot}]_t + \frac12 (\xi \cdot D (\xi \cdot D K_{i_1,\ldots,i_k}))(t,\phi_t(x)) \, \diff t,
\end{align*}
for all $t \in [0,T]$ and $x \in \mathbb{R}^n.$
\begin{remark}
We note that $G_{i_1,\ldots,i_k}$ and $H_{i_1,\ldots,i_k}$ do not satisfy the assumptions of the classical It\^o-Wentzell formula (Theorem 1.1. in \cite{kunita1981some}), since they are not continuous in time. However, as noted at the end of page 24 in \cite{flandoli2010well}, Kunita's theorems are stated for processes integrated with respect to continuous semimartingales, and hence the assumptions imposed are those necessary to treat semimartingales. In the present case, it is easy to prove that adaptability, $L^1$ and $L^2$ integrability in time of $G_{i_1,\ldots,i_k}$ and $H_{i_1,\ldots,i_k}$ (respectively), and regularity in space are enough. For instance, in \cite{krylov2011ito}, the author provides an extension of the classical It\^o-Wentzell formula that allows for drifts and diffusions $G_{i_1,\ldots,i_k}$ and $H_{i_1,\ldots,i_k}$ of such time regularity, and even for distribution-valued processes.
\end{remark}
By It\^o's product rule, we also obtain
\begin{align*}
&\diff \left(\prod_{\alpha=1}^k J^{i_\alpha}_{j_\alpha}(t,x) u_\alpha^{j_\alpha}(x)\right) = \sum_{p=1}^k \left(\prod_{\alpha \neq p}^k J^{i_\alpha}_{j_\alpha}(t,x) u_{\alpha}^{j_\alpha}(x) \right) u_p^{j_p}(x) \, \diff J^{i_p}_{j_p}(t,x) \\
&\quad + \frac12 \sum_{\substack{p,q = 1 \\ p \neq q}}^k \left(\prod_{\alpha \neq p,q}^k J^{i_\alpha}_{j_\alpha}(t,x) u_\alpha^{j_\alpha}(x) \right) u_p^{j_p}(x) u_q^{j_q}(x) \, \diff \left[J^{i_p}_{j_p}(\cdot,x),J^{i_q}_{j_q}(\cdot,x)\right]_t.
\end{align*}
We differentiate \eqref{flowmap} with respect to $x$ to arrive at 
\begin{align*}
\diff J_j^i(t,x) &= D_l ( b^i + (1/2) \xi^m D_m \xi^i)(t,\phi_t(x)) J_j^l(t,x) \, \diff t \\
&\quad +  D_l \xi^i(t,\phi_t(x)) J_j^l(t,x) \, \diff B_t,
\end{align*}
for $i,j = 1,\ldots,n$ for all $t \in [0,T]$ and $x \in \R^n.$

\noindent{\bf Step 3. Full equation for $F_t$.} 

By using all the formulas obtained in Step 2, one checks that for $\mathbb{P}$-a.s. $\omega$ independent of $u,$ $F_t(x)$ can be expressed as
\begin{align*}
F_t(x) - F_0(x) = & \int^t_0 \hat{G}^{1}_s(x) \, \diff s + \int^t_0 \hat{G}^{2}_s(x) \, \diff [W_{\cdot},B_{\cdot}]_s \\
&\quad + \int^t_0 \hat{H}^{1}_s(x) \, \diff W_s + \int^t_0 \hat{H}^{2}_s(x) \, \diff B_s,
\end{align*}
where 
\begin{align*}
&\hat{G}^{1}_s(x) :=  G_{i_1,\ldots,i_k}(s,\phi_s(x)) \prod_{\alpha = 1}^k J^{i_\alpha}_{j_\alpha}(s,x) u_\alpha^{j_\alpha}(x) \\
& + b^l(s,\phi_s(x)) D_l K_{i_1,\ldots,i_k}(s,\phi_s(x)) \prod_{\alpha = 1}^k J^{i_\alpha}_{j_\alpha}(s,x) u_\alpha^{j_\alpha}(x) \\
& + \sum_{p=1}^k K_{i_1,\ldots,i_k}(s,\phi_s(x)) D_l b^{i_p}(s,\phi_s(x)) \left(\prod_{\alpha \neq p}^k J^{i_\alpha}_{j_\alpha}(s,x) u_\alpha^{j_\alpha}(x) \right) J^l_{j_p}(s,x) u_p^{j_p}(x) \\
&\quad + \frac12 \xi^l(s,\phi_s(x)) \xi^m(s,\phi_s(x)) D_{lm}^2 K_{i_1,\ldots,i_k}(s,\phi_s(x)) \prod_{\alpha = 1}^k J^{i_\alpha}_{j_\alpha}(s,x) u_\alpha^{j_\alpha}(x) \\
&\quad + \frac12 \xi^m(s,\phi_s(x))D_m \xi^l(s,\phi_s(x)) D_l K_{i_1,\ldots,i_k}(s,\phi_s(x)) \prod_{\alpha = 1}^k J^{i_\alpha}_{j_\alpha}(s,x) u_\alpha^{j_\alpha}(x) \\
&\quad + \sum_{p=1}^k \Bigg[ \xi^m(s,\phi_s(x)) D_l \xi^{i_p}(s,\phi_s(x)) D_m K_{i_1,\ldots,i_k}(s,\phi_s(x)) \\
&\quad \quad + \frac12 D_l\left[\xi^m(s,\phi_s(x)) D_m \xi^{i_p}(s,\phi_s(x))\right] K_{i_1,\ldots,i_k}(s,\phi_s(x)) \Bigg] \\
&\qquad \times \left(\prod_{\alpha \neq p}^k J^{i_\alpha}_{j_\alpha}(s,x) u_\alpha^{j_\alpha}(x) \right) J^l_{j_p}(s,x) u_p^{j_p}(x) \\
&\quad + \frac12 \sum_{\substack{p,q=1 \\ p \neq q}}^k K_{i_1,\ldots,i_k}(s,\phi_s(x)) D_l \xi^{i_p}(s,\phi_s(x)) D_m \xi^{i_q}(s,\phi_s(x)) \\
&\qquad \times \left(\prod_{\alpha \neq p,q}^k J^{i_\alpha}_{j_\alpha}(s,x) u_\alpha^{j_\alpha}(x)\right) J^l_{j_p}(s,x) J^m_{j_q}(s,x) u_p^{j_p}(x) u_q^{j_q}(x), \\
&\hat{G}^{2}_s(x) := \xi^l(s,\phi_s(x)) D_l H_{i_1,\ldots,i_k}(s,\phi_s(x)) \prod_{\alpha = 1}^k J^{i_\alpha}_{j_\alpha}(s,x) u_\alpha^{j_\alpha}(x)\\
& + \sum_{p=1}^k H_{i_1,\ldots,i_k}(s,\phi_s(x)) D_l \xi^{i_p}(s,\phi_s(x)) \left(\prod_{\alpha \neq p}^k J^{i_\alpha}_{j_\alpha}(s,x) u_\alpha^{j_\alpha}(x) \right) J^l_{j_p}(s,x) u_p^{j_p}(x),\\
&\hat{H}^{1}_s(x) := H_{i_1,\ldots,i_k}(s,\phi_s(x)) \prod_{\alpha = 1}^k J^{i_\alpha}_{j_\alpha}(s,x) u_\alpha^{j_\alpha}(x),\\
&\hat{H}^{2}_s(x) := \xi^l(s,\phi_s(x)) D_l K_{i_1,\ldots,i_k}(s,\phi_s(x)) \prod_{\alpha = 1}^k J^{i_\alpha}_{j_\alpha}(s,x) u_\alpha^{j_\alpha}(x)  \\
& + \sum_{p=1}^k K_{i_1,\ldots,i_k}(s,\phi_s(x)) D_l \xi^{i_p}(s,\phi_s(x)) \left(\prod_{\alpha \neq p}^k J^{i_\alpha}_{j_\alpha}(s,x) u_\alpha^{j_\alpha}(x) \right) J^l_{j_p}(s,x) u_p^{j_p}(x).
\end{align*}

\noindent{\bf Step 4. Conclusion.}  By inspection and taking into account formula \eqref{Lieformula}, we obtain that for $\mathbb{P}$-a.s. $\omega$ independent of $\mathbf{U}:$
\begin{align*}
& \int^t_0 \hat{G}^{1}_s(x)  \, \diff s  = \int^t_0 \langle \phi_s^* \mathbf{G}(s,x), \mathbf{U}(x)\rangle \, \diff s + \int^t_0 \langle \phi_s^* \mathcal L_b \mathbf{K}(s,x), \mathbf{U}(x)\rangle \, \diff s \\
& \hspace{90pt} + \frac12 \int^t_0 \langle \phi_s^* \mathcal L_\xi \mathcal L_\xi \mathbf{K}(s,x), \mathbf{U}(x)\rangle \, \diff s, \\
& \int^t_0 \hat{G}^{2}_s(x) \, \diff [W_{\cdot},B_{\cdot}]_s = \int^t_0 \langle \phi^*_s \mathcal L_\xi \mathbf{H}(s,x), \mathbf{U}(x)\rangle \, \diff [W_{\cdot},B_{\cdot}]_s,\\
& \int^t_0 \hat{H}^{1}_s(x) \, \diff W_s = \int^t_0 \langle \phi^*_s \mathbf{H}(s,x), \mathbf{U}(x) \rangle \, \diff W_s, \\
& \int^t_0 \hat{H}^{2}_s(x) \, \diff B_s = \int^t_0 \langle \phi_s^* \mathcal L_\xi \mathbf{K}(s,x),\mathbf{U}(x) \rangle \, \diff B_s, \quad t \in [0,T], \quad x \in \mathbb{R}^n,
\end{align*}
where we remind $\mathbf{U} = u_1 \wedge \ldots \wedge u_k.$ Since $\mathbf{U}$ was chosen arbitrarily, the proof of \eqref{KIW} is finished.
The proof of \eqref{KIWpush} follows similar steps by noting that
\begin{enumerate}
    \item $(\phi_t)_* = \psi_t^*,$ where $\psi_t := \phi_t^{-1}.$
    \item Using the identity \eqref{flowinv}, one can show that the inverse flow satisfies
    \begin{align} \label{inversaflow}
    \begin{split}
        \diff \psi_t(x) &= \left(-b^i(t,x) \frac{\partial \psi_t}{\partial x_i}(x) + \xi^i(t,x) \frac{\partial \xi^j}{\partial x_i}(t,x) \frac{\partial \psi_t}{\partial x_j}(x) \right. \\
        &\qquad \left.+ \xi^i(t,x) \xi^j(t,x) \frac{\partial^2 \psi_t}{\partial x_j \partial x_i}(x) \right) \diff t - \xi^i(t,x) \frac{\partial \psi_t}{\partial x_i}(x) \diff B_t,
    \end{split}
    \end{align}
    for $t \in [0,T], x \in \R^n,$ where we recall that summation over repeated indices is implied.
    \item One can apply the It\^o-Wentzell formula to evaluate $K_{i_1,\ldots,i_k}(t,x)$ along $\psi_t(x)$ using \eqref{inversaflow}.
\end{enumerate}

\section{Background in tensor algebra and calculus} \label{tensor}
In this appendix, we provide some background on concepts from differential geometry that we use throughout this article. Although these concepts are naturally introduced on smooth manifolds, here we restrict ourselves to the Euclidean space $\mathbb{R}^n$, since (1) we only need this setting for our analysis, (2) we avoid defining manifolds and dealing with the arising technicalities.

We first introduce the notion of tensor fields on $\mathbb{R}^n$ and discuss some operations defined on them, such as raising and lowering indices when each tangent space (defined later on) is equipped with an inner product (the metric tensor). Next, we present some common classes of tensor fields that arise naturally in physics, namely, the tensors of alternating types. We also discuss some of the common operations than can be performed on tensor fields, including the Lie derivative, which is a fundamental operator that we need in order to define our main equation. Whereas in the present discussion, we assume for simplicity that the tensor fields we work with are smooth, in the proceeding appendix we will introduce tensor fields of different regularity classes. This will be fundamental for our analysis.

Let $V$ be a finite dimensional vector space with basis $\{e_1,\ldots,e_n\},$ and $V^*$ be its corresponding dual space with dual basis $\{e^1,\ldots,e^n\}$, defined via the relation $\left<e^i,e_j\right> = \delta_{ij}$, where $\delta_{ij}$ is the Kronecker delta (that is, $\delta_{ij}=1$ when $i=j$ and zero otherwise), where $\left<\cdot,\cdot\right>:V^* \times V \rightarrow \mathbb R$ is the natural pairing $\left<\alpha, v\right> := \alpha(v)$. For a contravariant vector $v \in V$, we express its components as $v = v^i e_i$ using raised indices, and for a covariant vector (covector) $\alpha \in V^*$, we express its components as $\alpha = \alpha_i e^i$ with lowered indices. We always assume Einstein's convention of summing over repeated indices.

\begin{definition}
Given an $n$-dimensional vector space $V$ and its dual $V^*,$ we define an $(r,s)$-tensor $\mathbf{T}$ to be a multilinear map
\begin{align*}
    \mathbf{T} : \underbrace{V^* \times \cdots \times V^*}_{r \text{ times}} \times \underbrace{V \times \cdots \times V}_{s \text{ times}} \rightarrow \mathbb R.
\end{align*}
Fixing a basis $\{e_1, \ldots, e_n\}$ for $V$ and $\{e^1, \ldots, e^n\}$ for $V^*$, we may interpret tensors as vectors with basis $\{e_{i_1} \otimes \cdots \otimes e_{i_r} \otimes e^{i_{r+1}} \otimes \cdots \otimes e^{i_{r+s}}\}_{i_k \in \{1,\ldots, n\}}$, where
\begin{align*}
    &e_{i_1} \otimes \cdots \otimes e_{i_r} \otimes e^{i_{r+1}} \otimes \cdots \otimes e^{i_{r+s}} \left(e^{j_1}, \ldots, e^{j_r}, e_{j_{r+1}}, \ldots, e_{j_{r+s}}\right) \\
    &:= \delta_{i_1}^{j_1} \cdots \delta_{i_r}^{j_r} \delta^{i_{r+1}}_{j_{r+1}} \cdots \delta^{i_{r+s}}_{j_{r+s}}.
\end{align*}
We can express $\mathbf{T}$ in coordinates as
\begin{align*}
    \mathbf{T} = T^{i_1,\ldots,i_r}_{i_{r+1},\ldots,i_{r+s}} e_{i_1} \otimes \cdots \otimes e_{i_r} \otimes e^{i_{r+1}} \otimes \cdots \otimes e^{i_{r+s}},
\end{align*}
where $T^{i_1,\ldots,i_r}_{i_{r+1},\ldots,i_{r+s}} =  \mathbf{T} \left(e^{i_1}, \ldots, e^{i_r}, e_{i_{r+1}}, \ldots, e_{i_{r+s}}\right)$ are its components for $i_1, \ldots, i_{r+s} \in \{1, \ldots, n\}$. Note that summation over repeated indices is assumed. We denote the space of $(r,s)$-tensors over $V$ by $\mathcal{T}^{(r,s)}_V$.
\end{definition}

Given two tensors $\mathbf{S} \in \mathcal{T}^{(r,s)}_V$ and $\mathbf{T} \in \mathcal{T}^{(r',s')}_V$, we define the tensor product $\otimes$ between these two tensors as the new tensor $\mathbf{S} \otimes \mathbf{T} \in \mathcal{T}^{(r+r', s+s')}_V$ given by
\begin{align}\label{eq:tensor-products}
    \mathbf{S} \otimes \mathbf{T} := S^{i_1,\ldots,i_r}_{i_{r+1},\ldots,i_{r+s}} T^{j_1,\ldots,j_{r'}}_{j_{r'+1},\ldots,j_{r'+s'}} \bigotimes_{a=1}^r e_{i_a} \otimes \bigotimes_{b=1}^{r'} e_{j_b}
    \otimes \bigotimes_{c=r+1}^{r+s} e^{i_c} \otimes \bigotimes_{d=r'+1}^{r'+s'} e^{i_d},
\end{align}
where for simplicity, we employed the notations $\bigotimes_{j=1}^l e_{i_j} := e_{i_1} \otimes \cdots \otimes e_{i_l}$ and $\bigotimes_{j=1}^m e^{i_j} := e^{i_1} \otimes \cdots \otimes e^{i_m}$ for any integers $l, m \in \mathbb{N}$.

\subsection{Smooth sections of tensor bundles}\label{app:smooth sections}
We wish to introduce the notion of tensor field, which is a function that assigns a tensor to every point in a manifold (in the present case, in $\R^n$). Tensors fields help formalise various objects of physical origin mathematically, such as Eulerian flow field, vorticity, and density.

First, we need to introduce the notion of tangent space at a point. Let $x \in \R^n.$ The tangent space at $x,$ denoted by $\mathcal T_x \R^n,$ is an equivalence class of smooth curves $c : (-1, 1) \rightarrow \R^n$ with $c(0) = x$, where we say that two curves $c_1$ and $c_2$ are equivalent if and only if $Dc_1(0) = Dc_2(0)$. One can check that this forms an $n$-dimensional vector space over $\R$. We define the tangent bundle as the tuple $(\mathcal{T} {\mathbb R^n}, \pi)$, where $\mathcal T\R^n := \{(x, v) : x\in \R^n, \, v\in \mathcal T_x\R^n\}$ and $\pi$ is the natural projection map $\pi : \mathcal{T} {\mathbb R^n} \rightarrow \mathbb R^n$, which is defined by $\pi(x, v) = x,$ for all $(x, v) \in \mathcal{T} {\mathbb R^n}$. Let $x \in \R^n.$ Since $\mathcal T_x \R^n$ is a vector space, there exists a corresponding dual space, denoted $\mathcal T_x^* \R^n$ and referred to as the cotangent space. We analogously define the cotangent bundle as the tuple $(\mathcal{T}^* {\mathbb R^n}, \pi)$, where $\mathcal T^*\R^n := \{(x, \alpha) : x\in \R^n, \, \alpha \in \mathcal T_x^*\R^n\}$ and $\pi$ is the corresponding natural projection.

Generalising further, we define the 
$(r,s)$-tensor bundle on $\mathbb R^n$ as the tuple $(\mathcal{T}^{(r,s)} {\mathbb R^n}, \pi)$, where $\mathcal{T}^{(r,s)} {\mathbb R^n} := \{(x, v) : x \in \R^n, \, v \in \mathcal T_{\mathcal T_x \R^n}^{(r,s)}\}$ and $\pi$ is the corresponding natural projection. For every point $x \in \mathbb R^n$, we have $\pi^{-1}(x) = \mathcal T_{\mathcal T_x \R^n}^{(r,s)}$, which we call the fibre of the bundle at $x$ and we will denote more conveniently by $\mathcal{T}_x^{(r,s)}\R^n$. We will often refer to $\mathcal{T}^{(r,s)}{\mathbb R^n}$ itself as the tensor bundle, omitting the projection map. Note that setting $(r,s) = (1,0)$, we recover the tangent bundle $\mathcal{T}\mathbb R^n,$ and setting $(r,s) = (0,1),$ we recover the cotangent bundle $\mathcal{T}^*\mathbb R^n$. A tensor field is defined as a ``section" of the tensor bundle, a concept which we introduce in the next paragraph.

A smooth section $s$ of the tensor bundle $(\mathcal{T}^{(r,s)}{\mathbb R^n}, \pi)$ is a smooth map $s : \mathbb R^n \rightarrow \mathcal{T}^{(r,s)}{\mathbb R^n}$ such that $\pi \circ s : \mathbb R^n \rightarrow \mathbb R^n$ is the identity map.
We define a smooth tensor field to be a smooth section of the tensor bundle. The space of smooth tensor fields will be denoted by $C^\infty(\R^n,\mathcal{T}^{(r,s)}{\mathbb R^n})$, or alternatively $\Gamma(\mathcal{T}^{(r,s)}{\mathbb R^n})$, as commonly done in the differential geometry literature. As special cases, a smooth section of the tangent bundle is called a smooth vector field and a smooth section of the cotangent bundle is called a smooth differential one-form. As standardly done in the differential geometry literature, we denote the set of all smooth vector fields by $\mathfrak{X}(\R^n)$ and the set of all smooth differential one-forms by $\Omega^1(\R^n)$.

\begin{remark}\label{app:remark on trivialisation}
In the Euclidean case, given $x \in \R^n$, the tangent space $\mathcal T_x\R^n$  can be canonically identified with the tangent space at zero $\mathcal T_0 \R^n$ via the following construction. Define the translation operator $L_x :\R^n \rightarrow \R^n$ by $L_xx' = x'-x,$ for $x' \in \R^n$. Given a basis $\{[c_1], \ldots, [c_n]\}$ for $\mathcal T_x\R^n$, we can construct a basis for $\mathcal T_0\R^n$ by  $\{[L_x c_1], \ldots, [L_x c_n]\}$. This defines an isomorphism $\mathcal T_xL_x : \mathcal T_x\R^n \rightarrow \mathcal T_0\R^n$ via the relation $\mathcal T_xL_x(\alpha_1[c_1] + \cdots + \alpha_n[c_n]) := \alpha_1[L_x c_1] + \cdots + \alpha_n[L_x c_n]$, for any $\alpha_1, \ldots, \alpha_n \in \R$. Thus, we can identify the tangent bundle $\mathcal T\R^n$ with the product space $\R^n \times \mathcal T_0\R^n$, which can further be identified with $\R^n \times \R^n$ upon fixing a basis for $\mathcal{T}_0\R^n$. Similarly, one can identify the cotangent bundle $\mathcal T^*\R^n$ with $\R^n \times \R^n$ via the adjoint of $\mathcal T_xL_x$ with respect to the natural pairing, namely, $\mathcal T^*_xL_x : \mathcal T_0^*\R^n \rightarrow \mathcal T_x^*\R^n$, and by fixing a dual basis for $\mathcal{T}_0^*\R^n$. Putting the previous facts together, we can identify the tensor bundle $\mathcal{T}^{(r,s)} \mathbb R^n$ with the direct product space $\R^n \times \mathcal{T}^{(r,s)}_{\mathbb R^n}$. 
\end{remark}

\begin{remark}\label{app:remark on canonical basis}
There is a canonical basis for $\mathcal T_0\R^n$ due to the presence of the standard coordinate basis $e_1 = (1, 0, 0 \ldots, 0)$, $e_2 = (0, 1, 0, \ldots, 0),$ etc, on $\R^n$. This is given by $\{[c_1], \ldots, [c_n]\}$, where $c_1(t) = (t, 0, 0 \ldots, 0)$, $c_2(t) = (0, t, 0 \ldots, 0)$, etc. By Remark \ref{app:remark on trivialisation} above, this also generates a canonical basis $\{[L_{-x}c_1], \ldots, [L_{-x}c_n]\}$ for $\mathcal T_x\R^n,$ for any $x \in \R^n$.
\end{remark}

\begin{remark}\label{app:interpretation of tensor fields}
Owing to Remark \ref{app:remark on trivialisation} above, a tensor field on $\R^n$ can be interpreted as a tensor-valued map $\mathbf{T}: \mathbb{R}^n \rightarrow \mathcal{T}^{(r,s)}_{\mathbb R^n}$, i.e., a function that assigns a tensor to every point in $\R^n$. In the cases $(r,s)=(1,0)$ and $(r,s)=(0,1)$, the tensor fields simply become vector-valued maps $\R^n \rightarrow \R^n$. 
\end{remark}

Given an $(r,s)$-tensor field $\mathbf{T}$ and an $(s,r)$-tensor field $\mathbf{S}$, we can define a natural pairing at each point $x \in \R^n$, denoted by angle brackets $\left<\cdot,\cdot\right> : \mathcal{T}^{(r,s)}_x\R^n \times \mathcal{T}^{(s,r)}_x\R^n \rightarrow \R$, by
\begin{align*}
    \left<\mathbf{T}(x),\mathbf{S}(x)\right> := T^{i_1,\ldots,i_r}_{j_1,\ldots,j_s}(x) S^{j_1,\ldots,j_s}_{i_1,\ldots,i_r}(x),
\end{align*}
where sum over repeated indices is assumed. 

Throughout the rest of this section, we assume sections to always be smooth.

\subsection{Raising/lowering indices and the bundle metric}
Consider a symmetric, positive definite, nondegenerate bilinear form $\gamma : V \times V \rightarrow \mathbb R$, which is a $(0,2)$-tensor whose components are given by $\gamma_{ij} = \gamma(e_i,e_j)$. Fixing a vector $v \in V$, one can construct an associated covector $v^\flat \in V^*$ through the relation $v^\flat(w) = \gamma(v,w),$ for any $w \in V$. Let $\gamma^\flat : V \rightarrow V^*$ be the map $\gamma^\flat :v \rightarrow v^\flat,$ for $v \in V$. This procedure is called ``lowering the index", since the components of $v^\flat$ now have lower indices ($(v^\flat)_i = \gamma_{ij}v^j$). We also introduce the inverse operation $\gamma^\sharp := (\gamma^\flat)^{-1} : V^* \rightarrow V$, which maps a covector $\alpha \in V^*$ to a vector $\alpha^\sharp \in V$ by $\alpha^\sharp = \gamma^\sharp (\alpha)$. This is called ``raising the index" for similar reasons. One can then define a $(2,0)$-tensor $\gamma^{-1}: V^* \times V^* \rightarrow \R$ by $\gamma^{-1}(\alpha, \beta) := \gamma(\alpha^\sharp,\beta^\sharp),$ whose components we simply denote by $\gamma^{ij} = \gamma^{-1}(e^i,e^j)$.

If $\mathbf{g}$ is a smooth $(0,2)$-tensor field such that for every $x \in \R^n$, $\mathbf{g}(x)$ defines a symmetric, positive definite, nondegenerate bilinear form on the tangent space $\mathcal{T}_x\R^n$, we say that $\mathbf{g}$ is a Riemannian metric or a metric tensor on $\mathcal{T}\R^n$. Applying the above procedure on each fibre, we can construct the cometric tensor $\mathbf{g}^{-1} \in \Gamma(\mathcal{T}^{(2,0)}\R^n)$, which defines a symmetric, positive definite, nondegenerate bilinear form on the cotangent spaces $\mathcal{T}_x^*\R^n$.
Combining the metric and cometric tensors, we can construct a metric on general tensor fields, defined as follows (see for instance \cite{bauer2010sobolev}).

\begin{definition}
Given a metric tensor $\mathbf{g}$ on $\mathcal{T}\R^n$, a bundle metric $\bar{\mathbf{g}}$ on $\mathcal{T}^{(r,s)}\R^n$ is defined by
\begin{align*}
\resizebox{1.0 \textwidth}{!}{
\begin{math}
    \bar{\mathbf{g}}_x(\mathbf{T}_x,\mathbf{S}_x) := g_{i_1 j_1}(x) \cdots g_{i_r j_r}(x) g^{i_{r+1} j_{r+1}}(x) \cdots g^{i_{r+s} j_{r+s}}(x) (T_x)^{i_1,\ldots,i_r}_{i_{r+1},\ldots,i_{r+s}} (S_x)^{j_1,\ldots,j_r}_{j_{r+1},\ldots,j_{r+s}},
\end{math}
}
\end{align*}
for $x \in \R^n$ and any $\mathbf{T}_x, \mathbf{S}_x \in \mathcal{T}_x^{(r,s)}\R^n$, where $g_{ij}$ are the components of the metric tensor $\mathbf{g}$ and $g^{ij}$ are the components of the cometric tensor $\mathbf{g}^{-1}$. For every $x \in \R^n$, this also equips each fibre $\mathcal{T}^{(r,s)}_x \R^n$ with a norm $\lvert \cdot\rvert_{x} : \mathcal{T}^{(r,s)}_x \R^n \rightarrow [0,\infty)$, defined by
\begin{align*}
    \lvert \mathbf{T}_x\rvert_{x} := \sqrt{\bar{\mathbf{g}}_x(\mathbf{T}_x,\mathbf{T}_x)}.
\end{align*}
\end{definition}
The Euclidean space is canonically equipped with the Euclidean metric tensor $g_{ij}(x) = \delta_{ij},$ for all $x \in \R^n$, where $\delta_{ij}$ is the Kronecker delta\footnote{This is operationally identical to the standard dot product employed in Euclidean geometry.}. In this case, the bundle metric simply becomes
\begin{align*}
    \bar{\mathbf{g}}_x(\mathbf{T}_x, \mathbf{S}_x) = (T_x)^{i_1,\ldots,i_r}_{j_{1},\ldots,j_{s}} (S_x)^{i_1,\ldots,i_r}_{j_{1},\ldots,j_{s}},
\end{align*}
where sum over repeated indices is assumed regardless of whether the indices are up or down. For simplicity, when the base metric is just the Euclidean metric, we denote the corresponding bundle metric by round brackets $(\cdot, \cdot)_x : \mathcal{T}^{(r,s)}_x\R^n \times \mathcal{T}^{(r,s)}_x\R^n \rightarrow \R^n$. Note that throughout the paper, this is the only metric tensor that we consider.

\subsection{Multi-vector fields and differential $k$-forms} \label{tensor-analysis}
In Section \ref{app:smooth sections}, we introduced the notion of smooth section of the tensor bundle and, in particular, defined (smooth) vector fields as the special case corresponding to $(r,s)=(1,0),$ and (smooth) differential one-forms as the special case corresponding to $(r,s)=(0,1)$. We now introduce an important generalisation of these two concepts, namely, (smooth) multi-vector fields and (smooth) differential $k$-forms.

We first introduce the exterior product $\wedge$, which is an operation $\wedge : \mathcal{T}^{(r,s)}_V \times \mathcal{T}^{(r',s')}_V \rightarrow \mathcal{T}^{(r+r',s+s')}_V$ defined by
$\mathbf{F} \wedge \mathbf{G} := \mathbf{F} \otimes \mathbf{G} - \mathbf{G} \otimes \mathbf{F}$, where the tensor product $\otimes$ was defined in \eqref{eq:tensor-products}. Given $x \in \R^n,$ we define $\bigwedge^k \mathcal{T}_x\R^n$ to be the space consisting of tensors of the form $v_1 \wedge v_2 \wedge \cdots \wedge v_k$, where $v_1, \ldots, v_k \in \mathcal{T}_x\R^n$. Likewise, we define $\bigwedge^k \mathcal{T}_x^*\R^n$ as the space consisting of tensors of the form $\alpha_1 \wedge \alpha_2 \wedge \cdots \wedge \alpha_k$, where $\alpha_1, \ldots, \alpha_k \in \mathcal{T}_x^*\R^n$. We define the alternating bundles $\bigwedge^k \mathcal{T}\R^n$ and $\bigwedge^k \mathcal{T}^*\R^n$ to be the tensor bundles whose fibres at point $x$ correspond to $\bigwedge^k \mathcal{T}_x\R^n$ and $\bigwedge^k \mathcal{T}^*_x\R^n$, respectively.

\noindent\textbf{Multi-vector fields.}
A smooth multi-vector field of degree $k$ (or simply, a $k$-vector field) $\mathbf{U}$ on $\mathbb R^n$ is a smooth $(k,0)$-tensor field such that for any permutation $\sigma \in S_k$ (here, $S_k$ denotes the permutation group) and $x \in \R^n$, we have
\begin{align*}
    \mathbf{U}(x)(\alpha_{\sigma(1)}, \ldots, \alpha_{\sigma(k)}) = \sgn(\sigma) \mathbf{U}(x)(\alpha_1, \ldots, \alpha_n), \text{ for all } \alpha_1, \ldots, \alpha_k \in \mathcal{T}^*_x \R^n,
\end{align*}
where $\sgn(\sigma)$ is the signature of the permutation $\sigma \in S_k$. One can check that smooth multi-vector fields are simply sections of $\bigwedge^k \mathcal{T}\R^n$.
We denote the space of smooth $k$-vector fields by $\X^k(\mathbb R^n)$.

\noindent\textbf{Differential $k$-forms.}
A smooth differential $k$-form (or simply, a $k$-form) $\mathbf{K}$ on $\R^n$ is a $(0,k)$-tensor field such that for any permutation $\sigma \in S_k$ and $x \in \R^n$,
\begin{align*}
    \mathbf{K}(x)(v_{\sigma(1)}, \ldots, v_{\sigma(k)}) = \sgn(\sigma) \mathbf{K}(x)(v_1, \ldots, v_n), \text{ for all } v_1, \ldots, v_k \in \mathcal{T}_x\R^n.
\end{align*}
Again, smooth differential $k$-forms are simply sections of $\bigwedge^k \mathcal{T}^*\R^n$.
We denote the space of smooth $k$-forms by $\Omega^k(\mathbb R^n)$.

As a standard convention in differential geometry, the canonical basis of $\mathcal T_x \R^n$ (i.e., the basis obtained by standard coordinates in $\R^n$, see Remark \ref{app:remark on canonical basis}) is denoted
$\{\partial/\partial{x_1}(x), \ldots, \partial/\partial{x_n}(x)\}$ and its corresponding dual basis is denoted by $\{\diff x_1(x), \ldots, \diff x_n(x)\}$. Note that these are related by $\left<\partial/\partial{x_i}(x), \diff x_j(x)\right> = \delta_{ij}$. Thus, in standard coordinates, a vector field $X$ and a one-form $\alpha$ can be expressed explicitly as
\begin{align*}
    X(x) = X^i(x) \frac{\partial}{\partial x_i}(x), \qquad \alpha(x) = \alpha_i(x) \diff x_i(x),
\end{align*}
respectively. Likewise, a $k$-vector field $\mathbf{U} \in \X^k(\R^n)$ and a $k$-form $\mathbf{K} \in \Omega^k(\R^n)$ can be expressed as
\begin{align*}
    \mathbf{U}(x) &= \sum_{i_1<\cdots<i_k}^n U^{i_1, \ldots, i_k}(x) \frac{\partial}{\partial x_{i_1}}(x) \wedge \ldots \wedge \frac{\partial}{\partial x_{i_k}}(x), \\ 
    \mathbf{K}(x) &= \sum_{i_1<\cdots<i_k}^n K_{i_1, \ldots, i_k}(x) \diff x_{i_1}(x) \wedge \ldots \wedge \diff x_{i_k}(x),
\end{align*}
respectively. Differential $k$-forms are especially important in differential geometry and physics since they allow to extend the notion of integration to general oriented manifolds. For a $k$-form $\mathbf{K} \in \Omega^k(\mathbb R^n)$, we define its integral over a $k$-dimensional submanifold $U \subseteq \mathbb R^n$ by
\begin{align*}
    \int_{U} \mathbf{K} := \sum_{i_1<\cdots<i_k}^n \int_{U} K_{i_1,\ldots,i_k}(x) \, \diff x_{i_1} \ldots \diff x_{i_k},
\end{align*}
where $\diff x_{i_1} \ldots \diff x_{i_k}$ is the Lebesgue measure on $\R^k$. This helps formalise the notions of circulation, flux, and density that arise in physics.

\subsection{Tensor calculus} \label{calculo}
Here, we review some elementary differential operations defined on smooth tensor fields.

\noindent\textbf{Pull-back and push-forward.}
Let $\varphi : \mathbb R^n \rightarrow \mathbb R^n$ be a smooth map. The push-forward $\varphi_* X$ of a vector field $X = X^i(x) \partial/\partial x_{i}(x)$ with respect to $\varphi$ is defined in coordinates as
\begin{align*}
    (\varphi_* X)(\varphi(x)) := X^i(x) \frac{\partial \varphi^j}{\partial x_i} (x) \frac{\partial}{\partial x_j}(\varphi(x)).
\end{align*}
The pull-back $\varphi^* \alpha$ of a one-form $\alpha = \alpha_i(x) \diff x_i(x)$ with respect to $\varphi$ is defined as
\begin{align*}
    (\varphi^* \alpha)(x) := \alpha_i(\varphi(x)) \frac{\partial \varphi^i}{\partial x_j}(x) \diff x_j(x).
\end{align*}
Furthermore, when $\varphi$ is a diffeomorphism, one can also define the pull-back of a vector field as $\varphi^* X := (\varphi^{-1})_* X,$ and the push-forward of a one-form as $\varphi_* \alpha := (\varphi^{-1})^* \alpha$.

In fact, this notion can be readily extended to arbitrary $(r,s)$-tensor fields. Indeed, for a smooth tensor field $\mathbf{T},$ we define its pull-back $\varphi^* \mathbf{T}$ with respect to a diffeomorphism $\varphi$ by
\begin{align*}
    &(\varphi^* \mathbf{T})(x)(\alpha^1(x),\ldots,\alpha^r(x),v_1(x),\ldots,v_s(x)) \\
    &\quad := \mathbf{T}(\varphi(x))(\varphi_*\alpha^1(\varphi(x)),\ldots,\varphi_*\alpha^r(\varphi(x)),\varphi_*v_1(\varphi(x)),\ldots,\varphi_*v_s(\varphi(x))),
\end{align*}
for any $\alpha^1,\ldots,\alpha^r \in \Omega^1(\mathbb R^n),$ and $v_1,\ldots,v_s \in \mathfrak X(\mathbb R^n)$.
This has the coordinate expression
\begin{align*}
    (\varphi^* \mathbf{T})^{i_1,\ldots,i_r}_{j_1,\ldots,j_s}(x) &= T^{p_1,\ldots,p_r}_{q_1,\ldots,q_s}(\varphi(x)) \frac{\partial \psi^{i_1}}{\partial x_{p_1}}(\varphi(x)) \cdots \frac{\partial \psi^{i_r}}{\partial x_{p_r}}(\varphi(x)) \frac{\partial \varphi^{q_1}}{\partial x_{j_1}}(x) \cdots \frac{\partial \varphi^{q_s}}{\partial x_{j_s}}(x),
\end{align*}
where $\psi := \varphi^{-1}$. Similarly, we can define the push-forward of an $(r,s)$-tensor field as $\varphi_* \mathbf{T} := (\varphi^{-1})^* \mathbf{T}$.

\noindent\textbf{Lie derivative.}
Let $\mathbf{T}$ be an arbitrary smooth $(r,s)$-tensor field and $X \in \mathfrak X(\mathbb R^n)$ a smooth vector field. Let $\varphi_t,$ $t \in [0,\infty)$ be the local flow of $X$, whose existence is guaranteed by the Cauchy-Lipschitz theorem. The Lie derivative of $\mathbf{T}$ with respect to $X$ is defined as
\begin{align*}
    \mathcal L_X \mathbf{T} := \lim_{t \rightarrow 0} \frac1t \left( \varphi_t^* \mathbf{T} - \mathbf{T}\right).
\end{align*}
The explicit formula (in coordinates) for the Lie derivative of a smooth tensor field $\mathbf{T}$ with respect to a smooth vector field $b \in \mathfrak X(\mathbb R^n)$ can be computed as
\begin{align} \label{Lieformula}
\begin{split}
    &(\mathcal L_b \mathbf{T})^{i_1,\ldots,i_r}_{j_1, \ldots, j_s}(x) = b^l(x)\frac{\partial T^{i_1,\ldots,i_r}_{j_1, \ldots, j_s}}{\partial x_l}(x) - T^{l,i_2\ldots,i_r}_{j_1,\ldots,j_s}(x) \frac{\partial b^{i_1}}{\partial x_l}(x) \\
    &\hspace{60pt} - \cdots - T^{i_1,\ldots,i_{r-1},l}_{j_1,\ldots,j_s}(x) \frac{\partial b^{i_r}}{\partial x_l}(x)  \\
    &\hspace{60pt} + T^{i_1\ldots,i_r}_{l,j_2\ldots,j_s}(x) \frac{\partial b^{l}}{\partial x_{j_1}}(x) + \cdots + T^{i_1,\ldots,i_r}_{j_1,\ldots,j_{s-1},l}(x) \frac{\partial b^l}{\partial x_{j_s}}(x),
\end{split}
\end{align}
for all $x \in \mathbb{R}^n$ and its corresponding $L^2$-adjoint operator applied to a smooth tensor $\mathbf{\Theta}$ is given by 

\begin{align}
\begin{split}
    \label{Lieformulaadjoint}
    &(\mathcal L^*_b \mathbf{\Theta})^{i_1,\ldots,i_r}_{j_1, \ldots, j_s}(x) = -\frac{\partial (b^l \Theta^{i_1,\ldots,i_r}_{j_1, \ldots, j_s})}{\partial x_l}(x)  - \Theta^{l, i_2, \ldots, i_r}_{j_1, \ldots, j_s}(x) \frac{\partial b^{l}}{\partial x_{i_1}}(x) \\
    &\hspace{60pt} - \cdots - \Theta^{i_1, \ldots, i_{r-1}, l}_{j_1, \ldots, j_s}(x) \frac{\partial b^{l}}{\partial x_{i_r}}(x) \\
    &\hspace{60pt} + \Theta^{i_1, \ldots, i_r}_{l, j_2, \ldots, j_s}(x) \frac{\partial b^{j_1}}{\partial x_{l}}(x) + \cdots + \Theta^{i_1, \ldots, i_r}_{j_1, \ldots, j_{s-1}, l}(x) \frac{\partial b^{j_s}}{\partial x_{l}}(x),
    \end{split}
\end{align}
for all $x \in \R^n$.

\section{Banach sections of tensor bundles} \label{back-banach}
Here, we introduce some Banach sections of the $(r,s)$-tensor bundle $\mathcal{T}^{(r,s)}\mathbb R^n$ that are used in this paper.
We only consider tensor fields on the Euclidean space $\mathbb R^n$ equipped with the standard Euclidean metric for simplicity; however, these notions can be extended to general Riemannian manifolds (see for instance \cite{scott1995Lp, dodziuk1981sobolev, bauer2010sobolev}). In Section \ref{tensor}, we introduced the space of smooth sections of the tensor bundle $\mathcal{T}^{(r,s)}\mathbb R^n$, denoted $\Gamma(\mathcal{T}^{(r,s)}\mathbb R^n)$. However, since in this paper we will work with sections of tensor bundles that are of lower regularity, we will employ the notation $\Gamma_{\text{alg}}(\mathcal{T}^{(r,s)}\mathbb R^n)$ to denote sections in the algebraic sense (i.e., maps $s : \mathbb R^n \rightarrow \mathcal{T}^{(r,s)}{\mathbb R^n}$ that satisfy $\pi \circ s(x) = x$ for all $x \in \R^n$ without the smoothness requirement), and $\Gamma_{\text{meas}}(\mathcal{T}^{(r,s)}\mathbb R^n)$ to denote algebraic sections that are measurable with respect to the Borel sigma-algebra.

\noindent\textbf{$L^p$-tensor fields.} For $1 \leq p < \infty$, we wish to define the $L^p$-norm of a measurable tensor field $\mathbf{T} \in \Gamma_{\text{meas}}(\mathcal{T}^{(r,s)}\mathbb R^n).$ For every $x \in \mathbb R^n$, one can induce a bundle metric on $\mathcal{T}^{(r,s)}\mathbb R^n$ by making use of the standard inner product on $\mathbb R^n$ (i.e. the Euclidean metric), which we denote by rounded brackets $( \cdot,\cdot)_x:\mathcal{T}^{(r,s)}_x \mathbb R^n \times \mathcal{T}^{(r,s)}_x\mathbb R^n \rightarrow \mathbb R$ (see Appendix \ref{tensor} for more details about the construction of a bundle metric). The corresponding associated norm is denoted by $\lvert \cdot\rvert_x:\mathcal{T}^{(r,s)}_x \mathbb R^n \rightarrow [0, \infty)$. For $1 \leq p < \infty$, we define the $L^p$-norm of $\mathbf{T} \in \Gamma_{\text{meas}}(\mathcal{T}^{(r,s)}\mathbb R^n)$ by
\begin{align*}
    \|\mathbf{T}\|_{L^p} := \left(\int_{\mathbb R^n} \lvert \mathbf{T}(x)
    \rvert_x^p \, \diff^n x\right)^{\frac1p},
\end{align*}
and for $p=\infty$, we define the supremum norm
\begin{align*}
    \|\mathbf{T}\|_{L^\infty} := \sup_{x \in \mathbb R^n} \lvert \mathbf{T}(x)\rvert_x.
\end{align*}
For $p \in [1,\infty]$, the $L^p$-section of $\mathcal{T}^{(r,s)} \mathbb R^n$, denoted by $L^p(\mathbb R^n, \mathcal{T}^{(r,s)}\mathbb R^n),$ is an equivalence class of tensor fields $\mathbf{T}$ with $\|\mathbf{T}\|_{L^p} < \infty,$ where two tensor fields $\mathbf{T}_1,\mathbf{T}_2$ are identified if and only if $\|\mathbf{T}_1-\mathbf{T}_2\|_{L^p} = 0$.

Let $\mathcal{B}(\mathbb R^n)$ be the Borel algebra. Given $U \in \mathcal{B}(\mathbb R^n)$ with compact closure, we define the $L^p$-norm of $\mathbf{T}$ on $U$ by
\begin{align*}
    \|\mathbf{T}\|_{L^p(U,\mathcal{T}^{(r,s)} \mathbb R^n)} := \left(\int_U \lvert \mathbf{T}(x) \rvert_x^p \,\diff^n x\right)^{\frac1p} =  \left(\int_{\mathbb{R}^n} \mathcal{X}_U \lvert \mathbf{T}(x) \rvert_x^p \,\diff^n x\right)^{\frac1p},
\end{align*}
for $1 \leq p < \infty,$ and
\begin{align}
    \|\mathbf{T}\|_{L^\infty(U,\mathcal{T}^{(r,s)} \mathbb R^n)} := \sup_{x \in U} \lvert \mathbf{T}(x) \rvert_x,
\end{align}
for $p = \infty$.
If $\|\mathbf{T}\|_{L^p(U,\mathcal{T}^{(r,s)} \mathbb R^n)} < \infty$ for any $U \in \mathcal{B}(\mathbb R^n)$ with compact closure, then we say that $\mathbf{T}$ is of class $L^p_{loc}(\R^n,\mathcal{T}^{(r,s)} \mathbb R^n)$. For convenience, we will denote by $\|\cdot\|_{L^p_R}$ the local $L^p$-norm on the open ball $B(0,R)$.

In the special case $p=2$, the space $L^2(\R^n,\mathcal{T}^{(r,s)}\mathbb R^n)$ is furthermore equipped with an inner product $\llangle \cdot,\cdot \rrangle_{L^2} : L^2(\R^n,\mathcal{T}^{(r,s)}\mathbb R^n) \times L^2(\R^n,\mathcal{T}^{(r,s)}\mathbb R^n) \rightarrow \R$, defined by
\begin{align*}
    \llangle \mathbf{F},\mathbf{G} \rrangle_{L^2} := \int_{\mathbb R^n} (\mathbf{F}(x), \mathbf{G}(x))_x \,\diff^n x,
\end{align*}
for any $\mathbf{F},\mathbf{G} \in L^2(\R^n,\mathcal{T}^{(r,s)} \mathbb R^n)$, making it a Hilbert space. More generally, for $\mathbf{F} \in L^q(\R^n,\mathcal{T}^{(r,s)}\mathbb R^n)$ and $\mathbf{G} \in L^p(\R^n,\mathcal{T}^{(r,s)}\mathbb R^n)$, where $q,p \in [1,\infty]$ such that $1/q + 1/p = 1$,  we have
\begin{align*}
    \llangle \mathbf{F},\mathbf{G} \rrangle_{L^2} &= \int_{\mathbb R^n} (\mathbf{F}(x), \mathbf{G}(x))_x \,\diff^n x \leq \int_{\mathbb R^n} \lvert \mathbf{F}(x)\rvert_x \lvert \mathbf{G}(x)\rvert_x \,\diff^n x \\
    &\leq \|\mathbf{F}\|_{L^q}\|\mathbf{G}\|_{L^p} < \infty,
\end{align*}
where we have made use of the classical Cauchy-Schwartz and H\"older's inequalities.

\noindent\textbf{$C^k$-tensor fields.}
Let $k \in \mathbb{N} \cup \{0\}.$ If every component $T^{i_1,\ldots,i_r}_{i_{r+1},\ldots,i_{r+s}}$ of an $(r,s)$-tensor field $\mathbf{T}$ is $k$-times differentiable, then we say that $\mathbf{T}$ is a $C^k$-section of the tensor bundle, denoted by $C^k\left(\R^n,\mathcal{T}^{(r,s)}\mathbb R^n\right)$. We note that this {\em is not a Banach space}, but a Fréchet space. Given an $(r,s)$-tensor field $\mathbf{T}$ of class $C^1$, we define its derivative to be the $(r,s+1)$-tensor field
\begin{align*}
\resizebox{1. \textwidth}{!}{
\begin{math}
    \nabla \mathbf{T}(x) := \frac{\partial T^{i_1,\ldots,i_r}_{i_{r+1},\ldots,i_{r+s}}}{\partial x_l}(x) \frac{\partial}{\partial x_{i_1}}(x) \otimes \cdots \otimes \frac{\partial}{\partial x_{i_r}}(x) \otimes \diff x_l \otimes \diff x_{i_{r+1}}(x) \otimes \cdots \otimes \diff x_{i_{r+s}}(x).
\end{math}
}
\end{align*}
Analogously, we can define the $k$-th derivative $\nabla^k \mathbf{T}$ for $C^k$-tensor fields $\mathbf{T}$, which is an $(r,s+k)$-tensor. Following standard conventions, we also use the notation $D$ for the derivative $\nabla$ (e.g. $D\mathbf{T}$, $D^2\mathbf{T}$). We denote by $C^k_0 (\R^n,\mathcal{T}^{(r,s)} \mathbb R^n )$ the class of $C^k$ tensor fields with compact support, which is a Banach space, equipped with the norm
\begin{align*}
    \|\mathbf{T}\|_{C_0^k} := \|\mathbf{T}\|_{L^\infty} + \sum_{m=1}^k \|\nabla^m \mathbf{T}\|_{L^\infty}
\end{align*}

We note that all the operators we defined in Section \ref{calculo} can be defined for tensors of class $C^1 (\R^n,\mathcal{T}^{(r,s)} \mathbb R^n ).$ In particular, the coordinate expressions for the Lie derivative and its $L^2$-adjoint in the smooth case \eqref{Lieformula}, \eqref{Lieformulaadjoint} remain the same.

\noindent\textbf{$W^{k,p}$-tensor fields.}
One can also define Sobolev sections of tensor fields in the same way one defines the usual Sobolev spaces. First we provide the notion of weak derivatives of tensor fields.
\begin{definition}
Given $\mathbf{T} \in L^1_{loc} (\R^n,\mathcal{T}^{(r,s)}\mathbb R^n)$, we say that $\mathbf{S} \in L^1_{loc} (\R^n,\mathcal{T}^{(r,s+1)} \mathbb R^n)$ is a weak derivative of $\mathbf{T}$ if it satisfies
$\llangle \mathbf{S}, \mathbf{\Theta} \rrangle_{L^2} = \llangle \mathbf{T}, \nabla^T \mathbf{\Theta} \rrangle_{L^2}$ for any test tensor field $\mathbf{\Theta} \in C^\infty_0 (\R^n,\mathcal{T}^{(r,s+1)}\mathbb R^n)$, where $\nabla^T$ is the $L^2$-adjoint operator of $\nabla : C^\infty_0 (\R^n,\mathcal{T}^{(r,s)}\mathbb R^n) \rightarrow C^\infty_0 (\R^n,\mathcal{T}^{(r,s+1)}\mathbb R^n)$, which in coordinates reads
\begin{align*}
\resizebox{1. \textwidth}{!}{
\begin{math}
  \nabla^T \mathbf{\Theta}(x) := -\frac{\partial \Theta^{i_1,\ldots,i_r}_{l,i_{r+1},\ldots,i_{r+s}}}{\partial x_l}(x) \frac{\partial}{\partial x_{i_1}}(x) \otimes \cdots \otimes \frac{\partial}{\partial x_{i_r}}(x) \otimes \diff x_{i_{r+1}}(x) \otimes \cdots \otimes \diff x_{i_{r+s}}(x),
\end{math}
}
\end{align*}
for all $x \in \R^n$. We denote the weak derivative $\mathbf{S}$ of a tensor field by $\nabla_w \mathbf{T}$, where ``$w$" is usually omitted if it is understood from the context.
\end{definition}

Weak notions of $k$-th derivatives are defined in a similar manner and we define the Sobolev norm $\|\cdot\|_{W^{k,p}}$ of a tensor field, for $k \in \mathbb{N} \cup \{0\}$ and $p \in [1,\infty],$ by
\begin{align*}
    \|\mathbf{T}\|_{W^{k,p}} := \|\mathbf{T}\|_{L^p} + \sum_{m=1}^k \|\nabla_w^m \mathbf{T}\|_{L^p}.
\end{align*}
We say that $\mathbf{T}$ is in the Sobolev section $W^{k,p}(\R^n,\mathcal{T}^{(r,s)}\mathbb R^n)$ if all of the weak derivatives $\nabla_w \mathbf{T}, \ldots, \nabla_w^k \mathbf{T}$ exist and $\|\mathbf{T}\|_{W^{k,p}} < \infty$. As usually, we identify two elements $\mathbf{T}_1, \mathbf{T}_2$ if and only if $\|\mathbf{T}_1-\mathbf{T}_2\|_{W^{k,p}} = 0$. Furthermore, if the section $\mathbf{T}_U := \mathcal{X}_{U} \mathbf{T} \in \Gamma_{\text{meas}}(\mathcal{T}^{(r,s)} \mathbb{R}^n ))$ is of class $W^{k,p}$ for every $U \in \mathcal{B}(\mathbb R^n)$ with compact closure, then we say that $\mathbf{T}$ is of class $W^{k,p}_{loc} (\R^n,\mathcal{T}^{(r,s)} \mathbb{R}^n )$.

\noindent\textbf{The vector field case.}
In the case $(r,s)=(1,0)$ a tensor field $\mathbf{T} \in \mathcal{T}^{(r,s)} \mathbb{R}^n$ is a vector field and hence can be interpreted as a function $b: \mathbb{R}^n \rightarrow \mathbb{R}^n$ (see Remark \ref{app:interpretation of tensor fields}). In this case, for the Banach classes of vector fields we will employ the notations $L^p(\mathbb{R}^n,\mathbb{R}^n),$ $W^{k,p}(\mathbb{R}^n,\mathbb{R}^n),$ $C^{k}(\mathbb{R}^n,\mathbb{R}^n),$ etc, instead of the geometric notations $L^p(\R^n,\mathcal{T}\mathbb{R}^n),$ $W^{k,p}(\R^n,\mathcal{T}\mathbb{R}^n),$ $C^{k}(\R^n,\mathcal{T}\mathbb{R}^n),$ etc. The reason behind this choice is making this article more accessible to readers who are less familiar with differential geometry.

\noindent\textbf{Distribution-valued tensor fields.}
Given the space $C^\infty_0(\R^n,\mathcal{T}^{(r,s)}\mathbb R^n)$ of smooth $(r,s)$-tensor fields with compact support, a distribution-valued tensor field $\mathbf{T}$ is a continuous (with respect to the LF topology) linear functional $\mathbf{T} :C^\infty_0(\R^n,\mathcal{T}^{(r,s)}\mathbb R^n) \rightarrow \mathbb R$. We denote the space of all distribution-valued tensor fields by $\D'(\R^n,\mathcal{T}^{(r,s)}\mathbb R^n)$, which is equipped with the weak-$*$ topology.

\begin{definition}[Mollification of tensor fields] \label{molliformula}
Consider a tensor field $\mathbf{T} \in L^1_{loc} (\R^n,\mathcal{T}^{(r,s)}\mathbb R^n)$ and a mollifier $\rho^\epsilon \in C^\infty_0(\mathbb R^n).$ We define the mollified tensor field $\mathbf{T}^\epsilon$ as the smooth $(r,s)$-tensor field satisfying
\begin{align}
\begin{split}
    &\mathbf{T}^\epsilon(x)(\alpha^1(x), \ldots ,\alpha^r(x), v_1(x), \ldots, v_s(x)) \\
    &\quad = \int_{\mathbb R^n} \rho^\epsilon(x-y) \mathbf{T}(y)(\tau_x^y\alpha^1(x), \ldots ,\tau_x^y\alpha^r(x), \tau_x^y v_1(x), \ldots, \tau_x^y v_s(x)) \,\diff^n y,
\end{split}\label{eq:mollification-of-tensors}
\end{align}
for any $\alpha^i \in \Gamma_{\text{alg}}(\mathcal{T}^* \R^n)$, $i=1, \ldots, r,$ $v_j \in \Gamma_{\text{alg}}(\mathcal{T} \R^n),$ $j=1,\ldots,s,$ where $\tau_x^y$ denotes the parallel transport from $x$ to $y$ along the geodesic with respect to the metric.
\end{definition}

When the metric is the Euclidean one, identifying vector fields and one-forms simply as maps $\R^n \rightarrow \R^n$ (see Remark \ref{app:interpretation of tensor fields}), we have $\tau_x^y v(x) = v(x) \in \R^n$ and $\tau_x^y \alpha(x) = \alpha(x) \in \R^n$ (in this case, the right hand side of \eqref{eq:mollification-of-tensors} makes sense by interpreting $\mathbf{T}$ as a map $\R^n \rightarrow \mathcal{T}_{\R^n}^{(r,s)}$). We employ the standard convolution notation $\mathbf{T}^\epsilon := \rho^\epsilon * \mathbf{T}$ to denote the mollification of a tensor.

\begin{definition}[Mollification of distribution-valued tensor fields] \label{distrimolliformula}
Given a distribution-valued tensor field $\mathbf{T} \in \D' (\R^n,\mathcal{T}^{(r,s)}\mathbb R^n)$ and a mollifier $\rho^\epsilon \in C^\infty_0(\mathbb R^n)$, we define the mollified tensor field $\mathbf{T}^\epsilon$ as the smooth $(r,s)$-tensor field satisfying
\begin{align*}
    \llangle \mathbf{T}^\epsilon, \mathbf{\Theta} \rrangle_{L^2} :=  \mathbf{T}( \rho^\epsilon * \mathbf{\Theta}),
\end{align*}
for any test tensor field $\mathbf{\Theta} \in C_0^\infty(\R^n,\mathcal{T}^{(r,s)} \mathbb R^n)$.
\end{definition}

\end{appendices}

\bibliography{biblio}
\nocite{label}

\end{document}